\documentclass[10pt]{article}

\usepackage[utf8]{inputenc}%       encodage du fichier source
\usepackage[T1]{fontenc}%          gestion des accents (pour les pdf)
\usepackage{textcomp}%             caractères additionnels
\usepackage{amsmath,amssymb,amsthm}% packages ams
\usepackage{lmodern}%              remplacer éventuellement par txfonts, fourier, etc.
\usepackage[a4paper]{geometry}%    taille correcte du papier
\usepackage{microtype}%            améliorations typographiques
\usepackage{caption}
\usepackage{listings}
\usepackage{moreverb}
\usepackage{mathrsfs}
\usepackage[sans]{dsfont}

\usepackage{setspace}
\usepackage[font=sf, labelfont={sf,bf}, margin=1cm]{caption}

\usepackage{import}
\usepackage{xifthen}
\usepackage{pdfpages}
\usepackage{transparent}
\usepackage{chngcntr}
\counterwithin{figure}{section}
\usepackage{subcaption}
\usepackage{hyperref}%             gestion des hyperliens
\hypersetup{
    colorlinks=true,
    linkcolor=blue,
    citecolor=magenta,
    urlcolor=blue,
    pdfborder={0 0 0}
}

\newtheorem{theorem}{Theorem}[section]
\newtheorem{proposition}[theorem]{Proposition}
\newtheorem{lemma}[theorem]{Lemma}
\newtheorem{corollary}[theorem]{Corollary}

\theoremstyle{remark}
\newtheorem{remark}[theorem]{Remark}

\newcommand{\R}{\mathbb{R}}
\newcommand{\N}{\mathbb{N}}
\newcommand{\Z}{\mathbb{Z}}
\newcommand{\Q}{\mathbb{Q}}

\newcommand{\Fc}{\mathcal{F}}

\newcommand{\cN}{\mathcal{N}}

\newcommand{\Tc}{\mathcal{T}}
\newcommand{\cT}{\mathcal{T}}

\newcommand{\Tf}{\mathbf{T}}

\newcommand{\Expect}[1]{\mathbb{E} \left[ #1 \right] }
\newcommand{\EXPECT}[2]{\mathbb{E}_{#1} \left[ #2 \right] }

\newcommand{\Prob}[1]{\mathbb{P} \left( #1 \right) }
\newcommand{\PROB}[2]{\mathbb{P}_{#1} \left( #2 \right) }

\renewcommand{\P}{\mathbb{P}}
\newcommand{\E}{\mathbb{E}}

\newcommand{\norme}[1]{\left\| #1 \right\| }

\newcommand{\floor}[1]{\left\lfloor #1 \right\rfloor}
\newcommand{\ceil}[1]{\left\lceil #1 \right\rceil}
\newcommand{\indic}[1]{ \mathbf{1}_{ \left\{ #1 \right\} } }
\newcommand{\eps}{\varepsilon}

\renewcommand{\d}{\mathrm{d}}

\DeclareMathOperator{\brw}{BRW}
\makeatletter
\newcommand*\bigcdot{\mathpalette\bigcdot@{.5}}
\newcommand*\bigcdot@[2]{\mathbin{\vcenter{\hbox{\scalebox{#2}{$\m@th#1\bullet$}}}}}
\makeatother

\usepackage{enumitem}

\usepackage{accents}
\newlength{\dhatheight}

\newcommand{\bbE}{\mathbb E}

\newcommand{\bbP}{\mathbb P}

\newcommand{\us}{\mathsf{u}}

\title{Thick points of 4D critical branching Brownian motion}
\date{\today}
\author{Nathana\"{e}l Berestycki\thanks{Universit\"at Wien, Fakult\"at f\"ur Mathematik. Oskar Morgenstern Platz 1, Wien 1090, Austria. } \and Tom Hutchcroft\thanks{California Institute of Technology. 1200 E California Blvd, Pasadena CA 91125, USA.} \and Antoine Jego\thanks{\'{E}cole Polytechnique F\'{e}d\'{e}rale de Lausanne. Rte Cantonale, 1015 Lausanne, Switzerland.}}

\numberwithin{equation}{section}

\begin{document}

\newcommand{\gs}{\mathsf{g}}
\newcommand{\hs}{\mathsf{h}}
\newcommand{\vs}{\mathsf{v}}
\newcommand{\ws}{\mathsf{w}}

\maketitle

\begin{abstract}
We study the thick points of branching Brownian motion and branching random walk with a critical branching mechanism, focusing on the critical dimension $d = 4$. We determine the exponent governing the probability to hit a small ball with an exceptionally high number of pioneers, showing that this has a second-order transition between an exponential phase and a stretched-exponential phase at an explicit value ($a = 2$) of the thickness parameter~$a$. 
We apply the outputs of this analysis to prove that the associated set of thick points $\cT(a)$ has dimension $(4-a)_+$,
so that there is a change in behaviour at $a=4$ but not at $a = 2$ in this case. Along the way, we obtain related results for the nonpositive solutions of a boundary value problem associated to the semilinear PDE
$\Delta v = v^2$
and develop a strong coupling between tree-indexed random walk and tree-indexed Brownian motion that allows us to deduce analogues of some of our results in the discrete case.

We also obtain in each dimension $d\geq 1$ an infinite-order asymptotic expansion for the probability that critical branching Brownian motion hits a distant unit ball, finding that this expansion is convergent when $d\neq 4$ and divergent when $d=4$. This reveals a novel, dimension-dependent critical exponent governing the higher-order terms of the expansion, which we compute in every dimension.
\end{abstract}

\newgeometry{margin=1.2in}

\newpage

\tableofcontents

\newgeometry{margin=1in}

\newpage

\setstretch{1.1}

\section{Introduction}

\subsection{Overview}

The main goal of this paper is to show that branching Brownian motion in dimension four is governed by a nontrivial multifractal geometry, and to compute the associated exponents. These exponents are shown to coincide partly with those that arise in the universality class of logarithmically correlated processes, but also display a completely novel double phase transition. 

\medskip
Our motivation comes in part from constructive field theory. Indeed, it has been well known at least since the pioneering work of Glimm and Jaffe (see e.g. \cite{GlimmJaffe} and the many references therein for an overview), that constructing nontrivial quantum field theories (QFTs) in four dimensions is a highly challenging task. While there are a good number of examples that are believed to have nontrivial limits in dimensions $d<4$ (which can then often be understood via Wilson's celebrated renormalisation group \cite{Wilson,wilsonfisher}, at least nonrigorously), the marginal case of dimension $d=4$ is much more delicate.
% {\color{red}[TH: The $\phi^3$ model has upper-critical dimension $6$, so presumably one could construct a non-trivial 4d "QFT" from this model using Wilson-Fisher. I think this model might be disqualified from giving a "true" QFT because it isn't reflection positive, but I don't know that the same objection doesn't also apply to our model.]} 
Many natural random fields (defined either in the discrete or in the continuum with appropriate counterterms) end up in the scaling limit as simply Gaussian, that is to say, trivial from the point of view of QFT. Let us in particular mention the recent remarkable results of Aizenman and Duminil-Copin \cite{ADC} for a proof of triviality in the Ising and $\phi_4^4$ cases. 

%\medskip Of course, it is possible to \emph{start} from a logarithmically correlated Gaussian field such as the membrane field and study its associated Gaussian multiplicative chaos measure
% , and study its properties.
%(see e.g.  \cite{10.1063/5.0058389, schiavo2021conformally} for such an approach in every even dimension).
%which includes the definition of a class of Gaussian random fields termed \emph{co-polyharmonic} that are non-trivial and conformally invariant in every even dimension.
% (more generally on even-dimensional manifolds satisfying a certain geometric admissibility condition), which enjoy a property of conformal invariance. 

%In this paper we take what we hope are the first steps toward a different approach: Constructing a non-trivial four-dimensional QFT as the scaling limit of a 
% However, what we have in mind here is the derivation of such a QFT from a 
% microscopic model of statistical mechanics. 
% natural model of statistical mechanics.

\medskip 
While the criticality (and, ultimately, the triviality) of the above theories in four dimensions typically stems from the nature of self-intersections of Brownian and/or random walk trajectories in dimension four (see \cite{Lawler_intersections} for a thorough and masterful introduction to this topic), our construction will instead arise from the intersections of a \textbf{branching system of Brownian particles}, at criticality, and more specifically from the consideration of the set of \textbf{thick points} for the occupation measure of such processes. 
% In this article we take what we hope to be the first steps towards the construction of a certain nontrivial quantum field theory in four dimensions from a natural model of statistical mechanics. 
 %Although our present results are far from the full construction of a nontrivial four-dimensional QFT, they nevertheless give an indication that such a construction should indeed be possible. 
Notably, we compute the Hausdorff dimension of this set of thick points, which should be thought of as the natural ``support'' for the measure describing the desired QFT, and find the limiting maximal thickness. 
This dimension is obtained using a methodology inspired by the theory of \textbf{Gaussian multiplicative chaos} (\cite{Kahane, BerestyckiGMC}; see \cite{RhodesVargas_survey,BerestyckiPowell} for an overview), which played a crucial role in the recent rigorous formulation and development of Liouville conformal field theory (\cite{DKRV, DOZZ, DuplantierSheffield}; see also \cite{BerestyckiPowell, VargasNotes} for an introduction and review).

\medskip It is useful at this stage to make a comparison with the case of dimension two, for which the above branching systems of particles should be replaced with either a single Brownian trajectory or a collection of these trajectories such as a Brownian loop soup \cite{LawlerWerner}. In this context, our results on the dimension of thick points may be loosely compared with those of Dembo, Peres, Rosen and Zeitouni \cite{DPRZ}, who in particular computed the dimension of the set of thick points for Brownian motion. (In this analogy, our result on the maximal thickness then corresponds to the  Erd\H{o}s--Taylor conjecture \cite{erdos_taylor1960} as proven in \cite{DPRZ}.) 
Remarkable further developments have recently taken place in this two-dimensional theory.
%In dimension two, the theory has however advanced significantly recently. 
For instance, a more precise geometric description of the structure of the set of thick points
%The connection between this result and (two-dimensional) quantum field theories
%
was first provided by \cite{BBK}, \cite{jegoBMC, jegoBMCc} and \cite{AidekonHuShi2018}, where a measure supported by this set of thick points for a Brownian trajectory was constructed. (These measures correspond formally to the exponential of multiples of the square root of local times, which are not \textit{a priori} well-defined.) A natural extension of this construction to the Brownian loop soup was given in \cite{ABJL}. Furthermore at the critical intensity $\theta = 1/2$ this work established a rigorous connection between the set of thick points of the associated loop soup and the Liouville measure, or more precisely the hyperbolic cosine of the Gaussian free field, which is closely related to the hyperbolic (2D) sine-Gordon theory. Whether these recent developments have analogues in the 4D theory we develop in this paper is a highly interesting question.

\medskip %Turning back to our four-dimensional context, we obtain various results of interest on branching random walk, branching Brownian motion and related particle systems, independently of any reference to QFT. 
Let us now describe our results in slightly more details.
In short, these are:

\begin{itemize}[leftmargin=*]
    \item In Theorem \ref{T:tail_general} we obtain the sharp exponent for the probability that a small ball is hit by an exceptionally large number of particles, when (1) particles are killed upon reaching a large sphere, or (2) when particles are not killed and live in the infinite volume of $\R^4$. More precisely we study the asymptotic probability that, starting from a single particle at distance of order $R$, the unit ball is hit by at least $(a/2)(\log R)^2$ pioneers, where $a>0$ is the thickness parameter. In the first case, the tails are always exponential and the above probability decays like $R^{-2-a + o(1)}$. However, in the unrestricted case, we prove that there is a (second order) phase transition from exponential to stretched exponential tails at the thickness parameter $a=2$.

    We emphasise that, unlike previous results on similar questions \cite{angel2021tail,asselah2022time}, the associated exponent is identified exactly rather than estimated up to constants. In particular, the phase transition which we identify at $a=2$ is a completely new phenomenon that was not even conjectured in these previous works.

    \item We study \emph{nonpositive} solutions to the semilinear partial differential equation 
    $$
    \Delta v = v^2,
    $$
    with prescribed boundary conditions. In Theorem \ref{T:proba_representation} we give a probabilistic representation to some solutions of this boundary value problem as well as criteria for uniqueness. In contrast with the much more studied case of nonnegative solutions (studied in particular by Mselati \cite{Mselati} and Dynkin \cite{Dynkin, dynkin1991probabilistic}; see also the work of Le Gall \cite{LeGall_survey, le1999spatial}), we find that uniqueness \emph{does not hold} for generic non-positive boundary values. Fortunately, we are still able to guarantee uniqueness under certain ``smallness'' conditions that suffice for our applications to thick points. See Theorems \ref{T:uniqueness} and \ref{T:proba_representation} for more details.

    \item In Theorem \ref{T:thick} we compute the ``fractal dimension'' of the set of thick points $\Tc_R(a)$ for branching Brownian motion, where the unit ball is hit by at least $(a/2)(\log R)^2$ pioneers. (See \eqref{E:def_thick} for a precise definition.) More precisely, we show that 
    $\# \Tc_R(a) \cap \Z^4 = R^{4-a + o(1)}$ with high probability as $R\to \infty$ if $a \in (0,4]$ and that $\Tc_R(a)$ is empty with high probability as $R\to \infty$ if $a>4$. 
    In particular, surprisingly, the phase transition for the existence of thick points in a branching Brownian motion at $a=4$ does \emph{not} coincide with the phase transition for the one-point exponent associated with the probability to be thick (which occurs at $a=2$).
    
    \item The above results refer to a notion of thickness which is defined in terms of \emph{pioneers} (defined in \eqref{E:pioneers}). However, we can also obtain analogous results  for a notion of thickness defined instead in terms of the \emph{occupation measure} (Theorem \ref{T:thick_BRW}). These results about the occupation measure can then be transferred to discrete models of branching random walk via a novel \emph{strong coupling theorem} (Theorem \ref{T:coupling_general}) which relies on a connection to the \emph{Horton--Strahler number} of a tree and is of independent interest.
    
    In the case of branching random walk, however, our results on the thick points in the sense of occupation measure are restricted to a suboptimal interval $a\in [0, a_0]$. 
    Interestingly, the lower-bound (which comes from a truncated second moment computation) is valid for the entire range of thicknesses; it is instead the upper-bound (which is traditionally easier in problems of this nature) that causes difficulties and imposes a restriction on the range of thicknesses we can treat.

    \item Finally, we obtain an expansion, to arbitrary order, of the probability that a branching Brownian motion hits the unit ball starting from radius $r\to \infty$ (Theorems \ref{T:hitprob4} and \ref{T:hitting_neq4}), in any dimension $d \geq 1$. This expansion has very different flavours depending on whether $d \leq 3$, $d=4$ or $d \geq 5$. For instance, in all dimensions except $d=4$, this probability can be written as a \emph{convergent} power series in $r^\beta$, where $\beta$ is a novel universal exponent. This exponent $\beta$, which we compute exactly in all dimensions, takes the ``mean-field'' value $\beta = d-4$ when $d\ge 5$ but is irrational when $d = 2,3$. 
    In the critical case $d=4$, the relevant series expansion is instead a \emph{divergent} asymptotic expansion taking a very different and more complex form. Prior to this work, only the leading terms were known (see Section \ref{SS:hitting_proba} for more details and references); the more precise results we establish are important for our study of thick points in dimension four, where some key computations require us to know the first \emph{three} terms of the asymptotic series.
\end{itemize}

\noindent
An additional technical contribution which should also be of independent interest is a finitary Tauberian theorem (Theorem \ref{T:tauberian}) which can be used under very weak assumptions compared to e.g.\ the usual Hardy--Littlewood Tauberian theorems.

%Some of the above results carry (with appropriate modifications) to the much easier case of dimensions $d\ge 5$, see the precise results below. 

\begin{remark}
Most of our results should carry through (with appropriate modifications) to the setting of \emph{super-Brownian motion}, which is the scaling limit of critical branching random walk (and critical branching Brownian motion) conditioned to survive for a long time. In the interest of space we did not pursue this extension.
\end{remark}

\subsection{Setup}

We now introduce the relevant definitions needed to state our results precisely. We start with the definition of branching random walk. More generally, we start with the definition of a discrete branching process (or Bienaym\'e--Galton--Watson tree) and its associated tree-indexed random walk, before moving on to branching Brownian motion, which is the continuous-time analogue.

%\medskip Branching Brownian motion is a stochastic particle system in which each particle performs independent Brownian trajectories and is replaced after a gestation time by a random number of particles, each evolving according to independent Branching Brownian motions. In this paper, the gestation times will be random and exponentially distributed with parameter 1 (independent of everything else). The genealogy of the process is described by a Galton--Watson process. 

%For $x \in D$, denote by $\P_x$ the law of a critical branching Brownian motion in $\R^d$ starting with a single particle at time 0. The initial particle evolves according to a Brownian motion until an independent exponential random variable with mean 1 rings. At that time, the particle dies with probability $1/2$ or branches into 2 particles with probability $1/2$. The subsequent particles

%\bigskip

\paragraph{Bienaymé--Galton--Watson trees.}
Fix $\xi$ a probability measure on $\{0,1,2,\dots\}$.
A Bienaymé--Galton--Watson tree (BGW) tree $\Tf$ with offspring distribution $\xi$ is a random rooted tree  whose law is characterised by the following properties:
\begin{enumerate}
\item The number of children of the root is distributed according to $\xi$;

\item Conditionally on the root having $n$ children, each subtree rooted at the children $v_1, \dots, v_n$ of the root are independent Bienaym\'e--Galton--Watson trees with offspring distribution $\xi$.
\end{enumerate}
We will denote the root vertex of this tree by $\varnothing$.
As is well known, the long-time behaviour of BGW trees depends in an essential way on the offspring distribution $\xi$. Throughout this article we will assume that $m= \sum_n n \xi(n) = 1$ and $\xi(1) \neq 1$, which corresponds to the classical \textbf{critical case} for the BGW tree. (When $m <1$ the tree is subcritical, with a survival probability decaying exponentially with time, while when $m>1$ the probability does not decay to zero). In fact we will often for simplicity restrict our attention to the critical binary branching mechanism $\xi = \frac12 \delta_0 + \frac12 \delta_2$,
although most of our results should hold in the universality class of nondegenerate critical offspring distributions $\xi$ with, say, a finite exponential moment.

%Particularly relevant to the current article is the phase transition of the extinction of the tree. Assuming that $\xi$ is not the Dirac mass at 1 (which is a degenerate case where each generation consists exactly of one individual), it is well known that the tree is finite almost surely if and only if the expectation of $\xi$ is at most 1: $\sum_{n \geq 1} n \xi(n) \leq 1$.
%Distributions $\xi$ with expectation equal to 1 are called critical. In most of this article, we will focus for simplicity on the critical binary branching $\xi = \frac12 \delta_0 + \frac12 \delta_2$ (although most of our results should hold in the universality class of nondegenerate critical offspring distributions $\xi$ with finite variance).

\paragraph{Tree-indexed random walk.}
Let $T$ be a tree rooted at $\varnothing$, and denote by $E(T)$ and $V(T)$ its sets of edges and vertices respectively. Given  a probability measure $\theta$ on $\R^d$, the $T$-indexed random walk $(S_T(v), v \in V(T))$ with increment distribution $\theta$ can be defined as follows: Let $\{X_e: e \in E(T)\}$ be i.i.d. increments with distribution $\theta$. For each $v \in V(T)$, let $e_1, \dots, e_n$ be the edges of the unique path from the root to $v$ and set $S_T(v) = S_T(\varnothing) + \sum_{i=1}^n X_{e_i}$.
In most of the current article, we will focus on the nearest neighbour random walk in $\Z^d$, meaning that $\theta(x) = 1/(2d)$ for all $2d$ neighbours $x \in \Z^d$ of the origin.

\medskip For a BGW tree $\Tf$, the $\Tf$-index random walk is simply called \textbf{branching random walk}. We will denote by $\P_{x}^{\brw,\theta}$ the law of the branching random walk corresponding to the critical binary branching mechanism and jump distribution $\theta$, starting from $x\in \Z^d$ (i.e.\ with $S_\Tf(\varnothing) = x$). If the jump distribution $\theta$ is not specified then we simply mean the nearest neighbour random walk as above.

\medskip Given $R>0$ and $x \in \Z^d$, it will also be convenient to denote by $\P_{x,R}^{\brw, \theta}$ the law of branching random walk, starting from $x$, and killed when the particles leave $B(0, R)$, the (Euclidean) ball of radius $R>0$. Equivalently, $\P_{x,R}^{\brw, \theta}$ is the law of the restriction of $(S_\Tf(v), v \in \Tf)$ to the subtree $\Tf_R$ consisting of those vertices $v \in \Tf$ such that $S_\Tf(u) \in B(0,R)$ for every $u \preceq v$,
where $u \preceq v$ means that $u$ lies on the unique path between $\varnothing$ and $v$. 
That is,
$
\Tf_R$ is defined to be $\{ v \in \Tf: \forall u \preceq v, S_{\Tf} (u) \in B(0,R)\}
$
and $\P_{x, R}^{\brw, \theta}$ is the law of $(S_{\Tf} (v), v \in \Tf_R)$. The radius $R$ can also take the value $+\infty$ in which case particles are simply not killed.

\medskip

An important quantity in this paper will be the number of \textbf{pioneers} in a given ball: Given  a point $x \in \R^d$ and a radius $r>0$, the set of pioneers $\cN_{x,r}$ on the ball $B=B(x,r)$ is defined to be
\begin{equation}
    \label{E:pioneers}
\cN_{x, r} = \{ v\in \Tf:S_{\Tf} (v) \in B, \text{ but }  S_{\Tf} (u) \notin B, \forall u \preceq v, u \neq v\}.
\end{equation}
We also write $N_{x,r}=\#\cN_{x,r}$ for the \emph{number} of pioneers on the ball $B(x,r)$. That is, $N_{x,r}$ counts the number of particles that enter the ball $B(x,r)$ without having an ancestor belonging to the same ball.
When $x$ is the origin we write $B_r=B(0, r)$ and write $N_r = N_{0,r}$ for the associated number of pioneers.

%\begin{itemize}
%\item
%Branching random walk:  let $\P^\theta_{x,R}$ be the law under which $S_\Tf$ is a $\Tf$-indexed random walk with one initial particle at $x$, with increment distribution $\theta$ and where particles are killed as soon as they reach $B(R)^c$.
%\item
%Branching Brownian motion: let $\P_{x,R}$ be the law under which $B_\Tf$ is a $\Tf$-indexed Brownian motion with one initial particle at $x$, with exponential gestation times with parameter 1 and where particles are killed as soon as they reach $B(R)^c$.
%\end{itemize}

%Let $\Tf$ be a critical Galton--Watson tree with offspring distribution $\xi = \frac12 \delta_0 + \frac12 \delta_2$. 

%Note that the restriction of $B_\Tf$ to the branching times is a $\Tf$-indexed random walk with increment distribution $\theta_1$ where $\theta_1$ is the law of $\sqrt{E} N$ where $E$ is an exponential random variable with parameter 1 and $N$ is an independent standard four-dimensional Gaussian random variable.

\paragraph{Branching Brownian motion.} The definition of branching Brownian motion can be done in a similar manner, replacing the discrete branching process by a continuous-time branching process and tree-indexed random walk by tree-indexed Brownian motion. As the definition is similar we only highlight the differences. The continuous-time branching process is a real tree $\Tf$ in which particles give birth at constant rate 1 and are replaced by a random number of offspring with distribution given by $\xi$. In other words, particles live for a duration which is exponentially distributed with mean one, and lifetimes are independent for different particles. Thus $\Tf$ consists of a connected union of segments, each with random positive lengths, rather than of a countable set of vertices and edges.  
% 
% \medskip
% 
For each segment $e$ of $\Tf$ of length $t>0$, we associate to this segment a displacement $X_e$ which is a Brownian motion in $\R^d$ starting from zero and of duration $t>0$; the Brownian displacements $X_e$ are taken to be independent across different segments of $\Tf$. 
In this manner, we may associate to $\Tf$ and this collection of independent Brownian motions a set of positions $(B_{\Tf} (v), v \in \Tf)$ obtained by summing the Brownian displacements along the geodesic path joining $v$ to the root $\varnothing$. Note that $B_{\Tf} (v)$ is not defined only at branching times but also for all other points in the tree, which encode times between branching events. We denote by $\P_{x,R}$ the law of branching Brownian motion, starting from $x \in \R^d$, killed outside the ball of radius $R$ about the origin. When we consider rotationally invariant events, we may simply write $\P_{\norme{x},R}$ instead of $\P_{x,R}$.

\medskip

We may also associate to $(B_{\Tf} (v), v \in \Tf)$ a notion of pioneer points exactly in the analogous manner to \eqref{E:pioneers}, which we also denote by $\cN_{x,r}$. We also write $N_{x,r}=|\cN_{x,r}|$ and $N_r=|\cN_{0,r}|$. Although we use the same notation for these random variables in both cases of branching random walk and branching Brownian motion, which case we are referring to will be clear from context
% will be clear in each case, and 
together with the different notation for the two laws $\P_{x,R}$ and $\P_{x, R}^{\brw}$.

\medskip

We are now ready to give formal statements of our main results. Before doing so, let us make note of the following convention:

\vspace{0.27cm}
\hrule

\medskip 

\noindent \textbf{Notational convention:} \emph{Unless otherwise specified, the dimension is set to $d=4$ throughout the remainder of the article.}

\medskip
\hrule

\subsection{The optimal tail exponent}

 Our first result concerns critical branching Brownian motion and identifies explicitly the exponent associated to the probability that a ball is hit by an exceptionally large number of pioneers, conditional on hitting the ball. To this end, define
\begin{equation}
    \psi : a \in [0,\infty) \mapsto \inf_{t \geq 1} \left( 2(t-1) + \frac{a}{t} \right) = \left\{ \begin{array}{ll}
    a & \text{if~} a \leq 2, \\
    2 \sqrt{2a} - 2 & \text{if~} a \geq 2.
    \end{array} \right.
\end{equation}
Note that $\psi'$ is continuous but $\psi''$ is discontinuous at $a =2$, so that $\psi$ has a second order transition at $a=2$. See Figure \ref{fig-psi} for a plot. 
Recall that $N_1$ denotes the number of pioneers on the unit sphere.

\begin{theorem}[Exponents for the number of pioneers]
\label{T:tail_general}
Let $x_0 >0$ and consider a thickness level $a>0$. The  asymptotic estimates 
\begin{align}
\label{E:T_no_killing}
\hspace{-2cm}\text{Without~killing:}&&\hspace{-1cm} \PROB{e^{-x_0}R,\infty}{ \left. N_1 \ge \frac{a}{2} (\log R)^2 \right\vert N_1 > 0 } & = R^{-\psi(a) + o(1)}, \qquad \text{ and } \\
\label{E:T_killing}
\hspace{-2cm}\text{With~killing:}&& \hspace{-1cm}\PROB{e^{-x_0}R,R}{ \left. N_1 \ge \frac{a}{2} (\log R)^2 \right\vert N_1 > 0 } & = R^{-a + o(1)}
\end{align}
both hold as $R \to \infty$.
\end{theorem}

\begin{figure}
\centering
   \begin{subfigure}{.35\columnwidth}
    \def\svgwidth{\columnwidth}
   %% Creator: Inkscape 1.2.2 (b0a84865, 2022-12-01), www.inkscape.org
%% Accompanies image file 'plot-psi.pdf' (pdf, eps, ps)
\begingroup%
  \makeatletter%
  \providecommand\color[2][]{%
    \errmessage{(Inkscape) Color is used for the text in Inkscape, but the package 'color.sty' is not loaded}%
    \renewcommand\color[2][]{}%
  }%
  \providecommand\transparent[1]{%
    \errmessage{(Inkscape) Transparency is used (non-zero) for the text in Inkscape, but the package 'transparent.sty' is not loaded}%
    \renewcommand\transparent[1]{}%
  }%
  \providecommand\rotatebox[2]{#2}%
  \newcommand*\fsize{\dimexpr\f@size pt\relax}%
  \newcommand*\lineheight[1]{\fontsize{\fsize}{#1\fsize}\selectfont}%
  \ifx\svgwidth\undefined%
    \setlength{\unitlength}{572.95763219bp}%
    \ifx\svgscale\undefined%
      \relax%
    \else%
      \setlength{\unitlength}{\unitlength * \real{\svgscale}}%
    \fi%
  \else%
    \setlength{\unitlength}{\svgwidth}%
  \fi%
  \global\let\svgwidth\undefined%
  \global\let\svgscale\undefined%
  \makeatother%
  \begin{picture}(1,0.90102222)%
    \lineheight{1}%
    \setlength\tabcolsep{0pt}%
    \put(0,0){\includegraphics[width=\unitlength,page=1]{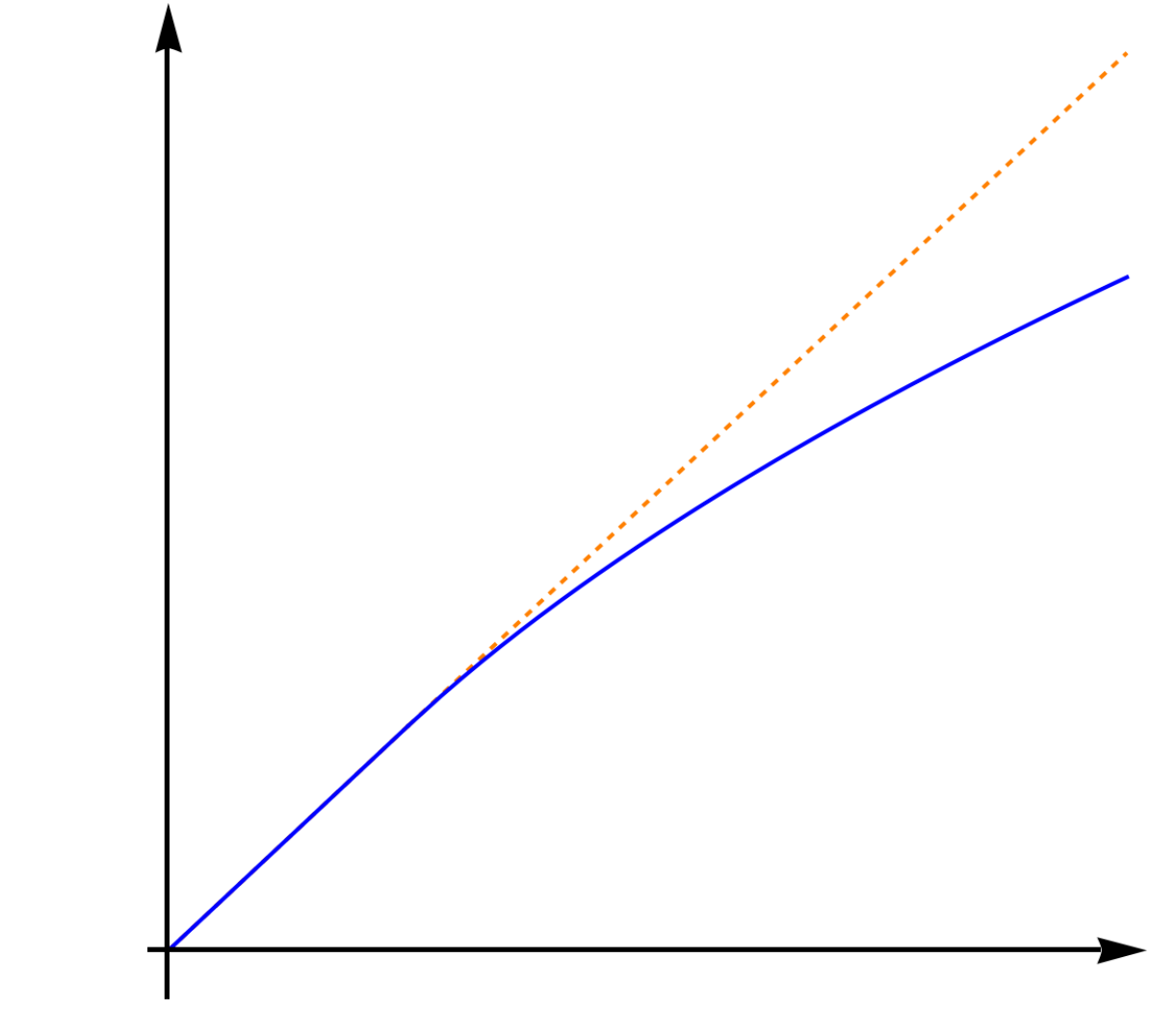}}%
    \put(-0.00211348,0.82923325){\color[rgb]{0,0,1}\makebox(0,0)[lt]{\lineheight{1.25}\smash{\begin{tabular}[t]{l}$\psi(a)$\\\end{tabular}}}}%
    \put(0.94601741,0.01947359){\color[rgb]{0,0,0}\makebox(0,0)[lt]{\lineheight{1.25}\smash{\begin{tabular}[t]{l}$a$\end{tabular}}}}%
    \put(0.37220395,0.00492237){\color[rgb]{0,0,0}\makebox(0,0)[lt]{\lineheight{1.25}\smash{\begin{tabular}[t]{l}$2$\end{tabular}}}}%
    \put(0,0){\includegraphics[width=\unitlength,page=2]{plot-psi.pdf}}%
  \end{picture}%
\endgroup%

   \end{subfigure}
\caption{Plot of $\psi$ in blue and of the identity function in dotted orange. The function $\psi$ is differentiable but not twice differentiable at $a=2$.}\label{fig-psi}
\end{figure}

We will also prove the following more general version version of the killed estimate \eqref{E:T_killing}, 
allowing us to consider pioneers on balls that are not necessarily unit balls.

\begin{theorem}
\label{T:maintail4}
Let $x_0 >0$ and $\eps>0$.
There exist constants $R_0=R_0(x_0,\eps)>0$ and $x_*=x_*(x_0,\eps) > x_0$ such that the estimate
$$
\biggl(\frac{r}{R}\biggr)^{(1 + \eps) a} \leq \mathbb{P}_{e^{-x_0}R,R}\biggl( N_{r} \ge \frac{a}{2}  \left(r \log \frac{R}{r}\right)^2  \bigg\vert \; N_{r} >0 \biggr) \leq \biggl(\frac{r}{R}\biggr)^{(1 - \eps) a}
% , \quad \quad a >0.
$$
holds for every $R \geq R_0$, $1 \leq r \leq e^{-x_*}R$, and $a>0$.
\end{theorem}

% \medskip

\paragraph{Context.} 
% We now recall some context concerning this result.
%We start by pointing out that the  analogous result for planar random walk is essentially trivial, since (if the walk visits a point) then its local time there will be exactly a geometric random variable. Of course, in that case it is necessary to kill a walk owing to recurrence, and the tails of the local time are similar to \eqref{E:T_killing}.
%
%\medskip 
% For super-Brownian motion, 
Le Gall and Merle \cite{LeGall-Merle} studied the occupation measure of the unit ball by a super-Brownian motion in four dimensions, started with one initial particle at distance $R$ and conditioned on hitting the unit ball. They proved that the total occupation measure, normalised by $\log R$, converges in distribution to an exponential random variable as $R \to \infty$. 
More recently, Angel, Hutchcroft and Jarai \cite{angel2021tail} and Asselah and Schapira \cite{asselah2022time} considered the analogous quantity for branching random walks, going beyond the distributional limit and into the large deviations regime.
% , which is the discrete analogue of the occupation measure of the unit ball.
% proving in particular that
% occupation measure instead of number of pioneers.
Translating their results into our continuous setting, they showed the existence of constants $0 <c_1 < c_2$ such that
\begin{equation}\label{T:AS_AHJ}
R^{-c_2 \min(a, \sqrt{a})} \leq
\PROB{e^{-x_0}R,\infty}{ \left. L(B(1)) \ge \frac{a m_1}{2} (\log R)^2 \right\vert N_1 > 0 }
\leq R^{-c_1 \min(a, \sqrt{a})},
\end{equation}
where $L(B(1))$ denotes the total occupation measure of the unit ball and $m_1 = \EXPECT{1,\infty}{L(B(1))}$ is a constant. 
This showed in particular that a naive extrapolation of Le Gall and Merle's result into the large deviations regime is incorrect. Still, it remained unclear precisely how the true tail probability would interpolate between the exponential tail for $a \ll 1$ and the stretched-exponential tail for $a \gg 1$, a question that is very relevant for understanding the structure of thick points. Theorem~\ref{T:tail_general}
 gives a very thorough answer to this question for branching Brownian motion\footnote{This result also concerns thick points defined in terms of pioneers rather than the occupation measure. As we discuss below and partly justify, we do not believe this to be an important difference.}, and suggests that one should recover the naive extrapolation of Le Gall and Merle's theorem into the large deviations regime \emph{under appropriate Dirichlet boundary conditions} (in reality, as we will see in Theorem \ref{T:laplace_sharp}, this is only true after an additional double-logarithmic shift in the mean of the exponential random variable).  The second-order phase transition between the ``linear'' and ``sublinear'' phases that occurs at $a=2$ is a completely new phenomenon, with the coarse estimates of \cite{angel2021tail,asselah2022time} being also consistent with a smooth interpolation between linear and sublinear behaviour.

\paragraph{Optimal strategies for large numbers of pioneers.} Our analysis also leads to probabilistic descriptions of   ``strategies'' the branching Brownian motion can take to have at least $\frac{a}{2}(\log R)^2$ pioneers on the unit sphere which are optimal up to an error of $R^{o(1)}$. This sharpens the analysis of \cite{asselah2022time} who gave similar strategies the process can take to have a large local time in the unit ball that achieve the correct probability up to a multiplicative constant in the exponent.
We now give a brief heuristic overview of these strategies, referring the reader to Section~\ref{S:strategy} for further detail.
% optimal strategies leading to the event of interest.

\medskip
First consider the case $a\leq 2$. In this case, Theorem~\ref{T:tail_general}
 implies that there is no major difference (i.e., no difference at the level of exponents) between the infinite-volume system and the system with a killing boundary on $\partial B(R)$. As part of our study of the structure of the set of thick points, we prove moreover that (again at the level of exponents), there is no distinction in this regime between the true behaviour of the model and the following simplified ``non-backtracking'' model. Let $\ell\geq 1$ and consider a sequence of intermediate scales $r_k=R^{1-k/\ell}$ with $k=1,\ldots,\ell$. Let $\mathcal{Z}_0$ be the single particle at the starting point $z$ at time zero and, for each $1\leq k\leq \ell$, let $\mathcal{Z}_k$ be the set of pioneers on the sphere $\partial B(0,r_k)$ that are descended from a particle in the set $\mathcal{Z}_{k-1}$ and for which the trajectory between these two particles remains inside the ball $B(0,r_{k-2})$. Thus, $|\mathcal{Z}_\ell|$ is a lower bound on $N_1$ (i.e., on the number of pioneers on the unit sphere), and counts only those pioneers that follow a reasonably simple  trajectory to reach the unit sphere. More specifically, these trajectories are ``non-backtracking on double-exponential scales''.
 It follows from the results of Section~\ref{S:strategy} that
 \[
 \P_{e^{-x_0}R,\infty}\left(|\mathcal{Z}_\ell| \geq \frac{a}{2}(\log R)^2  \text{ and } |\mathcal{Z}_k| \leq  \frac{a+\eps}{2}\left(r_k^2 \log \frac{R}{r_k}\right)^2 \text{ for every $1\leq k < \ell$} \; \bigg \vert \; N_1 > 0\right) \geq R^{-a+o(1)} 
 \]
 as $R\to \infty$ for each fixed $a,\eps>0$. When $a\in (0,2]$ it follows from Theorem~\ref{T:tail_general} that this probability is within an $R^{o(1)}$ factor of the probability that $N_1$ is at least $\frac{a}{2}(\log R)^2$, so that (at the level of exponents), this number of pioneers can be obtained from this simple non-backtracking strategy alone, with good control on the number of pioneers at each intermediate mesoscopic scale. This also holds for the model with killing on the outer sphere of radius $R$ for every $a>0$.
 % , at the level of exponents,
 % the process can have $\frac{a}{2}$
 
%  can moreover describe precisely the number of pioneers that will hit concentric spheres of intermediate mesoscopic scales [eq (7.1)], conditionally on $N_1 \geq a(\log R)^2/2$. Furthermore if $a<2$, we can guarantee that once a particle hits a sphere of such intermediate mesoscopic scale, say $R^t$ with $t \in (0,1)$, it is very likely that all of its descendants remain in the ball of radius 
% $$R^{\max (t, \sqrt{a/2}) + o(1)} \ll R.
% $$ 
% In particular, no ``macroscopic backtracking'' takes place. {\color{red}[TH: Phrase as optimal strategy? I guess the interpretation is that the walk accumulates local time in the "obvious" way, and that the tree doesn't do anything atypical other than survive for an appropriate amount of time?]}

\medskip Now consider the case $a>2$, with no killing. 
In this case, particles will first expand and reach a sphere of radius 
$R^{t+o(1)}$, where $t = \sqrt{a/2} >1$.
This occurs with probability $R^{-2t+o(1)}$.  Conditionally on this event, there will typically be $R^{2t+o(1)}$ pioneers on the sphere $\partial B(R^{t+o(1)})$. Particles will then stay in a slightly larger domain $B(R^{t+o(1)})$ and will generate the desired number of pioneers on the unit sphere by following the ``$a=2$ strategy'' started at this larger distance; this succeeds with probability $R^{-a/t +o(1)}$.
% (this corresponds to the ``$a=2$ strategy'' in the enlarged ball of radius $R^t$). 
Altogether, this scenario occurs with probability $R^{-2t-a/t+o(1)} = R^{-2-\psi(a)+o(1)}$. See Section~\ref{S:no_killing} and in particular Lemma \ref{L:june} for more details.

\paragraph{Analogous results for the local time.} Complementing Theorems \ref{T:tail_general} and \ref{T:maintail4} above, we also obtain analogous (but more partial) results concerning the behaviour of the total local time accumulated in a small ball, still for critical branching Brownian motion in $\R^4$. As mentioned above, we do not believe that there should be fundamental differences between measuring the thickness of a region via its occupation measure or via its number of pioneers; this is analogous to the fact that for a single Brownian trajectory, local time can be measured either in terms of occupation measure or in terms excursion counts between concentric circles (or spheres). We do however have to reparameterize by an appropriate constant factor to correctly compare between pioneers and local times.
For any point $x$ of the unit sphere, recall that
\begin{equation}
\label{E:m1}
m_1 = \EXPECT{1,\infty}{L(B(1))} = \int_{B(1)} \d y \int_0^\infty \d t ~p_t(x, y),
\quad \text{where} \quad
p_t(x,y) = \frac{1}{(2\pi t)^2} e^{-\frac{(x-y)^2}{2t} }
\end{equation}
is the heat kernel in $\R^4$. 
% We write the following theorem to also be applicable for balls that are not necessarily unit balls (see Theorem~\ref{T:maintail4} for a similar generalization of \eqref{E:T_killing}).
% {\color{red}[TH: Any reason we don't do this for pioneers here? We should at least mention Theorem 4.1.]}

\begin{theorem}[Exponents for the local time]
\label{T:local_time_tail}
There exists $a_0 > 0$ such that the following holds.
Let $x_0 >0$. For each $\eps >0$, there exists $x_* > x_0$ and $R_0>0$ such that
% large enough so that for all $R \geq R_0$, $,
\begin{align}
\label{E:T_local_lower}
\PROB{e^{-x_0}R, R}{ L(B(r)) \geq \frac{a m_1}{2} r^4 (\log R/r)^2 \vert N_r >0 } & \geq (R/r)^{-(1 + \eps) a},
% 
% \quad \quad a >0. 
\intertext{for every $R\geq R_0$, $r \leq e^{-x_*} R$ and $a>0$. If moreover $a \leq a_0$ then}
\label{E:T_local_upper}
\PROB{e^{-x_0}R, R}{ L(B(r)) \geq \frac{a m_1}{2} r^4 (\log R/r)^2 \vert N_r >0 } & \leq (R/r)^{-(1 - \eps) a}
\end{align}
for every $R\geq R_0$ and $r \leq e^{-x_*} R$.
\end{theorem}

\subsection{Pioneer thick points}

We now turn to our main result concerning thick points of critical branching Brownian motion in $\R^4$.
Let $z_0 \in B(0,R/2)$ be a starting point in the bulk of $B(R)$. 
For each thickness parameter $a>0$ we define the set of $a$-thick points to be
\begin{equation}
\label{E:def_thick}
\Tc_R(a) := \left\{ z \in  \Z^4: N_{z,1} \geq \frac{a}{2} (\log R)^2 \right\}.
\end{equation}
% {[\color{red} TH: Do we definitely want to require thick points be in $B(0,R)$? Of course it won't make a difference, but perhaps it's less natural.]}
(One could replace the unit lattice $\Z^4$ appearing in this definition with any rescaled lattice $\varepsilon \Z^4$ without changing any of our main results.)
We will also call these points $a$-pioneer-thick points if we want to distinguish them from thick points defined through the local time.

\begin{theorem}
\label{T:thick}
Let $\zeta$ denote the extinction time of the genealogical tree $\Tf$ and consider the conditional law $\P_{z_0,\infty} ( \,\cdot\, | \zeta > R^2)$ of the branching Brownian motion given that the branching process survives to time $R^2$, where $z_0\in B(0,R/2)$. As $R\to \infty$, the maximal thickness of a point converges in distribution to $2$, so that
% As $R\to \infty$, we have the distributional limits
% with particles killed on the boundary of the $R$-sphere, given that the entire genealogical tree 
\begin{equation}
\label{E:T_thick2}
\sup_{z \in \Z^4} \frac{N_{z,1}}{(\log R)^2} \to 2 \qquad \text{ in distribution as $R\to\infty$.}
\end{equation}
% in probability as $R \to \infty$. Moreover, for all $a \in (0,4)$, under the conditional law $\P_{z_0,R} ( \cdot | \zeta > R^2)$,
Moreover, for each $a\in (0,4]$, the number of thick points $\# \Tc_R(a)$ is equal to $ R^{4-a + o(1)}$ with high probability as $R\to \infty$, meaning more precisely that
\begin{equation}
\label{E:T_thick1}
\frac{ \log \# \Tc_R(a) }{\log R} \to 4-a \qquad \text{ in distribution as $R\to\infty$.}
% = R^{4-a + o(1)}
\end{equation}
% where $o(1) \to 0$ in probability as $R \to \infty$.
Moreover, both distributional limits are unchanged if we use the law $\P_{z_0, R} ( \,\cdot\, | \zeta> R^2)$, where particles are killed on the outer sphere of radius $R$, instead of $\P_{z_0, \infty} ( \,\cdot\, | \zeta> R^2)$.
\end{theorem}

The fact that the above results are unchanged when one adds or removes the killing boundary follows from the compactness of the support of branching Brownian motion. Indeed, under the measure $\P_{z_0, \infty} ( \cdot | \zeta> R^2)$, the probability that a particle reaches the sphere $\partial B(A R^2)$ converges to zero uniformly in $R$ as $A \to \infty$. Thus, we can \textit{a priori} add a killing boundary at a distance of order $R$ without changing the \emph{typical} behaviour of thick points. This means that it is the tail estimate \eqref{E:T_killing}, where the exponential behaviour persists into the entire large-deviations regime, that is relevant for the typical behaviour of thick points; the transition to stretched exponential behaviour in \eqref{E:T_no_killing} occurs due to \emph{global} atypical behaviours of the process that can be safely discarded in the context of Theorem~\ref{T:thick}.
Either way, whether we consider the number of thick points for a \emph{fixed, typical} realisation of branching Brownian motion as in Theorem \ref{T:thick}, or instead consider the probability  for one point to be thick with appropriate Dirichlet boundary conditions on a macroscopic sphere as in \eqref{E:T_killing} in Theorem \ref{T:tail_general} (i.e. we average over many realisations), we end up with a picture compatible with multiplicative chaos and log-correlated fields. 
(This contrasts with the boundaryless case  \eqref{E:T_no_killing}, where the nonlinear regime of $\psi(a)$ is incompatible with log-correlated fields.)

\medskip Again, we point out that Theorem \ref{T:thick} above may be viewed as an analogue of the celebrated result of Dembo, Peres, Rosen and Zeitouni \cite{DPRZ} for thick points of random walk and, in the case of \eqref{E:T_thick2}, as the analogue of their proof of the Erd\H{o}s--Taylor conjecture \cite{erdos_taylor1960}. We mention that there have been many extensions of this theorem (mainly in dimension two, which is the most interesting case for random walk, Brownian motion, and the Gaussian free field): see in particular \cite{bolthausen2001, daviaud2006, HuMillerPeres2010, BramsonDingZeitouni, BiskupLouidor_intermediate, BiskupLouidor_Extreme, rosen2006, BassRosen2007, jego2020, jegoRW, abe2022exceptional, abe2019exceptional}.

\medskip

Conceptually, the proof of Theorem \ref{T:thick} comes from a combination of the estimates obtained in Theorem \ref{T:tail_general} with a strategy inspired by Gaussian multiplicative chaos and in particular by \cite{BerestyckiGMC}. As in this and related works, this suggests it should be possible to construct a measure supported on the set of thick points and which should correspond informally to the exponential of the square root of the occupation measure, in a manner analogous to the recent works \cite{jegoBMC, AidekonHuShi2018, jegoRW, ABJL}, and \cite{jegoBMCc} in the critical case. We hope to return to this important question in future work. 

\begin{remark}
An analogous result holds for thick points defined in terms of local time for a ball. However, similarly to Theorem \ref{T:local_time_tail}, the analogue of the lower (resp. upper) bound \eqref{E:T_thick1} holds for any $a \in (0,4)$ (resp. $a \in (0,a_0)$). See also Theorem \ref{T:thick_BRW} in the discrete.
\end{remark}

\subsection{Hitting probabilities in every dimension}\label{SS:hitting_proba}
% To end our results in the branching Brownian motion case (see below for some results on the random walk case),
Our final main result on branching Brownian motion establishes very precise asymptotic estimates for the probability to hit a distant unit ball. We analyse these probabilities in every dimension, but
the expansions have very different flavours depending on whether $d \leq 3$, $d=4$ or $d \geq 5$ (the four-dimensional case being the most subtle).
%We will have to distinguish between three cases: high dimensions $d \geq 5$ (relatively easy case), low dimensions $d \in \{1,2,3\}$ (some work is needed) and critical dimension $d=4$ (significant work).
We start by recalling the leading order term of the hitting probability: as $r \to \infty$,
\begin{equation}
\label{E:hitting_leading_order}
\PROB{r,\infty}{N_1>0} \sim \left\{ \begin{array}{cc}
\hspace{0.26cm}(8-2d)r^{-2} & \text{if~} d \leq 3,\\
2  (\log r)^{-1} r^{-2}& \text{if~} d =4,\\
\hspace{0.84cm}c_d r^{-d+2} & \text{if~} d \geq 5.
\end{array} \right.
\end{equation}
(The leading constants $(8-2d)$, $2$, and $c_d$ depend on our choice of offspring distribution.)
These asymptotics were computed for super-Brownian motion by Dawson, Iscoe, and Perkins \cite{dawson1989super}, and for branching random walk in various works of Le Gall and Lin and Zhu; see \cite[Theorem 7]{10.1214/14-AOP947} for the case $d \leq 3$, \cite{Zhu3} and \cite{LeGall16} for the case $d =4$,  and \cite{LeGall16} and \cite{Zhu17} for the case $d \geq 5$.
% {\color{red}[TH: Should we make it clear that the constants $8-2d$ and $2$ are not always the same in these examples when one changes the offspring distribution etc?]}
(It seems that a full proof of this first-order estimate has not previously been written for branching Brownian motion, but this can be done very similarly to these other cases.) 

\medskip

In this section we go beyond these leading order estimates, establishing an infinite-order asymptotic expansion for the hitting probability in \emph{every} dimension. 
We start with the most interesting case of the critical dimension:

\begin{theorem}[Hitting probabilities in the critical dimension]
\label{T:hitprob4}
Let $d=4$.
There exists a constant $c_{\mathrm{hit}} \in \R$ such that
$$
\P_{r,\infty} (N_1>0) = \frac{2}{r^2 \log r } + \frac{2 \log \log r + c_{\mathrm{hit}}}{r^2 (\log r)^2} + o\!\left(\frac{1}{r^{2} (\log r)^{2}}\right),
\quad \quad \text{as~} r \to \infty.
$$
More generally, there exists a sequence of polynomials $(P_n)_{n \geq 1}$ such that 
\[
\PROB{r,\infty}{N_1>0} = \frac{1}{r^2} \sum_{n=1}^N \frac{P_n(\log \log r)}{(\log r)^n} + o\!\left( \frac{1}{(\log r)^{N}} \right) \quad \quad \text{as}~ r \to \infty \text{ for each $N \geq 1$}.
\]
The sequence of polynomials  $(P_n)_{n \geq 1}$ can be described recursively as the unique sequence of polynomials with $P_1 = 2$, $P_2 = 2X + c_{\mathrm{hit}}$, and 
\[
(2n-4)P_{n} - 2P_{n}' = \sum_{k=2}^{n-1} P_k P_{n+1-k} - (n-1)nP_{n-1} + (2n-1)P_{n-1}' - P_{n-1}''.
\]
for every $n\geq 3$.
\end{theorem}

It is a consequence of the recursive definition of the polynomials $(P_n)_{n\geq 1}$ that $P_n$ has degree $n-1$ and leading coefficient $2$ for every $n\geq 1$. For example,
\[
P_3 = 2X^2 + (2c_\text{hit}+1)X + \frac{c_\text{hit}^2+10c_\text{hit}+12}{2}.
\]
An interesting feature of this result is that we believe that the series $\sum_{n=1}^\infty  P_n(\log \log r)(\log r)^{-n}$ is \emph{divergent} for every finite $r$, so that the formal equality $\PROB{r,\infty}{N_1>0} = \frac{1}{r^2} \sum_{n=1}^\infty P_n(\log \log r)(\log r)^{-n}$ holds only in the sense of an asymptotic expansion. 
% {\color{red}[TH: Do we actually prove divergence?]}
% which is why the result can only be stated in the way we did
See Remark \ref{Rmk:divergent} for further details. As we discuss below, the divergence observed here is specific to the 4D case and makes the proof harder than in other dimensions. 
% {\color{red}[TH: Comments on universality? Comment on $c_\text{hit}$?]}

\begin{remark}
For other critical offspring distributions (with an exponential moment), we believe that the $2$ appearing as the coefficient of $(r^2\log r)^{-1}$ should be replaced by $2/\sigma^2$, where $\sigma^2$ is the variance of the offspring distribution. The polynomials $P_n$ should also satisfy a different, distribution-dependent recurrence relation but will still have degree $n-1$ for each $n\geq 1$.
\end{remark}

\paragraph*{Noncritical dimensions.}
Now suppose that $d \in \N \setminus \{4\}$, where $\N = \{ 1, 2, ,\ldots\}$. In this case we will obtain an expansion of the hitting probability $\P_{r, \infty} (N_1 > 0)$ of the ball of radius 1 as a \emph{convergent} series in powers of $r$; more precisely, after multiplication by $r^2$, this expansion is a power series in $r^\beta$, where $\beta$ is a novel 
% nontrivial (dimension-dependent but otherwise universal)
exponent that we now define. Set
\begin{equation}
\label{E:beta}
    \beta = \beta(d)  =
    \begin{cases}
        \displaystyle \frac{d-6 + \sqrt{d^2 - 20 d + 68}}{2}  & \text{ if } d \le 3,\\
          d-4 & \text{ if } d \ge 5,
         \end{cases}
    % ~=~
    % \left\{ \begin{array}{ll}
    % 1     &  \text{if~} d=1,\\
    % 2(\sqrt{2} -1) \approx 0.83     & \text{if~} d=2,\\
    % (\sqrt{17} - 3)/2 \approx 0.56 & \text{if~} d=3,\\
    %  d-4 & \text{if~} d\ge 5.
    % \end{array} \right.
\end{equation}
Thus, $\beta(1)=1$, $\beta(2)=2(\sqrt{2}-1)\approx 0.83$, $\beta(3)=(\sqrt{17}-3)/2\approx 0.56$, and $\beta(d)=d-4$ for $d\geq 5$.
Notice that both expressions for $\beta$ would degenerate to $0$ at $d=4$, and that the expression defining $\beta(d)$ for $d\le 3$ would not be real for $d >4$. Since $\beta(d)$ will be seen to govern the higher-order corrections to scaling for the critical probability when $d\neq 4$, the fact that the expression for $\beta(d)$ degenerates to $0$ when $d=4$ is related to the fact that the corrections are subpolynomially small when $d=4$.

\medskip

Given $d\in \N \setminus \{4\}$, we also define the sequence of coefficients $\alpha_\ell = \alpha_\ell (d)$ for $\ell =0, 1, 2, \ldots$ inductively by
\begin{equation}
\label{E:alpha}
\alpha_0 = (8-2d)_+, \quad
\alpha_1 = 1 \quad \text{and} \quad \alpha_\ell = \frac{1}{(\beta^2\ell + |2d-8|)(\ell-1)} \sum_{k=1}^{\ell - 1} \alpha_k \alpha_{\ell-k}, \quad \ell \geq 2
\end{equation}
where for $x \in \R$, $x_+ = \max(x,0)$.
Let $R_\alpha$ be the radius of convergence of the power series $x \mapsto \sum_{\ell=0}^\infty \alpha_\ell x^\ell$. One can check by induction that $0 \leq \alpha_\ell \leq (2\beta^2 + |2d-8|)^{1-\ell}$ for every $\ell \geq 1$ and hence that $R_\alpha \geq 2\beta^2 + |2d-8|>0$.

%\begin{equation}\label{E:alpha123}
%    \alpha_0 = 8-2d, \quad \alpha_1 = 1 \quad \text{and} \quad
%    \alpha_\ell = \frac{1}{(\beta^2 \ell + 8-2d)(\ell-1)} \sum_{k=1}^{\ell-1} \alpha_k \alpha_{\ell-k}, \quad \ell \geq 2
%\end{equation}

\begin{theorem}[Hitting probabilities in noncritical dimensions]\label{T:hitting_neq4}
Let $d \in \N \setminus \{4\}$ and let $\beta=\beta(d)$, $(\alpha_\ell)_{\ell\geq 0}=(\alpha_\ell(d))_{\ell\geq 0}$, and $R_\alpha=R_\alpha(d)$ be as above. There exists a constant $\mu_1 = \mu_1(d) \in \R$ such that
\[
\PROB{r,\infty}{N_1 > 0} = \sum_{\ell=0}^\infty \alpha_\ell \mu_1^\ell r^{-2-\beta \ell}  
% \left\{ \begin{array}{cc}
% \frac{6}{r^2} - \frac{|\mu_1|}{r^3} + \frac{1}{8} \frac{|\mu_1|^2}{r^4} %- \frac{1}{72} \frac{|\mu_1|^3}{r^5}
% + \dots & \text{if~} d=1, \\
% \frac{4}{r^2} - \frac{|\mu_1|}{r^{2+\beta}} + \frac{7+4 \sqrt{2}}{68} \frac{|\mu_1|^2}{r^{2+2\beta}} %- \frac{59+41\sqrt{2}}{3808} \frac{|\mu_1|^3}{r^{2+3\beta}} 
% + \dots & \text{if~} d=2, \\
% \frac{2}{r^2} - \frac{|\mu_1|}{r^{2+\beta}} + \frac{5+\sqrt{17}}{24} \frac{|\mu_1|^2}{r^{2+2\beta}} %- \frac{46+11\sqrt{17}}{708} \frac{|\mu_1|^3}{r^{2+3\beta}}
% + \dots & \text{if~} d=3, \\
% \frac{\mu_1}{r^{d-2}} + \frac{1}{2(d^2-7d+12)} \frac{\mu_1^2}{r^{2d-6}} + \dots
% & \text{if~} d \geq 5.
% \end{array} \right.
\]
for every $r > R_0(d):= \max(1,(|\mu_1|/R_\alpha)^{1/\beta})$.
Moreover, the constant $\mu_1$ is negative when $d \in \{1,2,3\}$ and is strictly between $0$ and $1$ when $d \geq 5$. Furthermore, when $d\geq 5$ the constant $\mu_1$ is the unique solution in $(0,1)$ to the equation $\sum_{\ell= 1}^\infty \alpha_\ell \mu_1^\ell = 1$.
\end{theorem}

Let us emphasise that the constants $\beta$ and $\mu_1$ both depend on $d$, as does the sequence $(\alpha_\ell)_{\ell\geq 0}$. 
Unlike $\beta$ and $(\alpha_\ell)_{\ell\geq 0}$, the constant $\mu_1$ is not explicit but is only defined implicitly (see Section \ref{S:123} for more details).
We believe that the exponent $\beta$ is \emph{universal} but that the constant $\mu_1$ and the sequence $(\alpha_\ell)_{\ell\geq 0}$ are not. 
% {\color{red}[TH: I guess this means the sequence $\alpha$ is not universal either?]}
For instance, if one considers another critical branching mechanism (satisfying appropriate moment conditions) then a similar expansion should hold with the \emph{same} value of $\beta$ but a possibly different value of $\mu_1$. With some additional work, this universality statement should follow from our arguments; see in particular the discussion below \eqref{E:june4}.
The proof of Theorem \ref{T:hitting_neq4} is contained in Section \ref{S:123}. 
% {\color{red}[TH: What about BRW/SBM? I guess for BRW there might be other ``lattice effect'' terms in the expansion? Should the sequence $\alpha$ be universal for BBM?]}

\begin{comment}
\paragraph*{High dimensions}

Let $d \geq 5$.
Let $(\alpha_\ell)_{\ell \geq 1}$ be the sequence defined by

Because for all $\ell \geq 1$, $\alpha_\ell \in (0,1]$, the function $\mu \in [0,1] \mapsto \sum_{\ell \geq 1} \alpha_\ell \mu^\ell$ is continuous, increasing, vanishes at 0 and is larger than $\alpha_1 = 1$ at $\mu =1$. Therefore, for all $s \in [0,1]$, there exists a unique $\mu_s \in [0,1)$ such that
\begin{equation}
\label{E:mu5}
\sum_{\ell \geq 1} \alpha_\ell \mu_s^\ell = s.
\end{equation}

\begin{lemma}\label{L:dimension5_hitting}
Let $d \geq 5$. For all $r \geq 1$,
\begin{equation}
\PROB{r,\infty}{N>0} = \sum_{\ell \geq 1} \alpha_\ell \mu_1^\ell r^{-(d-4)\ell -2}.
\end{equation}
\end{lemma}

The proof of this result is contained in Section \ref{S:preliminary_observations}.

\end{comment}

\subsection{Consequences for branching random walk}
So far, all of the results we have stated relied heavily on the continuous nature of $\R^d$ and its rotational symmetry. Indeed, much of our analysis works by using rotational symmetry to convert problems about the partial differential equation $\Delta v=v^2$ into problems about \emph{ordinary} differential equations, a technique that relies on rotational invariance in an essential way. Nevertheless, it is possible to apply our results to study lattice models such branching random walk on $\mathbb{Z}^d$, in which full rotational-invariance holds only in the scaling limit. Indeed, we establish
in Corollary \ref{C:coupling} a strong coupling between critical branching random walk and critical branching Brownian motion (both conditioned to survive for at least $R^2$ generations), which enables us to easily transfer many of the above results to the setting of branching random walk. 

\medskip

To give an example of what can be proven in this way, let us now state a theorem concerning local time thick points for branching random walk. Consider a $\mathbb{Z}^4$-valued branching random walk $S_\Tf$, indexed by a critical binary BGW tree $\Tf$, with centred, nearest-neighbour increments, and for each $x \in \Z^4$ consider the local time
% \[
$\ell_x = \sum_{v \in \Tf} \indic{S_\Tf(v) = x}$.
% \]
% be the local time of $S_\Tf$ at $x$. 
Recalling the definition  of the constant $m_1$ from \eqref{E:m1}, we define for each $a>0$ and $R\geq 1$ the set of local time thick points to be
\begin{equation}
\Tc_R^{\rm RW}(a) := \left\{ x \in \Z^4 : \ell_x \geq a \frac{16 m_1}{\pi^2} (\log R)^2 \right\}.
\end{equation}
The following theorem is a local-time analogue of Theorem~\ref{T:thick} for branching random walk.

\begin{theorem}\label{T:thick_BRW}
Let $\zeta$ denote the extinction time of the genealogical tree $\Tf$ and consider the conditional law $\P_{z_0,\infty}^\mathrm{BRW} ( \,\cdot\, |\, \zeta > R^2)$ of the branching random walk given that the branching process survives to time $R^2$, where $z_0\in \mathbb{Z}^4 \cap B(0,R/2)$.
% Let $z_0 \in B(R/2) \cap \Z^4$. Under $\P_{z_0,R}^{\brw} ( \cdot |\zeta> R^2)$ (where $\zeta$ denotes the extinction time of $\Tf$), we have
For each $a>0$ we have the distributional asymptotics
\[
\frac{1}{(\log R)^2} \max_{x \in \Z^4} \ell_x \geq \frac{64 m_1}{\pi^2} - o(1)
\quad \text{and} \quad
\left\{
\begin{array}{cc}
    \# \Tc_R^{\rm RW}(a) \geq R^{4-a+o(1)} & a \in (0,4), \\
    \# \Tc_R^{\rm RW}(a) \leq R^{4-a+o(1)} & a \in (0,a_0),
\end{array}
\right.
\]
as $R\to\infty$, where each $o(1)$ term denotes a random variable converging to $0$ in distribution as $R\to\infty$ when $a$ is fixed.
% where $o(1) \to 0$ in probability as $R \to \infty$.
Moreover, these distributional asymptotics are unchanged if we replace  $\mathbb{P}_{z_0,\infty}^\mathrm{BRW}(\,\cdot\,|\,\zeta>R^2)$ with the law $\mathbb{P}_{z_0,R}^\mathrm{BRW}(\,\cdot\,|\,\zeta>R^2)$, where particles are killed upon leaving the sphere of radius $R$.
\end{theorem}

Note that in the above theorem, the upper and lower bounds for the number of $a$-thick points match when $a \in (0,a_0)$. We conjecture that this is still the case for $a \in [a_0,4)$. 

\begin{remark}
It should not be difficult to extend this result to branching random walks with more general increment distributions. For example, if one considers increments that are centred, have all moments finite, and have covariance matrix $\Gamma$ then the same result should hold provided that one replaces  $\frac{16 m_1}{\pi^2}$ with $\frac{m_1}{\pi^2 \sqrt{\det \Gamma}}$ in the definition of the set of thick points.
\end{remark}

\subsection{Nonpositive solutions to the equation $\Delta v = v^2$ in four dimensions}

The semilinear partial differential equation 
$$
\Delta v = v^2
$$
with given boundary conditions in a domain $D$ of $\R^d$ plays a key role in our analysis. This PDE is well known to arise in connection with branching particle systems such as super Brownian motion, starting with Watanabe's original work \cite{watanabe1968limit}.
Since then, non-negative solutions of this PDE have been extensively studied, see e.g. \cite{dynkin1991probabilistic}. %showed that a compact set $K \subset \R^d$ is polar for super-Brownian motion (almost surely not visited by the cloud of particles when starting from outside of $K$) if and only if $K$ is removable for the equation $\Delta v = v^2$ (meaning that the only non-negative solution to $\Delta v = v^2$ in $\R^d \setminus K$ is the function $v=0$). 
Probabilistic representations  of non-negative solutions have also been greatly studied; such representations can also be used to study qualitative properties of this equation. We refer the reader to \cite{le1999spatial} for an overview and many further references to this topic.

% See also Lemma \ref{L:appendix_uniqueness}, below, where we recall a basic uniqueness result concerning non-negative solutions.

\medskip

It turns out, however, that for our purposes it is the \textbf{negative} solutions (or rather, nonpositive) to this PDE which are relevant. Surprisingly, very little seems to be known in that case, starting with the uniqueness of solutions on which our analysis depends crucially. While in the positive case the maximum principle easily shows solutions are unique (we recall this in Lemma \ref{L:appendix_uniqueness}), no such general argument can be made in the negative case. In fact, we explain that uniqueness cannot be expected to hold generally; see the discussion in Section \ref{SS:uniquenessprelim}. We show however that uniqueness holds for rotationally symmetric solutions in dimension $d=4$ with sufficiently small normal derivative. This is one of the key technical contributions of this paper. We state this as a theorem below:

\begin{theorem}\label{T:uniqueness}
Let $R> 1$ and $s<0$, and suppose $d=4$. The problem
\[
\begin{cases}
\Delta v_s = v_s^2  \quad \text{in} \quad B(R) \setminus \overline{B(1)}, \\
v_s = s \quad  \text{on} \quad \partial B(1), \\
v_s = 0 \quad  \text{on} \quad \partial B(R),\\
|(\partial v/\partial n) | < R^{-3}  \quad \text{on} \quad \partial B(R)
\end{cases}
%\right.
\]
has at most one rotationally invariant solution.
%
%that their normal derivative on the outer sphere $\partial B(R)$ is strictly smaller than~$R^{-3}$. 
\end{theorem}

See Theorem \ref{T:proba_representation} where we also give a probabilistic representation of these solutions. We believe that the hypotheses of this theorem are near-optimal: Numerics suggest that this uniqueness statement holds until $|(\partial v/\partial n) | < \lambda_c R^{-3}$ on $\partial B(R)$ where $\lambda_c \approx 2.43$. See Section~\ref{SS:uniquenessprelim} for a detailed discussion. Parts of this proof are not specific to the case of dimension $d=4$; in particular, Lemma \ref{L:critical_lambda}, which gives an equivalence between analytic and probabilistic descriptions of the uniqueness threshold, holds in all dimensions.

% \paragraph{Structure of the paper}
\subsection{Structure of the paper}
The rest of the paper is organised as follows:

\begin{itemize}[leftmargin=*]
\item Section \ref{S:preliminary_observations}: we relate the Laplace transform of the number of pioneers to boundary value problems for the PDE $\Delta v= v^2$ and make some related preliminary observations, including an interesting scale-invariance property for the number of particles frozen on the boundary of a large sphere.
\item
Section \ref{sec:uniqueness}: we give a probabilistic representation for \emph{some} nonpositive rotationally invariant solutions to $\Delta v = v^2$, and prove Theorem \ref{T:uniqueness}, which provides a uniqueness criterion for this PDE. (We also discuss non-uniqueness.) 
\item
Section \ref{sec:upper_bound}: by studying the asymptotic behaviour of solutions to $\Delta v = v^2$ and by applying the main result of Section \ref{sec:uniqueness}, we obtain the upper bound on the tail probability of $N_1$ with killing on $\partial B(R)$.
\item Section \ref{sec:series_expansion}: we continue our study of the asymptotic behaviour of the solutions. In particular, we obtain a divergent series expansion analogous to the one stated in Theorem \ref{T:hitprob4}.
\item
Section \ref{sec:tail4}:  we use the series expansion derived in the previous section to prove very precise estimates on the Laplace transform of the number of pioneers with Dirichlet boundary conditions. We then prove a ``finitary'' Tauberian theorem that allows us to transfer these estimates on the Laplace transform to the tail probabilities. This will conclude the proof of Theorem \ref{T:maintail4}.
\item
Section \ref{S:no_killing}: we prove the tail estimate \eqref{E:T_no_killing} in Theorem \ref{T:tail_general} when there is no killing, identifying precisely the transition at $a=2$ from linear to sublinear exponents.
\item
Section \ref{S:thick}: we prove our thick point result Theorem \ref{T:thick} based on a truncated second moment approach.
\item
Section \ref{S:occupation_measure}: using our estimates on the number of pioneers, we prove Theorem \ref{T:local_time_tail} concerning the tails of the occupation measure.
\item
Section \ref{sec:discrete}: for any tree $T$, we construct a strong coupling between a $T$-indexed random walk and a $T$-indexed Brownian motion. We then use this coupling to transfer results on branching Brownian motion to branching random walk.
\item
Section \ref{S:123}: we prove our expansion of the hitting probability when $d \neq 4$.
\end{itemize}

\paragraph*{Acknowledgements}

The authors thank Amine Asselah, Diederik van Engelenburg, Robin Khanfir, Bruno Schapira and Yilin Wang for stimulating discussions related to this work.
This paper was initiated during a Spring 2022 programme at the Mathematical Sciences Research Institute in Berkeley, California that was supported by NSF grant DMS-1928930; NB and AJ attended the full programme and began work on the project while TH attended one of the associated workshops.
% 
% is based upon work supported by the National Science Foundation under Grant No. DMS-1928930 while NB and AJ participated in a programme hosted by the Mathematical Sciences Research Institute in Berkeley, California, during the Spring 2022 semester. The work began when TH also visited this programme on the occasion of one of the workshops. 
The hospitality and stimulating atmosphere of the institute is gratefully acknowledged.
Part of the work also took place during visits by AJ and NB to Caltech and by AJ to the University of Vienna; we also acknowledge the hospitality of both institutions.
NB is supported by FWF Grant  10.55776/P33083 on ``Scaling limits in random conformal geometry'',
TH is supported by NSF grant DMS-2246494, and
AJ is supported by Eccellenza grant 194648 of the Swiss National Science Foundation and is a member of NCCR SwissMAP.

\section{Preliminary observations on the PDE \texorpdfstring{$\Delta v = v^2$}{Laplacian v = v squared}}\label{S:preliminary_observations}
The proofs of Theorems \ref{T:tail_general}, \ref{T:hitprob4} and \ref{T:hitting_neq4}  all rely on the analysis of the Laplace transform of the number of pioneers. In this section we introduce relevant notation, state the PDE satisfied by this Laplace transform (Lemma \ref{L:v_ODE}), and make some preliminary observations about the solutions of this PDE. The fact that the Laplace transform satisfies this PDE is not new, but our analysis of the resulting PDE is novel; a detailed discussion is given after the statement of Lemma~\ref{L:appendix_uniqueness}.
We work in an arbitrary dimension $d \geq 1$ throughout this section.

\medskip

We now introduce our notation for the Laplace transform of the number of pioneers. Let $R >1$ and for each $r \in (1,R)$ consider the quantity $s_c(r,R)$ defined by
\begin{equation}
\label{E:sc}
s_c(r,R) = \inf \{ s <0: \EXPECT{r,R}{ (1-s)^{N_1}} < \infty \},
\end{equation}
where $N_1$ is the number of pioneers on the unit sphere and we recall that $\mathbb{E}_{r,R}$ denotes the expectation taken with respect to branching Brownian motion started at distance $r$ from the origin and killed on the sphere of radius $R$.
With some abuse of terminology, $s_c(r,R)$ will often be referred to as the ``radius of convergence'' of the Laplace transform.
Observe that $s_c(r,R)$ does not actually depend on the parameter~$r$. Indeed, given $r_1, r_2 \in (1,R)$, if we define $E$ to be the event that the initial particle travels to the sphere of radius $r_2$ before the first branching occurs, then
\[
\EXPECT{r_1,R}{ (1-s)^{N_1}} \geq \EXPECT{r_1,R}{ (1-s)^{N_1} \mathbf{1}_E } = \PROB{r_1,R}{E} \EXPECT{r_2,R}{ (1-s)^{N_1}},
\]
and since $\PROB{r_1,R}{E} >0$ it follows that $s_c(r_1,R) \geq s_c(r_2,R)$. By exchanging the roles of $r_1$ and $r_2$ this shows that $s_c(r_1,R) = s_c(r_2,R)$, proving the claim that $s_c(\cdot,R)$ does not depend on $r$. From now on we will write $s_c(R)$ for the common value of $s_c(r,R)$ shared by all values of $r$ in $(1,R)$. 
For each $s \in (s_c(R),1]$, we define the function $\vs_s:y \in B(R) \setminus B(1)\to \R$ by
\begin{equation}
\label{E:vs}
\vs_s(y) := 1 - \EXPECT{y,R}{ (1-s)^{N_1} },
\end{equation}
which is related to the Laplace transform $\mathcal{L}_{y,R}(s)=\EXPECT{y,R}{e^{-sN_1}}$ of the law of $N_1$ by $\vs_s(y)=1-\mathcal{L}_{y,R}(-\log(1-s))$.
Since $\vs_s$ is rotationally invariant, we will also write $\vs_s(\norme{y}) = \vs_s(y)$. (Note that $\vs_s$ also depends on $R$, but we will suppress this from our notation.)
% Viewed as a function of $s$, $\vs_s(y)$ can be written in terms of the Laplace transform

\medskip

We now write down the boundary value problem satisfied by the function $\vs_s$.

\begin{lemma}\label{L:v_ODE}
Let $R>1$ and $s \in (s_c(R),1]$. The function 
$\vs_s$ is solution to the boundary value problem
\[
% \left\{
% \begin{array}{l}
\Delta \vs_s = \vs_s^2 \quad \text{in} \quad B(R) \setminus \overline{B(1)}, \qquad
\vs_s = s \quad \text{on} \quad \partial B(1), \qquad
\vs_s = 0 \quad \text{on} \quad \partial B(R).
% \end{array}
% \right.
\]
\end{lemma}

\begin{remark}
For our purposes, part of what it means for a function to solve a boundary value problem on some domain is that its extension to the closure of the domain is continuous. 
\end{remark}

\begin{proof}
In \cite{Lalley15} (see equation (7.19) therein) it is shown that if $d=1$ then $\vs_s$ satisfies the ODE
$\vs_s'' = \vs_s^2$ for every $s\in (1,R)$.
The same proof works in every dimension, meaning that $\Delta \vs_s = \vs_s^2$ in $B(R) \setminus \overline{B(1)}$. The boundary values are then easy to check.
See also \cite[Chapter V, Theorem 6]{le1999spatial} in the context of exit measures of super-Brownian motion.
\end{proof}

\begin{remark}
Throughout the paper we consider only branching Brownian motion with offspring distribution uniform on $\{0,2\}$. If one were to consider a different offspring distribution with probability generating function $G$ then the relevant PDE would be $\Delta v = 2G((1-v))-2(1-v)=:\Psi(v)$. The function $\Psi(v)$ is convex, non-negative, and satisfies $\Psi(v) \sim \sigma^2 v^2$ as $v\to 0$, where $\sigma^2$ is the variance of the offspring distribution. (Of course there are technicalities to address when $v$ lies outside of the radius of convergence of the series defining $\Psi$.) Since $\Psi(v)$ is non-negative and behaves like a multiple of $v^2$ when $v$ is small, it should with work  be possible to extend all of our analysis to other critical offspring distributions (with e.g.\ a finite exponential moment). The higher order corrections to $\Psi$ will lead to different higher-order terms in the series expansions of Theorem~\ref{T:hitprob4} and Theorem~\ref{T:hitting_neq4}, but the basic form of these expansions should be unchanged. We do not investigate these matters further in this paper.
\end{remark}

\begin{remark}
\label{remark:superBrownian_same_PDE}
The same PDE $\Delta v = v^2$ that we study is also relevant for super-Brownian motion, but with a different encoding between parameters of the generating function and boundary conditions. Indeed, letting $\mathbb{N}_x$ denote the excursion measure of the Brownian snake starting at $x$ and letting $Z$ be the total exit measure on the inner boundary $\partial B(1)$ when the process is killed upon leaving the annulus $B(R)\setminus B(1)$, it follows from \cite[Chapter V, Theorem 6]{le1999spatial} that the function
$\us_s(x)=4\mathbb{N}_x[1-e^{-sZ}]$ is a solution to the boundary value problem 
\[
% \left\{
% \begin{array}{l}
\Delta \us_s = \us_s^2 \quad \text{in} \quad B(R) \setminus \overline{B(1)}, \qquad
\us_s = 4s \quad \text{on} \quad \partial B(1), \qquad
\us_s = 0 \quad \text{on} \quad \partial B(R).
% \end{array}
% \right.
\]
Thus, aside from the various unimportant factors of $4$ that appear due to normalization conventions, the generating function $\us_{s/4}=4(1-\mathbb{N}_x[e^{-sZ/4}])$ satisfies the same boundary value problem as $\vs_s = 1 - \EXPECT{x,R}{ (1-s)^{N_1} }$; the only difference is that the generating function is taken with the parameter $s$ in the former case and $\log(1-s)$ in the latter case. Because of this connection, all of our work analyzing the non-positive solutions to the PDE $\Delta v=v^2$ can be applied directly to the super-Brownian case also. We do not pursue this further in this paper.
\end{remark}

Using that $\vs_s$ is rotationally invariant and expressing the Laplacian in spherical coordinates, Lemma \ref{L:v_ODE} can be rephrased as follows:
for every $R>1$ and $s \in (s_c(R),1]$, the function
$\vs_s:[1,R]\to\R$ is a solution to the boundary value problem
\begin{equation}
\label{E:L_ODEv}
% \left\{
% \begin{array}{l}
\vs_s'' + \frac{d-1}{r}\vs_s' = \vs_s^2 \quad \text{in} \quad (1,R), \qquad
\vs_s(1) = s, \qquad
\vs_s(R) = 0.
% \end{array}
% \right.
\end{equation}
(Recall that we identify $\vs_s(y)$ with $\vs_s(\|y\|)$ with an abuse of notations.) We will see throughout the paper that solutions to this boundary value problem have greatly differing behaviours in the three cases $d<4$, $d=4$, and $d>4$.

\paragraph{The maximum principle.} Since every solution to the PDE $\Delta u = u^2$ on some domain $D$ is subharmonic on $D$, it satisfies the maximum principle, meaning that if $u$ extends continuously to $\partial D$ then
\[
\sup_{x\in D}u(x) = \sup_{x\in \partial D} u(x).
\]
It follows in particular that if $u(x)$ is nonpositive on the boundary of $D$ then it is also nonpositive on the interior of $D$.

\begin{remark}
It is not true that if $u\ge 0$  on the boundary of a domain then also $u\ge 0$ in the interior of the domain. Indeed, let $d=1$ and $\lambda>0$ and consider the solution to the initial value problem $u''=u^2$, $u(0)=0$, $u'(0)=-\lambda$. It is easy to prove that there exists $x_0>0$ such that $u'(x_0)=0$ and that $u'(x)>0$ for $x>x_0$. Since $u$ is convex on its domain, it follows easily that $u(x)$  is positive for all sufficiently large $x$ in the domain of $u$. As such, letting $x_1>x_0$ be such that $u(x_1)=0$, the restriction of $u$ to $[0,x_1]$ solves the ODE $u''=u^2$ with zero boundary conditions, but is strictly negative on the interior of $[0,x_1]$. One can similarly obtain examples where the boundary values are positive but the function is non-positive on some parts of the interior.
\end{remark}

\paragraph{Uniqueness of non-negative solutions.} We now state and prove a uniqueness result concerning \emph{non-negative} solutions to $\Delta u = u^2$. We do so because we could not find this lemma in this form in the literature, although this type of uniqueness statement is well known. This also gives us the opportunity to recall where the non-negativity assumption is used in the proof. (We state the lemma only for the boundary conditions we intend to use it with; the proof is much more general.)

\begin{lemma}\label{L:appendix_uniqueness}
Let $d \geq 1$.
Let $s>0$ and $u, v: \R^d \setminus B(1) \to [0,\infty)$ be two continuous non-negative  solutions to the same boundary value problem
\[
\left\{
\begin{array}{l}
\Delta f = f^2 \quad  \text{in} \quad \R^d \setminus \overline{B(1)}, \\
f = s \quad \text{on} \quad \partial B(1), \\
f(x) \to 0 \quad \text{as} \quad x \to \infty.
\end{array}
\right.
\]
Then $u=v$.
\end{lemma}

\begin{proof}
We follow the proof of \cite[Chapter 5, Lemma 7]{le1999spatial}.
Let $w = u - v$ and let $D = \{ w > 0 \}$. Assume that $D$ is not empty and for each $R\geq 1$ let $D_R = D \cap B(R)$. \emph{Because $u$ and $v$ are non-negative}, we have that $\Delta w =u^2-v^2= w(u+v) \geq 0$ on $D$, so that $w$ is subharmonic on $D$. By the standard maximum principle, this implies that $\max_{D_R} w = \max_{\partial D_R} w$. On the other hand, by definition of $D$, $w$ is identically zero on $(\partial D) \cap ( B(R) \setminus \overline{B(1)} )$. Moreover, by assumptions on the boundary conditions of $u$ and $v$, $\max_{\partial B(1) \cup \partial B(R)} w \to 0$ as $R \to \infty$.
Wrapping up, this shows that $\max w = 0$ in $D$. This is absurd, proving that $D = \varnothing$ and that $u \leq v$. It follows by symmetry that $u \geq v$ also.
\end{proof}

%\paragraph{Previous work on the PDE $\Delta v=v^2$.}
%A relationship between the operator $v \mapsto \Delta v - v^2$ and super-Brownian motion was already established by Watanabe \cite{watanabe1968limit}.
%Since then, non-negative solutions to $\Delta v = v^2$ and their link to super-Brownian motion have been extensively studied. For instance, \cite{dynkin1991probabilistic} showed that a compact set $K \subset \R^d$ is polar for super-Brownian motion (almost surely not visited by the cloud of particles when starting from outside of $K$) if and only if $K$ is removable for the equation $\Delta v = v^2$ (meaning that the only non-negative solution to $\Delta v = v^2$ in $\R^d \setminus K$ is the function $v=0$). Probabilistic representations of non-negative solutions to this equation have also been greatly studied:
%see \cite{le1999spatial} for an overview of many results and references. 
% See also Lemma \ref{L:appendix_uniqueness}, below, where we recall a basic uniqueness result concerning non-negative solutions.
%
%\medskip
%
Our probabilistic solutions $\vs_s$ fall into the category of non-negative solutions when $s \in [0,1]$. However, the behaviour of $\vs_s$ when $s \in [0,1]$ gives very little information about the large deviations of the number of pioneers ${N_1}$, which is the main interest of the current paper. To get useful information about large deviations, we need to study $\vs_s$ when $s$ is negative, meaning that we are interested in \emph{nonpositive} solutions to the PDE $\Delta v = v^2$. 

For negative solutions, uniqueness is not guaranteed immediately. This lack of uniqueness accounts for much of the additional technical difficulty faced in our work.
One of the main contributions of our paper is therefore to analyse these nonpositive solutions; in particular to establish uniqueness of rotationally symmetric solutions in $\R^4$ subject to a certain smallness assumption, see Theorem \ref{T:uniqueness}.

\paragraph{Notational convention.} To ease the reading of the arguments below we will make use of the following conventions. When we consider a specific solution to a PDE or ODE coming from a probabilistic model we will use letters such as $\vs, \gs, \hs$; whereas general solutions to such problems will be denoted with regular letters such as $v, g, h$, etc.

\medskip 
 We now want to give a simple example of how one can use PDE representations such as Lemma \ref{L:v_ODE} to study the number of pioneers.
This example will be useful in the proofs of the main theorems and also features an application of Pringsheim's theorem (a.k.a.\ the Vivanti-Pringsheim theorem) in a friendly setting. This theorem states that if $f(z)=\sum_{n =0}^\infty a_n z^n$ is a function defined by a power series \emph{with non-negative coefficients} and with radius of convergence $R$ then there is no analytic continuation of $f$ to an open neighbourhood of $[0,R]$. (One often sees the conclusion of this theorem stated as ``$f$ has a singularity at $R$''. Let us stress however that it is possible for the power series defining the $k$th derivative of $f$ to converge at $R$ for every $k\geq 0$, as is the case when $a_n=e^{-\sqrt{n}}$.)
We will rely on this classical result once more in the proof of our probabilistic representation of nonpositive solutions to $\Delta v = v^2$ in Section~\ref{sec:uniqueness}.

\paragraph*{``Scale invariance'' of the number of pioneers.}

We now discuss how a form of scale invariance for the number of pioneers on some outer sphere can be deduced easily via ODE techniques. Fix $R\geq 1$ and consider running branching Brownian motion where we freeze particles not on the unit sphere but on the outer sphere $\partial B_R$. Denote by $\tilde N_R$ the number of particles that are frozen on this outer sphere, which we call pioneers as in our usual setting.
% Specifically, for $R \geq 1$, we do not freeze particles on the unit sphere and we only keep track of the number of pioneers on the outer sphere $\partial B(R)$, that we denote by $N_R$. 
% For each $y \in B(1)$ 
We define
\[
s_*(R) = \inf \{ s < 0 : \EXPECT{Ry}{ \left(1-\frac{s}{R^2} \right)^{N_R} } < \infty \}
\]
for some $y$ in the open ball  $B(1)$, 
where, similarly to the definition of $s_c(R)$ above, the quantity $s_*(R)$ does not depend on the choice of $y\in B(1)$.
% Similarly to $s_c(R)$ \eqref{E:sc},
% $s_*(R)$ does not depend on the starting point $Ry$.

\begin{lemma}\label{L:invariance}
Let $d \geq 1$. The quantity $s_*(R)$ does not depend on the choice of $R \geq 1$, and
% Moreover, t
the resulting constant $s_*=s_*(1)$ is a finite, negative number. Moreover, for each $y\in B(1)$, the function from $(s_*,1]$ to $\R$ defined by
% denote the resulting shared value, 
% We denote this constant by $s_*$.
% the constant $s_*$ is negative and finite and for all $y \in \overline{B(1)}$ and $s \in (s_*,1]$,
\begin{equation}
\label{E:invariance}
 % \qquad 
 s\mapsto R^2 \left( 1 - \EXPECT{Ry}{ \left(1-\frac{s}{R^2} \right)^{\tilde N_R} } \right)
\end{equation}
does not depend on $R \geq 1$.
\end{lemma}

% We mention that
\begin{remark}
Asselah and Schapira 
\cite{asselah2022time} also studied the number of pioneers on $\partial B(R)$ starting from a point inside $B(R)$ in the context of branching random walk. In our notation, they show that $\sup_R s_*(R) < 0$, meaning that they obtain a uniform-in-$R$ exponential upper bound on the tail of $N_R/R^2$. Here we obtain a much more precise estimate in the context of branching Brownian motion.
\end{remark}

\begin{proof}[Proof of Lemma~\ref{L:invariance}]
We first consider non-negative parameters $s \in [0,1]$.
For each $y \in \overline{B(R)}$ let $\vs_{s,R}(y) = 1 - \mathbb{E}_y[( 1 - s )^{\tilde N_R} ]$. Analogously to the boundary value problem satisfied by $\vs_s$ (Lemma \ref{L:v_ODE}), $\vs_{s,R}$ satisfies the boundary value problem
\[
% \left\{
% \begin{array}{l}
\Delta \vs_{s,R} = \vs_{s,R}^2 \quad \text{in} \quad B(R), \qquad
\vs_{s,R} = s \quad \text{on} \quad \partial B(R).
% \end{array}
% \right.
\]
We now make the appropriate change of variables by defining  $\ws_{s,R}(y) = R^2 \vs_{s/R^2,R}(Ry)$ for each $y \in \overline{B(1)}$, so that
%$\ws_{s,R}(y) = R^2 \left\{ 1 - \EXPECT{Ry}{ \left( 1 - \frac{s}{R^2} \right)^{N_R} } \right\}$
%Since $\Delta \ws_{s,R}(y) = R^4 \Delta \vs_{s/R^2,R}(Ry)$ and $\ws_{s,R}(y)^2 = R^4 \vs_{s/R^2,R}(Ry)^2$,
% The function
$\ws_{s,R}$ satisfies the boundary value problem
\[
% \left\{
% \begin{array}{l}
\Delta \ws_{s,R} = \ws_{s,R}^2 \quad \text{in} \quad B(1), \qquad
\ws_{s,R} = s \quad \text{on} \quad \partial B(1).
% \end{array}
% \right.
\]
By uniqueness of non-negative solutions to the above boundary value problem,
% (see e.g. \cite[Chapter 5, Corollary 8]{le1999spatial}),
the function $\ws_{s,R}$ does not depend on $R \geq 1$.

\medskip

Let $R_1, R_2 \geq 1$ and $y \in \overline{B(1)}$. We have just shown that $\ws_{s,R_1}(y) = \ws_{s,R_2}(y)$ for all $s \in [0,1]$. Since $\ws_{s,R_1}(y)$ and $\ws_{s,R_2}(y)$ are analytic in $s$ on their respective domains $(s_*(R_1),1]$ and $(s_*(R_2),1]$, this equality remains true for all $s \in (s_*(R_1) \vee s_*(R_2),1]$. We now show that $s_*(R_1) = s_*(R_2)$. The crucial observation is that $\ws_{s,R_2}(y)$ is a power series in $-s$ with non-negative coefficients. Thus, by Pringsheim's theorem, $s_*(R_2)$ is a singularity of $\ws_{s,R_2}(y)$ in the sense that $\ws_{s,R_2}(y)$ cannot be extended analytically to any open neighbourhood of $[s_*(R_2),0]$. This shows that $s_*(R_1) \geq s_*(R_2)$, since if it were not the case then $\ws_{s,R_1}(y)$ would provide an analytic continuation of $\ws_{s,R_2}(y)$ beyond the singularity $s_*(R_2)$.
Exchanging the roles of $R_1$ and $R_2$ we obtain that $s_*(R_1) = s_*(R_2)$ as claimed.

\medskip

The last remaining item  to check is the strict negativity and finiteness of $s_*$.
Let $R > 1$ (for instance take $R= 2$). We only need to show that $s_*(R) \in (-\infty, 0)$. While directly obtaining bounds that are uniform in $R$ would require some careful justification (as was done in \cite{asselah2022time}), we can instead conclude using very crude estimates since we already know that $s_*(R)$ does not depend on $R$. Because of the killing on $\partial B(R)$, the random variable $\tilde N_R$ can be stochastically dominated by the total progeny of a \emph{subcritical} Galton-Watson process whose Laplace transform has a positive radius of convergence. (For example, it is dominated by the total number of particles in a branching Brownian motion where we kill a particle at a branching event if it moved distance at least $2R$ since the previous branching event, so that the associated discrete genealogical tree has a Binomial$(2,p)$ offspring distribution for some $p<1/2$.) This shows that $s_*(R) < 0$. Similarly, one can stochastically dominate the total progeny of another subcritical Galton-Watson process by $\tilde N_R$ and conclude that $s_*(R)$ cannot be infinite.
\end{proof}

In the regime $R \gg 1$, we can approximate $\log(1 - s/R^2)$ by $-s/R^2$ and we will see that Lemma \ref{L:invariance} has the following corollary:

\begin{corollary}\label{C:outer_convergence}
Let $d \geq 1$ and $y \in B(1)$. The normalized probability 
$R^2 \P_{Ry}( \tilde N_R >0 )$ converges as $R \to \infty$ to a positive and finite quantity. Moreover, the law of the random variable $\tilde N_R / R^2$ under the conditional distribution $\P_{Ry}(\, \cdot\, \vert \tilde N_R >0 )$ converges to some limiting distribution as $R\to\infty$.
\end{corollary}

This could probably be inferred from the existence of the excursion measure for super-Brownian motion; we now give a short and direct argument based on the above. {We will omit some details of the proof since this corollary is tangential to the main results of the paper.}

\begin{proof}[{Sketch of proof}]
Since reaching $\partial B(R)$ occurs {with comparable probability to surviving for $R^2$ generations}\footnote{{This result is folklore, but we are not aware of a reference for it. A similar result is proven for branching random walk in \cite{MR3418547}, where the focus is on long-range models. The lower bound is easy since the locations of the particles in the branching random walk are Gaussian conditional on the genealogical tree. The upper bound can be proven using e.g.\ the method of \cite[Proof of Theorem 7.2]{MR4055195}, where the required bound on the expected number of pioneers in a sphere (which requires an involved analysis for the UST) is trivial for BBM.}}, the {normalized probability} $R^2 \P_{Ry}( \tilde N_R >0 )$ remains bounded away from 0 and from infinity as $R \to \infty$. 
Let $(R_n)_{n \geq 1}$ and $(R_n')_{n \geq 1}$ be two increasing sequences that diverge to infinity such that $R^2 \P_{Ry}( \tilde N_R >0 )$ converges along these two sequences. Let $\ell$ and $\ell'$ denote the two limits. We will show that $\ell = \ell'$. Using Lemma \ref{L:invariance}, we see that $\tilde N_R/R^2$ conditioned on $\{\tilde N_R>0\}$ converges in law along the two sequences $(R_n)_{n \geq 1}$ and $(R_n')_{n \geq 1}$ to some random variables $\tilde N_\infty$ and $\tilde N_\infty'$, respectively. Moreover, these limiting distributions are related by the fact that for all $s \in (s_*,1]$,
\[
\ell \Expect{1 - e^{-s \tilde N_\infty}}
= \ell' \Expect{1 - e^{-s \tilde N_\infty'}}.
\]
Assume without loss of generality that $\ell \geq \ell'$. Since the Laplace transforms of $\tilde N_\infty$ and $\tilde N_\infty'$ have a positive radius of convergence $|s_*|$, their laws are determined by their characteristic functions. The law of $\tilde N_\infty$ is therefore given by $(1-\ell'/\ell) \delta_0$ plus $\ell'/\ell$ times the law of $\tilde N_\infty'$. It can be shown that $\tilde N_\infty \neq 0$ a.s. (we omit these details) implying that $\ell = \ell'$. This concludes the proof of the convergence of $R^2 \P_{Ry}( \tilde N_R >0 )$ as $R \to \infty$. The convergence in law of $\tilde N_R/R^2$ under $\P_{Ry}(\cdot \tilde N_R>0)$ then follows.
\end{proof}

\begin{remark}
Following Remark~\ref{remark:superBrownian_same_PDE}, it should be possible to identify this limiting distribution in terms of the super-Brownian excursion measure. In the case $d=1$, this distribution can be described explicitly in terms of Weierstrass elliptic functions as explained in \cite{Lalley15}.
\end{remark}

% \red{[AJ: I'm not sure what's the limting distribution. We can discuss more on Thursday]}

% % {\color{red}[TH: For $d=1$ it seems like it's possible to express the solutions to the BVP, and hence the limiting distribution, in terms of the Weierstrass elliptic function. I haven't got anything nice to come out yet though. Did you think about whether the limiting distribution can be identified explicitly? Maybe we can at least write a question about it.]}
% {[\color{red}Write remark on super-Brownian case using slightly different form of ODE. Mention LZ15 1d calculations.]}
\paragraph{The autonomous form of the equation.} 
We end this section with the elementary but crucial observation that, after a change of variables, we can turn the ODE \eqref{E:L_ODEv} into an \emph{autonomous} equation. 
Indeed, let us define $L := \log R$ and
\begin{equation}
\label{E:gs}
\hs_s(x) = e^{2x} \vs_s(e^x)
\quad \text{and} \quad
\gs_s(x) := \hs_s(L-x)
= R^2 e^{-2x} \vs_s(R e^{-x}), \quad \quad
x \in [0,L].
\end{equation}
(Note that $\hs_s$ and $\gs_s$ both depend on the parameter $R$, but we suppress this from our notation.)
From the equation \eqref{E:L_ODEv}, one obtains that for each $s \in (s_c(R),1]$ the function $\hs_s$ 
% and $\gs_s$ are
is a solution to the boundary value problem
\begin{equation} \label{eq:hODE}
% \label{E:g_ODE_general_d}
% \left\{ \begin{array}{l}
\hs_s'' +(d-6)\hs_s' +(8-2d)\hs_s= \hs_s^2 \quad \text{in}~ (0,L), \qquad
\hs_s(0) = s, \qquad
\hs_s(L) = 0,
% \end{array} \right.
% \quad \quad \text{and} \quad \quad
% \left\{ \begin{array}{l}
% \gs_s'' -(d-6)\gs_s' +(8-2d)\gs_s = \gs_s^2 \quad \text{in}~ (0,L), \\
% \gs_s(0) = 0, \quad
% \gs_s(L) = s.
% \end{array} \right.
\end{equation}
with a similar boundary value problem satisfied by $\gs_s$.
In the case of primary interest that $d=4$, these boundary value problems may be written simply as
% From \eqref{E:L_ODEv}, one obtains that for each $s \in (s_c(R),1]$,
% the functions $\hs_s$ and $\gs_s$ are solutions to the boundary value problems
\begin{equation}
\label{E:g_ODE}
\left\{ \begin{array}{l}
\hs_s'' -2\hs_s' = \hs_s^2 \quad \text{in}~ (0,L), \\
\hs_s(0) = s, \quad
\hs_s(L) = 0,
\end{array} \right.
\quad \quad \text{and} \quad \quad
\left\{ \begin{array}{l}
\gs_s'' +2\gs_s' = \gs_s^2 \quad \text{in}~ (0,L), \\
\gs_s(0) = 0, \quad
\gs_s(L) = s.
\end{array} \right.
\end{equation}
% The associated ODEs $\hs_s'' - 2\hs_s' = \hs_s^2$ and $\gs_s'' + 2\gs_s' = \gs_s^2$ are both autonomous, so 
Although it may naively seem  that the equations have become simpler in the four-dimensional case, their behaviour in this case is in fact much more delicate.

\begin{remark}[Other dimensions]
  % In dimensions $d$ that are not necessarily equal to $d=4$, the equation analogous to \eqref{E:g_ODE} becomes 
  % $$
  % \hs''+ (d-6) \hs' + (8 -2d) \hs = \hs^2.
  % $$
Let us now point out a few ways in which the special role of the critical dimension $d=4$ can easily be seen from these equations. Most obviously, the coefficient $8-2d$ of the $\hs$ term in \eqref{eq:hODE} vanishes only when $d=4$, meaning that for $d \neq 4$ the equation admits the non-zero constant solution $\hs \equiv 8-2d$ while for $d=4$ there is no non-zero constant solution. For $d<4$ we will see that the solution of this equation, which corresponds to the critical BBM hitting probability via $h(x) = e^{2x} \PROB{e^{x},\infty}{N_1 >0}$, does in fact converge at infinity to the non-zero constant $8-2d$. (For $d>4$ this clearly cannot happen since the hitting probability is non-negative.) Further special features of the case $d=4$ become apparent when we try to linearize the equation in the large $x$ regime, assuming that $\hs$ is small at infinity.
The most naive way to linearize the equation is to neglect the $\hs^2$ term to get the equation
  % 
  % The special role played by the dimension $d=4$ can already be seen from this equation.
  % For this, consider a solution $h$ defined on $(0,\infty)$, such as $h(x) = e^{2x} \PROB{e^{2x},\infty}{N_1 >0}$, and assume that we are in a regime where $h$ is small at infinity.
  % The solutions to the linear equation
  $\hs'' + (d-6)\hs'+(8-2d)\hs=0$, whose solutions  are of the form $x \mapsto c_1 e^{-(d-4)x} + c_2 e^{2x}$. This suggests that when $d < 4$ non-trivial solutions cannot vanish at infinity, whereas, when $d > 4$, some non-trivial solutions vanish exponentially fast at infinity. The critical dimension $d=4$ is then a borderline case, as is consistent with \eqref{E:hitting_leading_order}. Moreover, the solutions to this linearized equation \emph{do not} have the property that $\hs^2$ is of lower order than $\hs''$ or $2\hs$ when $d\leq 4$, suggesting (correctly) that these solutions do not accurately reflect the true behaviour of the non-linear ODE in this case. (When $d<4$ one does however get a fairly good understanding of the \emph{remainder term} $\hs-(8-2d)$ by studying the naive linearization of the ODE it satisfies.) A more subtle but crucial difference between dimensions $d=4$ and $d>4$ is that when $d>4$ the terms $\hs''$, $\hs'$, and $\hs$ are all of the same order at infinity, with the $\hs^2$ term having smaller order, while for $d=4$ the $\hs'$ and $\hs^2$ terms have the same order at infinity with the $\hs''$ term having smaller order. The fact that the ``correction term'' $\hs''$ in the $d=4$ case involves a higher order of derivative than is found in the approximate equation $2\hs'=\hs^2$ leads the equation to have poor stability properties and ultimately to its asymptotic expansion being divergent, in contrast to the case $d>4$ where the ``correction term'' $\hs^2$ does not involve any derivatives.

\end{remark}

\section{Probabilistic representations of nonpositive solutions (proof of Theorem \ref{T:uniqueness})}\label{sec:uniqueness}

\subsection{Statement of result and preliminary steps}
\label{SS:uniquenessprelim}
We now initiate our study of the large deviations of the number of pioneers on the unit sphere in the four-dimensional case. We work exclusively in the case $d=4$ from now until Section \ref{S:123}.
% In the rest of the article (except Section \ref{S:123}), we assume that $d=4$.
The purpose of this section is to provide a probabilistic representation of nonpositive solutions of the autonomous equation $g''+2g' = g^2$. For each $\lambda >0$, let $g_\lambda$ be the unique maximal solution to the initial value problem
\begin{equation}
\label{E:g_lambda}
\left\{ \begin{array}{l}
g_\lambda'' + 2g_\lambda' = g_\lambda^2,\\
g_\lambda(0) = 0, \quad 
g_\lambda'(0) = -\lambda.
\end{array} \right.
\end{equation}
(This maximal solution is unique by the Cauchy--Lipschitz theorem since the ODE $g'' + 2g' = g^2$ is locally Lipschitz.)
See Figure \ref{fig1} for numerical approximations of $g_\lambda$ for different values of $\lambda$. We stress that $g_\lambda$ will always denote the solution to this initial value problem, whereas $\gs_s$ will denote the function introduced in \eqref{E:gs}, which is a specific solution to the boundary value problem \eqref{E:g_ODE}. (In particular, $\gs_s$ also depends implicitly on the parameter $R$ which does not appear in the definition of $g_\lambda$.) 

\medskip

The main result of this section reads as follows:

\begin{theorem}[Probabilistic representations for non-positive solutions with small derivative]
\label{T:proba_representation}
% For all $\lambda \in [0,1)$,
The function
$g_\lambda$ is defined and non-positive on $[0,\infty)$ for every $\lambda\in [0,2)$. Moreover, for each $\lambda\in [0,1)$ and $R=e^L >1$, the function $g_\lambda$ has the following probabilistic representation on $[0,L]$:
\begin{equation}
\label{E:uniqueness_weak}
g_{\lambda} (x) = (e^{-x} R)^2 \left( 1 - \EXPECT{e^{-x}R,R}{ (1-s)^{N_1} } \right) = \gs_s(x)
\quad \quad \text{with} \quad \quad s = g_{\lambda}(L)
\end{equation}
for every $x\in [0,L]$.
% , where $s=g_\lambda(L)$.
\end{theorem}

\begin{remark}
Numerical approximations to $g_\lambda$ suggest that it does \emph{not} remain non-positive when $\lambda$ is large, see Figure~\ref{F:large_lambda}. We stress however that our probabilistic representation of solutions should already fail to hold at smaller values of $\lambda$ than this, where numerics (Figure~\ref{F:gg'}) suggest that solutions remain non-positive but do not depend monotonically on $\lambda$ as they do when the probabilistic representation is valid.
% Indeed, Lemma~\ref{L:critical_lambda} characterises the values of $\lambda$ at which our probabilistic representation holds in terms of the monotonicity of $g_\lambda(L)$ as a function of $\lambda$: 
\end{remark}

% h''-2 h'= h^2
%  Once h' non-negative, remains non-negative.
%  h'''-2h'' = 2hh'
%  f = h'
%  f''-2f' = 2h f, h negative
% f'' - 2f' + 4 f >= 0.
% (fe^{-2t})'' = f'' e^{ct}-2cf'e^{ct}+4f e^{ct} >= 0
%  (fe^{-2t})' >= f(x_0)e^[-2x_0]
%  f(x)e^(-2x)>= f(x_0)e^(-2x_0) x

\medskip

This result implies immediately that solutions to the boundary value problem \eqref{E:L_ODEv} with $s<0$ are unique \emph{provided that their derivative at $R$ is sufficiently small}. (This derivative is automatically non-negative by the maximum principle.) Indeed, if $\tilde \vs_s$ is a solution to the boundary value problem \eqref{E:L_ODEv} with $s<0$ and we define $\tilde \gs_s$ by $\tilde \gs_s=R^2e^{-2x}\tilde \vs_s(Re^{-x})$ then $\tilde \gs_s$ solves the boundary value problem \eqref{E:g_ODE} and has derivative at zero given by
\begin{equation}
\label{E:gs'(0)}
\tilde \gs_s'(0)=\left[-2R^2e^{-2x}\tilde\vs_s(Re^{-x})-R^3e^{-3x}\tilde\vs_s'(Re^{-x}) \right]\,\Big\rvert_{x=0} =-R^3\tilde\vs_s'(R).
\end{equation}
Thus, if $0\leq \tilde \vs'_s(R)<R^{-3}$ then $0\leq -\tilde \gs_s'(0)<1$ and it follows from Theorem~\ref{T:proba_representation} that $\tilde \gs_s=\gs_s$ and $\tilde \vs_s=\vs_s$. (The hypothesis that $\vs_s'<R^{-3}$ in this uniqueness statement can be replaced by the hypothesis that $\vs_s'<\lambda_c(R) \cdot R^{-3}$, where $\lambda_c(R)\geq 1$ is defined in Proposition~\ref{P:derivative_proba_sol}.)
It follows that rotationally invariant solutions to the boundary value problem 
\[
\left\{
\begin{array}{l}
\Delta \vs_s = \vs_s^2 \quad \text{in} \quad B(R) \setminus \overline{B(1)}, \\
\vs_s = s \quad \text{on} \quad \partial B(1), \\
\vs_s = 0 \quad \text{on} \quad \partial B(R),
\end{array}
\right.
\]
with $s<0$ 
are also unique provided that their normal derivative on the outer sphere $\partial B(R)$ is strictly smaller than~$R^{-3}$. 
% {\color{red}[TH: Is non-negativity automatic for some reason? I think this might follow from my comment after Lemma 2.1.]}

% \begin{remark}
% \end{remark}

\medskip

Surprisingly, this uniqueness property is \emph{false} without this restriction on the derivative. As can be seen from Figure \ref{fig1}, as $-g'_\lambda(0) = \lambda$ increases, the corresponding solutions evolve monotonically until at least $\lambda =1$, but cease to be monotone in $\lambda$ at larger values of $\lambda$. In particular, the solutions for different values of $\lambda$ can cross each other. Hence if $x$ is such a crossing point we get two solutions to \eqref{E:g_ODE} with the same boundary values at $0$ and $x$ yet different solutions in the interior $(0,x)$, which shows that the boundary value problem \eqref{E:g_ODE} does not have a unique solution. 
% (and thus for \eqref{E:L_ODEv}) cannot hold.
 
\medskip

By the standard Cauchy--Lipschitz uniqueness theorem for solutions to ordinary differential equations with fixed initial value and derivative, Theorem \ref{T:proba_representation} will directly follow from the following result, whose proof occupies the rest of this section.

\begin{proposition}[Properties of the derivative of $\gs_s$ at $0$]
\label{P:derivative_proba_sol} Let $R> 1$ and consider the function $\gs_s$ defined in \eqref{E:gs}. 
\begin{enumerate}
 \item The function $\gs_s$ is differentiable at $0$ for each $s \in (s_c(R),0]$.
 \item The map
$
\lambda_R : s \in (s_c(R),0] \mapsto - \gs_s'(0)$
% \in [0,\lambda_c(R))
is a continuous decreasing function. 
% , where $\lambda_c(R)$ is defined by $\lambda_c(R) = \lim_{s \to s_c(R)^+} \lambda_R(s)$.
% Moreover, 
\item If we define 
$\lambda_c(R)$ by $\lambda_c(R) = \lim_{s \downarrow s_c(R)} \lambda_R(s)$ then $\lambda_c(R) \geq 1$ for every $R > 1$.
\end{enumerate}
\end{proposition}

\begin{remark}
The critical point $\lambda_c(R)$, which is defined to be the maximal slope at the origin that can be obtained from our probabilistic solutions, will play an important role throughout this section; we advise the reader to internalise its definition before moving on.
It is unlikely that the lower bound $\inf_{R>1} \lambda_c(R) \geq 1$ is optimal, but for our purposes it is enough to know that $\inf_{R>1} \lambda_c(R) >0$. Numerical approximations suggest that $\inf_R \lambda_c(R)$ is roughly $2.4$, see Figure~\ref{F:turning_point}.
\end{remark}
%As already alluded to, on the other hand, it seems that $\inf_R \lambda_c(R)$ is finite, i.e. that for some value of $\lambda$ large enough , $g_\lambda$ does not take the form of our probabilistic solutions $\gs_s$, $s \in ( \sup_R s_c(R),0]$. See Figure \ref{fig1}.

% \medskip

We start the proof of Proposition \ref{P:derivative_proba_sol} by showing that  $\gs_s$ is differentiable at 0 for each $s \in (s_c(R), 0]$. (Equivalently, $\vs_s$ is differentiable at $R$ for each $s \in (s_c(R), 0]$.)

\begin{lemma}\label{L:ODE_differentiable}
Let $R>1$, let $s \in (-s_c(R),1]$, and let $v : [1,R] \to \R$ be a solution to the boundary value problem \eqref{E:L_ODEv}.
Then $v$ is differentiable at $R$. Moreover, if $s\neq 0$ then $v'(R) \neq 0$.
% as soon as $s \neq 0$.
\end{lemma}

(Recall that part of what it means for $v : [1,R] \to \R$ to be a solution to the boundary value problem \eqref{E:L_ODEv} is that $v$ is continuous on $[1,R]$.)

\begin{proof}[Proof of Lemma \ref{L:ODE_differentiable}]
Using the ODE, we have
\[
\frac{\d}{\d r} (r^{d-1} v'(r)) = r^{d-1} \left( v''(r) + \frac{d-1}{r} v'(r) \right) = r^{d-1} v(r)^2 \geq 0.
\]
In other words, $r \in (1,R) \mapsto r^{d-1} v'(r)$ is nondecreasing. It remains to show that the limit of this function is finite as $r \to R$; this will prove that $\lim_{r \uparrow R} v'(r)$ is well defined and finite and hence that $v$ is differentiable at $R$ by the mean value theorem. By continuity of $v$, there exists some constant $C>0$ such that for all $r \in [1,R]$,
\[
\frac{\d}{\d r} (r^{d-1} v'(r)) = r^{d-1} v(r)^2 \leq C.
\]
Integrating this between $r_1 < r_2$ gives
\[
r_2^{d-1} v'(r_2) - r_1^{d-1} v'(r_1) \leq C(r_2-r_1)
\]
which shows that $r \in (1,R) \mapsto r^{d-1} v'(r)$ remains bounded as $r \to R$. This  concludes the proof that $v$ is also differentiable at $R$.
The fact that $v'(R) \neq 0$ if $s \neq 0$ follows from uniqueness of backwards solutions with fixed value and derivative at $R$: the only function satisfying the ODE with $v(R) = v'(R) = 0$ is the zero function, which does not satisfy the appropriate boundary condition at $r=1$ when $s\neq 0$.
\end{proof}

% In Corollary \ref{C:exc} below, we combine Lemmas \ref{L:v_ODE} and \ref{L:ODE_differentiable} to show that $\gs_s$ is differentiable at $0$.

% We start this section by stating and proving a corollary which is a consequence of Lemmas \ref{L:v_ODE} and \ref{L:ODE_differentiable}.

Lemmas \ref{L:v_ODE} and \ref{L:ODE_differentiable} have the following  corollary.

\begin{corollary}[Relationship between derivatives and excursion measures]
\label{C:exc}
Let $R> 1$. 
For each $s \in (s_c(R),1]$, the limits
% of escape probabilities 
% exists and is non degenerate:
\begin{equation}
\label{E:L_exc1}
\lim_{r \uparrow R} \frac{1}{R-r} \PROB{r,R}{{N_1}>0} \qquad \text{and} \qquad \lim_{r \uparrow R} \EXPECT{r,R}{ 1 - (1-s)^{N_1} \vert {N_1}>0}
\end{equation}
exist and are non-degenerate in the sense that both are finite and the first is positive.
% belongs to $(0,\infty)$.
% Moreover, for each $s \in (s_c(R),1]$, the limit of conditional expectations
% \begin{equation}
% \label{E:L_exc2}
% \lim_{r \uparrow R} \EXPECT{r,R}{ 1 - (1-s)^{N_1} \vert {N_1}>0}
% \end{equation}
% exists and is finite.
In particular, the law of the number of pioneers ${N_1}$ under the conditional measure $\PROB{r,R}{ \,\cdot\, \vert {N_1}>0 }$ converges weakly as $r \uparrow R$ to some limiting measure 
% random variable ${N_1}$ under
$\P_{\mathrm{exc},R}$. The critical value $s_c(R)$ is equal to the radius of convergence of the Laplace transform of ${N_1}$ under $\P_{\mathrm{exc},R}$, so that
\[
s_c(R) = \inf \left\{ s <0: \EXPECT{\mathrm{exc},R}{ (1-s)^{N_1}  } < \infty \right\}.
\]
Finally, for each $s \in (s_c(R),0]$, the function $\gs_s$ is differentiable at $0$ and satisfies
\begin{equation}
\label{E:probabilistic_gs'}
\gs_s'(0) = R^3 \left( \lim_{r \uparrow R} \frac{1}{R-r} \PROB{r,R}{{N_1}>0} \right) \EXPECT{\mathrm{exc},R}{ 1 - (1-s)^{N_1} }.
\end{equation}
\end{corollary}

\begin{remark}
% As a side remark, let us note that convergence
While we have restricted ourselves to proving distributional convergence of the number of pioneers, the argument we use to prove Corollary~\ref{C:exc} could be used to prove convergence of the law of the 
of the whole branching process under $\PROB{r,R}{ \,\cdot\, \vert {N_1}>0 }$ as $r \uparrow R$.
% could be obtained from the above argument.
This would define an excursion version of the branching Brownian motion. This will not be needed in the rest of this article, and is probably well known in the folklore of the literature on this subject, although we could not find a reference. 
%\red{Do you have any reference here?}
\end{remark}

\begin{proof}[Proof of Corollary \ref{C:exc}]
Lemmas \ref{L:v_ODE} and \ref{L:ODE_differentiable} show that for all $s \in (s_c(R),1]$, $\vs_s$ is differentiable at $R$. This derivative has the following probabilistic interpretation:
\begin{align}
\label{E:derivative_proba}
\vs_s'(R) & =
\lim_{r \uparrow R} \frac{\vs_s(R) - \vs_s(r)}{R-r} = - \lim_{r \uparrow R} \frac{1}{R-r} \EXPECT{r,R}{ 1 - (1-s)^{N_1} } \\
& = - \lim_{r \uparrow R} \frac{1}{R-r} \PROB{r,R}{{N_1}>0} \EXPECT{r,R}{ 1 - (1-s)^{N_1} \vert {N_1}>0 }.
\nonumber
\end{align}
When $s=1$, the last expectation on the right hand side is equal to 1 showing that the first limit in \eqref{E:L_exc1} exists, is finite, and is nonzero since $\vs_1'(R)\neq 0$ by Lemma \ref{L:ODE_differentiable}.
The convergence of the second limit in \eqref{E:L_exc1} and the expression \eqref{E:probabilistic_gs'} for $\gs_s'(0)$ then follows from this together with \eqref{E:derivative_proba} and the identity $\gs_s'(0) = -R^3 \vs_s'(R)$, which follows from the same calculation as \eqref{E:gs'(0)}.

\medskip

The only item remaining to check is that if $s < s_c(R)$ then $\EXPECT{\mathrm{exc},R}{ (1-s)^{N_1}  } = +\infty$.
Let $k \geq 1$ and suppose $r>R/2$.
The probability that a single particle starting on $\partial B(r)$ creates at least $k$ pioneers on the unit sphere is at least the probability that this particle creates at least one pioneer on $\partial B(R/2)$ multiplied by the probability that a single particle starting on $\partial B(R/2)$ creates at least $k$ pioneers on the unit sphere. Hence, for all $k \geq 1$,
\[
\PROB{\mathrm{exc},R}{{N_1} \geq k} = \lim_{r \uparrow R} \frac{\PROB{r,R}{{N_1} \geq k}}{\PROB{r,R}{{N_1} >0}}
\geq \left( \lim_{r \uparrow R} \frac{\PROB{r,R}{N_{R/2} >0}}{\PROB{r,R}{N_1 >0}} \right) \PROB{R/2,R}{{N_1} \geq k}.
\]
Reaching $\partial B(R/2)$ is easier than reaching $\partial B(1)$, so the ratio on the right hand side is at least 1. This proves that
$
\PROB{\mathrm{exc},R}{{N_1} \geq k} \geq \PROB{R/2,R}{{N_1} \geq k}
$ for every $k\geq 1$ and hence that the expectation of $(1-s)^{N_1}$ is also infinite under the excursion measure $\P_{\mathrm{exc},R}$ whenever $s < s_c(R)$.
This completes the proof.
\end{proof}

Corollary \ref{C:exc} shows in particular that $ - \gs_s'(0)$ is well-defined and finite, as claimed in Proposition \ref{P:derivative_proba_sol}. We may thus define $\lambda_R(s)$ (which appears in the statement of that proposition) as $\lambda_R(s) = - \gs_s'(0).$ Furthermore, using the expression \eqref{E:probabilistic_gs'}, we see immediately that $s\mapsto \lambda_R(s)$ is monotone and continuous (and indeed analytic) over $(s_c(R), 0]$.

To complete the proof of Proposition \ref{P:derivative_proba_sol}, it remains to prove the lower bound on the critical value $\lambda_c(R)$. (This is the most difficult part of the proposition to prove.) We now state two lemmas, proven in Sections \ref{SS:critical_lambda1} and \ref{SS:critical_lambda2} respectively, and show how they imply Proposition \ref{P:derivative_proba_sol}.

\medskip

The first of these lemmas, Lemma \ref{L:critical_lambda}, gives an alternative description of $\lambda_c(R)$ from a purely analytical point of view. This is a powerful tool to understand the critical point $\lambda_c(R)$, which was previously defined only implicitly as the maximal slope at the origin that can be obtained from our probabilistic solutions. This lemma will be proved in Section \ref{SS:critical_lambda1} using Pringsheim's theorem as a key input.

\begin{lemma}[Analytic description of $\lambda_c$]
\label{L:critical_lambda}
% $\lambda_c(R) = \tilde{\lambda}_c(R)$.
Recall that $R = e^L$. The equality
\[
 \lambda_c(R)=\sup \left\{ \lambda_2 >0: \lambda_1 \in (0,\lambda_2) \mapsto g_{\lambda_1}(L) \text{ is analytic and } \frac{\partial}{\partial \lambda_1} g_{\lambda_1}(L) <0 \text{ for every } \lambda_1 \in (0,\lambda_2)\right\}.
\]
holds for every $R=e^L>1$.
\end{lemma}

 % Assuming Lemma \ref{L:critical_lambda}, the lower bound on $\lambda_c(R)$ will follow from:
 The second lemma establishes various properties of the map $(\lambda,x)\mapsto g_\lambda(x)$, most notably that $g_\lambda(x)$ is monotone in $\lambda$ when $\lambda$ is sufficiently small. (The numerical solutions plotted in Figure~\ref{fig1} suggest that this is \emph{not} true for larger values of $\lambda$.) The proof of this result is contained in Section \ref{SS:critical_lambda2}, and is the most involved part of the proof of Proposition~\ref{P:derivative_proba_sol}.

\begin{figure}[t]
   \centering
   \includegraphics[trim=0.5cm 0.5cm 0.5cm 0.5cm, clip, height=3.8cm]{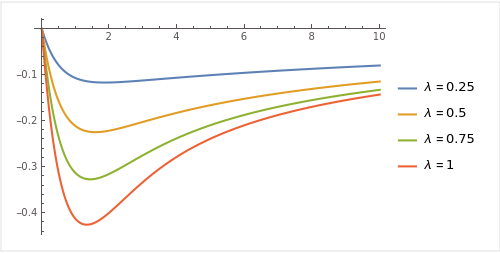}
      \includegraphics[trim=0.5cm 0.5cm 0.5cm 0.5cm, clip, height=3.8cm]{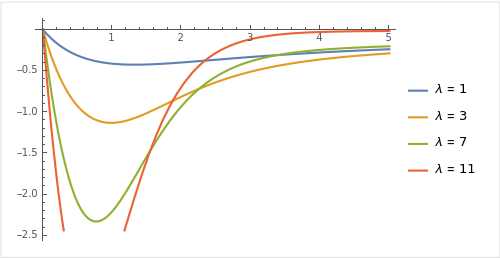}
% \includegraphics[trim=0.5cm 0.5cm 0.5cm 0.5cm, clip, height=3.7cm]{manuscript/images/large_lambda_critical.png}
   % \begin{subfigure}{.45\columnwidth}
   %  \def\svgwidth{\columnwidth}
   % \import{./figures/}{plot1.pdf_tex}
   % \caption{$\lambda = 0.25, ~0.5, ~0.75$ and $1$}
   % \end{subfigure}
   % \begin{subfigure}{.45\columnwidth}
   %  \def\svgwidth{\columnwidth}
   % \import{./figures/}{plot2.pdf_tex}
   % \caption{$\lambda = 3$ and $5$}
   % \end{subfigure}
\caption{Plots of numerical approximations of $g_\lambda(x)$ for different positive values of $\lambda$. Left: If $\lambda$ is small enough (e.g. $\lambda<1$), $g_\lambda$ possesses a probabilistic representation (Theorem \ref{T:proba_representation}), and the curves for different small values of $\lambda$ cannot cross each other. Right: The numerical solutions plotted here suggest that this non-intersection property fails for larger values of $\lambda$, even when the solutions are non-positive. Further numerics suggest that the critical value for solutions to be monotonic in $\lambda$ is around $2.4$, see Figure~\ref{F:turning_point}.}\label{fig1}
\end{figure}

\begin{figure}[t]
\centering
\includegraphics[trim=0.5cm 0.5cm 0.5cm 0.5cm, clip, width=0.65\textwidth]{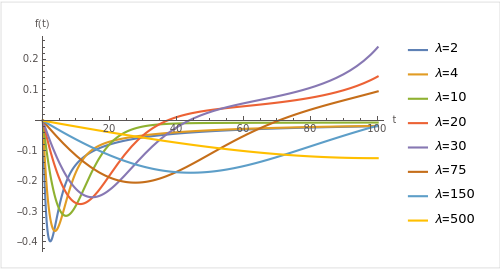}
\caption{
Left: Numerical plots of $u_\lambda(t):=\lambda^{-1} g_\lambda(\lambda^{-1} t)$ for various large values of $\lambda$. These numerical computations suggest that $g_\lambda$ \emph{can} become positive when $\lambda$ is large. The solutions with $\lambda=150$ and $\lambda=500$ cross the $x$-axis outside of the plot. Further numerics suggest that the critical value of $\lambda$ for $g_\lambda$ to remain non-positive is roughly $11.2$}\label{F:large_lambda}
% Right: The true critical value of $\lambda$ below which this solutions remain non-positive appears to be around $11.2$.}
\end{figure}
 
\begin{lemma}[Properties of $(\lambda,x)\mapsto g_\lambda(x)$]
% [Dependence of $g_\lambda$ on $\lambda$]
\label{L:weak_uniqueness_intermediate}
\hspace{1cm}
\begin{enumerate}
\item
For each $\lambda \in [0,2)$, the function $g_\lambda$, which was defined as the maximal solution to the initial value problem \eqref{E:g_lambda}, is defined and non-positive on $[0,\infty)$. 
\item The map $(\lambda,x) \in [0,2) \times [0,\infty) \mapsto g_\lambda(x)$ is real-analytic (viewed as a function of two real variables).
\item If $\lambda \in [0,1)$ then $\frac{\partial}{\partial \lambda} g_\lambda(x) < 0$ for every $x \geq 0$.
\end{enumerate}
\end{lemma}

 % We can now prove:
 We now briefly explain how Lemmas~\ref{L:critical_lambda} and \ref{L:weak_uniqueness_intermediate} imply Proposition \ref{P:derivative_proba_sol}.

\begin{proof}[Proof of Proposition \ref{P:derivative_proba_sol} given Lemmas~\ref{L:critical_lambda} and \ref{L:weak_uniqueness_intermediate}]
As already mentioned, by Corollary \ref{C:exc}, the function $\gs_s$ is differentiable at $0$ for every $s \in (s_c(R),0]$. The expression \eqref{E:probabilistic_gs'} for $\gs_s'(0)$ in terms of $s$ shows that $\lambda_R$ is continuous and decreasing. Finally, combining Lemmas \ref{L:critical_lambda} and \ref{L:weak_uniqueness_intermediate}, we obtain that $\lambda_c(R) \geq 1$.
\end{proof}

\subsection{Analytic description of \texorpdfstring{$\lambda_c$}{lambda c}: Proof of Lemma \ref{L:critical_lambda}}\label{SS:critical_lambda1}

In this section we prove Lemma \ref{L:critical_lambda}. We define
% Define
\begin{equation*}
\tilde{\lambda}_c(R) := \sup \left\{ \lambda_2 >0: \lambda_1 \in [0,\lambda_2] \mapsto g_{\lambda_1}(L) \text{ is analytic and } \frac{\partial}{\partial \lambda_1} g_{\lambda_1}(L) <0 \text{ for every } \lambda_1 \in [0,\lambda_2]\right\}
\end{equation*}
and aim to show that $\lambda_c(R)=\tilde \lambda_c(R)$.

\begin{proof}[Proof of Lemma \ref{L:critical_lambda}]
We start by showing that $\lambda_c(R) \leq \tilde{\lambda}_c(R)$.
By Corollary \ref{C:exc} (specifically \eqref{E:probabilistic_gs'}), the map $\lambda_R : s \in (s_c(R),0] \mapsto -\gs_s'(0) \in [0,\lambda_c(R))$ is analytic (since it is the Laplace transform of a random variable). Moreover, its derivative
\[
\frac{\d}{\d s} \lambda_R(s) = - R^3 \left( \lim_{r \uparrow R} \frac{1}{R-r} \PROB{r,R}{{N_1}>0} \right) \EXPECT{\mathrm{exc},R}{ {N_1} (1-s)^{{N_1}-1} }
\]
is negative on $(s_c(R),0]$ since ${N_1} \geq 1$ almost surely under the measure $\P_{\mathrm{exc},R}$. It follows that the inverse function
\[
\lambda \in [0,\lambda_c(R)) \mapsto g_\lambda(L) \in (s_c(R),0]
\]
is also analytic and has negative derivative. This implies that $\lambda_c(R) \leq \tilde{\lambda}_c(R)$.
% By definition of $\tilde{\lambda}_c(R)$, this shows that $\lambda_c(R) \leq \tilde{\lambda}_c(R)$.

\medskip

We now apply Pringsheim's theorem to prove the reverse inequality. Assume for contradiction that there exists $\eps >0$ such that $\lambda_c(R) + \eps < \tilde{\lambda}_c(R)$. By definition of $\tilde{\lambda}_c(R)$, this means that the map $f: \lambda \in [0,\lambda_c(R)+\eps] \mapsto g_\lambda(L)$ is analytic and has negative derivative. Let $[s_c(R) - \delta,0]$ be the image of $[0,\lambda_c(R)+\eps]$ under $f$. Since the derivative of $f$ is negative, $\delta$ is positive. Moreover, the inverse function $f^{-1} : [s_c(R) - \delta, 0] \to [0,\lambda_c(R)+\eps]$ is analytic and agrees with $\lambda_R$ on $(s_c(R),0]$. The contradiction comes from the fact that $s_c(R)$ is a singularity of $\lambda_R$, meaning that $\lambda_R$ cannot be analytically extended beyond $s_c(R)$. Indeed, $\lambda_R$ is a power series in $-s$ with non negative coefficients (see \eqref{E:probabilistic_gs'}) and $|s_c(R)|$ is its radius of convergence. By Pringsheim's theorem, $s_c(R)$ is therefore a singularity of $\lambda_R$.
This shows that $\lambda_c(R) + \eps$ cannot be smaller than $\tilde{\lambda}_c(R)$ which concludes the proof of the reverse inequality $\lambda_c(R) \geq \tilde{\lambda}_c(R)$.
\end{proof}

\begin{figure}[t]
\centering
\includegraphics[trim=0.5cm 0.5cm 2cm 1.5cm, clip, width=0.45\textwidth]{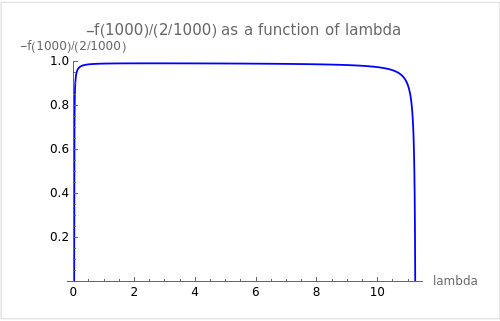} 
\hfill
\includegraphics[trim=0.5cm 0.5cm 2cm 1.5cm, clip, width=0.45\textwidth]{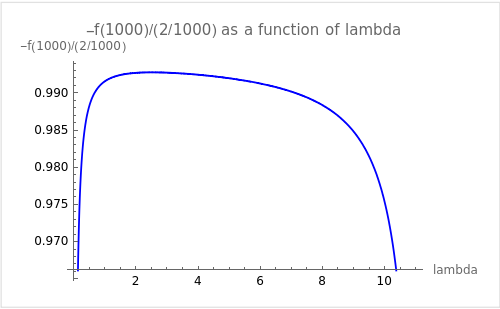}
\caption{Numerical plots of $-\frac{x}{2}g_\lambda(x)$ with $x=1000$ as a function of $\lambda$. Left: The asymptotics $g_\lambda(x)\sim -2/x$ as $x\to \infty$, proven for fixed $\lambda \in (0,1)$ in Section~\ref{sec:series_expansion}, is verified numerically for a larger range of $\lambda$ values. The convergence appears to hold uniformly away from $0$ and some critical value around $11.25$. Right: Zooming in on this plot near its maximum. The maximum of the function is attained numerically at $\lambda \approx 2.43$, suggesting that the critical value for our probabilistic representation of solutions to hold is around $\inf _R \lambda_c(R) \approx \lambda_c(e^{1000})\approx 2.43$.}\label{F:turning_point}
% Right: The true critical value of $\lambda$ below which this solutions remain non-positive appears to be around $11.2$.}
\end{figure}

\subsection{A Gr\"onwall-type lemma}

Before moving to the proof of Lemma~\ref{L:weak_uniqueness_intermediate}, we state and prove an elementary Gr\"onwall-type lemma which will be used in several places in this paper (including the proof of Lemma~\ref{L:weak_uniqueness_intermediate}). Taking $u=0$ in the first case yields the standard form of Gr\"onwall's lemma.

\begin{lemma}\label{L:Gronwall}
Let $a, b\in \R$ be such that $a < b$, and suppose that we are given two functions $u : [a,b] \to \R$ and $f :[a,b] \to \R$, where $u\in L^1([a,b])$ and $f$ is differentiable.
% Suppose that $f$ is differentiable and 
\begin{enumerate}
\item If there exists a constant $c$ such that
\begin{equation}\label{eq:G1}
f'(x) + cf(x) \leq u(x) \qquad \text{ for every } x\in (a,b)
\end{equation}
% for some constant $c \in \R$.
% Then for all $x,y \in [a,b]$, $x<y$,
then
\[e^{cy} f(y)  - e^{cx} f(x) \leq \int_x^y e^{ct} u(t) \d t\]
for every $x,y\in [a,b]$ with $x<y$.
\item If there exists a constant $c$ such that
% On the other hand, if instead $f$ and $u$ satisfy the inequality
% 
\begin{equation}\label{eq:G2}
x f'(x) + c f(x) \leq u(x) \qquad \text{ for every } x\in (a,b)
\end{equation}
then 
\[
y^c f(y) - x^c f(x) \leq \int_x^y t^{c-1} u(t) \d t
\]
for every $x,y\in [a,b]$ with $x<y$.
\end{enumerate}
We can reverse the signs of both inequalities in both cases.
\end{lemma}

\begin{proof}
First suppose that the inequality \eqref{eq:G1} holds. Multiplying the inequality by $e^{c x}$ leads to the differential inequality
$
\frac{\d}{\d x}(e^{cx} f(x)) \leq u(x) e^{cx}
$
for every
$x \in (a,b)$, which can then be integrated to reach the desired conclusion.
% and integrating this inequality between $x$ and $y$ leads to the desired conclusion.
Now suppose that the inequality \eqref{eq:G2} holds. In this case we multiply by $x^c$ to get that
$\frac{\d}{\d x} (x^c f(x)) \leq x^{c-1} u(x)$ for every $x\in(a,b)$, from which we may conclude as before. Both proofs work identically if the directions of all inequalities are reversed.
% We conclude by integrating between $x$ and $y$.
\end{proof}

\subsection{Proof of Lemma \ref{L:weak_uniqueness_intermediate}}\label{SS:critical_lambda2}

In this section we prove Lemma~\ref{L:weak_uniqueness_intermediate}, thereby concluding the proofs of Proposition~\ref{P:derivative_proba_sol} and Theorem~\ref{T:proba_representation}.

\medskip

We begin by observing that if $g_1$ and $g_2$ are solutions to the ODE $g'' + 2g' = g^2$, then the difference $f = g_2 - g_1$ satisfies the ODE
\[f'' + 2f' + \eps f =0 \qquad \text{ where } \qquad \eps = -g_1 - g_2.\]The following result studies properties of solutions to such equations that hold whenever $\eps$ is non-negative and bounded away from $1$. 
This will be useful when we will want to show that $\frac{\d}{\d \lambda} g_\lambda(x) < 0$ for all $x>0$ as stated in Lemma \ref{L:weak_uniqueness_intermediate}.

\begin{lemma}\label{L:ODE_negative}
Let $\eps : [0,\infty) \to [0,\infty)$ be a continuous nonnegative function with $\eta := \sup_{[0,\infty)} \eps < 1$. The maximal solutions to the ODE $f'' + 2f' + \eps f =0$ with $f(0) = 0$ and $f'(0) <0$ are defined on $[0,\infty)$ and satisfy $f(x) < 0$ for all $x>0$. More quantitatively,
the bound
% for all $x \geq 0$,
\begin{equation}
\label{E:L_ODE_negative}
f(x) \leq \frac{f'(0)}{2 \sqrt{1-\eta}} \left( e^{-(1-\sqrt{1-\eta})x} - e^{-(1+\sqrt{1-\eta}) x} \right)
\end{equation}
holds for every $x\geq 0$.
\end{lemma}

\begin{proof}[Proof of Lemma~\ref{L:ODE_negative}]
Since the ODE  $f''+2f' + \eps f = 0$ is linear, the fact that its maximal solutions 
% to the linear ODE $f''+2f' + \eps f = 0$ 
are defined on $[0,\infty)$ is standard and requires only that $\eps$ is continuous (or bounded). We focus on showing that if $f(0) = 0$ and $f'(0) < 0$ then $f$ stays nonpositive.
Let $c_\pm = 1 \pm \sqrt{1-\eta}$, which are real since $\eta<1$. 
Let $x_* = \inf \{ x > 0: f(x)  = 0 \}$ (with $\inf \emptyset = + \infty$ as usual). We are going to show that $x_* = +\infty$.
On $[0,x_*)$, we have $f'' + 2f' + \eta f \leq 0.$ The constants $c_+$ and $c_-$ have been defined so that $c_+ + c_- = 2$ and $c_+ c_- = \eta$, allowing us to write 
\begin{align*}
(f' + c_+ f)' + c_- (f' + c_+ f) = f'' + 2f' + \eta f \leq 0
\end{align*}
for every $x\in [0,x_*)$.
By Lemma \ref{L:Gronwall}, this leads to $f'(x) + c_+ f(x) \leq f'(0) e^{-c_- x}$ for all $x < x_*$. Applying Lemma \ref{L:Gronwall} once more, we get that
\[
f(x) \leq \frac{f'(0)}{2 \sqrt{1-\eta}} \left( e^{-(1-\sqrt{1-\eta})x} - e^{-(1+\sqrt{1-\eta}) x} \right) \quad \text{for every $x \in[0,x_*)$}.
\]
Since $f'(0)<0$, the function on the right hand side stays (strictly) negative on $(0,\infty)$  showing that $x_* = +\infty$. This concludes the proof.
\end{proof}

We will also need the following intermediate result in order to prove Lemma \ref{L:weak_uniqueness_intermediate}.

\begin{lemma}\label{L:g_small_first}
If $\lambda \in [0,2)$, then $g_\lambda$ is defined on $[0,\infty)$, is negative on $(0,\infty)$, and satisfies 
\begin{equation}
\label{E:L_g_small_first}
\lim_{x\to\infty}g_\lambda(x)=\lim_{x\to\infty}g_\lambda'(x)=0,
\end{equation}
uniformly on compact subsets of $\{ \lambda \in [0,2)\}$.
% $g_\lambda(x) \to 0$ and $g_\lambda'(x) \to 0$ as $x \to \infty$.
Moreover, $\sup_{[0,\infty)} |g_\lambda| \leq \lambda/2$ and $\sup_{[0,\infty)} |g_\lambda'| \leq 2\lambda$ for every $\lambda \in [0,2)$.
\end{lemma}

\begin{remark}
\label{remark:unimodal}
The proof of this lemma also shows that if $\lambda \in (0,2)$ then $g_\lambda$ is unimodal, i.e., is decreasing on some interval $[0,x_0]$ and increasing on the complementary interval $[x_0,\infty)$ for some $x_0=x_0(\lambda)>0$; we give a formal statement of this in Lemma~\ref{L:unimodal}. 
% \red{[TH: Do we want to state this properly and also prove unimodality for $g'$? This would fit nicely with the "three phases" picture in the next section, whose boundary points are the unique global optima of $g$ and $g'$ respectively. Maybe this discussion can go in Section 2.]}
\end{remark}

\begin{proof}[Proof of Lemma~\ref{L:g_small_first}]
Let $\lambda \in [0,2)$ and let $I$ be the maximal subinterval of $[0,\infty)$ containing $0$ on which $g_\lambda$ is defined. Integrating the inequality $g_\lambda'' + 2g_\lambda' \geq 0$ (using Lemma \ref{L:Gronwall}) leads to
\begin{equation}
\label{E:proof_negative3}
g_\lambda'(x) \geq - \lambda e^{-2x}
\quad \text{and} \quad
g_\lambda(x) \geq - \frac{\lambda}{2} (1-e^{-2x})	\qquad \text{ for every } x \in I.
\end{equation}
As in the proof of Lemma \ref{L:ODE_negative}, we set $x_* := \inf \{ x \in I: g_\lambda(x) \geq 0\}$ and will argue that $x_*=\sup I = \infty$. For each $x\in [0,x_*)$ we have by \eqref{E:proof_negative3} that $-\lambda/2 \leq g_\lambda \leq 0$ and hence that 
\[g_\lambda'' + 2 g_\lambda' = g_\lambda^2 \leq -\lambda g_\lambda /2 \qquad \text{ for every $x\in [0,x_*)$.}\] 
Using that $\lambda/2<1$ and $g_\lambda'(0)=-\lambda$ and arguing as in the proof of Lemma \ref{L:ODE_negative}, this implies that
\begin{equation}
\label{E:proof_negative1}
g_\lambda(x) \leq - \frac{\lambda}{2 \sqrt{1-\lambda/2}} \left( e^{-(1-\sqrt{1-\lambda/2})x} - e^{-(1+\sqrt{1-\lambda/2}) x} \right)
\end{equation}
for every $x < x_*$, while the inequality written in the proof of Lemma \ref{L:ODE_negative} as $f'(x)+c_+f(x)\leq f'(0)e^{-c_-x}$ applied with $f=g_\lambda$ and $\eta=\lambda/2$ yields that
% e following two estimates on $g_\lambda'$ and $g_\lambda$:
\begin{align}
\label{E:proof_negative2}
g_\lambda'(x) & \leq -\lambda e^{-(1-\sqrt{1-\lambda/2})x} - (1+\sqrt{1-\lambda/2}) g_\lambda(x) \\
& \leq -\lambda e^{-(1-\sqrt{1-\lambda/2})x} + (1+\sqrt{1-\lambda/2}) \frac{\lambda}{2} (1-e^{-2x}).
\nonumber
\end{align}
The inequality 
\eqref{E:proof_negative1} implies that $x_* = \sup I$. Together, these estimates prove that $g_\lambda$ and $g_\lambda'$ are uniformly bounded on $I$, which implies by  standard ODE theory that $g_\lambda$ is defined on $[0,\infty)$. 
% {\color{red}[TH: Is this OK or do we want a little more detail?]}
The fact that $\sup |g_\lambda| \leq \lambda/2$ and $\sup |g_\lambda'| \leq 2\lambda$ follows from \eqref{E:proof_negative3}, \eqref{E:proof_negative1}, and \eqref{E:proof_negative2} using the inequality
$\lambda + \frac{1}{2}(1+\sqrt{1-\lambda/2}) \lambda \leq 2 \lambda$.

\medskip

It only remains to check that $g_\lambda(x) \to 0$ and $g_\lambda'(x) \to 0$ as $x \to \infty$ uniformly on compact subsets of $\{ \lambda \in [0,2)\}$. 
First note that using the first line of
\eqref{E:proof_negative3}, we find $\limsup g'_\lambda(x) \le 2 \limsup |g_\lambda(x)| =0$. Together with the lower bound in \eqref{E:proof_negative1}, this shows that the uniform convergence of $g_\lambda(x)$ to 0 implies that of $g'_\lambda(x)$ to 0. We therefore focus on $g_\lambda$. Let $\lambda_* \in (0,2)$ and let $\eps>0$ be small.  
We define the following quantities: 
\begin{itemize}
    \item  $x_0 = \inf \{ x>0 : g'_\lambda(x) = 0 \}$

    \item $x_\eps = \inf \{x > 0: \sup_{[x,\infty)} |g_\lambda| \leq \eps \}$

    \item $y_\eps = \inf\{ x>x_0: g_\lambda(x) \ge - \eps$\}. 
\end{itemize}
We want to show that for any fixed $\eps>0$, $x_\eps$ is bounded uniformly in $\lambda \in [0,\lambda_*]$. If $\lambda \in [0,2\eps]$, then $\sup |g_\lambda| \leq \lambda/2 \leq \eps$ and $x_* = 0$. Take now $\lambda \in [2\eps, \lambda_*]$. Let us first show that $x_0$ is bounded uniformly.  If $x_0 \le 1$ there is nothing to show. 
 If $x_0$ is larger than 1, then for all $x \in [1,x_0]$, $g_\lambda''(x) + 2g_\lambda'(x) = g_\lambda(x)^2 \geq g_\lambda(1)^2$. Integrating this differential inequality (Lemma \ref{L:Gronwall}) gives that $g_\lambda'(x) \geq \frac12 g_\lambda(1)^2 (1-e^{2-2x}) + g_\lambda'(1)e^{2-2x}$ for every $x\in [1,x_0)$. In particular, this shows that $x_0$ is finite. Because $|g_\lambda(1)|$ is bounded away from zero uniformly in $\lambda \in [2\eps, \lambda_*]$ (which follows from \eqref{E:proof_negative1}) and because $|g_\lambda'(1)| \leq 2 \lambda$ is uniformly bounded away from infinity, the bound we obtain on $x_0$ is uniform in $\lambda \in [2\eps, \lambda_*]$.

Next, integrating the inequality $g_\lambda''+2g_\lambda' > 0$ between $x_0$ and $x>x_0$ (Lemma \ref{L:Gronwall}) yields that $g_\lambda'(x)>0$ for all $x>x_0$.
Therefore, for all $x \ge y_\eps$ we have $g_\lambda(x) \ge -\eps$, hence $x_\eps \le y_\eps$. 
It therefore suffices to show  that $y_\eps-x_0$ is bounded uniformly in $\lambda$ to conclude.
Either $y_\eps = x_0$, or for all $x \in (x_0,y_\eps)$,
$g_\lambda''(x) + 2g_\lambda'(x) = g_\lambda(x)^2 \geq \eps^2$. Integrating this inequality between $x_0$ and $x \in (x_0,y_\eps)$ with Lemma \ref{L:Gronwall}, we obtain that $g'(x) \geq \frac12 \eps^2 (1-e^{2(x_0-y_\eps)})$. Integrating further yields $g(y_\eps)-g(x_0) \geq \frac12 \eps^2 (y_\eps-x_0) - \frac14 \eps^4 (1-e^{2(x_0-y_\eps])})$. This provides the desired uniform bound on $y_\eps-x_0$. This concludes the proof.
\end{proof}

We conclude this section with the proof of Lemma \ref{L:weak_uniqueness_intermediate}.

\begin{proof}[Proof of Lemma \ref{L:weak_uniqueness_intermediate}]
We have already shown that $g_\lambda$ is well-defined and non-positive on $[0,\infty)$ when $\lambda \in [0,2)$. 
We next show that $(\lambda,x) \in [0,2) \times [0,\infty) \mapsto g_\lambda(x)$ is analytic. In fact, we will show  the stronger statement \eqref{E:analytic2} below.
For $x_0 \leq 0$ and $\mu \in \R$, let us denote by $g_{x_0,\mu}$ the unique maximal solution to the Cauchy problem $g''+2g' = g^2, g(0) = x_0$ and $g'(0) = \mu$. Fix $\lambda_*\in(0,2)$ and let
\[
\text{Data} = \{ (g_\lambda(x), g_\lambda'(x)): \lambda \in [0,\lambda_*], x \in [0,\infty) \}
\]
be the set of possible initial conditions (i.e., points in phase space) that are visited by the set of solutions to our original initial value problem with $\lambda \in [0,\lambda_*]$.
Lemma \ref{L:g_small_first} (specifically the uniform convergence \eqref{E:L_g_small_first}) implies that this set is a compact subset of $\R^2$.
We claim that the function
\begin{equation}
\label{E:analytic2}
(x_0,\mu,x) \in \mathrm{Data} \times [0,\infty) \mapsto (g_{x_0,\mu}(x),g_{x_0,\mu}'(x)) \in (-\infty,0] \times \R \quad \text{is~analytic.}
\end{equation}
The analyticity of $(\lambda,x) \in [0,2)\times [0,\infty) \mapsto g_\lambda(x)$ will be a direct consequence of \eqref{E:analytic2} since $\lambda_*$ can be as close to 2 as desired.
We now explain the proof of this claim. This will be a quick consequence of the fact that the ODE is autonomous and that the set $\text{Data}$ is compact (it is likely that the detailed derivation below is standard). 
Indeed, by compactness of the set Data and the Cauchy--Kovalevskaya theorem, there exists $\eps>0$ such that for all $(x_0,\mu) \in \text{Data}$, 
\begin{equation}
\label{E:analytic}
(x_0,\mu,x) \in \text{Data} \times [0,\eps] \mapsto (g_{x_0,\mu}(x),g_{x_0,\mu}'(x))
\quad \text{is~analytic.}
\end{equation}
For all $(x_0,\mu) \in \text{Data}$ and $x \in [\eps/2,3\eps/2]$, we have
\[
\left( g_{x_0,\mu}(x), g_{x_0,\mu}'(x) \right) = \left( g_{\hat{x}_0,\hat{\mu}}(x-\eps/2), g_{\hat{x}_0,\hat{\mu}}'(x-\eps/2) \right)
\quad \text{where} \quad \hat{x}_0 = g_{x_0,\mu}(\eps/2), \hat{\mu} = g'_{x_0,\mu}(\eps/2).
\]
By \eqref{E:analytic}, $\hat{x}_0$ and $\hat{\mu}$ are analytic functions of $x_0$ and $\mu$. Also, by definition of $\text{Data}$, $(\hat{x}_0, \hat{\mu}) \in \text{Data}$. So again by \eqref{E:analytic} $\left( g_{\hat{x}_0,\hat{\mu}}(x-\eps/2), g_{\hat{x}_0,\hat{\mu}}'(x-\eps/2) \right)$ are analytic functions of $\hat{x}_0$, $\hat{\mu}$ and $x \in [\eps/2,3\eps/2]$. Overall, we have written 
\[
(x_0,\mu,x) \in \text{Data} \times [\eps/2,3\eps/2] \mapsto (g_{x_0,\mu}(x),g_{x_0,\mu}'(x))
\]
as the composition of analytic functions. Putting things together, we have shown that the above map is analytic on $\text{Data} \times [0,3\eps/2]$. We can iterate this procedure to obtain analycity on $\text{Data} \times [0,\infty)$ as claimed in \eqref{E:analytic2}. 

\medskip

To conclude the proof, it remains to show that $\frac{\d}{\d \lambda} g_\lambda(x) <0$ for every $x \geq 0$ and $\lambda \in (0,1)$. Let $0 < \lambda_1 < \lambda_2 < \lambda_* < 1$.
We observe that the function $f = g_{\lambda_2} - g_{\lambda_1}$ satisfies the ODE
$f'' + 2f' = -\eps f$, $f(0)=0$ and $f'(0) = \lambda_1 - \lambda_2 <0$
with $\eps = -(g_{\lambda_1} + g_{\lambda_2})$. By Lemma \ref{L:g_small_first}, $\eps$ is nonnegative and $\sup_{[0,\infty)} \eps \leq (\lambda_1+\lambda_2)/2 < \lambda_*$ which is smaller than 1. Hence, by Lemma \ref{L:ODE_negative}, the function $f$ is negative on $(0,\infty)$ and satisfies the quantitative bound
\[
g_{\lambda_2}(x) - g_{\lambda_1}(x) \leq  - \frac{\lambda_2 - \lambda_1}{2\sqrt{1-\lambda_*}} \left( e^{-(1-\sqrt{1-\lambda_*})x} - e^{-(1+\sqrt{1-\lambda_*}) x} \right)
\]
for every $x\geq 0$, which implies that $\frac{\d}{\d \lambda} g_\lambda(x) <0$ for all $x \geq 0$ and $\lambda \in (0,\lambda_*)$. This concludes the proof since $\lambda_*$ can be as close to 1 as desired.
% of Lemma \ref{L:weak_uniqueness_intermediate}.
\end{proof}

\section{Tail estimates with killing I: Upper bounds}\label{sec:upper_bound}
In this section, we initiate the study of the long time behaviour of solutions to $g'' + 2g' = g^2$. Together with Theorem \ref{T:proba_representation}, this will be already enough to prove the upper bounds of Theorems \ref{T:maintail4} and \ref{T:thick}.
% with the former implying and the upper bounds concerning the killed process in Theorem \ref{T:tail_general}  (i.e., the upper bound of \eqref{E:T_killing}). \red{[AJ: I don't quite understand this sentence]} 
We prove relevant estimates on solutions to the ODE in Section~\ref{SS:g_small} then deduce the desired upper bounds on the tail of the number of pioneers in Section~\ref{SS:upper_bound}. We start in Section \ref{SS:hitting-proba-dirichlet} by recording for ease of future reference a lemma concerning the leading order term of the probability of hitting a small ball.

% \red{[blah]}

% Applying this theorem with $x_*=\log R$ yields the estimate \eqref{E:T_killing}. 
% We will prove the upper bound of this theorem in this section and the lower bound in \ref{sec:tail4}. We will return to the study of branching Brownian motion \emph{without killing} in Section~\ref{S:no_killing}.

\subsection{Preliminary estimates on the hitting probability}\label{SS:hitting-proba-dirichlet}

In Lemma \ref{L:hit_killing} below, we give preliminary estimates on the hitting probability both in finite and infinite volumes. These will be useful throughout the paper. Later, we will provide much sharper asymptotics in the infinite volume case; see Theorems \ref{T:hitprob4} and \ref{T:gh}.

%In Theorems \ref{T:hitprob4} and \ref{T:gh} we give very precise asymptotics on hitting probabilities in \emph{infinite} volume. For our study of thick points we will also need to have an estimate in \emph{finite} volume. Fortunately, all our applications will only require a leading-order estimate on this hitting probability, which we prove in this section.
% However, only the first term will be needed.

\begin{lemma}\label{L:hit_killing}
Let $x_0 >0$. For each $\delta >0$, there exists $R_0 >0$ and $x_*>x_0$ depending only on $x_0$ and $\delta$ such that for all $R \geq R_0$ and $r \in (0,e^{-x^*} R]$,
\begin{align}
\label{E:L_hit_no_killing}
\text{Without killing:} \quad
e^{-\delta} \frac{2}{(e^{-x_0}R)^2 \log (R/r)} & \leq
\PROB{e^{-x_0} R, \infty}{N_r > 0} \leq e^{\delta} \frac{2}{(e^{-x_0}R)^2 \log (R/r)}, \quad \text{and} \\
\label{E:L_hit_killing}
\text{With killing:} \hspace{51pt}
e^{-\delta} \frac{2(e^{2x_0} -1)}{R^2 \log (R/r)} & \leq
\PROB{e^{-x_0} R, R}{N_r > 0} \leq e^{\delta} \frac{2(e^{2x_0} -1)}{R^2 \log (R/r)}.
\end{align}
\end{lemma}

\begin{proof}
The estimate \eqref{E:L_hit_no_killing} is folklore. It follows also directly from Lemma \ref{lem:h_first_order}. We omit the details. We now explain how the estimate \eqref{E:L_hit_killing} is obtained.
Consider an intermediate scale $r' \in (r,e^{-x_0} R)$. Starting from $\partial B(e^{-x_0} R)$ and conditionally on $N_{r'}$, the ball $B(r)$ can be reached if, and only if, at least one pioneer on $\partial B(r')$ has had a progeny that reached $B(r)$. To proceed, we will choose the intermediate scale $r'$ in such a way that the following two properties will hold:
\begin{enumerate}
\item $r'$ will be sufficiently large so that, conditional on the event that $B(r)$ is hit, there will with high probability be exactly one pioneer on $\partial B(r')$ whose descendents hit $B(r)$. This will mean that the probability $B(r)$ is hit starting from $e^{-x_0}R$ is asymptotic to the expected number of pioneers on $\partial B(r')$ multiplied by the probability one of these particles hits $B(r)$.
\item 
$r'$ will be sufficiently small that the hitting probability of $B(r)$ with and without killing will be comparable when starting from $\partial B(r')$. 
\end{enumerate}
Once such an intermediate scale is found, we will be able to compute the desired finite-volume hitting probability to leading order by multiplying the expected number of pioneers on $\partial B(r')$ (which is just a Green's function computation) with the infinite-volume hitting probability from $\partial B(r')$, which is estimated in \eqref{E:L_hit_no_killing}.

First of all, arguing as in Corollary \ref{C:outer_convergence}, there exists $C>0$ such that for all starting point $y \in B(e^{-x_0}R)$,
\begin{equation}
\label{E:november1}
\PROB{y,\infty}{N_R>0} \leq \frac{C}{R^2}.
\end{equation}
We now introduce some parameters and the relevant scales.
Let $\delta >0$. Let $x_*>x_0$ be large enough so that for all $x \geq x_*$,
\begin{equation}\label{E:delta-cond}
\delta x > x_0, \quad
1-e^{-2\delta x} \geq e^{-\delta} \quad \text{and} \quad
e^{2\delta x} \geq C \delta^{-1} x,
\end{equation}
where $C$ is the constant appearing in \eqref{E:november1}.
Let $x \geq x_*$, $\xi = \delta x$ and $R \geq 1$. Denote by $r = e^{-x} R$ the target scale and $r' = e^{-\xi} R$ the intermediate scale.
By \eqref{E:L_hit_no_killing}, we may further assume that $x_*$ is large enough so that
\begin{equation}
\label{E:november2}
e^{-\delta} \frac{2}{(r')^2 \log (r'/r)} \leq \PROB{r',\infty}{N_r>0} \leq e^{\delta} \frac{2}{(r')^2 \log (r'/r)}.
\end{equation}
As already mentioned, the ball $B(r)$ can be reached if, and only if, at least one pioneer on $\partial B(r')$ has a descendent that reaches $B(r)$. Denoting $p = \PROB{r',R}{N_{r}>0}$, we thus have
\begin{align*}
\PROB{e^{-x_0} R, R}{N_r > 0}
= \EXPECT{e^{-x_0} R, R}{1- (1-p)^{N_{r'}} }.
\end{align*}
We bound $np - \frac12 n(n-1) p^2 \leq 1 - (1-p)^n \leq np$ to get that
\begin{equation}
\label{E:pf-hitting-dirichlet}
p \EXPECT{e^{-x_0} R, R}{N_{r'}} - \frac{1}{2} p^2 \EXPECT{e^{-x_0} R, R}{N_{r'} (N_{r'}-1)}
\leq \PROB{e^{-x_0} R, R}{N_r > 0} \leq
p \EXPECT{e^{-x_0} R, R}{N_{r'}}.
\end{equation}
To conclude, we need to estimate the first and second moment of $N_{r'}$ and the probability $p$. 
In the remainder of the proof, we will write $e^{\pm \delta}$ for a quantity that is bounded between $e^{-\delta}$ and $e^\delta$, the precise value of which may vary from line to line.

The first moment of $N_{r'}$ is explicit and equal to the probability that a Brownian path starting on the sphere $\partial B(e^{-x_0}R,R)$ hits $B(r')$ before reaching $\partial B(R)$:
\[
\EXPECT{e^{-x_0}R,R}{N_{r'}} = \frac{e^{2x_0}-1}{R^2/(r')^2-1} = e^{\pm \delta} \frac{e^{2x_0}-1}{R^2} (r')^2,
\]
where we used the second inequality in \eqref{E:delta-cond} in the second equality. Following the same computation as in \cite[Proof of Lemma 3.5]{asselah2022time}, the second moment is bounded by 
\[
\EXPECT{e^{-x_0}R,R}{N_{r'}^2} \leq \EXPECT{e^{-x_0}R,\infty}{N_{r'}^2}
\leq \frac{C}{R^2} (r')^4 \log \frac{R}{r'}.
\]
We next prove upper and lower bounds on the  probability $p = \PROB{r',R}{N_{r}>0}$. For the upper bound, we can simply ignore finite-volume effects to write $p \leq \PROB{r',\infty}{N_{r}>0}$, then estimate this infinite-volume probability using \eqref{E:november2}. We obtain that
\[
p \leq \PROB{r',\infty}{N_{r}>0} \leq e^\delta \frac{2}{(r')^2 \log (r'/r)} \leq e^{3\delta} \frac{2}{(r')^2 \log (R/r)}
\]
where the last inequality follows from $\log (r'/r) = x - \xi = (1-\delta)x = (1-\delta) \log(R/r) \geq e^{-2\delta} \log(R/r)$. For the lower bound, we have by \eqref{E:november1} and \eqref{E:november2},
\begin{equation}
\label{E:pf-lemma-hit3}
p \geq \PROB{r',\infty}{N_{r}>0} - \PROB{r',\infty}{N_{R}>0} \geq
e^{-\delta} \frac{2}{(r')^2 \log (r'/r)} - \frac{C}{R^2}.
\end{equation}
Using the third inequality in \eqref{E:delta-cond},
we further have
\[
C \frac{(r')^2}{R^2} \log (r'/r) = C e^{-2\xi} (x-\xi) = C (1-\delta) e^{-2 \delta x} x \leq \delta,
\]
showing that
\[
p \geq e^{-3\delta} \frac{2}{(r')^2 \log (r'/r)} \geq e^{-3\delta} \frac{2}{(r')^2 \log (R/r)}.
\]
Putting things together, it follows that
\[
p \EXPECT{e^{-x_0}R,R}{N_{r'}} = e^{\pm 4 \delta} \frac{2(e^{2x_0}-1)}{R^2\log(R/r)}
\quad \text{and} \quad
p^2 \EXPECT{e^{-x_0}R,R}{N_{r'}^2} \leq \frac{C}{R^2} \frac{\log (R/r')}{(\log (r'/r))^2} \leq \frac{C' \delta}{R^2 \log (R/r)},
\]
and injecting these two estimates in \eqref{E:pf-hitting-dirichlet} concludes the proof of \eqref{E:L_hit_killing}.
\end{proof}

\subsection{Study of \texorpdfstring{$g_\lambda$}{g lambda} when \texorpdfstring{$\lambda$}{lambda} is small}\label{SS:g_small}

Recall that the function $g_\lambda$ was defined in \eqref{E:g_lambda} as the maximal solution to the initial value problem $g''_\lambda+2g'_\lambda=g_\lambda^2$, $g_\lambda(0)=0$, $g_\lambda'(0)=-\lambda$, which is defined on the entire interval $[0,\infty)$ when $0\leq \lambda <2$ by Lemma~\ref{L:g_small_first}.
Our goal in this section is to establish a precise quantitative understanding of $g_\lambda$ when $\lambda$ is a small positive number.

\medskip

Our 
analysis will be based on a decomposition\footnote{These three phases are analogous to what are called the ``inner layer'', the ``matching region'', and the ``outer layer'' in fluid dynamics. The interested reader can find many fascinating discussions of associated phenomena by searching for ``boundary layer theory''.} of the phase portrait of $g_\lambda$ into three ``phases'' as depicted in Figure \ref{F:gg'}.  Before stating precise results, let us first describe these phases intuitively. In the first phase, which can be thought of as the ``integrated exponential phase'', $g^2_\lambda$ is negligible compared to $g'_\lambda$ and $g''_\lambda$ and the derivative $g_\lambda'$ approximately solves the ODE $(g_\lambda')'=-2(g'_\lambda)$. This phase lasts until $g'_\lambda$ reaches $0$. Following the integrated exponential phase, $g_\lambda$ then stays roughly constant until $g''_\lambda$ reaches~$0$. We think of this as the ``constant phase''. Finally, after $g''_\lambda$ has reached $0$, it will stay much smaller than $g'_\lambda$ and $g^2_\lambda$. In this final phase, $g_\lambda$ approximately solves the ODE $2g'_\lambda=g^2_\lambda$ and $g_\lambda(x)$ looks very much like~$-2/x$; we think of this as the ``power law'' phase. The power-law phase has a much larger ``natural time scale'' than the other phases:
if $\lambda$ is very small and one graphs the function $g_\lambda$ on the natural time-scale of the third phase (where $x$ is proportional to $1/\lambda$), the first two phases look like a near-instantaneous change in the values of $g_\lambda$ and $g'_\lambda$, while the third phase looks almost exactly like the rational function $2\lambda/(4\lambda x+4)$. See Figure~\ref{F:gg'2} for plots of numerical approximations to $g_\lambda$ supporting this picture. Precise versions of each of these claims will be proven in the remainder of this section.

\medskip

Let us now make these claims precise. Consider the parameters $x_0=x_0(\lambda)$ and $x_1=x_1(\lambda)$ defined by
\[
x_0 := \inf \{ x > 0: g'_\lambda(x) = 0 \}
\quad \text{and} \quad
x_1 := \inf \{ x>x_0: g_\lambda'' = 0 \}.
\]
The three phases discussed above will correspond to the intervals $[0,x_0]$, $[x_0,x_1]$, and $[x_1,\infty)$. These phases also have the following interpretation:

\begin{figure}[t]
\centering
\includegraphics[trim=0.5cm 0.5cm 0.5cm 0.5cm, clip, height=4.3cm]{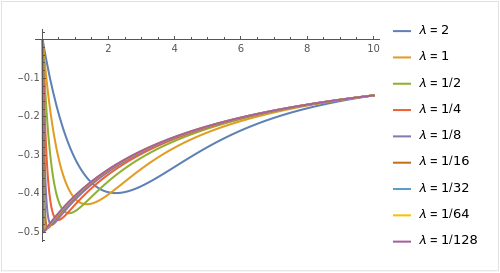} 
\hspace{0.25cm}
\includegraphics[trim=0.1cm 0.1cm 0.1cm 0.1cm, clip, height=4.3cm]{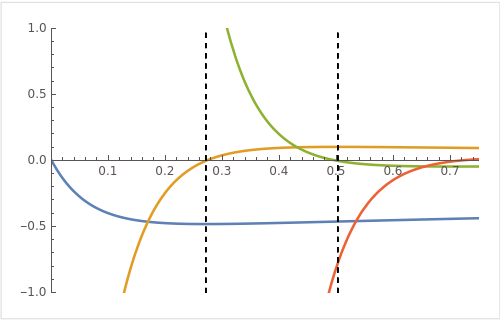}
\caption{
Left: Numerical plots of $u_\lambda(t):=\lambda^{-1} g_\lambda(\lambda^{-1} t)$ for various values of $0\leq \lambda \leq 2$. The first two phases become near-instantaneous under this scaling as $\lambda \downarrow 0$. Right: Numerical plots of $u_\lambda(t):=\lambda^{-1} g_\lambda(\lambda^{-1}t)$ (blue) and its first three derivatives (yellow, green, and red respectively) with $\lambda = 1/8$. The values $t=\lambda x_0$ and $t=\lambda x_1$ corresponding to $x_0$ and $x_1$ are represented by dashed lines. Here we have zoomed in on smaller values of $x$ to give more insight into the first two phases.}
\end{figure}\label{F:gg'2}

\begin{lemma}[Unimodality]
\label{L:unimodal}
Let $\lambda\in [0,2)$. Then $x_0$ and $x_1$ are both finite, the function $g_\lambda$ is strictly decreasing on $[0,x_0]$ and strictly increasing on $[x_0,\infty)$, and the derivative $g_\lambda'$ is strictly increasing on $[0,x_1]$ and strictly decreasing on $[x_1,\infty)$.
\end{lemma}

\begin{proof}
The definition of $x_0$ ensures that $g_\lambda$ is strictly decreasing on $[0,x_0]$, and since $g_\lambda(x)\to 0$ as $x\to \infty$ by Lemma~\ref{L:g_small_first} we must have that $x_0<\infty$. Meanwhile, using Lemma~\ref{L:Gronwall} to integrate the inequality $g''_\lambda+2g_\lambda'>0$ between $x>x_0$ and $x_0$ yields that $g'_\lambda(x)>0$ for all $x>x_0$, so that $g_\lambda$ is strictly increasing on $[x_0,\infty)$. We now prove the claim concerning $g'_\lambda$. When $x \in [0,x_0]$, $g_\lambda$ is negative and $g'_\lambda$ is non-positive, so that $g''_\lambda=g^2_\lambda-2g'_\lambda>0$ as desired, while the definition of $x_1$ ensures that $g'_\lambda$ is strictly increasing on $[x_0,x_1]$ also. Again, this easily implies that $x_1$ is finite since $g'_\lambda(x)\to 0$ as $x\to \infty$. Finally, differentiating both sides of the ODE $g''_\lambda+2g'_\lambda = g^2_\lambda$ yields that $g'''_\lambda+2g''_\lambda = 2g_\lambda g_\lambda'$, and since $g_\lambda$ is negative on $(0,\infty)$ and $g'_\lambda$ is positive on $(x_0,\infty)$ it follows that
$g'''_\lambda(x)+2g''_\lambda(x)<0$ for every $x>x_0$. Using Lemma~\ref{L:Gronwall} to integrate this inequality from $x_1$ to $x>x_1$ implies that $g''_\lambda(x)<0$ for every $x>x_1$, so that $g'_\lambda$ is strictly decreasing on $[x_1,\infty)$ as claimed.
\end{proof}

\begin{figure}[t]
\centering
\includegraphics[width=0.5\textwidth]{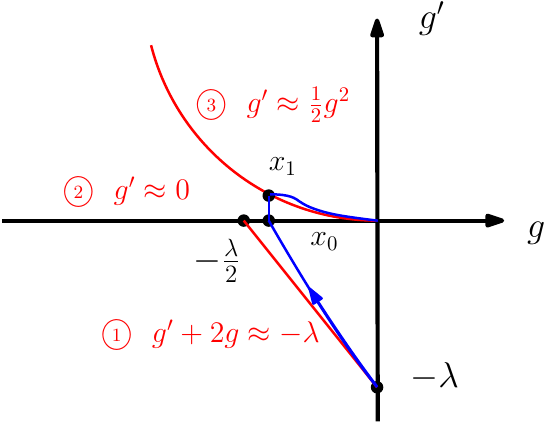}
\caption{
Schematic representation of the phase portrait of $g$. In blue, $g'$ against $g$ and in red approximations corresponding to different phases. In the first phase, $g^2$ is negligible compared to $g''$ and $g'$. In the second very short phase, $g$ is almost constant ($g'$ negligible). In the last one, $g''$ is negligible compared to $g'$ and $g^2$.}
\end{figure}\label{F:gg'}

We now begin to analyze these phases in more quantitative detail for small values of $\lambda$. Our first result in this direction provides estimates that are useful to understand the first phase $[0,x_0]$.

\begin{lemma}[Phase 1]
\label{L:Phase1}
% The estimate
We have that
  \begin{equation}\label{E:proof_ODE1}
x_0 = \frac{1}{2} \log \frac{8}{\lambda} + O \left( \lambda \log \frac{1}{\lambda} \right)
\quad \text{and} \quad
g_\lambda(x_0) = -\left(1+ O \left( \lambda \right) \right) \frac{\lambda}{2} \qquad \text{ as $\lambda \downarrow 0$}
\end{equation}
and that
\begin{equation}
\label{E:L_g_small_lambda2}
- \frac{\lambda}{2} (1-e^{-2x}) \leq g_\lambda(x) \leq - \frac{\lambda}{2} \left( 1 - e^{-2x} \right) + \frac{1}{8} \lambda^2 \left( x - \frac{1-e^{-4x}}{4} + 2x e^{-2x} - (1-e^{-2x}) \right)
\end{equation}
for every $\lambda \in [0,2)$ and $x\geq 0$.
% and that
% Moreover, 
% 
% Let $0\leq \lambda_*<2$. There exists $\lambda_1>0$ 
  % 
% \end{equation}
% as $\lambda \downarrow 0$.
% 
% for every $0\leq \lambda <2$, with all implicit constants being uniform on each compact subset of $[0,2)$.
% Moreover,
% the estimate
% \begin{equation}
% \label{E:L_g_small_lambda2}
% - \frac{\lambda}{2} (1-e^{-2x}) \leq g_\lambda(x) \leq - \frac{\lambda}{2} \left( 1 - e^{-2x} \right) + \frac{1}{8} \lambda^2 \left( x - \frac{1-e^{-4x}}{4} + 2x e^{-2x} - (1-e^{-2x}) \right)
% \end{equation}
% holds for every $\lambda \in [0,2)$ and $x\geq 0$.
\end{lemma}

\begin{proof}[Proof of Lemma~\ref{L:Phase1}]
To lighten notation we write $g=g_\lambda$.
The ODE satisfied by $g$ can be rewritten as
\begin{equation}
\label{E:proof_ODE10}
(g' e^{2x})' = g^2 e^{2x} \qquad \text{ for every $x\geq 0$}.
\end{equation}
% In particular,  the left hand side is non negative. 
It follows in particular that $(g' e^{2x})' \geq 0$, and
% is non-negative.
integrating this inequality between $0$ and $x$ leads to
\begin{equation}
\label{E:proof_ODE9}
g'(x) \geq -e^{-2x} \lambda \qquad \text{ and } \qquad g(x) \geq - \frac{\lambda}{2} (1-e^{-2x}).
\end{equation}
This lower bound on $g$ gives an upper bound on $g^2$, and plugging this upper bound back into \eqref{E:proof_ODE10} gives
\[
g'(x) e^{2x} + \lambda \leq \int_0^x \left( \frac{\lambda}{2} \right)^2 (e^t-e^{-t})^2 \d t
= \frac14 \lambda^2 ( \sinh(2x) - 2x ).
\]
Rearranging the above inequality, we have
\begin{equation}
\label{E:proof_ODE6}
g'(x) \leq - \lambda e^{-2x} + \frac{1}{8} \lambda^2 \left( 1 - e^{-4x} - 4xe^{-2x} \right).
\end{equation}
It follows that $g'(x)$ is negative for all $0\leq x \leq \frac{1}{2} \log \frac{8}{\lambda}$,
so that
% In other words,
\begin{equation}
\label{E:proof_ODE7}
x_0 \geq \frac{1}{2} \log \frac{8}{\lambda},
\end{equation}
yielding one side of the estimate on $x_0$ claimed in the statement of the lemma.
Next, integrating \eqref{E:proof_ODE6} gives
\begin{multline}
\label{E:proof_ODE20}
g(x)  \leq - \frac{\lambda}{2} \left( 1 - e^{-2x} \right) + \frac{1}{8} \lambda^2 \left( x - \frac{1-e^{-4x}}{4} + 2x e^{-2x} - (1-e^{-2x}) \right) \\
 = - \left( \frac{\lambda}{2} - O\left( \lambda^2 (1+x)\right) \right) \left( 1 - e^{-2x} \right),
\end{multline}
where the implied constants in the big-$O$ notation are uniform in $\lambda,x\geq 0$. Together with \eqref{E:proof_ODE9} this yields the bound on $g(x)$ claimed in \eqref{E:L_g_small_lambda2}.
% In particular, we see that for all $x \in [0,\log (1/\lambda)]$,
% \begin{equation}
% \label{E:proof_ODE8}
% g(x) \leq - \frac{\lambda}{2} \left(1 - C \lambda \log \frac{1}{\lambda} \right) \left( 1 - e^{-2x} \right).
% \end{equation}
% Recall that we obtained the upper bound \eqref{E:proof_ODE6} on $g'$ from the lower bound \eqref{E:proof_ODE9} on $g$. Similarly, 
Finally, using the ODE \eqref{E:proof_ODE10}, this upper bound on $g$ can be used to lower bound $g'$ in exactly the same way that we previously obtained the upper bound \eqref{E:proof_ODE6}  using the lower bound \eqref{E:proof_ODE9}. This yields that there exists a positive constant $C_1$ such that if $\lambda^2(1+x) \leq 1/C_1$ then
% the upper bound we now have on $g$ leads to the following lower bound on $g'$: for all $x \in [0,\log (1/\lambda)]$,
\[
g'(x) \geq - \lambda 
% \left(1 - C \lambda \log \frac{1}{\lambda} \right)
e^{-2x} + \frac{1}{8} \lambda^2 \left(1 - C_1\lambda^2(1+x) \right)^2 \left( 1 - e^{-4x} - 4xe^{-2x} \right).
\]
It follows that there exist positive constants $c$ and $C_2$ such that if $\lambda\leq c$ then $g'$ must vanish at some point $x$ smaller that $\frac{1}{2} \log \frac{8}{\lambda} + C_2 \lambda \log \frac{1}{\lambda}$, and hence that 
\[x_0 \leq \frac{1}{2} \log \frac{8}{\lambda} + C_2 \lambda \log \frac{1}{\lambda}\]
for every $0\leq \lambda \leq c$ as desired. 
% This is the desired upper bound on $x_0$. Together with \eqref{E:proof_ODE7}, this shows the estimate \eqref{E:proof_ODE1} on $x_0$.
The estimate on the value of $g(x_0)$ follows from \eqref{E:proof_ODE9} and \eqref{E:proof_ODE20}.
\end{proof}

We next study the second phase, where $g_\lambda$ remains roughly constant.

\begin{lemma}[Phase 2]
\label{L:Phase2}
$x_1 - x_0 \leq \frac12 \log{\frac{O(1)}{\lambda}}$,
$g_\lambda(x_1) = -\left(1+ O \left( \lambda |\log \lambda| \right) \right) \frac{\lambda}{2}$, and $g_\lambda'(x_1)=O(\lambda^2)$ 
as $\lambda \downarrow 0$.
\end{lemma}

\begin{proof}[Proof of Lemma~\ref{L:Phase2}]
By Lemma~\ref{L:unimodal}, $g$ is unimodal in the sense that $g' > 0$ on $(x_0,\infty)$. Since we also have by Lemma \ref{L:g_small_first} that $|g(x)| \leq \lambda / 2$ for all $x \geq 0$, and since the maximum of $g'(x)$ is attained at the point $x_1$ where $g''(x_1)=0$ and hence  $2g'(x_1)=g(x_1)^2 \leq \lambda^2/4$, it follows that
\begin{equation}
\label{E:Phase2b}
g'(x) \leq \frac{\lambda^2}{8},
\quad \quad x\geq x_0.
\end{equation}
Integrating shows that $g(x) - g(x_0) \leq \frac{\lambda^2}{8}(x-x_0)$ for $x \geq x_0$. Together with the estimate on $g(x_0)$ from Lemma~\ref{L:Phase1}, this yields the inequality
\begin{equation}
\label{E:Phase2a}
g(x) \leq - \frac{\lambda}{2} (1+O(\lambda |\!\log \lambda|)),
\quad \quad x \in [x_0,x_0+|\!\log \lambda|].
\end{equation}
We can now integrate the differential inequality
\[
(g')' + 2g' = g^2 \geq \frac{\lambda^2}{4} (1+O(\lambda |\log \lambda|))
\]
between $x_0$ and some $x \in [x_0,x_0+|\log \lambda|]$ (see Lemma \ref{L:Gronwall}) to get that
\[
g'(x) \geq \frac{\lambda^2}{8}(1+O(\lambda |\log \lambda|)) (1-e^{2(x_0-x)}),
\quad x \in [x_0,x_0+|\log \lambda|].
\]
This estimate together with \eqref{E:Phase2a} implies the following differential inequality for $g''$ on $[x_0, x_0+|\log \lambda|]$:
\[
(g'')' + 2g'' = 2g'g \leq - (1+O(\lambda |\log \lambda|)) \frac{1}{8} \lambda^3 (1-e^{2(x_0-x)}).
\]
Integrating {this inequality} using Lemma \ref{L:Gronwall} yields the following inequality for $x \in [x_0, x_0+|\log \lambda|]$,
\[
g''(x) \leq - (1+O(\lambda |\log \lambda|)) \frac{\lambda^3}{16} + \left( g''(x_0) + (1+O(\lambda |\log \lambda|)) \frac{\lambda^3}{16} (1+2(x-x_0))  \right) e^{2(x_0-x)}.
\]
By Lemma \ref{L:Phase1}, $g''(x_0) = g(x_0)^2 = (1+O(\lambda)) \frac{\lambda^2}{4}$, so the above right hand side becomes negative as soon as $x \geq x_0 + \frac{1}{2} |\log \lambda| + C$ for a sufficiently large constant $C$. That is,
\[
x_1 - x_0 \leq \frac{1}{2} |\log \lambda| + O(1).
\]
The estimates on $g(x_1)$ and $g'(x_1)$ stated in Lemma \ref{L:Phase2} then follow from \eqref{E:Phase2a} and \eqref{E:Phase2b} respectively.
\end{proof}

% 1+2(1/g)'=1-2g'/g^2 >=0
% (1/g)' >= -1
% (1/-g)' =< 1
% blows up in time at most 1/|g(0)|
% Wait till g hits -lambda/16.
% Wait till 
% |g|' >= |g|^2/2. 
% From time 0 to C, |g| changes by at least (C/2) lambda^2. So |g(C)|' must be at least lambda^2 / 2 by convexity. 

Finally, we study the third phase, where $g_\lambda$ has approximate power-law behaviour. 
% For now we will prove that this power-law decay holds only for $x$ belonging to this phase and of order at most $\lambda^{-3/2}$, returning to the study of $x\to \infty$ asymptotics in Section~\ref{sec:tail4}. We will apply this lemma only in the case that $\lambda x_\lambda \to \mu $ for some constant $\mu>0$.

\begin{lemma}[Phase 3]
\label{L:Phase3}
% Let $\eps>0$. The estimate
Suppose that we take $x_\lambda \to \infty$ as $\lambda \downarrow 0$ in such a way that $\liminf (\log 1/\lambda)^{-1} x_\lambda  >1$. Then
\begin{equation}
\label{eq:g_lambda_first_order_small_lambda}
g_\lambda(x_\lambda)\sim - \frac{2 \lambda}{\lambda x_\lambda + 4} \qquad \text{as $\lambda \downarrow 0$.}
\end{equation}
Moreover, the asymptotic estimate
\begin{equation}
\label{eq:g_lambda_first_order}
g_\lambda(x) \sim -\frac{2}{x} \qquad \text{as $x \to \infty$}
\end{equation}
holds for each fixed $\lambda \in (0,1)$.
% The inequality
% % There exists $\lambda_0>0$ such that
% \begin{equation}\label{E:proof_ODE3}
% 0 \leq 2g'_\lambda(y) - g_\lambda(y)^2 \leq  \frac{\lambda^3}{16}
% \end{equation}
% holds for every $\lambda \in [0,2)$ and  $x \geq x_1$. Moreover, for each $\mu_*>0$ there exists $x_*<\infty$ such that if $x>x_*$ and $\mu = \lambda x \leq \mu_*$ then
%     \[
% g_\lambda(x) = - \left( 1 + \frac{O(\log(\mu/x)) + O((\mu^2+1)\mu)}{x} \right) \frac{2\mu}{\mu + 4} \frac{1}{x}.
% \]
\end{lemma}

A consequence of this lemma is that $g_\lambda$ remains roughly constant from $x_1=O(\log 1/\lambda)$ until the much later time when $x$ is of order $1/\lambda$.

\begin{proof}[Proof of Lemma~\ref{L:Phase3}] 
We continue to write $g=g_\lambda$ and define $\delta  := 2g' - g^2=-g''$. We first prove that 
\begin{equation}\label{E:proof_ODE3}
 \sup_{x\geq x_1} |g''(x)| \leq  O(\lambda^3) \qquad \text{as $\lambda \downarrow 0$.}
\end{equation}
 The lower bound $g''(x)\geq 0$ for all $x\geq x_1$  follows from Lemma~\ref{L:unimodal}.
% We have $\delta' = -2 \delta - 2 g g'$. On $[x_0,\infty)$, we have $-2 gg' \geq 0$ and $\delta' + 2 \delta \geq 0$. Integrating between $y \geq x_1$ and $x_1$ (Lemma \ref{L:Gronwall}), we find that $\delta(y) \geq  0$ for all $y \geq x_1$. This is the first bound in \eqref{E:proof_ODE3}.
To prove the other bound, first observe that $\delta$ is a nonnegative function on $[x_1,\infty)$ vanishing at $x_1$ and converging to 0 at infinity (see Lemma \ref{L:g_small_first}). Therefore, it reaches its maximum at some point $x_2>x_1$ where $\delta'=-2 \delta - 2 g g'$ vanishes. This shows that $\max_{x \geq x_1} \delta(x) \leq \max_{x \geq x_1} |g g'(x)|$. 
% and hence by Lemma~\ref{L:g_small_first} that  
Since $|g|$ and $g'$ are both non increasing on $[x_1,\infty)$ we have that $\max_{x \geq x_1} |g g'(x)| = |gg'(x_1)|$, and the claim follows from Lemma~\ref{L:Phase2}.

\medskip

We now use the estimate $\sup_{x\geq x_1}|g''_\lambda(x)| =O(\lambda^3)$ to prove  \eqref{eq:g_lambda_first_order_small_lambda}. We begin by proving the claimed estimate under the additional assumption that $\limsup x_\lambda^2 \lambda^3=0$; large values of $x_\lambda$ satisfying $x_\lambda = \Omega(\lambda^{-3/2})$ will be handled afterwards, simultaneously to the proof of \eqref{eq:g_lambda_first_order}. Since $g_\lambda$ is increasing on $[x_0,\infty)$ and $x_0 \sim \frac{1}{2} \log (1/\lambda)$ as $\lambda \downarrow 0$, it suffices to prove the claim under the stronger assumption that $\liminf \lambda x_\lambda >0$.
Fix $\nu_*<1/2$; all implicit constants in the remainder of the proof may depend on $\nu_*$. Let $0< \nu <\nu_*$ be a parameter and let $x(\nu) := \inf \{ x>x_1: g(x) \geq - \nu \lambda \}$. On $[x_1,x(\nu)]$, the absolute value $|g|$ is larger than $\nu \lambda$, so \eqref{E:proof_ODE3} implies that
\[
\frac12 g^2 \leq g' \leq \frac{1}{2}g^2 +O(\lambda^3) = (1+O(\nu^{-2}\lambda)) \frac{1}{2}g^2 \quad \text{on~} [x_1,x(\nu)].
\]
Integrating this relation between $x_1$ and $x(\nu)$ by considering the implied inequalities on the derivative of $1/g$ yields
\begin{equation}
\label{E:proof_ODE11}
\frac12 (x(\nu) - x_1) \leq \frac{1}{g(x_1)} - \frac{1}{g(x_2)} \leq (1+O(\nu^{-2}\lambda)) \frac12 (x(\nu) - x_1).
\end{equation}
Together with \eqref{E:proof_ODE11}, this implies that
\[
x(\nu)-x_1 =  (1+O((\nu^{-2}\lambda) \left[\frac{2}{g(x_1)}-\frac{2}{g(x_2)} \right]
=(1+O(\nu^{-2}\lambda)) \left( \frac{2}{\nu} - 4 \right) \frac{1}{\lambda}
\]
where we safely absorbed the $1+O(\lambda)$ prefactor on $g(x_1)$ into the $(1+O((\nu^{-2}\lambda))$ prefactor in the second expression since $2/\nu\geq 2/\nu_*$ is bounded away from $4$. On the other hand,
by Lemmas \ref{L:Phase1} and \ref{L:Phase2}, $x_1 = o(1/\lambda)$, so that we can write simply
\[
x(\nu)
=(1+O(\nu^{-2}\lambda)) \left( \frac{2}{\nu} - 4 \right) \frac{1}{\lambda}.
\]
Inverting this expression yields the claim under the assumption that $\limsup x_\lambda^2 \lambda^3=0$. (The assumption that $\liminf \lambda x_\lambda>0$, which we made at the beginning of this paragraph, lets us pick $\nu_*<1/2$ and work only with $\nu<\nu_*$. The assumption that $\limsup x^2_\lambda \lambda^3=0$ ensures that the $O(\nu^{-2}\lambda)=O(x^2(\nu) \lambda^3)$ term is negligible.)

\medskip

We now prove \eqref{eq:g_lambda_first_order} along with the case of \eqref{eq:g_lambda_first_order_small_lambda}  in which $\liminf x_\lambda \lambda =\infty$; together with our previous analysis of the case $\limsup x_\lambda^2 \lambda^3$ this establishes \eqref{eq:g_lambda_first_order_small_lambda} in all relevant asymptotic regimes (indeed, there is a significant overlap in the regimes treated by the two arguments). 
Since $g''=g^2-2g$ is negative on $[x_1,\infty)$, 
% In this regime, 
% We first prove a simple lower bound on $g$, namely that there exists  $c\in \mathbb{R}$ such that
% \begin{equation}
% \label{E:proof_ODE_asymp1}
% g(x) \geq -\frac{2}{x+c} \quad \quad \text{for all } x \geq x_1.
% \end{equation}
% This bound follows directly from the fact that $\delta$ stays nonnegative on $[x_1,\infty)$.
we can integrate the relation $(-1/g)'=g'/g^2 \geq 1/2$ to obtain that
\[
-\frac{1}{g(x)} + \frac{1}{g(x_1)} \geq \frac{x-x_1}{2}
\]
for every $x\geq x_1$, and hence that
\[
g(x)\geq -\frac{2}{x-x_1-2/g(x_1)} = -\frac{2}{x-O(\lambda^{-1})} 
\]
for every $x \geq x_1$, which yields the desired asymptotic lower bound when $\liminf x_\lambda \lambda = \infty$. We now prove the matching upper bound.
Fix $0<M<\infty$ and let $x>x_1 \vee \lambda^{-1} M$. Recall that $g = g_\lambda$ is the unique solution to $g'' + 2g' = g^2$ with $g(0) = 0$ and initial slope $g'(0) = -\lambda$. Since $\lambda$ is larger than $M/x$ and smaller than $1$, Lemma~\ref{L:weak_uniqueness_intermediate} implies that $g$ is bounded above by the solution with initial slope $-M/x$, and it follows from the $\limsup x_\lambda^2 \lambda^3 =0$ case of the lemma treated above that 
\[g(x) \leq g_{M/x}(x) \sim -\frac{2M}{(M+4)x}\]
as $x\to \infty$.  The lower bound of the claimed estimate follows since $M$ was arbitrary. 
\end{proof}

\subsection{Upper bounds on pioneers}\label{SS:upper_bound}

In this section, we combine Theorem \ref{T:proba_representation} and the results of Section~\ref{SS:g_small} to establish a sharp upper bound on the tail of the number of pioneers. We then use this bound on the tail to bound from above the number of thick points, establishing the upper bound of Theorem \ref{T:thick}. We begin by proving the following proposition, which implies the upper bound of Theorem~\ref{T:maintail4}.

% \begin{theorem}
% \label{T:maintail4}
% Let $x_0 >0$ and $\eps>0$.
% There exist constants $R_0=R_0(x_0,\eps)>0$ and $x_*=x_*(x_0,\eps) > x_0$ such that the estimate
% $$
% \biggl(\frac{r}{R}\biggr)^{(1 + \eps) a} \leq \mathbb{P}_{e^{-x_0}R,R}\biggl( N_{r} \ge \frac{a}{2}  \left(r \log \frac{R}{r}\right)^2  \bigg\vert \; N_{r} >0 \biggr) \leq \biggl(\frac{r}{R}\biggr)^{(1 - \eps) a}
% % , \quad \quad a >0.
% $$
% holds for every $R \geq R_0$, $1 \leq r \leq e^{-x_*}R$, and $a>0$.
% \end{theorem}

\begin{proposition}\label{P:upper_bound_rough}
Let $x_0>0$. For each $\eps >0$, there exist constants $C=C(\eps) >0$ and $x_* = x_*(x_0,\eps) > x_0$  such that 
\begin{equation}
\label{E:P_upper_bound_rough}
\PROB{e^{-x_0} R,R}{ N_{r} \geq \frac{r^2 \log (R/r)}{2} n } \leq \frac{C}{(e^{-x_0} R)^2 \log(R/r)} e^{-(1-\eps)n}
\end{equation}
for every $1\leq r \leq e^{-x_*}R$ and $n \geq 1$.
\end{proposition}

To see that this implies the upper bound of Theorem~\ref{T:maintail4}, simply take $n = \lfloor a \log (R/r) \rfloor$ and apply the hitting probability estimate \eqref{E:L_hit_killing}
\begin{equation}
\label{eq:hitting_lower}
\PROB{e^{-x_0} R,R}{ N_{r} >0 } \geq \frac{c(x_0)}{(e^{-x_0} R)^2 \log(R/r)}.
\end{equation}
%which can be proven by the same method as \cite[Lemma 3.5]{asselah2022time}. (While the BRW in that lemma is not killed on the outer boundary, this does not seriously affect the proof since e.g.\ one can upper bound the second moment of the occupation time with killing by the second moment without killing, and first moments are trivial to estimate.)

% \begin{theorem}
% \label{T:maintail4}
% Let $x_0 >0$ and $\eps>0$.
% There exist constants $R_0=R_0(x_0,\eps)>0$ and $x_*=x_*(x_0,\eps) > x_0$ such that the estimate
% $$
% \biggl(\frac{r}{R}\biggr)^{(1 + \eps) a} \leq \mathbb{P}_{e^{-x_0}R,R}\biggl( N_{r} \ge \frac{a}{2}  \left(r \log \frac{R}{r}\right)^2  \bigg\vert \; N_{r} >0 \biggr) \leq \biggl(\frac{r}{R}\biggr)^{(1 - \eps) a}
% % , \quad \quad a >0.
% $$
% holds for every $R \geq R_0$, $1 \leq r \leq e^{-x_*}R$, and $a>0$.
% \end{theorem}

\begin{proof}[Proof of Proposition \ref{P:upper_bound_rough}]
Let $x=\log(R/r)$ so that $r=e^{-x}R$. 
Let $x_0>0$, let $\mu=\mu(\eps) >0$ be large enough to ensure that $t := \mu / (4 + \mu)$ is at least $1-\delta$, and let $x_1 =x_1(\eps,x_0)> x_0$ be such that if $x \geq x_1$ then $\mu / x \leq 1/2$. (The constant $1/2$ could be replaced by anything positive and strictly less than $1$.) Applying Theorem~\ref{T:proba_representation} with $\lambda=\mu/x$, we obtain that
\[
\frac{1}{(e^{-x_0} R)^2} g_{\mu/x}(x_0) = 
1 - \EXPECT{e^{-x_0}R,R}{ \left( 1 - \frac{ g_{\mu/x}(x) }{(e^{-x} R)^2} \right)^{\!N_r}\, } 
.
\]
Using the asymptotics for $g_{\mu/x}(x)$ proven in Lemma~\ref{L:Phase3} and the bound $g_{\mu/x}(x_0) \geq - \sup_{[0,\infty)} |g_{\mu/x}| \geq -\mu /(2x)$ from Lemma \ref{L:g_small_first}, we deduce that there exists a quantity $\delta(x,\mu)$ satisfying $\lim_{x\to\infty}\delta(x,\mu)=0$ such that
\[
1 - \EXPECT{e^{-x_0}R,R}{ \left( 1 + \frac{(1+\delta(x,\mu))2\mu}{(\mu + 4)x(e^{-x} R)^2} \right)^{\!N_r}\, } \geq - \frac{\mu}{2x (e^{-x_0} R)^2}
\]
% where $o(1) \to 0$ as $x \to \infty$. 
Recall that $t = \mu / (4 + \mu)$, so that $\mu/2=2t/(1-t)$. Rewriting the above slightly, we obtain that there exists a quantity $\delta'(x,\mu)$ satisfying $\lim_{x\to\infty}\delta'(x,\mu)=0$ such that
\[
\EXPECT{e^{-x_0}R,R}{ \exp \left( (1+\delta'(x,\mu)) t \frac{2N_r}{x(e^{-x} R)^2} \right) \indic{N_r>0} } \leq \PROB{e^{-x_0}R,R}{ N_r >0} + \frac{t}{1-t} \frac{2}{x (e^{-x_0} R)^2}.
\]
By Lemma \ref{L:hit_killing}, the probability that $N_r >0$ is not larger than
$(1+o(1)) 2 / (x (e^{-x_0}R)^2)$ as $x \to \infty$.
%the same probability without the killing on the sphere $\partial B(0,R)$, which is equal to $(1+o(1)) 2 / (x (e^{-x_0}R)^2)$ as $x \to \infty$. \red{[TH: Reference? Should we write a lemma with this hitting probability estimate? Is there a nice ``scale-invariance'' argument using the ODE? Maybe we should make a subsection explaining this estimate along with Kolmogorov's estimate on survival times and the probability of the BBM reaching a large radius, since currently there are several times we use these facts with sketchy explanations.]}
Using this and Markov's inequality, we can now bound the left hand side of \eqref{E:P_upper_bound_rough} by
\begin{align*}
e^{-(1+o(1))tn} \EXPECT{e^{-x_0}R,R}{ \exp \Big( (1+\delta'(x,\mu)) t \frac{2N_r}{x(e^{-x} R)^2} \Big) \indic{N_r>0} }
\leq \frac{1+o(1)}{1-t} \frac{2}{x(e^{-x_0} R)^2} e^{-t(1+o(1))n}.
\end{align*}
Since we took $\mu$ large enough to ensure that $t \geq 1 - \delta$, this proves the proposition.
\end{proof}

% The upper bound of Theorem \ref{T:maintail4} is a direct consequence of Proposition \ref{P:upper_bound_rough} above.
We now explain how the upper bound of Theorem \ref{T:thick} follows from Proposition~\ref{P:upper_bound_rough}.

\begin{proof}[Proof of Theorem \ref{T:thick} -- Upper bound]
We start by arguing that it is enough to show the upper bound with the killing on the outer sphere $\partial B(R)$ (or more precisely, on $\partial B(A R)$ for $A >0$ large). Indeed, conditionally on surviving $R^2$ generations, the probability that a particle reaches $\partial B(A R)$ goes to zero as $R \to \infty$ and $A \to \infty$. So we can \textit{a priori} restrict ourselves to this event and add a killing on $\partial B(A R)$.

\medskip

We now prove the upper bound under $\P_{z_0,R}( \,\cdot\, \vert \zeta > R^2 )$, where we recall that $\zeta$ denotes the extinction time of the branching process.
Let $a>0$. Recall the definition \eqref{E:def_thick} of the set $\Tc_R(a)$ of $a$-thick points. We compute
\begin{align*}
\EXPECT{z_0,R}{ \sum_{z \in \Z^4} \indic{z \in \Tc_R(a)} }
\leq \sum_{z \in \Z^4 \cap B(0,R)} \PROB{z_0,R}{ N_{z,1} \geq \frac{a}{2} (\log R)^2}.
\end{align*}
Let $z \in \Z^4 \cap B(0,R)$. By killing the particles on the sphere $\partial B(z,2R)$, which is strictly larger than the sphere $\partial B(R)$, we reduce the problem to a spherically symmetric situation once more. Proposition \ref{P:upper_bound_rough} then gives that
\[
\PROB{z_0,R}{ N_{z,1} \geq \frac{a}{2} (\log R)^2} \leq R^{-2-a+o(1)}.
\]
Because $\P_{z_0,R}(\zeta > R^2) = R^{-2+o(1)}$ (by Kolmogorov's theorem about the extinction probability of a critical Bienaym\'e Galton--Watson tree; see e.g. \cite[Corollary 1]{kesten1966galton}), the expectation of $\# \Tc_R(a)$ conditional on $\{ \zeta > R^2 \}$ is bounded by $R^{4-a+o(1)}$.
By Markov's inequality, we deduce that for any $\eps >0$,
\[\lim_{R \to \infty} \P_{z_0,R} \left( \# \Tc_R(a) \geq R^{4-a+\eps} \,\vert\, \zeta > R^2 \right) = 0.\]
This gives the desired upper bound on the number of thick points. It also gives the upper bound on $\sup_{z \in \Z^4} N_{z,1}$ since $\Tc_R(a)$ is empty with high $\P_{z_0,R}(\cdot \vert \zeta > R^2)$-probability as $R\to\infty$ for each $a>4$.
% i.e. $\sup_{z \in \Z^4} N_{z,1} \leq a (\log R)^2$.
\end{proof}

\section{Asymptotic expansion of solutions (proof of Theorem \ref{T:hitprob4})}
\label{sec:series_expansion}

The goal of this section is to prove the following theorem, which establishes infinite-order asymptotic expansions for solutions to the ODEs satisfied by the functions $\gs_s$ and $\hs_s$.  The following theorem immediately implies \ref{T:hitprob4} since
 \[
    h(x)=e^{2x}\mathbb{P}_{e^x,\infty}(N_1>0) 
 \]
 is a non-zero solution to the ODE $h''-2h'=h^2$ defined on $[0,\infty)$.

\begin{theorem}\label{T:gh}
Let $h$ be any solution to $h''-2h'=h^2$ defined on $[0,\infty)$ that is not identically zero, and let $g$ be any maximal solution to the initial value problem
\[
\left\{ \begin{array}{l}
g''+2g'=g^2,\\
g(0) = 0, ~g'(0) = -\lambda
\end{array} \right.
\]
with $\lambda \in (-1,0)$. Then $h(x)$ is positive for all $x \geq 0$, $g(x)$ is defined on $[0,\infty)$ and negative for all $x\geq 0$, and there exist two sequences $(Q_n)_{n \geq 1}$ and $(P_n)_{n \geq 1}$ of polynomials such that for any $N \geq 1$,
\[
g(x) = \sum_{n=1}^N \frac{Q_n(\log x)}{x^n} + o(x^{-N})
\quad \text{and} \quad
h(x) = \sum_{n=1}^N \frac{P_n(\log x)}{x^n} + o(x^{-N}),
\quad \quad \text{as~} x \to \infty.
\]
Moreover, there exist constants $C_\lambda, C_{*} \in \R$, with $C_*$ depending on the choice of $h$, such that $Q_1 = -2$, $Q_2(X) = 2 X + C_\lambda$, $P_1 = 2$, $P_2(X) = 2X + C_{*}$ and for all $n \geq 3$, $Q_n$ and $P_n$ are the unique polynomials satisfying the recurrence relations
\begin{equation}
\label{E:inductionQ}
(2n-4) Q_n -2Q_n' = Q_{n-1}'' - (2n-1) Q_{n-1}' + n(n-1) Q_{n-1} - \sum_{k=2}^{n-1} Q_k Q_{n+1-k}
\end{equation}
and
\begin{equation}
\label{E:inductionP}
(2n-4)P_{n} - 2P_{n}' = - P_{n-1}'' + (2n-1)P_{n-1}' - (n-1)nP_{n-1} + \sum_{k=2}^{n-1} P_k P_{n+1-k}.
\end{equation}
\end{theorem}

\begin{remark}\label{Rmk:divergent}
As already mentioned, the main difficulty in proving this theorem comes from the fact that the expansion $\sum_{n=1}^\infty \frac{P_n(\log x)}{x^n}$ diverges, i.e., is not summable for any finite $x$. Although we do not \emph{prove} that this series diverges, we note that this feature is already present for the solutions of the linearised equation
\[
h_{\text{lin}}'' - 2h_{\text{lin}}' = \frac{2}{x} h_{\text{lin}}, \quad h_{\text{lin}}(\infty) = 0.
\]
Indeed such solutions can be explicitly written in terms of the exponential integral $\text{Ei}(x) = -\int_{-x}^\infty \frac{e^{-t}}{t} \d t$:
\[
h_{\text{lin}}(x) = c\left( 1 + 2xe^{2x} \text{Ei}(-2x) \right), \quad c \in \R,
\]
from which we obtain that
\[
h_{\text{lin}}(x) = c \sum_{n=1}^{N-1} (-1)^{n+1} \frac{n!}{(2x)^n} + O \left( N! (2x)^{-N} \right).
\]
Although the series on the right hand side is also not a summable series, it is however a classical example of a Borel summable series. (Since we do not need this here, we do not explain the details of this notion. See \cite{hardy2000divergent} for background on this topic.) 
This suggests that the series in Theorem \ref{T:gh} may also be summable in the Borel sense. We do not investigate this in this paper.  
\end{remark}

In particular, from Theorem \ref{T:gh}, we see that there is a constant $C_\lambda$ such that 
 % of $g_\lambda$ \eqref{E:g_lambda}:
\begin{equation}
\label{E:asymp_g_Clambda}
g_\lambda(x) = -\frac{2}{x} + \frac{2\log x + C_\lambda}{x^2} + o(x^{-2})
\quad \quad \text{as}~ x \to \infty.
\end{equation}
The second goal of this section will be to study 
 the constant $C_\lambda$ appearing in this asymptotic expansion. 

The following lemma gives an alternative expression for $C_\lambda$ and establishes asymptotic estimates on $C_\lambda$ as $\lambda\downarrow 0$.

\begin{lemma}\label{L:C_lambda}
For all $\lambda \in (0,1)$ and $x_0 >0$,
\begin{equation}
\label{E:C_lambda}
C_\lambda = - \int_{x_0}^\infty \left( \frac{g_\lambda''(x)}{g_\lambda(x)^2} + \frac{1}{x} \right) \d x - \frac{2}{g_\lambda(x_0)} - x_0 - \log x_0.
\end{equation}
Moreover, the function $\lambda \in (0,1) \mapsto C_\lambda$ is continuous, decreasing, and satisfies
$C_{\lambda} \sim 4/\lambda$ as $\lambda \downarrow 0$.
\end{lemma}

% 
% \frac{d}{d\lambda}C_\lambda = -\int_1^\infty \frac{g_{xx\lambda}g^2-2 gg_{xx}g_\lambda}{g^4} + \frac{2 g_\lambda(1)}{g^2(1)}
% 

% u(t) = lambda^{-1}g(\lambda^{-1}t)
% g(t) = lambda u (\lambda t)
% g'=lambda^2 u'(lambda t)
% g''=\lambda^3 u''(lambda t)
% g''/g^2 = lambda u''(lambda t)/u(lambda t)^2
% s=lambda t, dt = lambda^{-1}ds
% int_{lambda^{-1}} [ lambda u''(lambda t)/u(lambda t)^2 + lambda / lambda t ] dt = int_1 [ u''(s)/u(s)^2 + 1 / s ] dt 
% all terms: int_1 [ u''(s)/u(s)^2 + 1 / s ] dt - 2/g(lambda^{-1}) - \lambda - log \lambda 
% 5-1

\begin{figure}[t]
\centering
\includegraphics[trim=0.5cm 0.5cm 2cm 1.5cm, clip, width=0.45\textwidth]{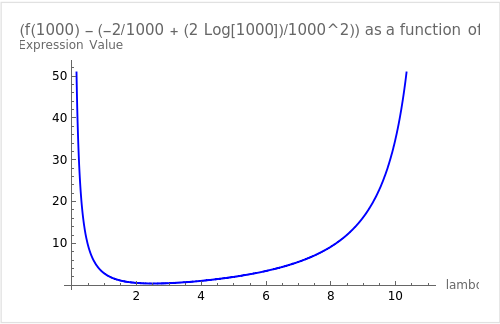} \hfill
\includegraphics[trim=0.5cm 0.5cm 2cm 1.5cm, clip, width=0.45\textwidth]{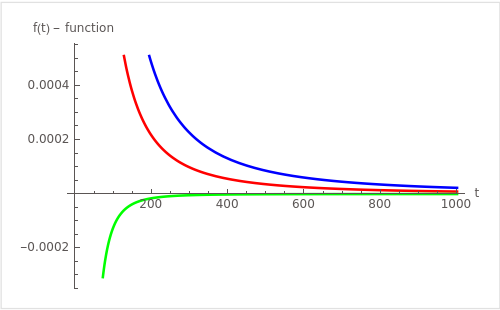}
% \hfill
% \includegraphics[trim=0.5cm 0.5cm 2cm 1.5cm, clip, width=0.45\textwidth]{manuscript/images/turning_point1_thick.png}
% \includegraphics[trim=0.5cm 0.5cm 0.5cm 0.5cm, clip, height=3.7cm]{manuscript/images/large_lambda_critical.png}
\caption{Left: Numerical plot of $x^2[g_\lambda(x)-(-\frac{2}{x}+\frac{2 \log x}{x^2})]$ with $x=1000$, which is a numerical approximation to $C_\lambda$, as a function of $\lambda$. Right: Numerical approximations to the difference of $g_{0.5}$ and the three approximations $-2/x$ (blue), $-2/x+2\log x /x^2$ (red), and $-2/x+(2\log x + 9.209) /x^2$ (green), where the numerical value $C_{0.5}\approx 9.209$ was extracted from the plot on the left.}
\label{F:C_lambda}
% Left: The asymptotics $g_\lambda(x)\sim -2/x$ as $x\to \infty$, proven for fixed $\lambda \in (0,1)$ in Section~\ref{sec:series_expansion}, is verified numerically for a larger range of $\lambda$ values. The convergence appears to hold uniformly away from $0$ and some critical value around $11.25$. Right: Zooming in on this plot near its maximum. The maximum of the function is attained numerically at $\lambda \approx 2.43$, suggesting that the critical value for our probabilistic representation of solutions to hold is around $\inf _R \lambda_c(R) \approx \lambda_c(1000)\approx 2.43$.}\label{F:turning_point}
% Right: The true critical value of $\lambda$ below which this solutions remain non-positive appears to be around $11.2$.}
\end{figure}

The fact that the right hand side of \eqref{E:C_lambda} is independent of $x_0$ follows from the fact that $g_\lambda$ satisfies the ODE $g_\lambda'' + 2g_\lambda' = g_\lambda^2$, so that $g_\lambda''/g_\lambda^2 = 1-2g'_\lambda/g_\lambda^2=(x+2/g_\lambda)'$.
% Note that the estimate \eqref{E:L_g_small_first} of Lemma~\ref{L:Phase1} gives a sense in which the approximation 
% shows that the asymptotic expansion \eqref{E:asymp_g_Clambda} already holds when $x$ is of order $1/\lambda$ and $\lambda \to 0$.

\medskip

We prove Theorem~\ref{T:gh} in Sections~\ref{S:g_expansion} and \ref{S:h_expansion} and Lemma~\ref{L:C_lambda} in Section~\ref{sec:C_lambda}.

\subsection{Series expansion of $g$}
% : proof of Theorem \ref{T:gh}}\label{sec:series_expansion}
\label{S:g_expansion}
This section is dedicated to the proof of the part of Theorem \ref{T:gh} concerning the solutions of $g''+2g'=g^2$. The derivation of the asymptotic expansion of solutions to $h''-2h'=h^2$ is closely analogous except that a different argument is required for the first-order estimate $h \sim 2/x$, and an appropriate replacement is needed for the unimodality properties proven for $g$ in Lemma~\ref{L:unimodal}; this is discussed in more detail in the next subsection.

% \red{[TH: Should we add some more remarks about the proof for $h$? We use a lot of things in this section that we developed for $g$ in the earlier sections. Is it necessary to re-do everything for $h$ or are there some tricks?]}
\medskip

%\red{[TH: Changes throughout this paragraph since I moved the first-order estimate to Section 4.1.]}
We now outline the main steps of the proof. The first-order estimate $g_\lambda(x) \sim -2/x$ was already proven in Lemma~\ref{L:Phase3}. 
% The proof of the high It will be decomposed into three parts. First, in Lemma~\ref{L:ODE_asymp1}, we obtain a first-order asymptotic estimate for $g$.
In Lemma \ref{L:ODE_asymp}, we warm up to the full asymptotic expansion by obtaining its second term, i.e., compute the asymptotic behaviour of $g$ up to an error of order $o(x^{-2})$. 
We will then show in Lemma \ref{L:existence_expansion} that the proof of this preliminary estimate can be generalised to obtain the existence of such an expansion for both $g_\lambda$ and $g_\lambda'$. The combinatorics of this step will be fairly involved and the induction relation \eqref{E:inductionQ} will be difficult to extract directly from the proof. Instead, we will show separately that if such an expansion holds for both $g_\lambda$ and $g_\lambda'$, then the polynomials $Q_n$ have to satisfy \eqref{E:inductionQ}. This final step will be straightforward since we will have already laid the groundwork to be able to differentiate term by term.

\medskip

We now begin to carry out this plan. We begin by estimating $g_\lambda$ and $g'_\lambda$ up to errors of order $o(x^{-2})$ and $o(x^{-3})$ respectively.

\begin{lemma}\label{L:ODE_asymp}
Let $\lambda \in (0,1)$. There exists a constant $C_\lambda \in \R$ such that
\begin{align}
g_\lambda(x) &= \frac{-2}{x} + \frac{2 \log x + C_\lambda}{x^2} + o\left(\frac{1}{x^2}\right)  \qquad \text{and}
% \quad \quad \mathrm{as~} x \to \infty.\\
\label{E:LODE_asymp_g}\\
g_\lambda'(x)&= \frac{2}{x^2} - \frac{4 \log x + 2C_\lambda-2}{x^3} + o\left(\frac{1}{x^3}\right)
\label{E:LODE_asymp_g'}
\end{align}
as $x\to \infty$.
\end{lemma}

% \red{[TH: Added the asymptotics of $g_\lambda'$ to this statement since we use it in the next lemma. Maybe we should make a separate lemma for the first order asymptotics of $g$ since the proof is self-contained?]}

\begin{proof}[Proof of Lemma~\ref{L:ODE_asymp}]
 % We continue to use the notation $g=g_\lambda$ and $\delta=2g'-g^2$ from Lemma~\ref{L:ODE_asymp1}. 
 %\red{[TH: First few sentences here are from beginning of old first-order proof.]}
 To lighten notation we write $g$ instead of $g_\lambda$. We also define $\delta := 2g' - g^2$. Although $\delta$ is equal to $-g''$, we prefer to view it as the `error term' describing the extend to which $g$ fails to satisfy the separable ODE $2g'=g^2$. The function $\delta$ satisfies
\begin{equation}
\label{E:proof_ODE_asymp2}
\delta' = -2\delta - 2gg' = -(2 - |g|) \delta + |g|^3.
\end{equation}
Define $x_1 := \inf \{ x>0: \delta(x) = 0 \}$ as in Section~\ref{SS:g_small}. We showed in Lemma~\ref{L:unimodal} that $x_1 < \infty$ and that $\delta$ is nonnegative on $[x_1,\infty)$. 
 We will iteratively prove a sequence of estimates on $g$ and $\delta$ of increasing strength.
% proving each successive bound by plugging the previous bounds into \eqref{E:proof_ODE_asymp2} via Lemma~\ref{L:Gronwall}.

% \medskip

% \noindent \textbf{Step 1:} We first prove a simple lower bound on $g$, namely that there exists  \red{$c\in \mathbb{R}$} such that
% \begin{equation}
% \label{E:proof_ODE_asymp1}
% g(x) \geq -\frac{2}{x+c} \quad \quad \text{for all } x \geq x_1.
% \end{equation}
% This bound follows directly from the fact that $\delta$ stays nonnegative on $[x_1,\infty)$. Indeed, integrating the relation $g'/g^2 \geq 1/2$ gives that
% \[
% -\frac{1}{g(x)} + \frac{1}{g(x_1)} \geq \frac{x-x_1}{2}, \quad x \geq x_1,
% \]
% and rearranging this inequality yields the claim with $c=-x_1-2/g(x_1)$. \red{[TH: It's not obvious to me that this is positive. It doesn't seem to matter so I changed it.]}
% % \[
% % -2/g >= x - x_1 - 2/g(x_1)
% % \] 

\medskip

\noindent \textbf{Upper bound on $\delta$:} Using the estimate $g(x)\sim -2/x$ in the differential equation \eqref{E:proof_ODE_asymp2}, we obtain that
\begin{equation}
\label{E:proof_ODE_asymp5}
\delta'(x) = - \left( 2 - \frac{2+o(1)}{x} \right) \delta(x) + \frac{8+o(1)}{x^3} \qquad \text{ as $x\to \infty$}.
\end{equation}
% It follows that if $x$ i
and hence that if $x$ is sufficiently large (so that the $(2+o(1))/x$ term has absolute value at most $1$ and the $(8+o(1))/x^3$ term is at most $16/x^3$) then
% We start by bounding $2- 2/(x+c) \geq 1$ when $x >x_2$ which gives
\[
\delta'(x) + \delta(x) \leq \frac{16}{x^3}.
\]
Let $x_2 \geq x_1$ be such that this estimate holds for all $x\geq x_2$.
After integration (Lemma~\ref{L:Gronwall}) this gives for all $x > x_2$,
\[
\delta(x)  \leq e^{-x}\int_{x_2}^x e^t \frac{16}{t^3} dt +  e^{-x+x_2} \delta(x_2)  =  O(x^{-3}).
\]
% \[
% \delta(x)\leq e^{-x}\int_{x_2}^x e^t \frac{8}{(t+c)^3} dt +  e^{-x+x_2} \delta(x_2) 
% \]
We now plug this bound back into \eqref{E:proof_ODE_asymp5} to obtain that
\[
\delta'(x) + 2\delta(x) \leq \frac{8+o(1)}{x^3} + O(x^{-4}) = \frac{8+o(1)}{x^3}.
\]
We integrate once more (using Lemma~\ref{L:Gronwall}) to obtain that
\begin{equation}
\label{E:proof_ODE_asymp6}
\delta(x) \leq e^{-2x}\int_{x_2}^{x} e^{2t} \left[\frac{8+o(1)}{x^3}\right] dt + e^{-2x+x_2}\delta(x_2) = \frac{4+o(1)}{x^3},
\end{equation}
where the final estimate follows by integration by parts.

% This proves the claimed bound $|g(x)| \geq \frac{1}{x+C}$.

\medskip

\noindent
\textbf{Lower bound on $\delta$:}
We now use the estimate $g(x)\sim -2/x$ to prove a lower bound on $\delta$. Plugging the bound \eqref{eq:g_lambda_first_order} into \eqref{E:proof_ODE_asymp2} and using that $|g|\delta\geq 0$ yields that
 % We now have a much better lower bound on $g$ that we can plug back into \eqref{E:proof_ODE_asymp2}:
\[
\delta'(x) + 2 \delta(x) \geq |g(x)|^3 \geq \frac{8+o(1)}{x^3}
\]
as $x\to \infty$,
and integrating this bound with Lemma \ref{L:Gronwall} implies that
\begin{equation}
\label{E:proof_ODE_asymp9}
\delta(x) \geq e^{-2x}\int_{x_2}^{x} e^{2t} \left[\frac{4+o(1)}{x^3}\right] dt + e^{-2x+x_2}\delta(x_2)= \frac{4+o(1)}{x^3} \qquad \text{as $x\to \infty$},
\end{equation}
where again we use integration by parts to bound the relevant integrals. Together with \eqref{E:proof_ODE_asymp6} this yields that $\delta(x)\sim 4/x^3$ as $x\to \infty$.

\medskip

\noindent \textbf{Improved lower bound on $g$:}
We now use the estimate $\delta(x)\sim 4/x^3$ to improve our lower bound on $g$. We have by \eqref{E:proof_ODE_asymp6} and \eqref{eq:g_lambda_first_order} that
% The upper bound \eqref{E:proof_ODE_asymp6} on $\delta$ and the estimate $|g(x)| \geq \frac{1}{x+C}$ gives
\[
2g'(x) - g(x)^2 = \delta(x) \sim 
\frac{4+o(1)}{x^3} \sim \frac{1}{x}g(x)^2 \qquad \text{as $x\to \infty$}
\]
and hence that
% In other words,
\[
-\left(\frac{1}{g(x)}+\frac{x}{2}\right)'
=\frac{g'(x)}{g(x)^2} - \frac{1}{2} = \frac{1+o(1)}{2x} \qquad \text{as $x\to \infty$.}
\]
It follows in particular that if $x$ is sufficiently large, then the left hand side term is at most $1/x$
%then
%\[
%-\left(\frac{1}{g(x)}+\frac{x}{2}\right)'
% \leq \frac{1}{x},
 % \qquad \text{as $x\to \infty$.}
%\]
and integrating this bound yields that
% if $x$ is sufficiently large then
% enough,
\begin{equation}
\label{E:proof_ODE_asymp8}
|g(x)| \geq \frac{2}{x + 2\log(x) + O(1)} \qquad \text{ as $x\to \infty$.}
\end{equation}
Note that this is not yet a sharp estimate since the coefficient of $\log x$ is $2$ rather than $1$.

\medskip

\noindent \textbf{Completing the proof:} We are now ready to conclude the proof. Consider the telescoping sum
\begin{align*}
\int_1^x \frac{\delta(t)}{g(t)^2}\d t = \log x + \int_1^x \left( \delta(t) - \frac{4}{t^3} \right) \frac{1}{g(t)^2} \d t + \int_1^x \frac{4}{t^3} \left( \frac{1}{g(t)^2} - \frac{t^2}{4} \right) \d t.
\end{align*}
The estimates \eqref{eq:g_lambda_first_order}, \eqref{E:proof_ODE_asymp6}, \eqref{E:proof_ODE_asymp8} and \eqref{E:proof_ODE_asymp9} on $|g|$ and $\delta$ show that the two integrals on the right hand side converge absolutely as $x \to \infty$. Therefore, there exists some $C_0 \in \R$ such that
\begin{equation}
\label{E:proof_C_lambda1}
\int_1^x \frac{\delta(t)}{g(t)^2}\d t = \log x + C_0 + o(1) \quad \quad \text{as~} x \to \infty.
\end{equation}
But by definition of $\delta$, the left hand side can also be expressed as
\[
\int_1^x \frac{\delta(t)}{g(t)^2}\d t = \int_1^x \left[\frac{2g'}{g(t)^2}-1 \right]\d t =
-\frac{2}{g(x)} + \frac{2}{g(1)} - (x-1).
\]
Rearranging, we obtain that
\begin{equation}
\label{E:proof_C_lambda2}
g(x) = \frac{-2}{x + \log x + C_0 - 2/g(1) -1 + o(1)} \qquad \text{ as $x\to\infty$,}
\end{equation}
concluding the proof of \eqref{E:LODE_asymp_g}. The estimate \eqref{E:LODE_asymp_g'} follows and the relation $2g'=g^2+\delta$ since $\delta\sim 4/x^3$.
\end{proof}

We next extend this analysis to arbitrary order, without yet determining the relevant polynomials.

\begin{lemma}\label{L:existence_expansion}
Let $\lambda \in (0,1)$. There exist two sequences of polynomials $(Q_n)_{n \geq 1}$ and $(R_n)_{n \geq 2}$ such that for any $N \geq 1$,
\begin{equation}
\label{E:L_existence_expansion}
g_\lambda(x) = \sum_{n=1}^N \frac{Q_n(\log x)}{x^n} + o(x^{-N}),
\quad g_\lambda'(x) = \sum_{n=2}^{N+1} \frac{R_n(\log x)}{x^n} + o(x^{-N-1}),
\quad \quad \text{as~} x \to \infty.
\end{equation}
\end{lemma}

\begin{proof}[Proof of Lemma~\ref{L:existence_expansion}]
We will prove the statement by induction on $N$, the case $N=2$ having already been treated in Lemma \ref{L:ODE_asymp}. Let $N \geq 2$ and suppose that there exist polynomials $(Q_n)_{n = 1, \ldots, N}$ and $(R_n)_{n = 2, \ldots, N+1}$ such that \eqref{E:L_existence_expansion} holds. As in the proof of Lemma \ref{L:ODE_asymp}, we let $\delta := 2g' - g^2$, which satisfies
% (see \eqref{E:proof_ODE_asymp2})
\begin{align*}
\delta' + 2\delta & = -2gg' = -2 \left( \sum_{n=1}^N \frac{Q_n(\log x)}{x^n} \right) \left( \sum_{n=2}^{N+1} \frac{R_n(\log x)}{x^n} \right) + o(x^{-N-2})  = \sum_{n=3}^{N+2} \frac{Q_n^{(1)}(\log x)}{x^n} + o(x^{-N-2}).
\end{align*}
Here and in the rest of the proof, we write $(Q_n^{(1)})_{n \geq 1}$, $(Q_n^{(2)})_{n \geq 1}$, and so for certain polynomials arising in our computations. Although these polynomials can be written explicitly in terms of $(Q_n)_{n=1 \dots N}$ and $(R_n)_{n = 2 \dots N+1}$, we keep the dependence inexplicit for now and return to their precise computation later once we know we can differentiate term-by-term.
% As already alluded to, we adopt such a notation since writing down explicitly these polynomials in terms of $(Q_n)_{n=1 \dots N}$ and $(R_n)_{n = 2 \dots N+1}$ would lead to very lengthy formulas that would be difficult to analyse.
Using that $(e^{2x}\delta)'=e^{2x}(\delta'+2 \delta)$, we can integrate this equality to obtain that
% Integrating leads to (consider the derivative of $e^{2x} \delta(x)$)
\begin{align}
\nonumber
\delta(x) & = e^{-2x} \delta(1) + e^{-2x} \sum_{n=3}^{N+2} \int_1^x e^{2t} \frac{Q_n^{(1)}(\log t)}{t^n} \d t + e^{-2x} \int_1^x e^{2t} o(t^{-N-2}) \d t \\
& = \sum_{n=3}^{N+2} \frac{Q_n^{(2)}(\log x)}{x^n} + o(x^{-N-2})
\label{E:proof_existence1}
\end{align}
as $x\to \infty$. 
It follows that 
\begin{align*}
\frac{\delta(x)}{g(x)^2} = (1 + o(x^{-N+1})) \left( \sum_{n=3}^{N+2} \frac{Q_n^{(2)}(\log x)}{x^n} \right) \left( \sum_{n=1}^N \frac{Q_n(\log x)}{x^n} \right)^{-2}.
\end{align*}
Since $Q_1=-2$ is a constant function, we can expand the negative binomial to rewrite this as
\begin{equation}
\label{eq:deltaQ3}
\frac{\delta(x)}{g(x)^2} =  \sum_{n=1}^{N} \frac{Q_n^{(3)}(\log x)}{x^n} + o(x^{-N}) = \frac{1}{x}+\sum_{n=2}^{N} \frac{Q_n^{(3)}(\log x)}{x^n} + o(x^{-N}),
\end{equation}
where we used the first-order asymptotics for $\delta$ and $g$ computed in Lemma~\ref{L:ODE_asymp} to compute the first term of the asymptotic series.
Note that here it is crucial that the polynomial $Q_1$ is constant (equal to $-2$), otherwise the above negative-binomial expansion would involve rational functions of $\log x$ instead of polynomials.

\medskip

The estimate \eqref{eq:deltaQ3} implies that
\[
\int_x^\infty \left( \frac{\delta(t)}{g(t)^2} - \frac{1}{t} \right) \d t
= \sum_{n=2}^N \int_x^\infty \frac{Q_n^{(3)}(\log t)}{t^n}\d t + o(x^{-N+1})
= \sum_{n=1}^{N-1} \frac{Q_n^{(4)}(\log x)}{x^n} + o(x^{-N+1})
\]
as $x\to \infty$, where the expansion of the integrals appearing here can be established using iterated integration by parts. Meanwhile, the definition of $\delta$ and the preliminary estimate we obtained in Lemma \ref{L:ODE_asymp} yield that
\[
\lim_{y \to \infty} \int_x^y \left( \frac{\delta(t)}{g(t)^2} - \frac{1}{t} \right) \d t = \lim_{y \to \infty} \left( \frac{2}{g(x)} - \frac{2}{g(y)} - (y-x) - \log \frac{y}{x} \right) = \frac{C_\lambda}{2} + \frac{2}{g(x)} + x + \log x,
\]
% But the estimate we obtained on $\delta/g^2$ also shows that
% \[
% \int_x^\infty \left( \frac{\delta(t)}{g(t)^2} - \frac{1}{t} \right) \d t
% = \sum_{n=2}^N \int_x^\infty \frac{Q_n^{(3)}(\log t)}{t^n} + o(x^{-N+1})
% = \sum_{n=1}^{N-1} \frac{Q_n^{(4)}(\log x)}{x^n} + o(x^{-N+1})
% \]
% which implies that
% Together, these estimates imply that
so that
\[
- \frac{2}{g(x)} = x + \log x + \frac{C_\lambda}{2} - \sum_{n=1}^{N-1} \frac{Q_n^{(4)}(\log x)}{x^n} + o(x^{-N+1})
\]
as $x\to \infty$.
By using once more the power series of $(1+u)^{-1}$, this gives exactly the estimate we wanted for $g$:
\[
g(x) = \sum_{n=1}^{N+1} \frac{Q_n^{(5)}(\log x)}{x^n} + o(x^{-N-1}).
\]
The estimate for $g'$ is then a consequence of this estimate together with \eqref{E:proof_existence1}:
\[
2g'(x) = \delta(x) + g(x)^2 = \sum_{n=1}^{N+2} \frac{Q_n^{(6)}(\log x)}{x^n} + o(x^{-N-2}).
\]
This concludes our inductive proof.
\end{proof}

We conclude this section by proving the part of Theorem \ref{T:gh} concerning $g$:

\begin{proof}[Proof of Theorem \ref{T:gh} for the function $g$]
By using the asymptotic expansions of $g$ and $g'$ from Lemma \ref{L:existence_expansion} and by using the fact that $g'' = g^2 - 2g'$, we obtain for free an asymptotic expansion for $g''$ of a similar form. Since $g$ and $g'$ both converge to zero as $x\to \infty$, we can write $g'(x)=- \int_x^\infty g''(t)\d t$ and $g(x)=-\int_x^\infty g'(t) \d t$, which allows us to reobtain the asymptotic expansions of $g$ and $g'$ by integrating the asymptotic expansion of $g''$ term by term. 
% Integrating the expansion of $g''$ once (resp. twice) yields an expansion of $g'$ (resp. $g$), that is to say we can differentiate twice the expansion of $g$.
Wrapping up, we have shown that there exists a sequence of polynomials $(Q_n)_{n \geq 1}$ such that for any $N \geq 1$,
\begin{align*}
g(x) & = \sum_{n=1}^N \frac{Q_n(\log x)}{x^n} + o(x^{-N}) \\
g'(x) & = \sum_{n=2}^{N+1} \frac{Q'_{n-1}(\log x) - (n-1)Q_{n-1}(\log x)}{x^n} + o(x^{-N-1}) \\
g''(x) & = \sum_{n=3}^{N+2} \frac{Q_{n-2}''(\log x)-(2n-3)Q_{n-2}'(\log x)+(n-1)(n-2)Q_{n-2}(\log x)}{x^n} + o(x^{-N-2}).
\end{align*}
Because $g$ satisfies $g''+2g' = g^2$, we must have $2Q_1' - 2 Q_1 = Q_1^2$ and that 
\begin{align*}
Q_{n-2}'' - (2n-3) Q_{n-2}' + (n-1)(n-2) Q_{n-2} + 2Q_{n-1}' -2(n-1)Q_{n-1} = \sum_{k=1}^{n-1} Q_k Q_{n-k}
\end{align*}
for every $n \geq 3$.
It follows from Lemma~\ref{L:ODE_asymp} that $Q_1=-2$ and $Q_2(X)=2X+C_\lambda$,  
% (we know that $Q_1$ cannot be equal to zero by Lemma \ref{L:ODE_asymp} - alternatively Lemma \ref{L:ODE_asymp} already gave the first two terms $Q_1$ and $Q_2$)
and rearranging the above display leads to the recurrence \eqref{E:inductionQ}. One can also easily check that this recurrence relation uniquely determines $Q_{n}$ in terms of $Q_1,\ldots,Q_{n-1}$ for every $n\geq 2$, since the matrix that must be inverted to solve for the coefficients of $Q_n$ is upper-triangular with diagonal terms all equal to $2n-4$.
(For $n=2$ the term $(2n-4)Q_{n}$ vanishes and the recurrence relation does not put any constraints on the constant term of $Q_2$.) This concludes the proof.
\end{proof}

\subsection{Series expansion of $h$}
\label{S:h_expansion}

In this section we explain how to modify the analysis of $g$ carried out in the previous section to derive the series expansion for $h$. We will focus on the first-order asymptotics $h(x)\sim 2/x$ together with the appropriate replacement for the unimodality properties of $g$ derived in Lemma~\ref{L:unimodal} (namely, the fact that $h$ is totally monotone), with the rest of the proof being very similar. 

We begin with the following proposition, which establishes non-negativity of all solutions to $h''-2h'=h^2$ defined on all of $[0,\infty)$ together with the fact that these solutions are always totally monotone. The latter property  replaces the unimodality properties of $g$ established in Lemma~\ref{L:unimodal}. The fact that these properties hold for \emph{all} solutions defined on $[0,\infty)$ indicates an important difference between the two ODEs (or equivalently between the forward and backward evolution of either ODE).

\begin{proposition}
\label{P:h_totally_monotone}
Let $h$ be any solution to $h''-2h'=h^2$ defined on $[0,\infty)$. If $h$ is not identically zero then the following hold:
\begin{enumerate}
\item $h(x)>0$ for all $x\geq 0$.
\item $h(x)\to 0$ as $x\to \infty$.
\item $h$ is totally monotone in the sense that $(-1)^n h^{(n)}(x)>0$ for every $n \geq 1$.
\item $h^{(n)}(x)\to 0$ as $x\to \infty$ for every $n\geq 1$.
\end{enumerate}
\end{proposition}

We note that this proposition together with the uniqueness of non-negative solutions to boundary value problems for the ODE $h''-2h'=h^2$ (which can be proven as in Lemma~\ref{L:appendix_uniqueness}) has the following simple corollary.

\begin{corollary}
There exists a totally monotone function $h_\infty:(0,\infty)\to (0,\infty)$ with $h_\infty(x)\to \infty$ as $x\to 0$ and $h_\infty(x)\to 0$ as $x\to \infty$ such that if $h$ is any solution to $h''-2h'=h^2$ defined on $[0,\infty)$ then there exists $c>0$ such that $h(x)=h_\infty(x+c)$ for every $x\geq 0$.
\end{corollary}

\begin{remark}
The quantity $h_\infty(x)$ can be written in terms of the measure given to the set $\{$the unit ball is hit$\}$ by the four-dimensional superBrownian excursion measure started at $e^x$.
\end{remark}

We break the proof of Proposition~\ref{P:h_totally_monotone}
into a few steps, beginning with the following lemma.

\begin{lemma}
\label{L:h_decreasing}
Let $h$ be any solution to $h''-2h'=h^2$ defined on $[0,\infty)$. If $h$ is not identically zero then $h'(x)<0$ for all $x\geq 0$.
\end{lemma}

\begin{proof}[Proof of Lemma~\ref{L:h_decreasing}]
 Suppose not, so that there exists $x_0\geq 0$ with $h'(x_0)\geq 0$. Since $h''(x_0)=2h'(x_0)+h(x_0)^2\geq 0$ and $h'$ cannot vanish on any non-trivial interval (by uniqueness of forward and backward solutions to the IVP associated to the ODE), there must exist $x_1>x_0$ with $h'(x_1)>0$. Using that $h''-2h' \geq 0$, it follows by Lemma~\ref{L:Gronwall} that $h'(x)\geq e^{2(x-x_1)}$ for every $x \geq x_1$. Integrating this inequality, it follows that there exists $x_2\geq x_1$ such that $h'(x)$ and $h(x)$ are both bounded below by $1$ for all $x\geq x_2$, and hence that $h''(x)\geq 0$ for all $x\geq x_2$.
Now, let $y_0=x_2$ and for each $n\geq 1$ let $y_n\geq y_{n-1}$ be minimal such that $h'(y_n)=2^n$. For each $x\geq y_n$ we have that 
\[
h(x) = h(y_n) + \int_{y_n}^x h'(t) \d t \geq 2^n (x-y_n) 
\]
and hence by Lemma~\ref{L:Gronwall} that
\[
h'(x) \geq e^{2(x-y_n)}h(y_n) + 2^{2n} \int_{y_n}^x   (t-y_n)^2  \d t \geq \frac{1}{3}2^{2n} (x-y_n)^3
\]
for every $x\geq y_n$. This implies that
$y_{n+1}-y_n \leq (3 \cdot 2^{-n})^{1/3}$
for every $n\geq 0$, and since $\sum (3 \cdot 2^{-n})^{1/3} <\infty$ this shows that $h'$ blows up in finite time and contradicts the assumption that $h$ is defined on all of $[0,\infty)$.
\end{proof}

\begin{lemma}
\label{L:h_limit}
    Let $h$ be any solution to $h''-2h'=h^2$ defined on $[0,\infty)$. If $h$ is not identically zero then $h(x)>0$ for all $x\geq 0$ and $h(x)\to 0$ as $x\to \infty$.
\end{lemma}

\begin{proof}[Proof of Lemma~\ref{L:h_limit}]
Since $h'(x)<0$ for all $x\geq 0$, $h(x)$ converges to some $c \in [-\infty,\infty)$ as $x\to \infty$, and it suffices to prove that $c=0$. We will rule out the cases $c>0$ and $c<0$ separately.

\medskip

First suppose for contradiction that $c$ is positive, in which case $h(x)$ is positive for all $x\geq 0$. We first claim that $h''(x) \geq 0$ for every $x \geq0$. Suppose that this is not the case, so that there exists $x_0 \geq 0$ with $h''(x)< 0$. The second derivative $h''$ satisfies the differential inequality $h'''-2h''=2hh' \leq 0$, so that (by Lemma~\ref{L:Gronwall}) $h''(x)\leq e^{2(x-x_0)} h''(x_0)$ for every $x\geq x_0$. Integrating this inequality twice implies that $h(x)$ is eventually negative, contradicting the assumption that $c\geq 0$. This completes the proof that the second derivative is non-negative. Now, since $h(x)\to c$ and $h'(x)$ is monotone, we must have that $h'(x)\to 0$ as $x\to \infty$. But the ODE then implies that $h''(x)\to c^2>0$ as $x\to \infty$, a contradiction. This rules out the case that $c$ is positive.

\medskip

Now suppose that $c$ is negative, and let $x_0\geq 0$ be such that $h(x) \leq 0$ for all $x \geq x_0$. In this case, the second derivative satisfies the differential inequality $h'''-2h''=2hh'\geq 0$ for every $x\geq x_0$, and it follows by a similar argument to above that $h''(x) < 0$ for all $x> x_0$. As such, $h'(x)\leq h'(x_0)<0$ for all $x\geq x_0$, so that $c$ cannot be finite and negative. 

\medskip

We now argue as in the proof of Lemma~\ref{L:h_decreasing} that if $c=-\infty$ then in fact $h'$ blows up to negative infinity in finite time. Let $x_1\geq x_0$ be such that $h(x)\leq -1$ for all $x\geq x_1$ and, for each $n\geq 0$ let $y_n\geq x_0$ be minimal such that $h'(y_n) \leq - 2^n$. Integrating the differential inequality $h''-2h'\geq 1$ on $[x_1,\infty)$ yields that $y_n<\infty$ for every $n\geq 0$. Since $h'$ is decreasing on $[x_1,\infty)$, it follws as in the proof of Lemma~\ref{L:h_decreasing} that
\[h(x)\leq -2^n(x-y_n) \text{ and that } h'(x) \leq -2^{2n} \int_{y_n}^x (t-y_n)^3\]
for every $n\geq 0$ and $x\geq y_n$, so that $y_{n+1}\leq (3\cdot 2^{-n})^{1/3}$ for every $n\geq 0$. It follows as before that $h'$ blows up to negative infinity in finite time, contradicting the assumption that $h$ was defined on $[0,\infty)$.
\end{proof}

\begin{proof}[Proof of Proposition~\ref{P:h_totally_monotone}]
We have already shown that $h(x)>0$ for all $x\geq 0$, $h'(x)<0$ for all $x\geq 0$, and that $h(x)\to 0$ as $x\to \infty$. 
% It suffices to prove that $h$ is totally monotone, the claim that $h^{(n)}(x)\to 0$ as $x\to \infty$ for each $n \geq 0$ following by this and an easy induction. 
% 
% \medskip
We now prove by induction on $n$ that $(-1)^n h^{(n)}(x)>0$ for all $x\geq 0$ and that $h^{(n)}(x)\to 0$. Suppose that $n\geq 2$ and that $(-1)^n h^{(k)}(x)>0$ and $h^{(k)}(x)\to 0$ for all $0\leq k < n$ and $x\geq 0$. The $n$th derivative $h^{(n)}$ satisfies 
\begin{equation}
\label{E:h(n)_ODE}
h^{(n+1)}-2h^{(n)} = (h^2)^{(n-1)} =  \sum_{k=0}^{n-1} \binom{n-1}{k}h^{(k)}h^{(n-1-k)}.
\end{equation}
By the induction hypothesis, the right hand side is positive if $n$ is odd and negative if $n$ is even. Thus, it follows by Lemma~\ref{L:Gronwall} that if $(-1)^nh^{(n)}(x_0) \leq 0$ for some $x_0\geq 0$ then $(-1)^n h^{(n)}(x) \leq 0$ for all $x\geq x_0$. This is inconsistent with the induction hypothesis that $(-1)^{n-1} h^{(n-1)}(x)$ converges to $0$ from above as $x\to \infty$, so that we must have $(-1)^nh^{(n)}(x)>0$ for all $x\geq 0$ as claimed. 

It remains to prove that $h^{(n)}(x)\to 0$ as $x\to \infty$. Since $h^{(n-1)}(x)\to 0$ as $x\to \infty$, we must have that $\liminf_{n\to \infty} |h^{(n)}(x)|=0$.
% Suppose for contradiction that there exists $\eps>0$ such that $|h^{(n)}(x)|\geq 2\eps$ for an unbounded set of $x\geq 0$. 
By \eqref{E:h(n)_ODE} and the induction hypothesis, there exists $x_0 <\infty$ such that 
\[
|h^{(n+1)}(x)-2h^{(n)}(x)| \leq \eps
\]
for all $x\geq x_0$. Applying Lemma~\ref{L:Gronwall}, we deduce that 
\[
|h^{(n)}(x)-e^{2(x-x_1)}h^{(n)}(x_1)|\leq \eps  \int_{x_1}^x e^{2x-2t} \d t \leq \frac{\eps}{2} e^{2(x-x_1)} 
\]
for every $x\geq x_1 \geq x_0$. Thus, if $x_1 \geq x_0$ is such that $|h^{(n)}(x_1)| \geq \eps$ then $|h^{(n)}(x)| \geq \frac{\eps}{2} e^{2(x-x_1)}$ for every $x\geq x_1$, contradicting the fact that  $\liminf_{n\to \infty} |h^{(n)}(x)|=0$. This completes the proof of the induction step and hence of the proposition.
\end{proof}

We now compute the first-order asymptotics of $h$.

\begin{lemma}
\label{lem:h_first_order}
Let $h$ be any solution to $h''-2h'=h^2$ defined on $[0,\infty)$. If $h$ is not identically zero then $h(x)\sim 2/x$ as $x\to \infty$.
\end{lemma}

\begin{proof}

Define $\delta=h''=2h'+h^2$. Since $h$ is totally monotone and $\delta=h''$, we have that $\delta(x)\geq 0$ for all $x\geq 0$ and hence that $h$ satisfies the differential inequality $2h' \geq -h^2$. Integrating the resulting inequality on the derivative of $1/h$ yields that
\[
\frac{1}{h(x)}-\frac{1}{h(0)} = \int_0^x -\frac{h'(t)}{h^2(t)} \d t \leq \frac{x}{2}
\]
% 
% \[
% \frac{1}{h(x)}-\frac{1}{h(0)} \leq \frac{x}{2}
% \]
and hence that 
$h(x) \geq \frac{2+o(1)}{x}$
as $x\to \infty$.

We now prove the matching upper bound. We first claim that $\delta(x)=|h''(x)|=o(|h'(x)|)$ as $x\to \infty$.
Since $h$ is totally monotone, $|h''(x)/h'(x)|$ is decreasing  and hence converges from above to some limit $c\geq 0$. If $c>0$ then we have that $h''(x) \geq -ch'(x)$ for all $x\geq 0$ and hence by Gr\"onwall (Lemma~\ref{L:Gronwall}) that $|h'(x)| \leq e^{-cx} |h'(0)|$ for every $x\geq 0$. Since $h(x)\to 0$ as $x\to \infty$, this inequality is inconsistent with the sub-exponential lower bound $h(x) \geq (2+o(1))/x$, so that we must have $c=0$ as claimed. Using that $h''(x)=o(|h'(x)|)$ as $x\to \infty$, we can approximately simplify the ODE $h''+2h'=h^2$ by writing
\[
h'(x) \sim - \frac{1}{2} h^2.
\]
The claim follows by integrating the resulting asymptotic estimate $(1/h)'\sim 1/2$.
\end{proof}

\begin{remark}
It should be possible to adapt this proof to give an alternative derivation of the first-order asymptotics of $g_\lambda$ also. 
\end{remark}

Given Propositon~\ref{P:h_totally_monotone} and Lemma~\ref{lem:h_first_order}, the rest of the proof of the part of Theorem~\ref{T:gh} concerning $h$ is very similar to the part concerning $g$, and we omit the details.

%\subsection{Asymptotic behaviour of $g_\lambda$ when $\lambda$ is small}\label{sec:g_small}
%\input{sections_tex_files/gsmall4}

\subsection{Proof of Lemma \ref{L:C_lambda}}\label{sec:C_lambda}

In this section we prove Lemma~\ref{L:C_lambda}, providing an integral expression for the constant $C_\lambda$ which allows us to deduce continuity of $\lambda$ and the asymptotic expression $C_\lambda \sim 4/\lambda$ as $\lambda \downarrow 0$.

\begin{proof}[Proof of Lemma \ref{L:C_lambda}]
As usual, we will denote $\delta_\lambda = 2g_\lambda' - g_\lambda^2 = - g_\lambda''$.
By \eqref{E:proof_C_lambda1} and \eqref{E:proof_C_lambda2}, $C_\lambda$ can be written as
\begin{equation}
\label{E:C_lambda_integral_1}
C_\lambda = \int_1^\infty \left( \frac{\delta_\lambda(x)}{g_\lambda(x)^2} - \frac{1}{x} \right) \d x - \frac{2}{g_\lambda(1)} - 1.
\end{equation}
The first integration bound was chosen in an arbitrary way, meaning that for any $x_0 >0$, $C_\lambda$ can be written as
\begin{equation}
C_\lambda = \int_{x_0}^\infty \left( \frac{\delta_\lambda(x)}{g_\lambda(x)^2} - \frac{1}{x} \right) \d x - \frac{2}{g_\lambda(x_0)} - x_0 - \log x_0.
\end{equation}
This corresponds to the formula stated in \eqref{E:C_lambda}. We emphasise that the term $\log x_0$ was not apparent when $x_0$ was chosen to be equal to 1 since $\log 1 = 0$. (Note that the equation $g_\lambda'' + 2g'_\lambda = g_\lambda^2$ easily implies that the expression on the right hand side of the above equation does not depend on $x_0$.)

\medskip

We now prove continuity of $C_\lambda$. We use \eqref{E:C_lambda_integral_1}.
The maps $\lambda \mapsto g_\lambda(1)$ and
$
\lambda \mapsto \frac{\delta_\lambda(x)}{g_\lambda(x)^2} - \frac{1}{x}
$, $x>1$, are continuous since solutions to initial value problems depend continuously on their initial conditions.
 By \eqref{E:C_lambda} and the dominated convergence theorem, to conclude that $C_\lambda$ is continuous it is enough to check that for each $\lambda \in (0,1)$ there is a neighbourhood of $\lambda$ on which $\frac{\delta_\lambda(x)}{g_\lambda(x)^2} - \frac{1}{x}$ is dominated by some integrable function. 
% This domination was checked in the proof of Lemma \ref{L:ODE_asymp} for a fixed $\lambda$; we now explain the uniformity of this domination.
Let $0 <\lambda_1 < \lambda_2<1$. By Theorem \ref{T:proba_representation}, $g_\lambda(x)$ is monotone in $\lambda$, so for all $\lambda \in [\lambda_1, \lambda_2]$,
$
|g_{\lambda_1}(x)| \leq |g_\lambda(x)| \leq |g_{\lambda_2}(x)|,
$
meaning that we have uniform control on $g_\lambda(x)$ in the sense that
\[
\sup_{\lambda \in [\lambda_1,\lambda_2]} |g_\lambda(x)| = |g_{\lambda_2}(x)| = \frac{2}{x} + O\left(\frac{\log x}{x^2} \right)
\]
and
\[
\inf_{\lambda \in [\lambda_1,\lambda_2]} |g_\lambda(x)| = |g_{\lambda_1}(x)| = \frac{2}{x} + O\left(\frac{\log x}{x^2} \right).
\]
Here and below we use the notation that $f(x) = O(g(x))$ to mean that there exists there exists a constant $C>0$ such that $|f(x) | \le C |g(x)|$. In particular the reader should keep in mind that the left hand side might well be negative.

Since we obtained the bounds \eqref{E:proof_ODE_asymp6} and \eqref{E:proof_ODE_asymp9} on $\delta_\lambda$ using only first-order estimates on $|g_\lambda|$, we deduce that the same bounds on $\delta$ hold uniformly in $\lambda \in [\lambda_1, \lambda_2]$ in the sense that
\[
\inf_{\lambda\in [\lambda_1,\lambda_2]} \delta_\lambda(x) = \frac{4}{x^3} + O\left(\frac{\log x}{x^4} \right)\qquad   \text{ and } \qquad \sup_{\lambda\in [\lambda_1,\lambda_2]} \delta_\lambda(x) = \frac{4}{x^3} + O\left(\frac{\log x}{x^4} \right)
\]
as $x\to \infty$. It follows that
\[
\left|\frac{\delta_\lambda(x)}{g_\lambda(x)^2}-\frac{1}{x}\right|=O\left(\frac{\log x}{x^2}\right)
\]
uniformly in $\lambda \in [\lambda_1,\lambda_2]$. This provides the required domination by an integrable function and concludes the proof of continuity of $C_\lambda$. The fact that $C_\lambda$ is decreasing on $(0,1)$ follows immediately from the fact that $g_\lambda(x)$ is a decreasing function of $\lambda\in (0,1)$ as proven in Lemma~\ref{L:weak_uniqueness_intermediate}.
% This proves that the domination we have is actually locally uniform in $\lambda$ concluding the proof of the continuity of $C_\lambda$.

\medskip

\begin{comment}
We apply \eqref{E:C_lambda} with $x_0$ of the form $x_0 = \mu_0 / \lambda$ where $\mu_0$ is any positive number. Using Lemma \ref{L:Phase3} to estimate the term $2/g_\lambda(\mu_0/\lambda)$ that appears there, we obtain that
\begin{align}
\nonumber
C_\lambda &
= \int_{x_0}^\infty \left( \frac{\delta_\lambda(x)}{g_\lambda(x)^2} - \frac{1}{x} \right) \d x + \frac{\mu_0 + 4}{\lambda} - \frac{\mu_0}{\lambda} + o(\lambda^{-1}) \\
& = \frac{4}{\lambda} +
\frac{1}{\lambda} \int_{\mu_0}^\infty \left( \frac{\delta_\lambda(t/\lambda)}{g_\lambda(t/\lambda)^2} - \frac{\lambda}{t} \right) \d t
+ o (\lambda^{-1})
\label{E:Clambda_formula2}
\end{align}
after a change of variable. \red{[TH: More detail in the rest of this argument? I don't quite follow as written.]}
Now, by Lemma \ref{L:g_small_lambda},
\[
\frac{t}{\lambda} g_{\lambda}(t/\lambda) \to -\frac{2t}{t+4}
\]
uniformly in $t>0$. Integrating the equation \eqref{E:proof_ODE_asymp2} satisfied by $\delta_\lambda$ and using the above asymptotics, one deduces that
\[
\frac{t^3}{\lambda^3} \delta_\lambda(t/\lambda) \to \frac{4t^3}{(t+4)^3}
\]
uniformly in $t>0$. Plugging these two estimates back to \eqref{E:Clambda_formula2}, we finally have
\begin{align*}
C_\lambda = \frac{4}{\lambda} - 4 \int_{\mu_0}^\infty \frac{\d t}{t(t+4)} + o(\lambda^{-1}).
\end{align*}

\bigskip
\end{comment}

We now show that $C_\lambda \sim 4/\lambda$ as $\lambda \downarrow 0$. Roughly speaking, we will check that the asymptotic expansion \eqref{E:asymp_g_Clambda} is a good approximation as soon as $x \geq 1/\lambda$ and conclude using Lemma \ref{L:Phase3}. The following line of argument is close to the proof of Lemma \ref{L:ODE_asymp}.
Let $\lambda \in (0,1/2]$. All the big-O written in the rest of the proof are uniform with respect to $\lambda$. Let $x_\lambda = 1/\lambda$ and $x \geq x_\lambda$. As noticed above, the upper bounds on $|g_\lambda|$ and $|\delta_\lambda|$ are uniform in $\lambda(0,1/2]$:
\begin{equation}
\label{E:proof_Clambda2}
|g_\lambda(x)| \leq |g_{1/2}(x)| \leq \frac{2}{x} + O(1) \frac{\log x}{x}
\end{equation}
and
\begin{equation}
\label{E:proof_Clambda1}
\delta(x) \leq \frac{4}{x^3} + O(1) \frac{\log x}{x^4}.
\end{equation}
Obtaining uniform lower bounds require more care. (Indeed, we recall from Lemma \ref{L:g_small_first} that $g_\lambda \to 0$ pointwise as $\lambda \to 0$.) First, notice that since $1/x \leq \lambda$,
\[
|g_\lambda(x)| \geq |g_{\lambda=1/x}(x)| \sim \frac{2}{5x}
\quad \quad \text{as}~ x \to \infty
\]
by Lemma \ref{L:Phase3}. Hence, there exists $\lambda_0 \in (0,1/2]$ small enough, such that for all $\lambda \in (0,\lambda_0]$ and $x \geq x_\lambda$, $|g_\lambda(x)| \geq 1/(5x)$. Injecting this estimate in \eqref{E:proof_Clambda1} shows that
\[
-\left( \frac{1}{g_\lambda(x)} + \frac{x}{2} \right)' = \frac{g_\lambda'(x)}{g(x)^2} - \frac12 = \frac{\delta_\lambda(x)}{2g_\lambda(x)^2} \leq \frac{O(1)}{x}, \quad \quad x \geq x_\lambda,
\]
where we stress again that all big-$O$ estimates are uniform in $\lambda \in (0,1/2]$ and $x\geq x_\lambda$.
Integrating between $x_\lambda$ and some $x \geq x_\lambda$ shows that
\[
|g_\lambda(x)| \geq \frac{2}{x + O(1) \log(x/x_\lambda) + 2/|g_\lambda(x_\lambda)| - x_\lambda}.
\]
By Lemma \ref{L:Phase3}, the difference $2/|g_\lambda(x_\lambda)| - x_\lambda$ is asymptotic to $4/\lambda$ as $\lambda \to 0$. Using in particular that it is bounded by $O(1/\lambda)$, we can inject this improved lower bound on $|g_\lambda|$ back in \eqref{E:proof_Clambda1} to obtain that
\[
-\left( \frac{1}{g_\lambda(x)} + \frac{x}{2} \right)' = \frac{\delta_\lambda(x)}{2 g_\lambda(x)^2} \leq \frac{1}{2x} + \frac{O(1)}{\lambda x^2} + O(1) \frac{\log x}{x^2}, \quad \quad x \geq x_\lambda.
\]
Integrating this inequality finally shows that
\[
|g_\lambda(x)| \geq \frac{2}{x + \log(x/x_\lambda) + 2/|g_\lambda(x_\lambda)| - x_\lambda + O(1)/(\lambda x) + O(1)\log x/x}.
\]
In particular,
\[
C_\lambda = \lim_{x \to \infty} \frac{2}{|g_\lambda(x)|} - x - \log x \leq \frac{2}{|g_\lambda(x_\lambda)|} - x_\lambda - \log x_\lambda.
\]
We can now use this lower bound on $|g_\lambda|$ to get a lower bound on $\delta_\lambda$ using the differential inequality $\delta_\lambda' + 2\delta_\lambda = |g_\lambda|\delta + |g_\lambda|^3 \geq |g_\lambda|^3$, which follows from \eqref{E:proof_ODE_asymp2}. We get that 
\[
\delta_\lambda(x) \geq \frac{4}{x^3} + O(1) \frac{1}{\lambda x^4} + O(1) \frac{\log x}{x^4}, \quad \quad x \geq 2x_\lambda.
\]
See the paragraph around
\eqref{E:proof_ODE_asymp9} for details. This lower bound on $\delta_\lambda$ together with the upper bound \eqref{E:proof_Clambda2} on $|g_\lambda|$ shows that
\[
\int_{2x_\lambda}^\infty
\left( \frac{\delta_\lambda(x)}{g_\lambda(x)^2} - \frac{1}{x} \right) \d x \geq O(1) \int_{2x_\lambda}^\infty \left( \frac{1}{\lambda x^2} + \frac{\log x}{x^2} \right) \d x \geq O(1) \frac{1}{\lambda x_\lambda} + O(1) \frac{\log x_\lambda}{x_\lambda} \geq O(1).
\]
Applying the formula \eqref{E:C_lambda} for $C_\lambda$ to $x_0 = 2x_\lambda$, we get that
\begin{align*}
C_\lambda & = \frac{2}{|g_\lambda(2x_\lambda)|} - 2x_\lambda - \log 2x_\lambda + \int_{2x_\lambda}^\infty
\left( \frac{\delta_\lambda(x)}{g_\lambda(x)^2} - \frac{1}{x} \right) \d x  \geq \frac{2}{|g_\lambda(2x_\lambda)|} - 2x_\lambda - \log 2x_\lambda + O(1).
\end{align*}
Wrapping up, we have proved that
\[
\frac{2}{|g_\lambda(2x_\lambda)|} - 2x_\lambda - \log 2x_\lambda + O(1) \leq C_\lambda \leq \frac{2}{|g_\lambda(x_\lambda)|} - x_\lambda - \log x_\lambda.
\]
By Lemma \ref{L:Phase3}, both sides of the above inequalities are asymptotically equivalent to $4/\lambda$ as $\lambda \to 0$. This concludes the proof.
\end{proof}

\section{Tail estimates with killing II: Lower bounds (proof of Theorem \ref{T:maintail4})}\label{sec:tail4}
The aim of this section is to prove the lower bound of Theorems \ref{T:maintail4}, completing the proof of that theorem. The proof relies on the following precise asymptotic of the Laplace transform of the number of pioneers, which we prove in this section using Theorem~\ref{T:hitprob4} as an input.

\begin{theorem}\label{T:laplace_sharp}
Let $x_0 >0$. For each $\delta >0$, there exist constants $T_0, R_0>0$ and $x_* >x_0$ depending only on $x_0$ and $\delta$ such that
\begin{equation}
 e^{-\delta} \frac{2(e^{2x_0} - 1 )}{R^2 x} \frac{x}{T} \leq \EXPECT{e^{-x_0} R,R}{ \exp \left\{ \left( 1 - \frac{T}{x} \right) \frac{2N_{e^{-x}R}}{(x+\log x)(e^{-x}R)^2} \right\} -1} \leq e^{\delta} \frac{2(e^{2x_0} - 1 )}{R^2 x} \frac{x}{T}
\end{equation}
for all $T \in [T_0,2 T_0]$, $R \geq R_0$, and $x \in [x_*,\log R]$.
\end{theorem}

% The expression $e^{\pm \delta}$ on the right hand side stands for a real number in the interval $ [e^{ - \delta}, e^{+ \delta}]$. 

\begin{remark} 
As we will see in Lemma \ref{L:hit_killing} below, the term $\frac{2(e^{2x_0} - 1 )}{R^2 x}$ corresponds to the asymptotic behaviour of the hitting probability of $\partial B(0,e^{-x}R)$ starting from a single particle on $\partial B(0,e^{-x_0} R)$. Therefore, Theorem \ref{T:laplace_sharp} says in a very strong and  precise way that the law of the number of pioneers $N_{e^{-x} R}$ conditioned to be positive is close to that of an exponential random variable with mean $\frac12 (x+\log x) (e^{-x}R)^2$.

\medskip
 
The logarithmic shift in the mean (which corresponds to shifting by $\log \log R$ in the original scale) is important in the sense that the theorem would not be true without it. This logarithmic shift will eventually come from the second order term in the expansion in Theorem \ref{T:gh}. To see why it matters, one simply needs to expand
\[
\left( 1 - \frac{T}{x} \right) \frac{1}{x+\log x} = \frac{1}{x} - \frac{\log x}{x^2} - \frac{T}{x^2} + \cdots.
\]
The dependence in $T$, which is the key feature of this theorem, appears only from the third term onward; if we replaced $\log x$ by any other quantity differing by a divergent term, this effect would overwhelm that of changing $T$. 
% (Note that this does not explain why $\log x$ is the correct second term in the mean of the exponential, only that this second term is important to identify correctly.
\end{remark}

%\red{[TH: Is it worth saying anything more about the role of the $\log x$ term here, which is really a $\log\log$ term?]}

Before proving Theorem \ref{T:laplace_sharp}, we will first explain how we will use it to deduce Theorem \ref{T:maintail4}. The idea is of course to convert information on the Laplace transform into information on the tail. This is usually achieved by the means of the classical Tauberian theorem of Hardy--Littlewood. However, the assumptions of that theorem are too restrictive to apply given only the information we obtain in Theorem \ref{T:laplace_sharp}. Roughly speaking, in order to use such a theorem we would require asymptotics of the Laplace transform for values of $1-\tfrac{T}{x}$ arbitrarily close to $1$. 
In our case, even if we take $x$ as large as possible (namely, take $x = \log R$ so that we consider the pioneers on the ball of radius $1$) our estimates on the Laplace transform hold only on an interval of the type $[1- t_1/\log R, 1- t_2/\log R]$. The Hardy--Littlewood theorem also requires an assumption of regular or slow variation, something which we do not have at our disposal.
% 
% \medskip
% 
% To deduce Theorems
To proceed, we therefore require a finitary improvement of the classical Tauberian theorem. 
Fortunately, we manage to prove such a finitary improvement using fairly soft arguments in Section
\ref{S:Tauberian}; we anticipate that this will be useful in other problems too. This is an improvement over the classical theorem in the sense that the assumptions are less restrictive, but of course this is reflected in the conclusion too, which is not as strong.

\medskip

% Section \ref{sec:tail4}
% The rest of this section is organised as follows.
% Assuming Theorem \ref{T:laplace_sharp}, and using our Tauberian theorem, we prove Theorem \ref{T:maintail4} in Section \ref{S:proof_tail}.
% The rest of Section \ref{sec:tail4} will be dedicated to proving Theorem \ref{T:laplace_sharp}, i.e. the desired estimates on the Laplace transform of the number of pioneers.
% In Section \ref{sec:intermediate4}, we state without proofs key intermediate results related to this Laplace transform. We then show how Theorems \ref{T:hitprob4} and \ref{T:laplace_sharp} follow from these intermediate results. 
% The rest of Section \ref{sec:tail4} will be dedicated to the proofs of the intermediate results stated in Section \ref{sec:intermediate4}.

We prove Theorem~\ref{T:laplace_sharp} in Section~\ref{sec:intermediate4}, formulate and prove our finitary Tauberian theorem in Section~\ref{S:Tauberian}, and then deduce Theorem~\ref{T:tail_general} in Section~\ref{S:proof_tail}.

\subsection{Estimates on the Laplace transform}\label{sec:intermediate4}
In this section we prove Theorem \ref{T:laplace_sharp}. The proof will rely on the results of Sections~\ref{sec:upper_bound} and \ref{sec:series_expansion}.
% The proofs of these intermediate results will be the concern of Sections \ref{sec:series_expansion} and \ref{sec:C_lambda}.

\begin{proof}[Proof of Theorem \ref{T:laplace_sharp}]
To ease notations, we will  prove the statement only for $x = L$ so that $e^{-x}R = 1$. The same line of argument proves the result in general: Simply replace the number of pioneers $N_1$ on the unit sphere by the renormalised number of pioneers $N_{e^{-x}R}/(e^{-x}R)^2$ on the sphere $\partial B(e^{-x}R)$ and each occurrence of $L$ by $x$. Throughout the proof, we will write $e^{\pm \delta}$ for a quantity that is bounded between $e^{-\delta}$ and $e^\delta$, the precise value of which may vary from line to line.

\medskip

We first introduce some parameters.
Let $x_0>0$ be fixed, and let $\delta >0$ be small.
% and let $A>0$ be much larger than $1/\delta$.
Let us first pick $\mu_0 = \mu_0(\delta)$ small enough so that, defining $A := 1/\mu_0$,
\begin{equation}
\label{E:proof_laplace4}
\frac{1}{\delta} e^{-\delta A/2} \leq \delta
\quad \text{and} \quad
A \geq 8 e^{10\delta}.
\end{equation}
By Lemmas \ref{L:Phase3} and \ref{L:C_lambda}, for all $\mu >0$, we have that
\[
\lim_{\xi \to \infty} \xi g_{\mu/\xi}(\xi) = \frac{-2}{1+4/\mu}
\quad \text{and} \quad \lim_{\xi \to \infty} \frac1\xi C_{\mu/\xi} = \frac{4}{\mu}
\]
uniformly with respect to $\mu$ on compact subsets of $(0,\infty)$. Thus, by choosing $\mu_0$ smaller if necessary, we can find $\xi = \xi(\delta,\mu_0) \geq \max\{x_0,3\mu_0\}$ large such that if $\mu \in [\mu_0,3\mu_0]$ then
% \mu \in [1/A,2/A]$,
% so that
% \[
% g_{\mu/x}(x) \sim \frac{-2}{x+ C_{\mu/x}}
% \]
% as $\mu \to x$, uniformly in $\mu \in [1/A,2/A]$.
% Since $A$ is much larger than $1/\delta$, there exists $x_2 \geq x_0$ such that if $x \geq x_2$ then 
\begin{equation}\label{E:proof_laplace1}
\frac{e^{-\delta}}{C_{\lambda = \mu/\xi}} \leq -\frac{1}{2} g_{\lambda = \mu/\xi}(\xi) \leq \frac{e^\delta}{C_{\lambda = \mu/\xi}}
\quad \text{and} \quad
4e^{-\delta} \frac{\xi}{\mu} \leq C_{\lambda=\mu/\xi} \leq 4e^\delta \frac{\xi}{\mu}.
\end{equation}
Moreover, by Proposition \ref{P:upper_bound_rough} applied with $\eps=1/2$, we may assume that $\xi$ is large enough (depending only on $x_0$ and $\delta$) so that for all $n \geq 1$,
\begin{equation}
\label{E:proof_laplace3}
\PROB{e^{-x_0} R,R}{N_{e^{-\xi}R} \geq \frac{\xi(e^{-\xi} R)^2}{2} n} \leq \frac{C(x_0)}{\xi (e^{-x_0} R)^2} e^{-n/2} ; \quad \text{ and } \quad e^{-2 \xi} \le 1- e^{-\delta}
\end{equation}
for some constant $C(x_0)$ depending only $x_0$.

Now that our parameters are defined, we can start more concretely the proof.
Since $\xi \geq 3\mu_0$, it follows from Theorem \ref{T:proba_representation} that 
\[
1 - \EXPECT{e^{-\xi}R,R}{ \left( 1 - g_\lambda(L) \right)^N }
= \frac{1}{(e^{-\xi}R)^2} g_\lambda(\xi)
\]
for all $\lambda \in [\mu_0/\xi,3\mu_0/\xi]$. Applying Theorem \ref{T:gh} (more precisely \eqref{E:asymp_g_Clambda}), it follows that
\[
g_\lambda(L) = \frac{-2}{L + \log L + C_\lambda +o(1)}
\]
as $L\to \infty$,
where, here and in what follows, all uses of asymptotic notation are taken to be uniform in $\lambda \in [\mu_0/\xi,3\mu_0/\xi]$ (the fact that the estimate is uniform in this range follows by continuity of $C_\lambda$). Writing
\[
N \log \left( 1 + \frac{2}{L + \log L + C_\lambda + o(1)} \right) = \left( 1 - \frac{C_\lambda +o(1)}{L} \right) \frac{2N}{L + \log L}
\]
and using \eqref{E:proof_laplace1}, we deduce that there exists a quantity $\eta(\lambda, L)$ which is jointly continuous in $\lambda$ and $L$ with $\eta(\lambda, L)\to 0$ as $L\to \infty$ uniformly over $\lambda \in [\mu_0/\xi, 3\mu_0/\xi]$, such that
% leads to
\[
1 - \EXPECT{e^{-\xi}R,R}{ \exp \left\{ \left( 1 - \frac{C_\lambda + \eta(\lambda, L)}{L} \right) \frac{2N}{L+\log L} \right\} }
= - \frac{2 e^{\pm \delta}}{(e^{-\xi}R)^2 C_\lambda}
\]
for every $\lambda \in [\mu_0/\xi,3\mu_0/\xi]$.
Starting now from a point on the sphere $\partial B(e^{-x_0} R)$, we have
\begin{align*}
&  \EXPECT{e^{-x_0}R,R}{ 1 - \exp \left\{ \left( 1 - \frac{C_\lambda + \eta(\lambda, L)}{L} \right) \frac{2N}{L + \log L} \right\} } \\
&\hspace{2cm} = \sum_{n = 1}^\infty \PROB{e^{-x_0}R,R}{N_{e^{-\xi}R} = n } \left( 1 - \EXPECT{e^{-\xi}R,R}{\exp \left\{ \left( 1 - \frac{C_\lambda + \eta(\lambda, L)}{L} \right) \frac{2N}{L + \log L} \right\} }^n \right) \\
&\hspace{2cm} = \sum_{n = 1}^\infty \PROB{e^{-x_0}R,R}{N_{e^{-\xi}R} = n }
\left( 1 - \left( 1 + \frac{2e^{\pm \delta}}{(e^{-\xi}R)^2 C_\lambda} \right)^n \right),
\end{align*}
where $e^{\pm \delta}$ may represent a different quantity bounded between $e^{-\delta}$ and $e^{\delta}$ in each term of the sum.
The linearisation $(1+u)^n \approx 1+nu$ is a good approximation as long as $nu$ is small. Using the estimate \eqref{E:proof_laplace1} on $C_\lambda$ and recalling that $A = 1/\mu_0, \lambda = \mu/\xi$, we deduce after some elementary computations that for $n \leq n_1 := \delta (e^{-\xi} R)^2 A\xi$,
\[
1 - \left( 1 + \frac{2e^{\pm \delta}}{(e^{-\xi}R)^2 C_\lambda} \right)^n = - \frac{2e^{\pm 10\delta}}{(e^{-\xi}R)^2 C_\lambda} n.
\]
This implies that 
\begin{align}
\label{E:proof_laplace2}
& \EXPECT{e^{-x_0}R,R}{1- \exp \left\{ \left( 1 - \frac{C_\lambda+\eta(\lambda, L)}{L} \right) \frac{2N}{L + \log L} \right\} }  \\
\nonumber
&\hspace{3cm} = - e^{\pm 10 \delta}
\EXPECT{e^{-x_0}R,R}{ N_{e^{-\xi}R} \indic{ N_{e^{-\xi}R} \leq n_1 } } \frac{2}{(e^{-\xi}R)^2 C_\lambda} \\
\nonumber
&\hspace{6cm} + \sum_{n = n_1+1}^\infty \PROB{e^{-x_0}R,R}{N_{e^{-\xi}R} = n } \left( 1 - \left( 1 + \frac{2e^{\pm \delta}}{(e^{-\xi}R)^2 A\xi} \right)^n \right).
\end{align}
It remains to argue that the contribution of the sum above is small and that the restriction to the event $\{ N_{e^{-\xi}R} \leq n_1 \}$ leaves the expectation of $N_{e^{-\xi} R}$ almost unchanged. This follows from \eqref{E:proof_laplace3}: for all $k \geq 1$, 
%according to this proposition, if $\xi$ is large enough, then we can bound for all $n \geq 1$,
%\[
%\PROB{e^{-x_0} R,R}{N_{e^{-\xi}R} \geq \frac{\xi(e^{-\xi} R)^2}{2} n} \leq \frac{\red{O(1)}}{\xi (e^{-x_0} R)^2} e^{-n/2}.
%\]
%This gives for all $k \geq 1$
\begin{align*}
& \Big| \sum_{n = kn_1+1}^{(k+1)n_1} \PROB{e^{-x_0}R,R}{N_{e^{-\xi}R} = n } \left( 1 - \left( 1 + \frac{2e^{\pm \delta}}{(e^{-\xi}R)^2 A\xi} \right)^n \right) \Big| \\
&\hspace{3cm} \leq \PROB{e^{-x_0}R,R}{N_{e^{-\xi}R} \geq kn_1 + 1} \left( 1 + \frac{2e^{\pm \delta}}{(e^{-\xi}R)^2 A\xi} \right)^{(k+1)n_1} \\
&\hspace{3cm} \leq \frac{C(x_0)}{\xi (e^{-x_0} R)^2} \exp \left( - k\delta A + 2(k+1) e^{\pm 10 \delta} \delta \right) \leq \frac{C(x_0)}{\xi (e^{-x_0} R)^2} \exp \left( - k\delta A / 2 \right)
\end{align*}
since $A\geq 8e^{10\delta}$. This proves that the sum on the right hand side of \eqref{E:proof_laplace2} is at most
\[
\frac{C'(x_0)}{\xi (e^{-x_0} R)^2} \frac{1}{\delta A} \exp  \left( - \delta A / 2 \right)
\leq \frac{\delta C'(x_0)}{(e^{-x_0} R)^2 C_\lambda},
\]
where we used \eqref{E:proof_laplace4} and the fact that $C_\lambda$ is of order $A \xi$ (see \eqref{E:proof_laplace1}) in the last inequality.
Similarly, one can show that the expectation on the right hand side of \eqref{E:proof_laplace2} equals $e^{\pm 10\delta} \EXPECT{e^{-x_0}R,R}{ N_{e^{-\xi}R} }$.
%\[
%\EXPECT{e^{-x_0}R,R}{ N_{e^{-\xi}R} \indic{ N_{e^{-\xi}R} \leq n_1 } } = e^{o(1)} \EXPECT{e^{-x_0}R,R}{ N_{e^{-\xi}R} }
%\]
The expectation $\EXPECT{e^{-x_0}R,R}{ N_{e^{-\xi}R} }$ is equal to the probability that a Brownian trajectory starting on the sphere $\partial B(e^{-x_0} R)$ hits $\partial B(e^{-\xi} R)$ before $\partial B(R)$ which is equal to 
$$
\frac{e^{2x_0} - 1}{e^{2\xi} - 1} = e^{\pm\delta} e^{-2\xi} ( e^{2x_0} -1 ),
$$ 
by \eqref{E:proof_laplace3}.
Wrapping things up, we have obtained that
\[
\EXPECT{e^{-x_0}R,R}{1- \exp \left\{ \left( 1 - \frac{C_\lambda+\eta(\lambda,L)}{L} \right) \frac{2N}{L + \log L} \right\} }
= - (1 \pm C''(x_0) \delta) \frac{2 (e^{2x_0} -1)}{R^2 C_\lambda}.
\]
By continuity of $\lambda \mapsto C_\lambda + \eta(\lambda,L)$ for each fixed $L$, the intermediate value theorem and \eqref{E:proof_laplace1}, $\{ C_\lambda+\eta(\lambda,L), \lambda \in [A/\xi, 3A/\xi]\}$ contains the interval $[T_0,2T_0]$ with $T_0 = C_{2.99A/\xi}$ if $L$ is large enough. The above estimate thus implies the claim.
\end{proof}

\subsection{A finitary Tauberian theorem}\label{S:Tauberian}
In this section we prove the following finitary Tauberian theorem, which will be applied when we deduce Theorem~\ref{T:maintail4} from Theorem~\ref{T:laplace_sharp}.

\begin{theorem}\label{T:tauberian}
For each $0<a<1$ and $0<c<C<\infty$ there exists $\delta>0$ such that if $T>1$ and $X$ is a non-negative random variable satisfying
\[ 
\frac{e^{-\delta}}{1-\lambda} \leq \bbE \left[e^{\lambda X} \right] \leq \frac{e^{\delta}}{1-\lambda}
\]
for every $1-C/T \leq \lambda \leq 1-c/T$ then
% $\mu([1-\eps,1+\eps]) \geq \eps/e$.
\[
\bbP\Bigl((1-a)T \leq X \leq (1+a)T\Bigr) \geq \delta a T 
% \exp\left[-\frac{s+s \eps}{1-s} \right]
 \exp\left[-(1+a) T \right].
\]
\end{theorem}

\begin{remark}
The proof of this theorem also yields an upper bound of the same form. We focus on the lower bound since upper bounds on the tail are easy to obtain from generating function estimates via Markov's inequality.
\end{remark}

We begin the proof of this theorem with the following lemma.

\begin{lemma}
\label{lem:Tauberian}
For each $0<\eps<a<1$, $0<c_1<c_2<1$, and $A<\infty$ there exists $\delta>0$ such that if $\mu$ is a probability measure on $[0,\infty)$ satisfying
\[ 
\frac{e^{-\delta}}{1-\lambda} \leq \int_0^\infty e^{\lambda x} \d \mu(x) \leq \frac{e^{\delta}}{1-\lambda}
\]
for every $c_1\leq \lambda \leq c_2$ then
% $\mu([1-\eps,1+\eps]) \geq e^{-1+\delta} \eps$.
\[
\mu([y-a,y+a]) \geq (1-\eps)(1-e^{-2a})e^{-y+a}
\]
for every $0\leq y \leq A$.
\end{lemma}

Note that $e^{-y+a}(1-e^{-2a})$ is equal to the probability that an Exponential$(1)$ random variable belongs to the interval $[y-a,y+a]$.

\begin{proof}[Proof of Lemma~\ref{lem:Tauberian}]
Suppose not. Then there exists $0<\eps<a<1$, $0<c_1,c_2<1$, $A<\infty$, a sequence of measures $(\mu_n)_{n\geq 1}$ and a sequence of numbers $(y_n)_{n\geq 0}$ in $[0,A]$ such that 
\[
\frac{e^{-1/n}}{1-\lambda} \leq \int_0^\infty e^{\lambda x} \d \mu_n(x) \leq \frac{e^{1/n}}{1-\lambda}
\]
for every $n\geq 1$ and $c_1\leq \lambda \leq c_2$ but for which $\mu_n([y_n-a,y_n+a])\leq (1-\eps)(1-e^{-2a})e^{-y_n+a}$ for each $n\geq 1$. Taking a subsequence if necessary, we may take $y_n$ to converge to a limit $y\in [0,A]$. Letting $0<\eta<1$ be such that $(1-e^{-2(a-\eta)})e^{-y+a-\eta}>(1-\eps)(1-e^{-2a})e^{-y+a-\eps}$, it follows that $\mu_n([y-a+\eta,y+a-\eta]) \leq \mu_n([y_n-a,y_n+a])\leq (1-\eps)(1-e^{-2a})e^{-y_n+a}$ for all sufficiently large $n$.
If $\mu$ is the law of an Exp(1) random variable then $\mu([y-a+\eta,y+a-\eta]) = (1-e^{-2(a-\eta)})e^{-y+a-\eta}>(1-\eps)(1-e^{-2a}) e^{-y+a-\eps}$, so to reach a contradiction it suffices to prove that $\mu_n$ converges weakly to $\mu$ as $n\to\infty$; this is clear since the moment generating function of any subsequential limit of the $\mu_n$'s agrees with the moment generating function of $\mu$ on the interval $[c_1,c_2]$.
%To prove this, note that if $\mu'$ is any subsequential limit of the $\mu_n$'s then $\mu'$ must have the same MGF as $\mu$ on the interval $[c_1,c_2]$.
%By Pringheim's Theorem the power series representation of the moments of $\mu'$ must hold up to the first positive singularity of the MGF, and it follows that $\mu$ and $\mu'$ have the same moments. Since $\mu$ and $\mu'$ have exponential tails they are characterised by their moments and are therefore equal. [I guess there is a more standard way to do this.]
 % Thus $(\mu)_{n\geq 1}$ does not converge weakly to $\mu$, contradicting the above lemma.
\end{proof}

\begin{proof}[Proof of Theorem \ref{T:tauberian}]
We will prove the theorem by applying Lemma~\ref{lem:Tauberian} to a certain biased and rescaled version of $X$.
 Fix $T>1$ and suppose that $X$ satisfies
\[
\frac{e^{-\delta}}{1-\lambda}\leq \bbE\left[e^{\lambda X}\right]\leq \frac{e^{\delta}}{1-\lambda}
\]
for every $1-C/T \leq \lambda \leq 1-c/T$. Let $s=1-(c+C)/2T$ and let $Z_s$ be a random variable whose law is obtained by biasing the law of $X$ by $e^{sX}$ and multiplying the resulting random variable by $(1-s)$.  The moment generating function of $Z_s$ satisfies
\[
\frac{e^{-2\delta}}{1-\lambda} \leq \bbE\left[e^{\lambda Z_s}\right] = \frac{\bbE\left[e^{s X+(1-s)\lambda X}\right]}{\bbE\left[e^{s X}\right]} \leq \frac{e^{2\delta}}{1-\lambda}
\] 
for every $\lambda$ satisfying $1-C/T\leq s+(1-s)\lambda \leq 1-c/T$. This condition holds if and only if 
$1-2C/(c+C) \leq \lambda \leq 1-2c/(c+C)$. Thus, it follows from Lemma~\ref{lem:Tauberian} (applied with ``$y$'' equal to $T$, ``$a$'' equal to $(1-s)Ta/2=(c+C) a/4$, and ``$\eps$'' equal to $1/2$) that if $\delta>0$ is sufficiently small as a function of $a$, $c$, and $C$ then
\[
\bbP\left(1-a \leq \frac{Z_s}{(1-s)T} \leq 1+ a\right) \geq (1-\eps)\left(1-\exp\left(-\frac{(1-s)Ta}{2}\right)\right) e^{-1+a}
\]
and hence that
\begin{align*}
\bbP\left((1-a)T \leq X \leq (1+a)T\right) 
&= \bbE [e^{ sX}] \bbE\left[e^{-\frac{s Z_s}{1-s}} \mathbf{1}_{
\left\{1-a \leq \frac{Z_s}{(1-s)T} 
\leq 1+ a\right\} } \right] 
\\&\geq  \frac{1}{2} \frac{e^{-\delta}}{1-s} \left(1-\exp\left(-\frac{(1-s)Ta}{2}\right)\right)  \exp\left[-s(1+a)T -1+a\right].
\end{align*}
Since $s\leq 1$, the claim follows by using the approximation $1-\exp[-\frac{1}{2}(1-s)Ta]\ge c'(1-s)Ta$ (where $c'>0$ depends only on $c$ and $C$) and noting that the two factors of $1-s$ cancel.
\end{proof}

\subsection{Proof of Theorem \ref{T:maintail4}}
\label{S:proof_tail}
This section combines the estimates on the Laplace transform (Theorem \ref{T:laplace_sharp}) together with the Tauberian theorem from Section \ref{S:Tauberian} to prove Theorem \ref{T:maintail4}.

\begin{proof}[Proof of Theorem \ref{T:maintail4}]
Let $\eps >0$. We apply Theorem \ref{T:tauberian} with ${a=\eps}$, $c=1$, and $C=2$, letting $\delta >0$ be small enough that the statement made in Theorem \ref{T:tauberian} holds. Now, by Theorem \ref{T:laplace_sharp}, there exists $u_0 >0$ (that we referred to as $T_0$ in Theorem \ref{T:laplace_sharp}), $R_0 >0$ and $x_* >x_0$ large enough such that for all $u_0<u<2u_0$, $R>R_0$, $x>x_*$,
\[
\EXPECT{e^{-x_0} R,R}{ 1 - \exp \left\{ \left( 1 - \frac{u}{x} \right) \frac{2N_{e^{-x}R}}{(x+\log x)(e^{-x}R)^2} \right\} } = - e^{\pm \delta/3} \frac{2(e^{2x_0} - 1 )}{R^2 x} \frac{x}{u},
\]
where, as before, we write $e^{\pm \delta/3}$ for a quantity bounded between $e^{-\delta/3}$ and $e^{\delta/3}$ whose precise value may vary from line to line.
Also, by Lemma \ref{L:hit_killing}, if $x$ and $R$ are large enough,
\begin{equation}
\label{E:proof_Tail1}
\PROB{e^{-x_0} R,R}{ N_{e^{-x}R} >0 } = e^{\pm \delta/3} \frac{2(e^{2x_0} - 1 )}{R^2 x}.
\end{equation}
These two estimates together imply that
\[
\EXPECT{e^{-x_0} R,R}{ \left. \exp \left\{ \left( 1 - \frac{u}{x} \right) \frac{2N_{e^{-x}R}}{(x+\log x)(e^{-x}R)^2} \right\} \right\vert N_{e^{-x}R} >0} = 1 + e^{\pm 2\delta/3} \frac{x}{u}.
\]
Fix $u \in [u_0,2u_0]$. If we take $x$ to be large enough that the first term on the right hand side is much smaller than the second one, we obtain that
\[
\EXPECT{e^{-x_0} R,R}{ \left. \exp \left\{ \left( 1 - \frac{u}{x} \right) \frac{2N_{e^{-x}R}}{(x+\log x)(e^{-x}R)^2} \right\} \right\vert N_{e^{-x}R} >0 } = e^{\pm \delta} \frac{x}{u}
\]
whenever $x$ and $R$ are sufficiently large.
Applying Theorem \ref{T:tauberian} to $T = (1-\eps)^{-1} x/u_0$ and the random variable $X = \frac{2N_{e^{-x}R}}{(x+\log x)(e^{-x}R)^2}$ conditioned to be positive, we obtain that
\[
\PROB{e^{-x_0} R,R}{ \left. N_{e^{-x}R} \geq \frac12 \frac{x}{u_0} (x+\log x) (e^{-x} R)^2 \right\vert N_{e^{-x}R} >0}
\geq \frac{\delta \eps x}{(1-\eps)u_0} \exp\left[-\frac{1+\eps}{1-\eps} \cdot \frac{x}{u_0}\right]
\]
Using \eqref{E:proof_Tail1} to rewrite this estimate slightly, we have shown for each $\eps >0$ that there exists $\delta'>0$ such that
\begin{equation}
\label{E:proof_Tail2}
\PROB{e^{-x_0} R,R}{ N_{e^{-x}R} \geq \frac{x(e^{-x} R)^2}{2} \frac{x}{u_0}  } \geq \frac{\delta'}{R^2 x} \exp\left[-\frac{1+\eps}{1-\eps} \cdot \frac{x}{u_0}\right]
\end{equation}
for all $u_0 >0$ large enough (depending on $\eps$) and for all $x >x_0$ large enough (depending on $\eps$ and $u_0$).
This shows that the lower bound of the desired form holds \emph{for all sufficiently small values of $a$.} 

\medskip
We now bootstrap this result to obtain any thickness level $a>0$ instead of small levels $1/u_0$ as follows.
Let $x_0' = x_0 + 1$ be an intermediate level.
Starting from a single particle on the sphere of radius $e^{-x_0} R$, we can first condition on the event that the branching process emanating from this particle contains at least $\ceil{au_0} (e^{-x_0'}R)^2$ pioneers on the sphere of radius $e^{-x_0'}R$. This occurs with some probability larger than $c(a,u_0,x_0,x_0')R^{-2}$. Conditionally on this event, to create at least $a x^2 (e^{-x} R)^2 / 2$ pioneers on the sphere $\partial B(e^{-x}R)$, we can ask each of the $\ceil{au_0}$ sets of $(e^{-x_0'}R)^2$ pioneers on $\partial B(e^{-x_0'}R)$ to create at least $\frac{1}{u_0} x^2 (e^{-x} R)^2 / 2$ pioneers. By \eqref{E:proof_Tail2} applied to $x_0'$ instead of $x_0$, this probability is at least
\[
\left( \frac{C}{x} e^{-(1+\eps) x/u_0} \right)^{\ceil{au_0}} \geq e^{-(1+2\eps) \frac{\ceil{au_0}}{u_0}x}.
\]
Because $\eps$ and $u_0$ can be chosen arbitrarily small and large respectively, this concludes the proof.
\end{proof}

\section{Tail estimates without killing (proof of Theorem \ref{T:tail_general})}

\label{S:no_killing}

In this we complete the proof of Theorem \ref{T:tail_general}. Since we have already proved above Theorem \ref{T:maintail4},  which implies directly the part of Theorem \ref{T:tail_general} with killing, it suffices to prove
\eqref{E:T_no_killing}, which is the part of  that concerns the process in infinite volume, with no killing. In fact we will prove the following more general theorem which allows for a killing boundary at an arbitrary distance.

\begin{theorem}
\label{T:tail_generaler}
Let $x_0 >0$ and consider a thickness level $a>0$. For each $s \in [1,\infty]$, we have that
% The  asymptotic estimates 
\begin{align}
\label{E:T_some_killing}
 \PROB{e^{-x_0}R,R^s}{ \left. N_1 \ge \frac{a}{2} (\log R)^2 \right\vert N_1 > 0 } & = R^{-\psi_s(a) + o(1)}
\end{align}
as $R \to \infty$, where we set $\psi_s(a):=\inf_{1\leq t \leq s}(2(t-1)+at^{-1})$ (note that $s=\infty$ is allowed, in which case $\psi_s(a) $ coincides with the function $\psi$ in Theorem \ref{T:tail_general}).
\end{theorem}

%\red{[TH: Probably the proof of the upper bound can be made to give an estimate on the Laplace transform, which is stronger. This might be useful for a sharp analysis of the local time.]}

To this end, we first state and prove the following intermediate result:

\begin{lemma}\label{L:june}
For each fixed $t>1$ and $a>0$, the estimate
\begin{equation}
    \label{E:L_june}
\PROB{R,R^t}{\mathrm{hit}~ \partial B(R^t) \text{ and } N_1 \geq \frac{a}{2} (\log R)^2} = R^{-2t - a/t + o(1)}
\end{equation}
holds as $R\to \infty$.
\end{lemma}

\begin{proof}[Proof of Lemma~\ref{L:june}]
The lower bound follows quickly from the estimate \eqref{E:T_killing} of Theorem \ref{T:tail_general}, concerning the process with killing on scale $R$. Indeed, starting with one initial particle on $\partial B(R)$, the probability that it reaches $\partial B(R^t/2)$ with $2R^{2t}$ pioneers is equal to $R^{-2t + o(1)}$. We then divide these pioneers into two conditionally independent groups of $R^{2t}$ particles. The probability that the first group reaches $\partial B(R^t)$ is of constant order. Meanwhile, by \eqref{E:T_killing}, the probability that the second group generates $\frac{a}{2} (\log R)^2 = \frac{a}{2t^2} (\log R^t)^2$ pioneers on $\partial B(1)$ before reaching $\partial B(R^t)$ is equal to $(R^t)^{-a/t^2 + o(1)} = R^{-a/t+o(1)}$. Altogether, this shows that the left hand side of \eqref{E:L_june} is at least $R^{-2t - a/t +o(1)}$.

\medskip

The upper bound requires more work. For this, we will need to rely again on the PDE satisfied by the Laplace transform.
Fix $u\in (0,1/t)$ and let $s = -2 u / \log R$.
% for some $u \in (0,1/t)$. 
By Theorem \ref{T:proba_representation} and Lemma \ref{L:Phase3}, there exists a positive constant $C(u)$ such that
\begin{equation}
\label{E:november4}
1 - \EXPECT{R,R^t}{\left( 1 - s \right)^{N_1}} \geq - \frac{C(u)}{R^2 \log R}.
\end{equation}
Let us define for $x \in [0,t \log R]$,
\[
\gs(x) := (R^t e^{-x})^2 \EXPECT{R^t e^{-x},R^t}{1 - (1-s)^{N_1}},
\quad
\tilde{\gs}(x) := (R^t e^{-x})^2 \EXPECT{R^t e^{-x},R^t}{1 - (1-s)^{N_1} \indic{N_{R^t}=0}}.
\]
We are interested in the function $\mathsf{f} := \tilde{\gs} - \gs \geq 0$, which is given by
\[\mathsf{f}(x) = (R^t e^{-x})^2 \EXPECT{R^t e^{-x},R^t}{(1-s)^{N_1} \indic{N_{R^t}>0}}.\]
The two functions
$\gs$ and $\tilde{\gs}$ satisfy
\[
\left\{ \begin{array}{ll}
\gs'' + 2\gs' = \gs^2,\\
\gs(0) = 0,
~\gs(t \log R)= s,
\end{array} \right.
\quad \text{and} \quad
\left\{ \begin{array}{ll}
\tilde{\gs}'' + 2\tilde{\gs}' = \tilde{\gs}^2,\\
\tilde{\gs}(0) = 1,
~\tilde{\gs}(t \log R)= s,
\end{array} \right.
\]
while $\mathsf{f}$ satisfies the ODE $\mathsf{f}'' + 2\mathsf{f}' = (\gs+\tilde{\gs})\mathsf{f}$. 

\medskip

Let $\eta >0$.
Since $\sup_{[0,\infty)} |g_\lambda| \leq \lambda/2$ (Lemma \ref{L:g_small_first}) and the slope at the origin of $\gs$ goes to zero as $R \to \infty$ (Lemma \ref{L:Phase3}), we have that
$\gs \geq - \eta/2$ for all sufficiently large $R$. We now fix such a value of $R$ and claim that
\begin{equation}
\label{E:june2}
\mathsf{f}(x) \leq e^{-(2-\eps)x}, \quad \quad x \in [0,t\log R],
\quad \quad \text{where} \quad \eps = 1 -\sqrt{1-\eta}.
\end{equation}
Once the claim \eqref{E:june2} is proven, we may apply
this estimate with $x = (t-1) \log R$ to deduce that
% $u \in (0,1/t)$,
\[
\EXPECT{R,R^t}{(1+2u/\log R)^{N_1} \indic{N_{R^t}>0}} \leq R^{-2t+o(1)}
\]
as $R\to \infty$ (recalling that $u\in (0,1/t)$ was fixed), from which the claimed upper bound of \eqref{E:L_june} may be deduced via Markov's inequality and optimising over $u\in (0,1/t)$ by taking $u$ as close to $1/t$ as desired.
% and 
% Markov's inequality will then conclude the proof of the upper bound: the left hand side of  is at most $R^{-2t-a/t+o(1)}$.

\medskip

We now prove the claimed inequality \eqref{E:june2}.
Since $\tilde{\gs} \geq \gs$, we have $\gs + \tilde{\gs} \geq -\eta$ and
$
\mathsf{f}'' + 2\mathsf{f}' \geq -\eta \mathsf{f}.
$
As in the proof of Lemma \ref{L:ODE_negative}, if we let $c_{\pm} = 1 \pm \sqrt{1-\eta}$ we have that $c_+ + c_- = 2$ and $c_+ c_- = \eta$ and hence that
\[
(\mathsf{f}' + c_+ \mathsf{f})' + c_- (\mathsf{f}' + c_+ \mathsf{f}) = \mathsf{f}'' + 2\mathsf{f}' + \eta \mathsf{f} \geq 0.
\]
Integrating between $x$ and $t \log R$ for some $x \in [0,t \log R]$ (using Lemma \ref{L:Gronwall}) leads to
\[
(\mathsf{f}'(x) + c_+ \mathsf{f}(x)) e^{c_- x} \leq (\mathsf{f}'(t \log R) + c_+ \mathsf{f}(t \log R)) e^{c_- t \log R}.
\]
Because $\mathsf{f}(t \log R) = 0$ and $\mathsf{f}'(t \log R) \leq 0$, this implies that $\mathsf{f}' + c_+ \mathsf{f} \leq 0$ in $[0,t \log R]$.
Integrating further between $0$ and some $x \in [0,\log R]$ (using Lemma \ref{L:Gronwall} a second time) leads to \eqref{E:june2}, which concludes the proof.
\end{proof}

We can now prove Theorem \ref{T:tail_generaler}.

\begin{proof}[Proof of Theorem \ref{T:tail_generaler}]
Fix $a>0$ and $s\in [1,\infty]$. The lower bound is a direct consequence of \eqref{E:T_killing} when $a \leq 2$. If $a>2$, letting $t_* =\min\{s,\sqrt{a/2}\}$, we have by Lemma \ref{L:june} that
\[
\PROB{R,R^s}{N_1 \geq \frac{a}{2} (\log R)^2} \geq \PROB{R,R^{t_*}}{N_1 \geq \frac{a}{2} (\log R)^2, \mathrm{hit~} \partial B(R^{t_*})} = R^{-2t_* - a/t_*+o(1)}.
\]
The exponent $t_*$ was chosen to minimise the above exponent, which is equal to $\psi_s(a)$.
This concludes the proof of the lower bound. 

\medskip

For the upper bound, let $\eps >0$ be the inverse of a large integer, let $t_k = 1 + \eps k$ for each $k \geq 0$, and let $k_{\max}$ be the largest integer such that 
$t_k\leq t_*=\min\{s,\sqrt{a/2}\}$.
% \leq \psi_{\red{s}}(a) + 2$ 
% \red{so that
% \[t_{k_{\max}} \leq \frac{1}{2}\psi_{s}(a) + 1 +\eps \]
% }
Let $K := \max \{ k \geq 0: N_{R^{t_k}} >0 \}$, where $N_{R^{t_k}}$ is the number of pioneers on the large sphere $\partial B(R^{t_k})$ when particles are frozen both on this sphere and the unit sphere.
We have by a union bound that
\begin{align*}
\PROB{R,R^s}{N_1 \geq \frac{a}{2} (\log R)^2}
\le \sum_{k=0}^{k_{\max}} \PROB{R,R^s}{N_1 \geq \frac{a}{2} (\log R)^2, K = k}
+ \PROB{R,R^s}{K \geq k_{\max} +1 }.
\end{align*}
If $t_*=s$ then the last probability in the above display is equal to zero. Otherwise, this probability is
\[\PROB{R,R^s}{N_{R^{t_{k_{\max}+1}}} >0} = R^{-2t_{k_{\max}+1}+o(1)} \leq R^{-2\sqrt{a/2}+o(1)} = R^{-2-\psi_s(a)+o(1)}\] by definition of $k_{\max}$. In either case, we have that
\begin{equation}
    \label{eq:big_K}
\PROB{R,R^s}{N_{R^{t_{k_{\max}+1}}} >0} \leq R^{-2-\psi_s(a)+o(1)}.
\end{equation}
We now take $k \in \{0,\dots, k_{\max} \}$ and work on the event $\{K=k\}$.
We can decompose $N_1$ as $N_1 = N_1^{(1)} + N_1^{(2)}$ where $N_1^{(1)}$ is the number of pioneers on $\partial B(1)$ that \emph{do not} have an ancestor in $\partial B(R^{t_k})$, and $N_1^{(2)}=N_1-N_1^{(1)}$ is the number of pioneers on $\partial B(1)$ that \emph{do} have an ancestor in $\partial B(R^{t_k})$.
% before (resp. after) the particles have reached $\partial B(R^{t_k})$.
We have by a union bound that
\begin{align}
\nonumber & \PROB{R,\infty}{N_1 \geq \frac{a}{2} (\log R)^2, K = k}
\leq \PROB{R, \infty}{N_{R^{t_k}} \geq R^{2t_k} (\log R)^{3/2}} \\
& ~~~~~~~~~~~~~~ + \PROB{R,R^{t_{k+1}}}{N_1^{(1)} + N_1^{(2)} \geq \frac{a}{2} (\log R)^2, 0 < N_{R^{t_k}} < R^{2t_k} (\log R)^{3/2} }.
\label{E:june3}
\end{align}
By Lemma \ref{L:invariance}, the first probability on the right hand side decays faster than any polynomial in $R$.
Meanwhile, the second probability is at most
\begin{align*}
\PROB{R,R^{t_{k+1}}}{N_1^{(1)} \geq \frac{a}{2} (\log R)^2, 0 < N_{R^{t_k}}} 
+ \sum_{n=1}^{1/\eps} \P_{R,R^{t_{k+1}}} \Big( N_1^{(1)} \in [1-n\eps, 1-(n-1)\eps] \frac{a}{2} (\log R)^2, & \\
N_1^{(2)} \geq \frac{(n-1)\eps a}{2} (\log R)^2, 0 < N_{R^{t_k}} < R^{2t_k} (\log R)^{3/2} \Big). &
\end{align*}
Conditionally on $N_{R^{t_k}}$, the random variables $N_1^{(1)}$ and $N_1^{(2)}$ are independent. 
Moreover, by \eqref{E:november4} and the conditional independence of particles after hitting $\partial B( R^{t_k})$, we have for all $v \in (0,1)$,
\[
\EXPECT{R,R^{t_{k+1}}}{\left. \left( 1 + \frac{2v}{\log R^{t_{k+1}}} \right)^{N_1^{(2)}} \right\vert N_{R^{t_k}} < R^{2t_k} (\log R)^{3/2} } \le  \exp( C(v) (\log R)^{1/2} ) = R^{o(1)}.
\]
By Markov's inequality, this implies that, conditionally on $N_{R^{t_k}} < R^{2t_k} (\log R)^{3/2}$, the probability (for the branching Brownian motion with killing at $R^{t_{k+1}}$) that $N_1^{(2)}$ is at least $\frac{(n-1)\eps a}{2} (\log R)^2$ is at most $R^{-(n-1) \eps a / t_{k+1} +o(1)}$.
%Theorem \ref{T:laplace_sharp} (and arguing as in Proposition \ref{P:upper_bound_rough}) \red{add a reference for Laplace transform starting from nonmacro distance...?}. 
Moreover, by Lemma \ref{L:june},
\[
\P_{R,R^{t_{k+1}}} \Big( N_1^{(1)} \geq (1-n\eps) \frac{a}{2} (\log R)^2, 0 < N_{R^{t_k}} \Big)
\leq R^{-2t_k - (1-n\eps) a/t_k +o(1)}.
\]
This shows that the second probability in the right hand side of \eqref{E:june3} is at most
\begin{align*}
R^{-2t_k -a/t_k + o(1)}
+ \sum_{n=1}^{1/\eps} R^{-2t_k - (1-n\eps) a/t_k -(n-1) \eps a / t_{k+1} +o(1)}
\leq R^{-2t_k -a/t_k + 2a\eps+o(1)},
\end{align*}
where in the last inequality we made elementary simplifications and, in particular, used the bound $$
\frac{1}{t_{k+1}} = \frac{1}{t_k + \eps} \geq \frac{1}{t_k} \Big(1-\frac{\eps}{t_k}\Big) \geq \frac{1}{t_k} - \frac{\eps}{t_k^2} \ge \frac{1}{t_k} - \eps$$
(as $t_k \ge 1$).
Wrapping up, we have shown that
\[
\PROB{R,\infty}{N_1 \geq \frac{a}{2} (\log R)^2, K = k} \leq R^{-2t_k -a/t_k + 2a\eps+o(1)}
\]
for each $0\leq k \leq k_\mathrm{max}$
and hence that
\begin{multline}
\label{eq:small_K}
\PROB{R,\infty}{N_1 \geq \frac{a}{2} (\log R)^2, K \leq k_\mathrm{max}} \leq k_\mathrm{max}\cdot R^{-\min_{0\leq k \leq k_\mathrm{max}}(2t_k +at_k^{-1}) + 2a\eps+o(1)} \leq R^{-\psi_s(a)+2a\eps+o(1)}.
\end{multline}
The claimed upper bound follows from \eqref{eq:big_K} and \eqref{eq:small_K} since $\eps>0$ was arbitrary. 
% We have obtained that
% \red{Since $\eps>0$ was arbitrary, it follows that}
% \[
% \PROB{R,\infty}{N_1 \geq \frac{a}{2} (\log R)^2} \leq R^{-2-\psi_{\red{s}}(a)+o(1)} + R^{2a\eps+o(1)} R^{- \inf_{t \geq 1} (2t +a/t)}.
% \]
% $2+\psi(a)$ coincides with $\inf_{t \geq 1} (2t +a/t)$. Since $\eps$ can be arbitrarily small, this concludes the proof of the upper bound.
\end{proof}

\section{Thick points (proof of Theorem \ref{T:thick})}
\label{S:thick}
In this section we apply the main results of Sections \ref{sec:upper_bound} and \ref{sec:tail4} to prove Theorem~\ref{T:thick}. 
The upper bound in this theorem has already been proved at the end of Section \ref{SS:upper_bound},
so that it remains only to prove the (more difficult) lower bound. 
% The overall strategy is explained in Section \ref{S:strategy}

\subsection{Strategy of the proof}\label{S:strategy}

We now explain the overall strategy behind the proof of the lower bound of \eqref{E:T_thick1} in Theorem \ref{T:thick}. In particular, we will state three key propositions (Propositions \ref{P:1point_backtracking}, \ref{P:1point_after_backtracking} and \ref{P:2point}), and explain how these propositions imply Theorem \ref{T:thick}. The rest of Section \ref{S:thick} is then dedicated to the proofs of these three key propositions.

\medskip

As is common in the study of analogous questions for other models, we use a truncated second moment approach. The truncation we use will be described in terms of \emph{non-backtracking trajectories}, a notion we now introduce.
% We start by introducing a notion of \emph{non-backtracking} trajectories.
Let $\eta_1 >0$ be the inverse of a large integer, let $z \in \Z^4$, and define $r_k = R^{1-k \eta_1}$ for each $k \in \{1, \dots, 1/\eta_1\}$.
We say that a pioneer in $\partial B(z,r)$ is \textbf{non-backtracking} if, for each $1\leq k \leq 1/\eta_1$, the trajectory taken by the ancestors of the pioneer does not revisit the sphere $\partial B(z,r_k \log R)$ after first visiting the sphere $\partial B(z,r_k)$. (Note that this makes sense even if $r$ is larger than some of the $r_k$.)
% In addition to killing the particles on the sphere $\partial B(0,R)$, we will also kill each particle that reaches the sphere $\partial B(z,r_k \log R)$ if it possesses an ancestor that has reached $\partial B(z,r_k)$ and we do that for all $k \in \{1, \dots, 1/\eta_1\}$.
We denote the number of non-backtracking pioneers on the sphere $\partial B(z,r)$ by $\hat{N}_{z,r}$, noting that $\hat{N}_{z,r}$ depends implicitly on the choice of $\eta_1$ and $R$.

\medskip

In the first key proposition, we show that restricting the branching process to non-backtracking trajectories does not reduce the exponent for a single point to be thick
% thick number of pioneers on the sphere $\partial B(z,1)$. 
(as can be seen by comparing this result with Theorem \ref{T:maintail4}). The constants $1/2$ and $1/10$ appearing in this proposition are somewhat arbitrary.

\begin{proposition}\label{P:1point_backtracking}
Let $z,z_0 \in B(0,R/2)$ with $|z-z_0| \geq R/10$. We have
\[
\PROB{z_0,R}{ \hat{N}_{z,1} \geq \frac{a}{2} (\log R)^2 } \geq R^{-2-a+o(1)}
\quad \quad \text{as} \quad R \to \infty.
\]
\end{proposition}

The next key proposition will state that the exponent for a point to be thick remains the same if, in addition to restricting to non-backtracking trajectories, we also impose a ``good event'' stating that the number of non-backtracking pioneers on each mesoscopic scale is of order consistent with the desired level of thickness at the unit scale.
% further restrict  the number of pioneers $\hat{N}_{z,r_k}$ by the non-backtracking process on each sphere $\partial B(z,r_k)$, $k=1 \dots 1/\eta_1$. 
Let $b>a$ be a parameter close to $a$ and define the good event $G_R(z)$ by
\begin{equation}
G_R(z) := \left\{ \forall k=1, \dots, 1/\eta_1, \hat{N}_{z,r} \leq \frac{b}{2} r_k^2 \left( \log \frac{R}{r_k} \right)^2 \right\}.
\end{equation}
The definition of this good event is inspired by the Gaussian and Brownian multiplicative chaos theories; see e.g. \cite{BerestyckiGMC, jegoBMC}.
In the next result, we show that adding this second layer of good event (the first layer being the restriction to non-backtracking particles) keeps the first moment almost unchanged (rather than merely keeping the exponent the same).

\begin{proposition}\label{P:1point_after_backtracking}
Let $z_0, z \in B(0,R)$ with $|z-z_0| \geq R/10$. There exists $\eps_1 = \eps_1(a,b,\eta_1) >0$ such that for all $R$ large enough,
\begin{equation}
\label{E:1st_moment_good}
\PROB{z_0,R}{ \hat{N}_{z,1} \geq \frac{a}{2} (\log R)^2, G_R(z)^c }
\leq R^{-2-a-\eps_1}.
\end{equation}
\end{proposition}

Finally, we bound the relevant truncated two-point function, where the truncation both restricts to non-backtracking trajectories and imposes that the relevant good events hold.

\begin{proposition}\label{P:2point}
Let $\eta_3>0$ be a small parameter, $z_0 \in B(R/2) \setminus B(R/4)$ and $z, w \in B(0,R^{1-\eta_3})$. There exists $\eps_2 = \eps_2(a,b,\eta_1,\eta_3)>0$ such that
\begin{equation}
\label{E:2point_good}
\PROB{z_0,R}{ \hat{N}_{z,1}, \hat{N}_{w,1} \geq \frac{a}{2} (\log R)^2, G_R(z), G_R(w) } \leq R^{-2-2a+o(1)} \left( \frac{R}{|z-w|} \right)^{a+\eps_2}.
\end{equation}
Moreover, for $a$ and $\eta_3$ fixed, $\eps_2(a,b,\eta_1,\eta_3) \to 0$ as $b \to a^+$ and $\eta_1 \to 0$.
\end{proposition}

Assuming the three propositions above, we are now ready to prove the lower bound in Theorem \ref{T:thick}. In a nutshell, an application of Paley--Zygmund inequality will show that the probability of having a large number of thick points decays slower than any power of $R$. We will then take advantage of the numerous independent branching Brownian motions present in a branching Brownian motion conditioned to hit $\partial B(R/2)$.

\begin{proof}[Proof of Theorem \ref{T:thick} -- Lower bound]
Without loss of generality, we change the conditioning: rather than conditioning on surviving at least $R^2$ generations, we will condition on the event that at least one particle reaches the outer sphere $\partial B(R)$. We will denote by $\P^*_{z_0,R}$ this conditional law. (Any conditioning that guarantees that the BGW tree survives at least $R^{2+o(1)}$ generations with high probability would do.)
% First of all, notice that if we knew that, for each $a < 4$, the set of $a$-thick points $\Tc_R(a)$ is not empty with high $\P^*_{z_0,R}$-probability as $R \to \infty$, this would imply that $\liminf_{R \to \infty} \sup_{z \in B(R) \cap \Z^4} N_{z,1} / (\log R)^2 \geq 2$ in probability.
Note that
it suffices to prove \eqref{E:T_thick1}, since this trivially implies the lower bound on the maximal thickness.
Let $\eta_3 >0$ be as in Proposition \ref{P:2point} (small, but much larger than $\eta_1$) and define the following subset of $\Tc_R(a)$:
\[
\widehat{\Tc}_R(a) := \left\{ z \in B(0,R^{1-\eta_3}) \cap \Z^4: \hat{N}_{z,1} \geq \frac{a}{2} (\log R)^2 \right\}.
\]
Let $\delta >0$. By Propositions \ref{P:1point_backtracking} and \ref{P:1point_after_backtracking},
\begin{align}
\label{E:poitiers1}
\E_{z_0,R}\left[ \sum_{z \in \Z^4} \indic{z \in \widehat{\Tc}_R(a)} \mathbf{1}_{G_R(z)} \right] \geq R^{-2+4(1-\eta_3)-a+o(1)}.
\end{align}
Under $\P^*_{z_0, R}$ there is at least one particle hitting $\partial B(R)$. Pick such a particle uniformly at random, and let $E$ be the event that the ancestors of this particle gave birth to  at least $R^2$ particles, each starting in $B(R/2) \setminus B(R/4)$.
It is easy to see that the conditional probability of $E$ is uniformly bounded away from 0: $\inf_R \P_{z_0,R}^*(E) >0$, and that conditionally on $E$ under $\P^*_{z_0, R}$, and conditionally on the location of these branching events, the descendants form $R^2$ independent branching Brownian motions starting from the given locations in $B(R/2) \setminus B(R/4)$. We deduce from \eqref{E:poitiers1} that
\begin{align*}
\E_{z_0,R}^*\left[ \sum_{z \in \Z^4} \indic{z \in \widehat{\Tc}_R(a)} \mathbf{1}_{G_R(z)} \right]
& \geq \P_{z_0,R}^*(E) R^2 \inf_{z_0' \in B(R/2) \setminus B(R/4)} \E_{z_0',R}\left[ \sum_{z \in \Z^4} \indic{z \in \widehat{\Tc}_R(a)} \mathbf{1}_{G_R(z)} \right] \\
& \geq R^{4(1-\eta_3)-a+o(1)}.
\end{align*}
Meanwhile, by Proposition \ref{P:2point},
\begin{align*}
\E_{z_0,R}\left[ \left( \sum_{z \in \Z^4} \indic{z \in \widehat{\Tc}_R(a)} \mathbf{1}_{G_R(z)} \right)^2 \right]
\leq R^{-2-2a+o(1)} \sum_{z,w \in \Z^4 \cap B(0,R^{1-\eta_3})} \left( \frac{R}{|z-w|} \right)^{a+\eps_2}.
\end{align*}
If $b$ is close enough to $a$ and $\eta_1$ is sufficiently small, then $a+\eps_2<4$, so that the above sum can be bounded by $O(R^{8(1-\eta_3)+\eta_3(a+\eps_2)})$ and the above second moment is bounded above by $R^{-2+8(1-\eta_3) - 2a + \eta_3(a+\eps_2) +o(1)}.$ To deduce a similar bound regarding the second moment under the conditioned probability $\P_{z_0,R}^*$, we simply use the fact that the probability of hitting $\partial B(R)$ is of order $R^2$ so that
\begin{align*}
\E_{z_0,R}^*\Bigg[ \Big( \sum_{z \in \Z^4} \indic{z \in \widehat{\Tc}_R(a)} \mathbf{1}_{G_R(z)} \Big)^2 \Bigg]
& \leq \frac{1}{\PROB{z_0,R}{\text{hit~} \partial B(R)}} \E_{z_0,R}\Bigg[ \Big( \sum_{z \in \Z^4} \indic{z \in \widehat{\Tc}_R(a)} \mathbf{1}_{G_R(z)} \Big)^2 \Bigg] \\
& \leq R^{8(1-\eta_3) - 2a + \eta_3(a+\eps_2) +o(1)}.
\end{align*}
Using the Paley--Zygmund inequality, we deduce that
\begin{align*}
& \P^*_{z_0,R}\left( \# \Tc_R(a) \geq R^{4(1-\eta_3) - a - \delta} \right) \\
&\hspace{2cm} \geq \P^*_{z_0,R}\left( \sum_{z \in \Z^4} \indic{z \in \widehat{\Tc}_R(a)} \mathbf{1}_{G_R(z)} \geq \frac12 \E^*_{z_0,R}\left[ \sum_{z \in \Z^4} \indic{z \in \widehat{\Tc}_R(a)} \mathbf{1}_{G_R(z)} \right] \right) \\
&\hspace{2cm} \geq \frac14 \E^*_{z_0,R}\left[ \sum_{z \in \Z^4} \indic{z \in \widehat{\Tc}_R(a)} \mathbf{1}_{G_R(z)} \right]^2 \left/ \E^*_{z_0,R}\left[ \left( \sum_{z \in \Z^4} \indic{z \in \widehat{\Tc}_R(a)} \mathbf{1}_{G_R(z)} \right)^2 \right] \right.  \geq R^{-\eta_3(a+\eps_2)-o(1)}.
\end{align*}
Rephrasing slightly, we have proved that for any $\delta >0$, 
\begin{equation}
\label{E:poitiers2}
\P^*_{z_0,R}\left( \# \Tc_R(a) \geq R^{4 - a - \delta} \right) \geq R^{-o(1)}.
\end{equation}
Using again the fact that $\PROB{z_0,R}{\text{hit~} \partial B(R)} = R^{-2 +o(1)}$, this implies the following estimate for the unconditioned probability: for any $\delta>0$,
\begin{equation}
\label{E:poitiers3}
\P_{z_0,R}\left( \# \Tc_R(a) \geq R^{4 - a - \delta} \right) \geq R^{-2+o(1)}.
\end{equation}

To conclude the proof, it remains to replace the $R^{o(1)}$ on the right hand side of \eqref{E:poitiers2} by some quantity that goes to 1 as $R \to \infty$.
Let $\delta>0$ be small and let $\eps>0$ be small enough so that $(4-a-\delta)/(1-\eps) < 4-a$.
Let $E$ be the event that for all $k= \floor{R^\eps/4}, \floor{R^\eps/4}+1,\dots, \floor{R^\eps/2} - 1$, the number of pioneers on the sphere $\partial B(R/2 + k R^{1-\eps})$ is at least $R^{2-\eps/2}$. For each $k$ and each such pioneer, consider the number of thick points generated by that pioneer where particles are killed on the next sphere $\partial B(R/2 + (k+1) R^{1-\eps})$. These numbers of thick points stochastically dominate $\# \Tc_{2 R^{1-\eps}}(a)$ under $\P_{w_0,2 R^{1-\eps}}$, where $w_0 \in \partial B(R^{1-\eps})$.
If $\# \Tc_R(a) \leq R^{4 - a - \delta}$, then all these subsets of $\Tc_R(a)$ contain at most $R^{4 - a - \delta} = (R^{1-\eps})^{(4-a-\delta)/(1-\eps)}$ elements (recall that $(4-a-\delta)/(1-\eps) < 4-a$). On the event $E$, there are at least $c R^{\eps} R^{2-\eps/2} = c R^{2+\eps/2}$ of these subsets. By a repeated application of the Markov property and by \eqref{E:poitiers3} applied to $2 R^{1-\eps}$ and to a starting point on $\partial B(R^{1-\eps})$ instead of $R$ and $z_0$, we deduce that
\begin{align*}
\P^*_{z_0,R}\left(E, \# \Tc_R(a) \leq R^{4 - a - \delta} \right)
& \leq R^{2+o(1)} \P_{z_0,R}\left(E, \# \Tc_R(a) \leq R^{4 - a - \delta} \right) \\
& \leq R^{2+o(1)} (1-R^{-2+o(1)})^{cR^{2+\eps/2}} = R^{2+o(1)} e^{-R^{\eps/2+o(1)}} \to 0
\end{align*}
as $R \to \infty$.
Together with the fact that $\P^*_{z_0,R}(E) \to 1$ as $R \to \infty$, this concludes the proof of the fact that $\P^*_{z_0,R}(\# \Tc_R(a) \leq R^{4 - a - \delta} ) \to 0$ as $R \to \infty$.
\end{proof}

\subsection{One-point estimate in a non-backtracking scenario}

The goal of this section is to prove Proposition \ref{P:1point_backtracking}. We start with a preliminary lemma, which is essentially just a restatement of some of our earlier results.
Recall that $\eta_1$ stands for the inverse of a large integer and that the decreasing sequence of radii $(r_k)_{k}$ is defined by $r_k = R^{1-k \eta_1}$, for $k=1, \dots, 1/\eta_1$.

\begin{lemma}
Let $\eta_1>0$ be the inverse of an integer. Let $[\nu_-, \nu_+] \subset (0,1)$ be a compact interval of $(0,1)$.
For all $\nu \in [\nu_-,\nu_+]$ and $k \in \{1, \dots, 1/\eta_1\}$,
\begin{equation}
\label{E:rennes1}
\EXPECT{r_k,r_k \log R}{ \exp \left( \frac{2 \nu}{r_{k+1}^2 \log \frac{r_k}{r_{k+1}}} N_{r_{k+1}} \right) } 
= 1 + (1+o(1)) \frac{2}{r_k^2 \log(r_k/r_{k+1})} \frac{\nu}{1-\nu}
\end{equation}
where $o(1) \to 0$ as $R \to \infty$, uniformly in $\nu \in [\nu_-,\nu_+]$ and $k \in \{1, \dots, 1/\eta_1\}$.
\end{lemma}

\begin{proof} We will use the function $g_\lambda$ as defined in \eqref{E:g_lambda} and studied extensively in Sections \ref{sec:uniqueness}, \ref{sec:upper_bound}, and~\ref{sec:series_expansion}.
Let $\mu = 4\nu / (1-\nu)$. Because $\nu$ belongs to a compact subspace of $(0,1)$, $\mu$ belongs to a compact subspace of $(0,\infty)$. 
Letting $x_0 = \log \log R$ and $x = \eta_1 \log R + \log \log R$, we have $r_k = (r_k \log R) e^{-x_0}$ and $r_{k+1} = (r_k \log R) e^{-x}$ and, by Theorem~\ref{T:proba_representation}, that
\[
1 - \EXPECT{r_k,r_k \log R}{ \left(1 - \frac{1}{r_{k+1}^2} g_{\lambda}(x) \right)^{N_{r_{k+1}}} } = \frac{1}{r_k^2} g_{\lambda}(x_0)
\]
for all $\lambda \in (0,1)$. Taking $\lambda = \mu / x$, we have by Lemma \ref{L:Phase3} that
\[
g_\lambda(x) = - (1+o(1)) \frac{2\mu}{\mu +4} \frac1x = -(1+o(1)) 2\nu \frac{1}{\eta_1 \log R}.
\]
Also, by \eqref{E:L_g_small_lambda2} and because $1 \ll x_0 \ll \frac{1}{\lambda}$, we have
\[
g_\lambda(x_0) = -(1+o(1)) \frac{\lambda}{2} = (1+o(1)) \frac{2\nu}{1-\nu} \frac{1}{\eta_1 \log R}.
\]
Putting things together, we obtain the desired claim.
\end{proof}

Let $\eta_2 >0$ be much smaller than $\eta_1$ and define for all $k =1 \dots 1/\eta_1$, $I_k = \frac{a}{2} (1-\eta_2)^k r_k^2 \log (\frac{R}{r_k})^2$. The main intermediate step in the proof of Proposition \ref{P:1point_backtracking} is summarised by the following lemma.

\begin{lemma}\label{L:1point1step}
Let $k \in \{1, \dots, 1/\eta_1 - 1\}$. For each $i =1, \ldots, I_k$, let $N^{(i)}$ be i.i.d. random variables distributed according to the number of pioneers on $\partial B(r_{k+1})$, starting from a single particle on $\partial B(r_k)$ and where particles are killed on $\partial B(r_k \log R)$. We have
\[
\Prob{ \sum_{i=1}^{I_k} N^{(i)} \geq I_{k+1} } \geq R^{-a \eta_1 (1+O(\eta_2/\eta_1))}
\]
where the implicit constants in $O(\eta_2/\eta_1)$ are uniform provided that $R$ is sufficiently large.
% for some constant $C>0$ independent of $R$ large enough.
\end{lemma}

\begin{proof}
The idea of the proof can be easily explained with the following toy model. Let $n$ be a large integer and $E_1, \dots, E_n$ be i.i.d. exponential variables with mean 1. We are interested in the probability that $E_1 + \dots + E_n$ exceeds $t n$, where $t$ is some fixed real number larger than $1$. To achieve this atypical behaviour, an optimal strategy consists in shifting the mean of each individual exponential variable to $t$, i.e. to shift the probability measure by
$
t^{-n} e^{(1-1/t)(E_1 + \dots + E_n)}.
$
After this shift, the sum $E_1 + \dots + E_n$ exceeds $tn$ with a fairly high probability (asymptotic to 1/2 by the central limit theorem, but more importantly larger than any exponential in $n$).
Since the numbers of pioneers behave in many ways like exponential random variables, we will be able to use such a strategy as well. Moreover, since we have a good control on the Laplace transform of the number of pioneers, we will be able to control the cost of such a shift.

We now define the change of measure that we will use in the proof. Let $k \in \{1,\dots, 1/\eta_1\}$, let $\nu = 1/(1+k)$, 
let
\[
t=\frac{2\nu}{r_{k+1}^2 \log \frac{r_k}{r_{k+1}}}
\]
and let $\Q$ be the probability law
\[
\frac{\d \Q}{\d \P} = \left. \exp \left( t \sum_{i=1}^{I_k} N^{(i)} \right) \right/ \Expect{ \exp \left( t \sum_{i=1}^{I_k} N^{(i)} \right) }.
\]
By \eqref{E:rennes1} and by the special choice of $\nu$, we have
\begin{align}
\label{E:rennes5}
& \Expect{ \exp \left( t \sum_{i=1}^{I_k} N^{(i)} \right) } 
= \left( 1 + \frac{2(1+o(1))}{r_k^2 \log(r_k/r_{k+1})} \frac{\nu}{1-\nu} \right)^{I_k} \\
&\hspace{2cm} = \exp \left((1+o(1)) (1-\eta_2)^k a \frac{\nu}{1-\nu} \log^2 \left( \frac{R}{r_k} \right) \Big/ \log \left( \frac{r_k}{r_{k-1}} \right) \right) = R^{(1+o(1)) (1-\eta_2)^k a \eta_1 k}.
\nonumber
\end{align}
Define $I_{k+1}^+ = \frac{1+\eta_2}{1-\eta_2} I_{k+1}$.
We have
\begin{align*}
 \Prob{ \sum_{i=1}^{I_k} N^{(i)} \geq I_{k+1} }
&\geq \Prob{ I_{k+1} \leq \sum_{i=1}^{I_k} N^{(i)} \leq I_{k+1}^+ } \\
& = R^{(1+o(1)) (1-\eta_2)^k a \eta_1 k} \E_{\Q} \left[ \indic{I_{k+1} \leq \sum_{i=1}^{I_k} N^{(i)} \leq I_{k+1}^+} \exp \left( -t \sum_{i=1}^{I_k} N^{(i)} \right) \right]
\end{align*}
When the sum of the $N^{(i)}$ is smaller than $I_{k+1}^+$, we can bound from below the exponential in the expectation above by $\exp \left( -t I_{k+1}^+ \right)$. This leads to
\begin{align*}
& \Prob{ \sum_{i=1}^{I_k} N^{(i)} \geq I_{k+1} } \geq
R^{ ((1+o(1))(1-\eta_2)^k k  -  (1-\eta_2)^{k-1}(1+\eta_2)(k+1)) \eta_1 a} \Q \left( I_{k+1} \leq \sum_{i=1}^{I_k} N^{(i)} \leq I_{k+1}^+ \right) \\
&\hspace{7cm} = R^{ - (1 + O(\eta_2/\eta_1)) (1+o(1)) \eta_1 a} \Q \left( I_{k+1} \leq \sum_{i=1}^{I_k} N^{(i)} \leq I_{k+1}^+ \right).
\end{align*}
To conclude the proof it is enough to show that the above probability appearing on the second line is equal to $(1+o(1))$. To do this, we will show separately that $\Q \left( \sum_{i=1}^{I_k} N^{(i)} \leq I_{k+1} \right)$ and $\Q \left( \sum_{i=1}^{I_k} N^{(i)} \geq I_{k+1}^+ \right)$ vanish as $R \to \infty$. We start by bounding the first probability and we let $\mu = \nu + \frac{k}{k+1} \frac{1}{\sqrt{1-\eta_2}} - 1$ (which is positive and smaller than $\nu$). By Markov's inequality, we have
\begin{align*}
& \Q \left( \sum_{i=1}^{I_k} N^{(i)} \leq I_{k+1} \right)
\leq \exp \left( \frac{\mu}{\nu} t I_{k+1} \right) \E_\Q \left[ \exp \left( - \frac{\mu}{\nu} t \sum_{i=1}^{I_k} N^{(i)} \right) \right] \\
&\hspace{4cm} = R^{-(1+o(1))(1-\eta_2)^k a \eta_1 k} R^{\mu (1-\eta_2)^{k+1} (k+1)^2 a \eta_1 }
\Expect{ \exp \left( \frac{\nu-\mu}{\nu}t \sum_{i=1}^{I_k} N^{(i)} \right) }
\end{align*}
where the first power of $R$ on the right hand side comes from the normalisation of the Radon--Nikodym derivative $\d \Q / \d \P$; see \eqref{E:rennes5}.
By \eqref{E:rennes1}, this is further equal to
\begin{align*}
& \exp \left[ (1+o(1)) \left(-(1-\eta_2)^k a \eta_1 k
+ \mu (1-\eta_2)^{k+1} (k+1)^2 a \eta_1
+ (1-\eta_2)^k a \eta_1 \frac{\nu-\mu}{1-(\nu-\mu)} k^2 \right) \log R \right] \\
& =\exp \left[ (1+o(1)) a\eta_1 (1-\eta_2)^k \left( -k + \mu (1-\eta_2) (k+1)^2 + \frac{\nu-\mu}{1-(\nu-\mu)}k^2 \right) \log R \right].
\end{align*}
Our specific choice for $\mu$ leads to
\begin{align*}
\Q \left( \sum_{i=1}^{I_k} N^{(i)} \leq I_{k+1} \right)
& \leq \exp \left[  -(1+o(1))a\eta_1 (1-\eta_2)^k ( 2\sqrt{1-\eta_2}-2+\eta_2 ) k(k+1)   \log R \right]\\
& \leq \exp\left[  - \frac15 a \eta_1 \eta_2^2 k(k+1) \log R \right] \to 0
\end{align*}
as $R \to \infty$. To bound the other probability, we proceed similarly. We define this time $\mu = \left( 1 - \frac{1}{\sqrt{1+\eta_2}} \right) \frac{k}{k+1}$ (which is positive and smaller than $1-\nu$). By Markov's inequality and \eqref{E:rennes1},
\begin{align*}
& \Q \left( \sum_{i=1}^{I_k} N^{(i)} \geq I_{k+1}^+ \right)
\leq \exp \left( - \frac{\mu}{\nu}t I_{k+1}^+ \right) \E_\Q \left[ \exp \left( \frac{\mu}{\nu}t \sum_{i=1}^{I_k} N^{(i)} \right) \right] \\
& = R^{ -(1-\eta_2)^k a \eta_1 k }
R^{ - \mu (1-\eta_2)^k (1+\eta_2) (k+1)^2 a \eta_1 }
\Expect{ \exp \left( \frac{\nu + \mu}{\nu}t  \sum_{i=1}^{I_k} N^{(i)} \right) } \\
& =\exp\left[ \left( -(1-\eta_2)^k a \eta_1 k - \mu (1-\eta_2)^k (1+\eta_2) (k+1)^2 a \eta_1 
+ (1-\eta_2)^k a \eta_1 \frac{\nu+\mu}{1-(\nu+\mu)} k^2 \right) \log R \right] \\
& = \exp\left[{\left( a\eta_1 (1-\eta_2)^k \left( -k - \mu (1+\eta_2) (k+1)^2 + \frac{\nu+\mu}{1-(\nu+\mu)}k^2 \right) \right)} \log R \right].
\end{align*}
Our choice for $\mu$ leads to
\begin{align*}
\Q \left( \sum_{i=1}^{I_k} N^{(i)} \geq I_{k+1}^+ \right)
& \leq \exp \left[ - a\eta_1 (1-\eta_2)^k \left( 2 + \eta_2 - 2 \sqrt{1+\eta_2} \right) k(k+1) \log R \right] \\
& \leq \exp \left[ - \frac15 a \eta_1 \eta_2^2 k(k+1) \log R \right] \to 0
\end{align*}
as $R \to \infty$. This concludes the proof of Lemma \ref{L:1point1step}.
\end{proof}

We can now prove Proposition \ref{P:1point_backtracking}.

\begin{proof}[Proof of Proposition \ref{P:1point_backtracking}]
With the same notations as in Lemma \ref{L:1point1step}, we have
\begin{align*}
\PROB{e^{-x_0}R,R}{ \hat{N}_1 \geq I_{1/\eta_1} }
\geq \PROB{e^{-x_0}R,R}{ N_{r_1} \geq I_1 } \prod_{k=1}^{1/\eta_1-1} \PROB{r_k,r_k \log R}{\sum_{i=1}^{I_k} N_{r_{k+1}}^{(i)} \geq I_{k+1} }.
\end{align*}
By Theorem \ref{T:maintail4}, the first probability on the right hand side is at least
$R^{-2-a\eta_1 (1-\eta_2) +o(1)}$. By Lemma \ref{L:1point1step}, each probability appearing in the product is at least $R^{-a\eta_1 (1+O(\eta_2/\eta_1)) }$. Using the definition of $I_{1/\eta_1}$, this gives
\[
\PROB{e^{-x_0}R,R}{ \hat{N}_1 \geq \frac{a}{2} (1-\eta_2)^{1/\eta_1} (\log R)^2} \geq R^{-2-a (1+O(\eta_2/\eta_1))}.
\]
Because $\eta_2$ can be as small as desired, this concludes the proof.
\end{proof}

\subsection{Truncated first moment}

We now prove Proposition~\ref{P:1point_after_backtracking}, which states that we can add the good event $G_R(z)$ to our non-backtracking thick point event without significantly changing its probability.

\begin{proof}[Proof of Proposition \ref{P:1point_after_backtracking}]
Since the left hand side of \eqref{E:1st_moment_good} increases as the value of $b$ approaches $a$, we can assume that $b$ is close enough to $a$ to ensure that $\sqrt{b}(1-\eta_1) < \sqrt{a}$.
Let $\eps>0$ be small enough so that
\[
\frac{a+\eps}{b} \frac{a+\eps}{a} < 1-\eps,
\quad \sqrt{\frac{b}{a}}(1-\eta_1) < 1-\eps
\quad \text{and} \quad \eta_1^2 (\sqrt{b} - \sqrt{a})^2 > \eps.
\]
Now let  $\delta>0$ be small enough such that for all $p \geq 0$,
\begin{equation}
\label{E:rennes2}
\frac{1+(p+1)\delta}{(1+p\delta)^2} \frac{a+\eps}{b} \frac{a+\eps}{a} < 1-\eps,
\quad \sqrt{\frac{b}{a} (1+\delta)}(1-\eta_1) < 1-\eps
\quad \text{and} \quad
\eta_1 (- \delta b + \eta_1 (\sqrt{b} - \sqrt{a})^2 ) > \eps.
\end{equation}
Let $J_k = \frac{b}{2} r_k^2 \left( \log \frac{R}{r_k} \right)^2$.
By a union bound, the left hand side of \eqref{E:1st_moment_good} is at most
\begin{align*}
\sum_{k=1}^{1/\eta_1 -1 } \PROB{e^{-x_0}R,R}{ \hat{N}_1 \geq \frac{a}{2} (\log R)^2, \hat{N}_{r_k} \geq J_k }.
\end{align*}
Fix $k \in \{1, \dots, 1/\eta_1 -1 \}$. Let $N^{(i)}_1, i \geq 1$, be i.i.d. random variables having the same law as the number of pioneers on the unit sphere with one initial particle on $\partial B(r_k)$ and where particles are killed on $\partial B(r_k \log R)$. Conditionally on $\hat{N}_{r_k} = n$, the random variable $\hat{N}_1$ is stochastically dominated by $\sum_{i=1}^n N_1^{(i)}$ (there are additional killings in the variable $\hat{N}_1$, so that their distributions are not identical). Therefore, we can bound
\begin{align}
\nonumber
& \PROB{e^{-x_0}R,R}{ \hat{N}_1 \geq \frac{a}{2} (\log R)^2, \hat{N}_{r_k} \geq J_k } \\
& \leq \sum_{p=0}^\infty \PROB{e^{-x_0}R,R}{ 1 + p \delta \leq \frac{\hat{N}_{r_k}}{J_k} < 1 + (p+1)\delta }
\P \Bigg( \sum_{i=1}^{(1+(p+1)\delta)J_k} N_1^{(i)} \geq \frac{a}{2} (\log R)^2 \Bigg).
\label{E:rennes3}
\end{align}
By Proposition \ref{P:upper_bound_rough}, the first probability on the right hand side is at most
\[
\PROB{e^{-x_0}R,R}{N_{r_k} \geq (1 + p \delta) J_k}
\leq
R^{-2} e^{-(1+o(1))(1+p \delta)b \log \frac{R}{r_k} }
= R^{-2-(1+o(1))(1+p \delta)b \eta_1 k}.
\]
When $(1+p \delta)b \eta_1 k \geq a + \eps$, we simply bound the second probability in \eqref{E:rennes3} by 1. The contribution of the sum over $p$ large enough so that $(1+p \delta)b \eta_1 k \geq a + \eps$ is then at most $R^{-2-a-\eps+o(1)}$. We now control the remaining part of the sum and take $p \geq 0$ such that $(1+p \delta)b \eta_1 k < a + \eps$.
Let
\[
\nu = 1 - \sqrt{\frac{b}{a} (1+(p+1)\delta)} \eta_1 k.
\]
The first two inequalities in \eqref{E:rennes2} ensure that $\nu \in (\eps, 1-\eps)$. Applying Markov's inequality and recalling that $r_k = R^{1-k\eta_1}$, we have
\begin{align*}
& \P \Bigg( \sum_{i=1}^{(1+(p+1)\delta)J_k} N_1^{(i)} \geq \frac{a}{2} (\log R)^2 \Bigg)
\leq R^{- \frac{a \nu}{1-k\eta_1} } \EXPECT{r_k,r_k \log R}{ \exp \left( \frac{2\nu N_1}{\log r_k} \right) }^{(1+(p+1)\delta)J_k}.
\end{align*}
Similarly to \eqref{E:rennes1},
\[
\EXPECT{r_k,r_k \log R}{ \exp \left( \frac{2\nu N_1}{\log r_k} \right) } = 1 + (1+o(1)) \frac{2}{r_k^2 \log r_k} \frac{\nu}{1-\nu}.
\]
After simplifications, this leads to
\begin{equation}
\label{E:rennes6}
\P \Bigg( \sum_{i=1}^{(1+(p+1)\delta)J_k} N_1^{(i)} \geq \frac{a}{2} (\log R)^2 \Bigg)
\leq \exp \left( - \frac{1+o(1)}{1-k \eta_1} \left( \sqrt{(1+(p+1)\delta)b} \eta_1 k - \sqrt{a} \right)^2 \log R \right).
\end{equation}
Overall this gives that
\begin{align*}
& \PROB{e^{-x_0}R,R}{ N_{r_k} \geq (1 + p\delta)J_k }
\P \Bigg( \sum_{i=1}^{(1+(p+1)\delta)J_k} N_1^{(i)} \geq \frac{a}{2} (\log R)^2 \Bigg)\\
& \leq R^{-2} \exp \left( - (1+o(1)) \left( (1+p \delta)b \eta_1 k + \frac{1}{1-k \eta_1} \left( \sqrt{(1+(p+1)\delta)b} \eta_1 k - \sqrt{a} \right)^2 \right) \log R \right) \\
%& = R^{-2} R^{(1+o(1)) \delta b \eta_1 k} \exp \left( - (1+o(1)) \left( (1+(p+1) \delta)b \eta_1 k + \frac{1}{1-k \eta_1} \left( \sqrt{(1+(p+1)\delta)b} \eta_1 k - \sqrt{a} \right)^2 \right) \log R \right) \\
%& = R^{-2} R^{-(1+o(1))a} R^{(1+o(1)) \delta b \eta_1 k} R^{- (1+o(1)) \frac{\eta_1 k}{1-\eta_1 k} \left( \sqrt{ (1+(p+1)\delta)b } - \sqrt{a} \right)^2 } \\
& = R^{-2} \exp \left( - (1+o(1)) \left( a - \delta b \eta_1 k + \frac{\eta_1 k}{1-\eta_1 k} \left( \sqrt{ (1+(p+1)\delta)b } - \sqrt{a} \right)^2 \right) \log R \right).
\end{align*}
By the last inequality in \eqref{E:rennes2},
we can bound from below the exponent
\[
- \delta b \eta_1 k + \frac{\eta_1 k}{1-\eta_1 k} \left( \sqrt{ (1+(p+1)\delta)b } - \sqrt{a} \right)^2 \geq \eta_1(- \delta b + \eta_1 (\sqrt{b} - \sqrt{a})^2) > \eps.
\]
We have shown that when $p$ is small enough so that $(1+p \delta)b \eta_1 k < a + \eps$, the probability in \eqref{E:rennes3} is at most $R^{-2 - a - \eps +o(1)}$. Since we have already bounded the remaining contribution of the sum, this concludes the proof.
\end{proof}

\begin{remark}
    We note that the appearance of a quadratic exponent in \eqref{E:rennes6} after taking the square root of the thickness parameters $a$ and $b$ is highly suggestive of an isomorphism-type picture, similar to the isomorphism relating the occupation of random walk trajectories to square of Gaussian free field \cite{LeJan, MarcusRosen}.
\end{remark}

\subsection{Truncated second moment}

We now prove Proposition~\ref{P:2point}, which upper bounds the probability for two points to be thick points under the nonbacktracking condition, and each satisfying the good event.

\begin{proof}[Proof of Proposition \ref{P:2point}]
Let $z,w \in B(0, R^{1- \eta_3})$, and recall that $\eta_3$ is fixed while $\eta_1 \to 0$. If $2|z-w| \leq R^{\eta_1} \log R$, then we can simply bound the left hand side of \eqref{E:2point_good} by the probability that a single point is thick which is at most $R^{-2-a+o(1)}$
%\[
%R^{-2-a+o(1)} \leq R^{-2-2a+o(1)} \left( %\frac{R}{|z-w|} \right)^a
%\]
and so (after elementary algebra) this probability is smaller than the right hand side of \eqref{E:2point_good} with $\eps_2 = a\eta_1/(1+\eta_1)$.

Thus we suppose now that $2|z-w| \ge R^{\eta_1} \log R$, hence we can find $k \in \{k_0, \dots, 1/\eta_1-1\}$ such that
\begin{equation}
\label{E:pf_2point2}
r_k \log R \leq 2|z-w| \leq r_{k-1} \log R.
\end{equation}
Here $k_0 = k_0( \eta_3, \eta_1) $ can be chosen arbitrarily large if $\eta_1$ is small enough for any given $\eta_3>0$. By the definition of the good event, the left hand side of \eqref{E:2point_good} is at most
\begin{align*}
\P_{z_0,R} \Big( \hat{N}_{z,1}, \hat{N}_{w,1} \geq \frac{a}{2} (\log R)^2, \hat{N}_{z,r_k}, \hat{N}_{w,r_k} \leq \frac{b}{2} r_k^2 \Big( \log \frac{R}{r_k} \Big)^2 \Big).
\end{align*}
Let $\delta >0$ be the inverse of a large integer and denote by $J_k = \frac{b}{2} r_k^2 \left( \log \frac{R}{r_k} \right)^2$. The above probability is equal to 
\begin{align*}
& \sum_{0 \leq p_1,p_2 < 1/\delta}
\P_{z_0,R} \Big( \hat{N}_{z,1}, \hat{N}_{w,1} \geq \frac{a}{2} (\log R)^2, p_1 \delta < \frac{\hat{N}_{z,r_k}}{J_k} \leq (p_1+1)\delta, p_2 \delta < \frac{\hat{N}_{w,r_k}}{J_k} \leq (p_2+1)\delta \Big).
\end{align*}
Let $p_1, p_2 \in \{0,\dots, 1/\delta -1\}$. Without loss of generality, let us assume that $p_2 \geq p_1$. Using the nonbacktracking condition, the descendants of the particles counted in $\hat N_{z,r_k}$ and $\hat N_{w,r_k}$ respectively, and which contribute to $\hat N_{z,1}$ and $\hat N_{w,1}$ respectively, are independent from one another. Hence we may bound the probability appearing in the above sum by
\begin{align}
\label{E:pf_2point1}
\PROB{z_0,R}{ N_{w,r_k} > p_2 \delta J_k }
\prod_{j=1}^2
\P \Bigg( \sum_{i=1}^{(p_j+1) \delta J_k} N^{(i)}_1 \geq \frac{a}{2} (\log R)^2 \Bigg)
\end{align}
where $N^{(i)}_1, i \geq 1$, are i.i.d.\ random variables distributed according to the number of pioneers on the unit sphere starting from a single particle on $\partial B(r_k)$ and where particles are killed on $\partial B(r_k \log R)$. By \eqref{E:rennes6}, for $j=1,2$,
\begin{align*}
\P \Bigg( \sum_{i=1}^{(p_j+1) \delta J_k} N^{(i)}_1 \geq \frac{a}{2} (\log R)^2 \Bigg)
\leq \exp \left( - \frac{1+o(1)}{1-k \eta_1} \left( \sqrt{(p_j+1)\delta b} \eta_1 k - \sqrt{a} \right)^2 \log R  \right)
\end{align*}
and by Proposition \ref{P:upper_bound_rough},
\begin{align*}
\PROB{z_0,R}{ N_{w,r_k} > p_2 \delta J_k } & \leq R^{-2+o(1)} \exp \left( - bp_2 \delta \log \frac{R}{r_k} \right) \\
& \leq R^{-2+o(1) - bp_2 \delta \eta_1 k }
\leq R^{-2+o(1) - \frac12 \sum_{j=1}^2 bp_j \delta \eta_1 k }.
\end{align*}
In the last inequality above, we used the fact that $p_2 \geq p_1$.
Putting things together, and combining terms suitably, \eqref{E:pf_2point1} is at most
\begin{align*}
& R^{-2-2a+o(1)} \exp \Bigg( - \frac{k \eta_1}{1-k \eta_1} \sum_{j=1}^2 \left( a - 2 \sqrt{(p_j+1)\delta a b} + (p_j+1) \delta b \eta_1 k + \frac{1-k\eta_1}{2}p_j \delta b \right) \log R \Bigg) \\
& = R^{-2-2a+o(1)} \exp \Bigg( - \frac{k\eta_1}{1-k \eta_1} \sum_{j=1}^2 \left( \left( \sqrt{a} - \sqrt{(p_j+1)\delta b} \right)^2 - (1-k \eta_1) \left( (p_j+1)\delta b - \frac12 p_j \delta b \right) \right) \log R \Bigg) \\
& \leq R^{-2-2a+o(1)} \exp \Bigg( \frac12 k \eta_1 \sum_{j=1}^2 (p_j+2) \delta b \log R \Bigg).
\end{align*}
By \eqref{E:pf_2point2},
\[
\log \frac{R}{2|z-w|} \geq (k-1) \eta_1 \log R - \log \log R = (1+o(1)) \left(1 - \frac1k \right) k \eta_1 \log R.
\]
So we have obtained that \eqref{E:pf_2point1} is at most
\[
R^{-2-2a+o(1)} \left( \frac{R}{|z-w|} \right)^{\frac12 \left(1 - \frac1k \right)^{-1} \sum_{j=1}^2 (p_j+2) \delta b }.
\]
Since $p_1,p_2 \leq 1/\delta -1$, the exponent $\frac12 (1 - \frac1k )^{-1} \sum_{j=1}^2 (p_j+2) \delta b$ is bounded by $(1 - \frac{1}{k_0})^{-1} (1 + \delta)b$ which can be made arbitrarily close to $(1 - \frac1{k_0} )^{-1} b$ by choosing $\delta$ small enough. Recall further that $k_0 = k_0( \eta_1, \eta_3)$ can be chosen arbitrarily large by taking $\eta_1$ small enough for any given $\eta_3>0$. 
%we look at points $z,w \in B(R^{1-\eta_3})$, the distance $|z-w|$ is at most $2 R^{1-\eta_3}$ and the associated $k$ has to be at least of order $\eta_3 / \eta_1$. For a fixed $\eta_3$ this can be made arbitrarily large by choosing $\eta_1$ very small. 
Hence the exponent $\frac12 (1 - \frac1{k_0} )^{-1} \sum_{j=1}^2 (p_j+2) \delta b$ can be made as close to $a$ as desired. This concludes the proof.
\end{proof}

\section{Occupation measure of branching Brownian motion (proof of Theorem~\ref{T:local_time_tail})}\label{S:occupation_measure}

In this section, we prove Theorem~\ref{T:local_time_tail}, which concerns the probability that a small ball has a large occupation measure for branching Brownian motion. We start with the lower bound, which follows quickly from Theorem \ref{T:maintail4}:

\begin{proof}[Proof of Theorem \ref{T:local_time_tail} -- Lower bound]
Let $x_0 >0$, $\eps >0$. By Theorem \ref{T:maintail4}, there exists $x_* >0$ and $R_0 >0$ such that for all $R \geq R_0$, $r \leq e^{-x_*} R$, $a >0$,
\[
\PROB{e^{-x_0}R,R}{N_r \geq \frac{a}{2} r^2 (\log R/r)^2 \vert N_r>0 } \geq (R/r)^{-(1+\eps)a}.
\]
The local time of $B(r)$ can be decomposed as $\sum_{i=1 \dots N_r} L_i$ where $L_i$ is the occupation measure accumulated by the progeny of the $i$-th pioneer of $\partial B(r)$. The $L_i$'s are independent of $N_r$ and are i.i.d. and distributed according to the occupation measure of $B(r)$ under $\P_{z,R}$ where $z$ is any point of $\partial B(r)$. Letting $n = \frac{(1+\eps)^2 a}{2} r^2 (\log R/r)^2$, we have
\begin{align*}
\PROB{e^{-x_0}R,R}{L(B(r)) \geq \frac{am_1}{2} r^4 (\log R/r)^2 \vert N_r>0 }
\geq (R/r)^{-(1+\eps)^3 a} \Prob{ \sum_{i=1}^{n} L_i \geq  \frac{am_1}{2} r^4 (\log R/r)^2 }.
\end{align*}
Recall that since the branching mechanism is critical, $\Expect{L_i}$ is equal to the expected value of the local time in $B(r)$ of a single Brownian motion that starts on $\partial B(r)$ and that is killed on $\partial B(R)$. If $x_*$ is large enough, this expectation is close to the one in infinite volume, more precisely it is larger than $(1+\eps)^{-1} m_1 r^2$ (recall the definition \eqref{E:m1} of $m_1$).
The probability on the right hand side of the above display is therefore at least
\[
\Prob{ \sum_{i=1}^{n} L_i \geq  (1+\eps)^{-1} \Expect{L_1} n }.
\]
Because the $L_i$'s are bounded from below by 0, this probability is extremely close to 1.
This is for instance a consequence of standard large deviations for i.i.d. nonnegative random variables,  sometimes called Benett's inequality (see e.g. \cite[Theorem 2.9]{boucheron2013concentration}) or \cite[Theorem (9.5)]{Durrett}). 
This concludes the proof.
\end{proof}

We now move on to the proof of the upper bound \eqref{E:T_local_upper} in Theorem \ref{T:local_time_tail}. We first explain the difficulty of its proof. Let $r \geq 1$.
As in the proof of the lower bound, we can decompose the occupation measure of the ball $B(r)$ as $\sum_{i=1 \dots N_r} L_i$ where $L_i$ is the occupation measure accumulated by the progeny of the $i$-th pioneer of $\partial B(r)$.
We want to argue that, given $N_r$, the sum $\sum_{i=1 \dots N_r} L_i$ is concentrated around $\Expect{L_1} N_r$. In the lower bound, this step was straightforward and we only used the nonnegativity of the $L_i$'s. For the more subtle upper bound, we need to guarantee a certain decay of their upper tail, which is the content of the following theorem.

\begin{theorem}
There exists $\lambda_0 >0$ such that for all $\lambda \in [0,\lambda_0]$,
\begin{equation}
\label{E:local_time_laplace}
\EXPECT{r,R}{ \exp \left( \frac{\lambda}{r^4 \log (R/r)} L(B(r)) \right) } = 1 + \frac{\lambda m_1 +o(1)}{r^2\log (R/r)}
\end{equation}
where $o(1) \to 0$ as $r/R \to 0$ and $R \to \infty$.
\end{theorem}

%The fact that $v(\lambda) <2$ for all $\lambda < \lambda_0$ is significant in our subsequent use of this theorem.

\begin{proof}
We first claim that there exists $\lambda_0 >0$ and $C>0$ such that for all $\lambda \in [0,2\lambda_0]$, $R$ large enough and $R/r$ large enough,
\begin{equation}
\label{E:local_time_laplace4}
\EXPECT{r,R}{ \exp \left( \frac{\lambda}{r^4 \log (R/r)} L(B(r)) \right) } \leq 1 + \frac{C}{r^2\log (R/r)}.
\end{equation}
This is a reformulation of a result obtained in \cite{asselah2022time} in the context of branching random walk; see in particular (9.2) therein. 
%The fact that the constant $v(\lambda)$ is smaller than 2 can always be guaranteed by decreasing the value of $\lambda_0$. 
Their proof is based on the following estimate which is the content of \cite[Proposition 9.1]{asselah2022time}: there exists $c, C>0$ such that for all $r \geq 1$ and $t>r^4$,
\begin{equation}
\label{E:AS}
\sup_{x \in B(2r)} \PROB{x,2r}{L(B(2r))>t} \leq \frac{C}{r^2} \exp \left( - c t/r^4 \right).
\end{equation}
To get \eqref{E:local_time_laplace4} from this estimate, Asselah and Schapira decompose $L(B(r))$ as a sum over successive ``waves'' (or excursions) between $\partial B(r)$ and $\partial B(2r)$ of the occupation measure of $B(r)$ accumulated during the wave. Each such elementary occupation measure is handled thanks to \eqref{E:AS} wheareas the total number of terms involved in the sum is handled thanks to their knowledge on the number of pioneers. See \cite[Section 9]{asselah2022time} for more details. The exact same approach can be done in the continuum (in fact more precise estimates on these pioneers were already obtained here in the previous sections). This leads to \eqref{E:local_time_laplace4}; we omit the details.

\begin{comment}
the following: there exists $\lambda>0$ and $C>0$ such that
\begin{equation}
\EXPECT{r,R}{ \exp \left( \frac{\lambda}{r^4 \log (R/r)} L(B(r)) \right) } \leq 1 + \frac{C(\lambda)+o(1)}{r^2\log (R/r)},
\end{equation}
where $C(\lambda) \to 0$ as $\lambda\to 0$. We thus obtain the result by taking $\lambda < \lambda_0$
small enough (but fixed). 
We omit the details.
\end{comment}

We will need in addition the following moment estimate:
for all $p \geq 1$, there exists $C_p >0$ such that for all $r \geq 1$,
\begin{equation}
\label{E:local_time_laplace5}
\EXPECT{r,\infty}{L(B(r))^p} \leq C_p r^{-2+4p}.
\end{equation}
This type of statement is relatively to prove by induction on $p\ge 1$. As already explained, by criticality of the branching mechanism, the first moment coincides with the first moment of the local time of a single Brownian motion, which is equal to $m_1 r^2$. For the second moment, one needs to handle the local times generated by two particles. One can decompose over the generation of the most recent common ancestor. The remaining trajectories are independent and the second moment factorises into the square of the first moment. Integrating out the randomness of the most recent common ancestor, one gets \eqref{E:local_time_laplace5} for $p=2$. This procedure can be iterated to control moments of arbitrary order $p\ge 1$ by considering the first branching event, leading to a diagrammatic sum representing the possible genealogical structure of the $p$ particles. It has been carried out in great detail in the context of branching random walk in \cite[Section 3]{angel2021tail}. We emphasise that we do not require any control on the growth of the constants $C_p$, which would require a careful analysis of these diagrammatic sums as was done in \cite{angel2021tail}. Here we will actually only use \eqref{E:local_time_laplace5} with $p=2$ and $4$, and thus omit the details.

We now explain how we obtain \eqref{E:local_time_laplace}. Let us denote by $X = L(B(r)) / (r^4\log(R/r))$. Let $\lambda \in [0,\lambda_0]$. Using the mean value theorem we can bound
\begin{align*}
0\le \EXPECT{r,R}{e^{\lambda X} - 1 - \lambda X} & \leq \frac{\lambda_0^2}{2} \EXPECT{r,R}{e^{\lambda_0 X} X^2}
= \frac{\lambda_0^2}{2} \EXPECT{r,R}{X^2} + \frac{\lambda_0^2}{2} \EXPECT{r,R}{(e^{\lambda_0 X} -1 ) X^2} \\
& \leq \frac{\lambda_0^2}{2} \EXPECT{r,R}{X^2} + \frac{\lambda_0^2}{2} \EXPECT{r,R}{(e^{\lambda_0 X} -1 )^2}^{1/2} \EXPECT{r,R}{X^4}^{1/2}
\end{align*}
where we used Cauchy--Schwarz in the last inequality. By \eqref{E:local_time_laplace5},
\[
\EXPECT{r,R}{X^2} \leq \frac{C}{r^2 (\log (R/r))^2}
\quad \text{and} \quad
\EXPECT{r,R}{X^4}^{1/2}  \leq \frac{C}{r (\log (R/r))^2}
\]
and
by \eqref{E:local_time_laplace4},
\[
\EXPECT{r,R}{(e^{\lambda_0 X} -1 )^2}
\leq \EXPECT{r,R}{e^{2\lambda_0 X} } - 1
\leq \frac{C \lambda_0^2}{r^2 \log (R/r)}.
\]
Thus (we absorb the dependence on $\lambda_0$ into the constant $C$),
\[
0\le \EXPECT{r,R}{e^{\lambda X} - 1 - \lambda X} 
\leq \frac{C}{r^2 (\log(R/r))^2}.
\]
Since $\EXPECT{r,R}{L(B(r))} = (1+o(1))m_1 r^2$, this concludes the proof of \eqref{E:local_time_laplace}.
\end{proof}

Using directly the estimate \eqref{E:local_time_laplace} as well as Theorem \ref{T:laplace_sharp}, we find that for all $\lambda < \min(\lambda_0, 2/m_1)$,
\begin{align*}
\PROB{e^{-x_0} R, R}{ L(B(r)) \geq \frac{am_1}{2} r^4 (\log R/r)^2 }
\leq R^{-2+o(1)} (R/r)^{-\frac{am_1 \lambda}{2}}.
\end{align*}
If we were allowed to choose $\lambda$ as close to $2/m_1$ as possible, this would already imply the upper bound of Theorem \ref{T:local_time_tail} (without any restriction on $a \in (0, a_0)$). However, we have no control over the $\lambda_0$ in \eqref{E:local_time_laplace} at this stage, and hence we will need to proceed in a more complicated way, using once more the notion of ``non-backtracking'' particles from Section~\ref{S:thick}.

\begin{remark}
    In fact, we believe that the optimal value for $\lambda_0$ in \eqref{E:local_time_laplace} is equal to $2/m_1$. However, it seems unlikely that the approach of \cite{asselah2022time} can directly achieve this level of precision.
\end{remark}

\begin{proof}[Proof of Theorem \ref{T:local_time_tail} -- Upper bound]
In this proof, we will assume that $a < \min (2, \lambda_0 m_1)$ where $\lambda_0>0$ is the constant appearing in \eqref{E:local_time_laplace}.
Let $x_0 >0$, $k_{\max} \geq 1$ be a large integer and $\eta_1 = 1/(k_{\max}+1)$. Let $\eta_2>0$ (inverse of a large integer) and $\eps >0$ be both much smaller than $\eta_1$.
Let $R_0>0$ and $x_*>0$ be large enough so that Theorems \ref{T:maintail4} and \ref{T:laplace_sharp} and the estimate \eqref{E:local_time_laplace} apply with errors at most $\eps$. Let $R > R_0^{1/\eta_1}$ and $r \in [1, e^{-x_*/\eta_1} R]$.

For all $k=0, \dots, k_{\max}+1$, let $r_k = r (R/r)^{k \eta_1}$. By definition, $r_0 = r$ and $r_{k_{\max}+1} = R$. We have tuned the parameters to ensure that $r_1 \geq R_0$ and $r_{k+1} / r_k \geq e^{x_*}$ for all $k$.
For each $k \in \{1, \dots, k_{\max}\}$, we let $\mathbf{P}_{r_k}$ be the set of pioneers on $\partial B(r_k)$ emanating from the pioneers on $\partial B(r)$. By convention, we also set $\mathbf{P}_{r}$ to be the set of pioneers on $\partial B(r)$.
For each $k \in \{0, \dots, k_{\max}\}$, let $\mathbf{L}_k$ be the local time in $B(r)$ induced by $\mathbf{P}_{r_k}$ where particles are killed on $\partial B(r_{k+1})$. In words, $\mathbf{L}_k$ captures the local time in $B(r)$ of particles that backtracked exactly to level $r_k$.
By construction, the total local time of $B(r)$ is equal to
$
\sum_{k=0}^{k_{\max}} \mathbf{L}_k.
$
Let $(\Fc_k)_{k=0, \dots, k_{\max}}$ be the filtration:
\begin{equation}
\label{E:filtration}
\Fc_k = \sigma ( \# \mathbf{P}_{r_\ell}, \ell = 0, \dots, k ), \quad k =0, \dots, k_{\max}.
\end{equation}
Let $k \in \{ 1, \dots, k_{\max} \}$. Note that conditionally on $\Fc_k$, $(\mathbf{L}_k, \mathbf{L}_{k+1}, \dots, \mathbf{L}_{k_{\max}})$ is independent from $(\mathbf{L}_0, \mathbf{L}_1, \dots, \mathbf{L}_{k-1})$ and is distributed as follows.
%\red{[TH: I got a bit confused about whether there was an off-by-one error in the indexing here. It's probably fine but please check again.]} 
Consider a branching Brownian motion with $\# \mathbf{P}_{r_k}$ initial particles located on $\partial B(r_k)$. Let $\ell_k$ be the total local time of $B(r)$ where particles are killed on $\partial B(r_{k+1})$. For $i = 1, \dots, k_{\max} -k$, let $\ell_{k+i}$ be the local time of $B(r)$ generated by the pioneers of $\partial B(r_{k+i})$ where particles are killed on $\partial B(r_{k+i+1})$. Then the law of $(\mathbf{L}_k, \mathbf{L}_{k+1}, \dots, \mathbf{L}_{k_{\max}})$ conditionally given $\Fc_k$ is the same as the law of $(\ell_k, \dots, \ell_{k_{\max}})$.
%Note however that the conditional law of $(\mathbf{L}_0, \mathbf{L}_1, \dots, \mathbf{L}_{k-1})$ given $\Fc_k$ is not as explicit.

For each $\mathbf{k} \in\{ 0, \dots, k_{\max}\}$, we define the events
\[
E_{\mathbf{k}} := \bigl\{ \max \{ k=0, \dots, k_{\max}: \# \mathbf{P}_{r_k} \geq 1 \} = \mathbf{k} \bigr\}, \quad
F_{\mathbf{k}} := \biggl\{ \forall k = 0, \dots, \mathbf{k}, \# \mathbf{P}_{r_k} \leq r_k^2 \frac{(\log R/r)^2}{\log \log R/r} \biggr\}.
\]
Note that $F_\mathbf{k}$ is measurable with respect to $\Fc_\mathbf{k}$. By Lemma \ref{L:invariance}, the probability
$\Prob{ F_{k_{\max}}^c \vert N_r >0 }$ decays faster than any polynomial in $R/r$ (here we use that $\sup_R s_*(R) < 0$).
In the rest of the proof, we will work on the event $F_{k_{\max}} \cap\{ N_r>0\}$.
We first keep track of the contribution of $\mathbf{L}_0$:
\begin{align}
\label{E:local_upper1}
& \P \Big( \sum_{k = 0}^{k_{\max}} \mathbf{L}_k \geq \frac{a m_1}{2} r^4 (\log R/r)^2, F_{k_{\max}} \Big)
\leq \Prob{ \mathbf{L}_0 \geq \frac{a m_1}{2} r^4 (\log R/r)^2 }
\\
&\hspace{1cm} + \sum_{p=0}^{1/\eta_2-1} \P \Big( \mathbf{L}_0 \in [p, p+1) \frac{a m_1 \eta_2}{2} r^4 (\log R/r)^2,
\sum_{k = 1}^{k_{\max}} \mathbf{L}_k \geq (1-(p+1) \eta_2) \frac{a m_1}{2} r^4 (\log R/r)^2, F_{k_{\max}} \Big).
\nonumber
\end{align}
We first claim that the killing on $\partial B(r_1)$ and the estimate \eqref{E:local_time_laplace} are enough to show that (recall that $\eps$ and $\eta_2$ are much smaller than $\eta_1$ so that $\eps'$ defined below is small)
\begin{equation}
\label{E:local_upper5}
\Prob{ \mathbf{L}_0 \geq \frac{a m_1}{2} r^4 (\log R/r)^2 \,\Big \vert\, N_r >0 }
\leq (R/r)^{-a+\eps'}
\quad \text{where} \quad
\eps' = m_1 \lambda_0 \frac{\eps + \eta_2}{2\eta_1} + \eps.
\end{equation}
Indeed, we decompose $\mathbf{L}_0 = \sum_{i=1}^{N_r} L_i$ where $L_i$, $i \geq 1$, are i.i.d. and distributed as the local time of $B(r)$ starting with one particle on $\partial B(r)$ and where particles are killed on $\partial B(r_1)$. We have 
\begin{align}
\nonumber
& \Prob{ \mathbf{L}_0 \geq \frac{a m_1}{2} r^4 (\log R/r)^2 \,\Big \vert\, N_r >0 }
\leq \PROB{e^{-x_0}R,R}{ N_r \geq \frac{a}{2} r^2 (\log R/r)^2 \,\Big \vert\, N_r >0 } \\
&\hspace{1.75cm} + \sum_{p=0}^{1/\eta_2-1} \PROB{e^{-x_0}R,R}{ N_r \in [n_p,n_{p+1}) \,\Big \vert\, N_r >0 }
\Prob{ \sum_{i=1}^{n_{p+1}} L_i \geq \frac{am_1}{2} r^4 (\log R/r)^2 }
\label{E:local_time_laplace6}
\end{align}
where $n_p = p \eta_2 \frac{a}{2}r^2 (\log R/r)^2$, $p \geq 0$. By Markov's inequality and \eqref{E:local_time_laplace}, the last probability in the above display is at most
\begin{align*}
(R/r)^{- \frac{am_1 \lambda_0}{2\eta_1}} \EXPECT{r,r_1}{ \exp \left( \frac{\lambda_0}{r^4 (\log r_1/r)} L(B(r)) \right)}^{n_{p+1}}
& = (R/r)^{- \frac{am_1 \lambda_0}{2\eta_1}} \left( 1 + \frac{m_1 \lambda_0 }{r^2 \eta_1(\log R/r)}(1+\eps) \right)^{n_{p+1}}\\
%& = (R/r)^{- \frac{am_1 \lambda_0}{2\eta_1} + \frac{(p+1)\eta_2 a}{2 \eta_1} m_1 \lambda_0 (1+ \eps)}.
%\leq (R/r)^{- (1 - \eps (p+1) \eta_2) \frac{am_1 \lambda_0}{4\eta_1}} \\
& = \exp\left( - \frac{am_1 \lambda_0}{2\eta_1} ( 1 - (p+1) \eta_2  (1+ \eps))\log (R/r)\right).
\end{align*}
Moreover, by Theorem \ref{T:maintail4}, $\PROB{e^{-x_0}R,R}{N_r \geq n_p\vert N_r>0} \leq (R/r)^{-a p\eta_2 + \eps}$.
This shows that the sum in \eqref{E:local_time_laplace6} is at most
\begin{align*}
(R/r)^{-a+o(1)} \sum_{p=0}^{1/\eta_2-1}
\exp \left( \left(
(1-p\eta_2)a (1 - \frac{m_1 \lambda_0}{2\eta_1}) + m_1 \lambda_0 \frac{\eps (p+1)\eta_2 + \eta_2}{2\eta_1} + \eps \right)
\log \frac{R}{r} \right).
\end{align*}
We simply bound $1 - \frac{m_1 \lambda_0}{2\eta_1} \leq 0$ and $(p+1) \eta_2 \leq 1$ to get that the sum in \eqref{E:local_time_laplace6} is at most $(R/r)^{-a+\eps'+o(1)}$ where $\eps'$ is defined in \eqref{E:local_upper5}.
Together with Theorem \ref{T:maintail4} this establishes the claim \eqref{E:local_upper5}, which handles the first term on the right hand side of \eqref{E:local_upper1}.

% We now come back to \eqref{E:local_upper1} where we need to take care of the sum.
It remains to bound the sum over $p$ appearing on the right hand side of \eqref{E:local_upper1}. By using that $\mathbf{L}_0$ and $(\mathbf{L}_1, \dots, \mathbf{L}_{k_{\max}})$ are independent conditionally on $\Fc_1$, we can further bound the left hand side of \eqref{E:local_upper1} by
\begin{align}
\label{E:local_upper3}
\Prob{ \mathbf{L}_0 \geq \frac{a m_1}{2} r^4 (\log R/r)^2 } & + \sum_{p=0}^{1/\eta_2-1} \E \Bigg[ \Prob{ \mathbf{L}_0 \in [p, p+1) \frac{a m_1 \eta_2}{2} r^4 (\log R/r)^2 \,\Big\vert\, \Fc_1 } \times \\
& \times \Prob{ \sum_{k = 1}^{k_{\max}} \mathbf{L}_k \geq (1-(p+1) \eta_2) \frac{a m_1}{2} r^4 (\log R/r)^2, F_{k_{\max}} \,\Bigg\vert\, \Fc_1 } \Bigg].
\nonumber
\end{align}
Let $p \in \{0,\dots, 1/\eta_2 -1\}$.
We now focus on bounding the last probability above and denote by $\alpha = (1-(p+1)\eta_2) a$. The event inside the probability can occur only on $\bigcup_1^{k_{\max}} E_{\mathbf{k}} \cap F_{\mathbf{k}}$. Let $\mathbf{k} \in \{1, \dots, k_{\max} \}$. By definition of $E_{\mathbf{k}}$ and then by Cauchy--Schwarz,
\begin{align}
\nonumber
& \Prob{ \sum_{k = 1}^{k_{\max}} \mathbf{L}_k \geq \frac{\alpha m_1}{2} r^4 (\log R/r)^2, E_{\mathbf{k}} \cap F_{\mathbf{k}} \,\Bigg \vert\, \Fc_1 }
\leq \Prob{ \sum_{k = 1}^{\mathbf{k}} \mathbf{L}_k \geq \frac{\alpha m_1}{2} r^4 (\log R/r)^2, \# \mathbf{P}_{r_{\mathbf{k}}} > 0, F_{\mathbf{k}} \,\Bigg \vert\, \Fc_1 } \\
& \leq \Prob{ \sum_{k = 1}^{\mathbf{k}} \mathbf{L}_k \geq \frac{\alpha m_1}{2} r^4 (\log R/r)^2, F_{\mathbf{k}} \,\Bigg \vert\, \Fc_1 }^{1/2}
\Prob{ \# \mathbf{P}_{r_{\mathbf{k}}} > 0 \,\Bigg \vert\, \Fc_1 }^{1/2}.
\label{E:local_upper2}
\end{align}
On the event that $\# \mathbf{P}_1 \leq r_1^2 (\log R/r)^2 / \log \log R/r$ (which is measurable w.r.t. $\mathcal{F}_1$), we bound crudely, 
\begin{equation}
\Prob{ \# \mathbf{P}_{r_{\mathbf{k}}} > 0 \,\big\vert\, \Fc_1 }^{1/2} \mathbf{1}_{F_1} \leq (r_1/ r_{\mathbf{k}}) (R/r)^{\eps} = (R/r)^{- ({\mathbf{k}}-1)\eta_1 +\eps} = (R/r)^{- \frac{\mathbf{k}-1}{k_{\max} + 1} +\eps}.
\label{eq:vienne1}
\end{equation}
We now bound the first probability in \eqref{E:local_upper2}. For this purpose, let 
\begin{equation}
\mu_0 = \min( \lambda_0, \frac{2}{(1+\eps)m_1}) \quad \text{and} \quad \lambda = \mu_0 \frac{k_{\max} +1}{\mathbf{k}+1}.
\label{eq:vienne2}
\end{equation}
By Markov's inequality and by independence of $\mathbf{L}_{\mathbf{k}}$ and $(\mathbf{L}_0, \dots, \mathbf{L}_{\mathbf{k}-1})$ conditionally on $\Fc_{\mathbf{k}}$,
\begin{align}
\Prob{ \sum_{k = 1}^{\mathbf{k}} \mathbf{L}_k \geq \frac{\alpha m_1}{2} r^4 (\log R/r)^2 ,F_{\mathbf{k}} \big\vert \Fc_1 }
\leq (R/r)^{- \frac{\lambda \alpha m_1}{2} }
\Expect{ \mathbf{1}_{F_{\mathbf{k}}} \exp \left( \frac{\lambda}{r^4 (\log R/r)} \sum_{k=1}^{\mathbf{k}} \mathbf{L}_k \right) \big\vert \Fc_1 } \nonumber \\
= (R/r)^{- \frac{\lambda \alpha m_1}{2} } \Expect{ \mathbf{1}_{F_{\mathbf{k}}} \Expect{ \exp \left( \frac{\lambda}{r^4 (\log R/r)} \sum_{k=1}^{\mathbf{k}-1} \mathbf{L}_k \right) \big\vert \Fc_{\mathbf{k}} } \Expect{ \exp \left( \frac{\lambda}{r^4 (\log R/r)} \mathbf{L}_{\mathbf{k}} \right) \big\vert \Fc_{\mathbf{k}} } \big\vert \Fc_1 }.\label{eq:vienne3}
\end{align}
Moreover, by the explicit description of the conditional law of $\mathbf{L}_k$ given $\Fc_k$ (see below \eqref{E:filtration}) and by \eqref{E:local_time_laplace} and Proposition \ref{P:upper_bound_rough}, 
\begin{align*}
\Expect{ \exp \left( \frac{\lambda}{r^4 (\log R/r)} \mathbf{L}_{\mathbf{k}} \right) \big\vert \Fc_{\mathbf{k}} }
= \EXPECT{r_{\mathbf{k}}, r_{{\mathbf{k}}+1}}{ \EXPECT{r,r_{{\mathbf{k}}+1}}{ \exp \left( \frac{\lambda}{r^4 (\log R/r)} L(B(r)) \right) }^{N_r} }^{\# \mathbf{P}_{r_{\mathbf{k}}} } \\
= \EXPECT{r_{\mathbf{k}}, r_{{\mathbf{k}}+1}}{ \left( 1 + \frac{(1+\eps)m_1\lambda}{r^2 (\log R/r)} \right)^{N_r} }^{\# \mathbf{P}_{r_{\mathbf{k}}} }
= \left( 1 + \frac{2}{r_{\mathbf{k}}^2 \log (r_{\mathbf{k}}/r)} \frac{\frac{(1+\eps)m_1 \lambda}{2} \frac{{\mathbf{k}}+1}{k_{\max}+1}}{1-\frac{(1+\eps)m_1 \lambda}{2} \frac{{\mathbf{k}}+1}{k_{\max}+1}} \right)^{\# \mathbf{P}_{r_{\mathbf{k}}} }.
\end{align*}
In the second equality we used that $\lambda / \log(R/r) \leq \lambda_0 / \log (r_{{\mathbf{k}}+1}/r)$ and in the third equality we used that $(1+\eps)m_1 \lambda / \log(R/r) < 2 / \log(r_{{\mathbf{k}}+1}/r)$.
On the event $F_{\mathbf{k}}$, the right hand side of the above display is bounded by $(R/r)^{2\eps}$. We have obtained that 
\begin{align}
\Expect{ \mathbf{1}_{F_{\mathbf{k}}} \exp \left( \frac{\lambda}{r^4 (\log R/r)} \sum_{k=1}^{\mathbf{k}} \mathbf{L}_k \right) \big\vert \Fc_1 }
\leq (R/r)^{2\eps} \Expect{ \mathbf{1}_{F_{\mathbf{k}-1}} \exp \left( \frac{\lambda}{r^4 (\log R/r)} \sum_{k=1}^{\mathbf{k}-1} \mathbf{L}_k \right) \big\vert \Fc_1 }.\label{eq:vienne4}
\end{align}
By iterating, we find that the left hand side of the above display is at most $(R/r)^{2\eps k_{\max}}$. Putting together \eqref{eq:vienne1}, \eqref{eq:vienne2}, \eqref{eq:vienne3}, \eqref{eq:vienne4}, we obtain that \eqref{E:local_upper2} is at most
\[
\exp \left( - \left( \frac{\mu_0 \alpha m_1}{4} \frac{k_{\max}+1}{{\mathbf{k}}+1} + \frac{{\mathbf{k}}-1}{k_{\max}+1} -(k_{\max}+1) \eps \right) \log \frac{R}{r} \right)
\leq (R/r)^{2\eta_1+ \eps/\eta_1} (R/r)^{- \sqrt{\mu_0 m_1 \alpha}}
\]
using that for all $u \in (0,1)$, $c_1 u + c_2/u \geq 2 \sqrt{c_1 c_2}$ (applied to $u = \frac{{\mathbf{k}}+1}{k_{\max}+1}$).
Since we assumed at the very beginning of the proof that $a < \min(2,\lambda_0 m_1) = \mu_0 m_1$, we have $\alpha \leq \mu_0 m_1$ and we can bound $\sqrt{\mu_0 m_1 \alpha} \geq \alpha$.
Injecting this estimate back in \eqref{E:local_upper3} and recalling that we denoted $\alpha = (1-(p+1)\eta_2)a$, we obtain that
\begin{align*}
& \Prob{ \sum_{k = 0}^{k_{\max}} \mathbf{L}_k \geq \frac{a m_1}{2} r^4 (\log R/r)^2, F_{k_{\max}} }
\leq \Prob{ \mathbf{L}_0 \geq \frac{a m_1}{2} r^4 (\log R/r)^2 } \\
& + (R/r)^{2\eta_1+\eps/\eta_1} \sum_{p=0}^{1/\eta_2-1} (R/r)^{-(1-(p+1)\eta_2)a } \Prob{ \mathbf{L}_0 \in [p, p+1) \frac{a m_1 \eta_2}{2} r^4 (\log R/r)^2 }.
\end{align*}
By \eqref{E:local_upper5} (where $\eps'$ is also defined), this leads to 
\[
\Prob{ \sum_{k = 0}^{k_{\max}} \mathbf{L}_k \geq \frac{a m_1}{2} r^4 (\log R/r)^2, F_{k_{\max}} \vert N_r >0 }
\leq (R/r)^{2\eta_1+\eta_2/\eta_1 a+\eps/\eta_1+\eps'} (R/r)^{-a}.
\]
Because ${2\eta_1+\eta_2/\eta_1 a+\eps/\eta_1+\eps'}$ can be arbitrary small, this concludes the proof.
\end{proof}

\section{Strong coupling and local time of branching random walk (proof of Theorem \ref{T:thick_BRW})}\label{sec:discrete}
\subsection{Strong coupling of tree-indexed random walk and Brownian motion}

In this section, we state and prove our strong coupling between tree-indexed random walk and tree-indexed Brownian motion.
This is the analogue for tree-indexed processes of the celebrated dyadic or KMT coupling  giving a strong approximation of Brownian motion by random walk. This subject is in itself one with a long and distinguished history; we refer for instance to \cite[Chapter 7]{lawler2010random} for a summary of some results on this topic.

All the trees considered in section will be continuous trees with finitely many branching points, i.e. each edge will carry a length. Tree-indexed random walks will then also be defined on the edges by linear interpolation.
Our result will follow from a novel connection between this problem and the so-called Horton--Strahler number of a tree, which we now define:

\paragraph*{Horton--Strahler number}
The Horton--Strahler number $H(T)$ of a finite tree $T$ is defined recursively as follows. If the tree $T$ consists of a single vertex (the root), we set $H(T) = 1$. Otherwise, let $T_1, \dots, T_m$ be the subtrees hanging off the children of the root. We set
\[
H(T) =
\left\{
\begin{array}{ll}
\max_{i = 1, \dots, m} H(T_i), & \text{if this maximum is achieved by only one of the trees $T_1,\ldots,T_m$} 
% \quad \text{argmax}_{i=1, \dots, m} H(T_i) = 1;
\\
\max_{i = 1, \dots, m} H(T_i) + 1, & \text{otherwise.} 
% \quad \text{argmax}_{i=1, \dots, m} H(T_i) \geq 2.
\end{array}
\right.
\]
% \red{[TH: Is argmax here supposed to mean the *number* of children achieving the max?]}
$H(T)$ can be thought of as a way of measuring the branching complexity of the tree $T$. For instance, if $T$ is simply a path, then $H(T) = 1$, no matter how long this path is. On the other hand, if $T$ is a ``perfect binary tree'', where each interior vertex has two children and where all leaves have the same depth, then $H(T)$ is equal to the depth of $T$ which is also equal to $\log_2 (\# V(T)+1)$. In general, it follows directly from its recursive definition that the Horton--Strahler number of a tree can be bounded by $\log_2 (\# V(T)+1)$.
See the introduction of \cite{10.1214/21-EJP678} for more background on the Horton--Strahler number, especially in the context of Bienaymé--Galton--Watson trees.

\medskip

Let $\theta$ be a probability distribution on $\R^d$. Assume that $\theta$ is centred with some finite exponential moment. Denote by $\Gamma$ the covariance matrix of $\theta$, which we suppose invertible.

\begin{theorem}\label{T:coupling_general}
There exist $c_1, c_2, c_3 >0$ such that the following holds. Let $T$ be a continuous tree with finitely many branching points and edges of unit length. Let $S_T$ be a $T$-indexed random walk (with linear interpolation between neighbouring branching points) with increments distributed according to $\theta$ and let $B_T$ be a $T$-indexed Brownian motion. There exists a coupling between $S_T$ and $B_T$ such that for all $x \geq 0$,
\[
\Prob{ \sup_{t \in T} |\Gamma^{-1/2} S(t) - B(t)| \geq c_1 H(T) \log (1+d(T)) + c_1 \log \ell(T) + x } \leq c_2 e^{-c_3 x}
\]
where $d(T)$ denotes the depth of $T$ and $\ell(T)$ its number of leaves.
\end{theorem}

The error obtained in Theorem \ref{T:coupling_general} above is fairly small if $H(T)$ is small compared to the depth of the tree. This is the case for critical Galton--Watson trees conditioned to survive a long time. See Corollary \ref{C:coupling} and \eqref{E:Horton_GW} below. On the other hand, if $H(T)$ is large, it does not seem reasonable to expect a very good coupling between a $T$-indexed random walk and a $T$-indexed Brownian motion.

\begin{proof}
Let $T$ be a finite tree. By a ``path'' in the tree $T$, we will mean a monotone path $v_1, \dots, v_m$ such that for all $i=2, \dots, m$, $v_i$ is a child of $v_{i-1}$. We will moreover say that two paths are disjoint if they use disjoint sets of edges.

To obtain a strong coupling between $S_T$ and $B_T$, one can proceed as follows. We partition the tree $T$ into a set of disjoint paths. On each of these paths, we use the standard (i.e., dyadic or KMT) strong coupling between random walk and Brownian motion. We then ``weld'' these successive couplings to get an overall coupling on the whole tree. 
It is however crucial to ensure that on each path from the root to a leaf of the tree, we do not weld too many couplings. Indeed, each such welding potentially increases the error in the overall coupling. In other words, we need to partition $T$ using a special set of paths that form a network of ``highways''. By this, we mean that each path from the root to one of the leaves will use only a few highways.

Finding such a set of highways depends very much on the branching complexity of the tree. This is where the Horton--Strahler number $H(T)$ of $T$ will be useful.
Indeed, an equivalent definition of $H(T)$ goes as follows. Starting with $T_1 = T$, define recursively $T_{i+1}$ as the tree obtained from $T_i$ by erasing each leaf of $T_i$ and each path of nodes with a single child leading to leaves of $T_i$. It can be proved by induction on $H(T)$ that the minimal number of steps needed to erase all vertices of $T$ is equal to $H(T)$.
From this equivalent definition of $H(T)$, it is clear that we can partition $T$ into a set of disjoint paths $\wp_1, \dots, \wp_m$ in such a way that each path from the root to one of the leaves is the concatenation of at most $H(T)$ paths in $\{ \wp_1, \dots, \wp_m\}$.

By the multidimensional strong approximation of Einmahl \cite{einmahl1989extensions}, for each path $\wp_l, (l=1, \ldots, m)$, we can find a probability space $(\Omega_l, \Fc_l, \P_l)$ and a random walk $S_l$ with increment distribution $\theta$ and Brownian motion $B_l$ defined on $(\Omega_l, \Fc_l, \P_l)$, both starting at 0 and indexed by the path $\wp_l \subset T$ (rather than the more conventional interval $[0, \text{len}(\wp_l)]$), such that for all $x>0$,
\begin{equation}
\label{E:coupling1}
\P_l \left( \sup_{t \in \wp_l} |\Gamma^{-1/2} S_l(t) - B_l(t)| \geq c'_1 \log d(T) + x \right) \leq c'_2 e^{-c'_3 x}.
\end{equation}
Note that we simply bound the length of $\wp_l$ by the depth of $T$. This rough bound will be enough for our purposes.
We now consider the product space $\Omega = \prod_{l=1}^m \Omega_l$, $\Fc = \bigotimes_{l=1}^m \Fc_l$ and $\P = \bigotimes_{l=1}^m \P_l$. Let $v \in T$ (recall that $T$ is a continuous tree). Denote by $\varnothing$ the root of this tree, and let $[\varnothing, v]$ denote the unique geodesic path, parameterised by length, between the root $\varnothing$ and the point $v\in T$. We decompose this geodesic path according to the partition $\wp_l, l = 1, \ldots, m$ as follows:
$$
[\varnothing, v] = \cup_{l=1}^m ( \wp_l \cap [\varnothing, v]) = : \cup_{l=1}^m [ t_l, u_l]
$$
where $t_l = t_l(v), u_l = u_l(v)\in T$ and $[t_l, u_l]$ corresponds to the path $\wp_l \cap [\varnothing, v]$ (parameterised by length). (If $\wp_l$ does not intersect $[\varnothing, v]$ then by convention we take $t_l,u_l$ to be the starting point of $\wp_l$.) With this definition we can now specify a coupling between the $T$-indexed walk $(S_T(v))_{v\in T}$ and $T$-indexed Brownian motion $(B_T(v))_{v\in T}$, as follows: namely, for $v\in T$ we set
$$
S_T(v) = \sum_{l=1}^m S_l (u_l) \quad \text{and} \quad
B_T(v) = \sum_{l=1}^m B_l (u_l).
$$
It is easy to check with the Markov property that this indeed defines a valid coupling between the two processes.
%the geodesic path (parameterised by length) corresponding to the intersection of $[\varnothing, v]$ with $\wp_l$. 
\begin{comment}
Denoting by $\varnothing = v_1, \dots, v_j = v$ the branching points of the unique path from the root $\varnothing$ to $v$, we define
\[
S_T(v) = \sum_{i = 1}^j \sum_{l=1}^m S_l( v_i) \indic{v_i \in \wp_l}
\quad \text{and} \quad
B_T(v) = \sum_{i = 1}^j \sum_{l=1}^m B_l( v_i) \indic{v_i \in \wp_l}.
\]
This defines $S_T$ and $B_T$ on the branching points of $T$. We then define $S_T$ (resp. $B_T$) on the edges by linear interpolation (resp. by using the relevant pieces of Brownian trajectories).
\end{comment}
%By construction, $S_T$ and $B_T$ are $T$-indexed random walk and $T$-indexed Brownian motion as desired.
Moreover, by \eqref{E:coupling1}, for any given leaf $v$, the maximal error in this coupling on the geodesic $[\varnothing,v]$ from the root $\varnothing$ to $v$ satisfies 
\[
\max_{u \in [\varnothing,v]}|\Gamma^{-1/2} S_T(u) - B_T(u)| \preceq c'_1 H(T) \log d(T) + c'_4 \sum_{l=1}^{H(T)} E_l
\]
where $E_l, l=1, \dots, H(T)$, are i.i.d. exponential random variables with mean 1, and $\preceq$ stands for stochastic domination. 
By Chernoff's inequality,
\[
\Prob{ \sum_{l=1}^{H(T)} E_l \geq 2H(T) + x } \leq c'_5 e^{-c'_6 x}, \quad x >0.
\]
We now need to take the maximum of this error over all leaves of the tree. This is handled by the following fact, which follows by a simple union bound argument:
the maximum of $N$ (not necessarily independent) exponential random variables with mean 1 is stochastically dominated by $\log N + X$ where the tails of $X$ decay exponentially fast, uniformly in $N$ and the joint law of the relevant exponential random variables.
\end{proof}

We now rephrase Theorem \ref{T:coupling_general} in the context of Galton--Watson trees. Let $\xi$ be a nondegenerate critical offspring distribution: $\sum_{n \geq 1} \xi(n) n = 1$ and $\xi(1) < 1$. Let $\Tf_n$ be a (discrete) critical Galton-Watson tree with offspring distribution $\xi$, conditioned to have size $n$. Depending on the value of $n$, this conditioning might be degenerate. We will say that $n$ is \textbf{$\xi$-admissible} if the event we are conditioning on has a positive probability. As before, we will view $\Tf_n$ as a continuous tree by assigning length 1 to each edge.

Given $\Tf_n$, and given a probability distribution $\theta$ on $\R^d$ with vanishing mean, invertible covariance matrix $\Gamma$ and a finite exponential moment,
let $S_{\Tf_n}$ be a $\Tf_n$-indexed random walk with increment distribution $\theta$, and let $B_{\Tf_n}$ be a $\Tf_n$-indexed Brownian motion.
In other words, $S_{\Tf_n}$ is simply a branching random walk conditioned on the size of its total progeny being $n$. We stress however that in the rest of the article we used a different notion of branching Brownian motion, in which the branching times were not deterministically equal to 1, but were instead random and exponentially distributed with mean 1. We explain below the proof of Corollary \ref{C:coupling} how we remedy this issue.

\begin{corollary}\label{C:coupling}
There exist constants $c_1, c_2, c_3$ such that the following holds. For all $\xi$-admissible integer $n \geq 1$, there exists a coupling between $S_{\Tf_n}$ and $B_{\Tf_n}$ such that
\[
\Prob{ \sup_{t \in \Tf_n} |\Gamma^{-1/2} S_{\Tf_n}(t) - B_{\Tf_n}(t)| \geq c_1 \log(n)^2 + x } \leq c_2 e^{-c_3 x}.
\]
\end{corollary}

\begin{proof}
This follows directly from Theorem \ref{T:coupling_general} and from the deterministic bounds $H(\Tf_n) \leq \log_2 n$, $\log d(\Tf_n) \leq \log n$ and $\log \ell(\Tf_n) \leq \log n$. Note that these bounds are not very wasteful: under the extra assumption that the offspring distribution has finite variance,
\begin{equation}
\label{E:Horton_GW}
\frac{H(\Tf_n)}{\log_2 n} \xrightarrow[n \to \infty]{\P} \frac12
\quad \text{and} \quad
\frac{\log d(\Tf_n)}{\log n} \xrightarrow[n \to \infty]{\P} \frac12.
\end{equation}
The first convergence is \cite[Theorem 1.1]{10.1214/21-EJP678}, while the second is a direct consequence of Aldous' convergence of the depth first walk rescaled by $n^{1/2}$ to the Brownian excursion \cite{aldous1993continuum}; see also \cite[Theorem 1]{Marckert}.
\end{proof}

When we will want to compare a branching Brownian motion (with exponentially distributed branching times) to a branching random walk, we will proceed as follows. We will first restrict the branching Brownian motion to the branching times. This leads to a branching random walk with increments distributed as $\sqrt{E} G$ where $E$ is an exponential random variable with parameter 1 and $G$ is an independent standard Gaussian random variable in $\R^4$. We will then use Corollary \ref{C:coupling} to compare it with branching Brownian motion with deterministic branching times. Finally, we will use Corollary \ref{C:coupling} once more to compare the latter to our initial target branching random walk.

\subsection{Thick points of branching random walk}

We now have all the ingredients to prove Theorem \ref{T:thick_BRW}.

\begin{proof}[Proof of Theorem \ref{T:thick_BRW}]
We start by proving the lower bound.
Let $\theta$ be the distribution on $\Z^d$ such that for any neighbour $x$ of the origin, $\theta(\{x\}) = 1/(2d)$.
Let $\theta_1$ and $\theta_2$ be the law respectively of $W_E$ and $W_1$ where $W$ is a 4d Brownian motion and $E$ is an independent exponential random variable with mean 1.
Let $\Tf$ be a Galton--Watson tree with offspring distribution $\frac12 \delta_0 + \frac12 \delta_2$ conditioned to survive at least $R^2$ generations. We will consider three $\Tf$-indexed random walks $S_\Tf^1$, $S_\Tf^2$ and $S_\Tf$ with respective increment distributions $\theta_1$, $\theta_2$ and $\theta$. To give us some room, we will kill the particles of $S_\Tf^1$,  $S_\Tf^2$, and  $S_\Tf$ when they reach $B(R/4)^c$, $B(R/2)^c$, and $B(R)^c$ respectively.
We will denote their respective discrete local times on $B(x,r)$ by $L^1_{x,r}$, $L_{x,r}^{2}$, and $L_{x,r}$.
Without loss of generality, assume that each branching random walk starts at the origin.

The strategy of the proof is as follows. We will use Theorem \ref{T:thick} and a standard concentration result to show a thick point result for $S_\Tf^1$ in terms of local times of \emph{mesoscopic} balls. These local times of mesoscopic balls have the advantage of being very stable: if one perturbs slightly the spatial motions of a branching process, then one only needs to increase slightly the radii to ensure that the local time of balls have not decreased. Note that such a property might not hold in general for the number of pioneers. With the help of the strong coupling previously established and this stability of local time of balls, we will then successively transfer this thick point result to $S_\Tf^2$ and then to $S_\Tf$.

\medskip

Let $a >0$ be a thickness level. Let $\eps>0$ be small and $\eta >0$ be much smaller than $\eps$. Let $r = R^\eta$ be a small mesoscopic scale.
Let us start by considering a critical Branching Brownian motion $B_\Tf$ with exponentially distributed gestation time.
Let us define the discrete local time of the ball $B(x,r)$ as the sum over each branching point $v \in \Tf$ (including the leaves) of the indicator function that just before the branching at $v$ the particle belongs to $B(x,r)$. This process of discrete local times has the same law as $(L_{x,r}^{1}, x \in \R^d)$.
We first claim that
\begin{equation}
\label{E:discrete1}
\P_{z_0,R/4}^{\brw,\theta_1}\Bigl(
\# \bigl\{ x \in 10 r \Z^4 : L^1_{x,r} \geq \frac{am_1}{2} r^4 (\log R/r)^2 \bigr\} \geq (R/r)^{4-a-2\eps} \,\Big\vert\, \zeta \geq R^2 \Bigr) \to 1
\end{equation}
as $R \to \infty$.
Indeed, we can write for all $x \in \Z^4$ and $r >0$,
\[
L^1_{x,r} = \sum_{i = 1}^{N_{x,r}} L_i
\]
where $N_{x,r}$ is as before the number of pioneers on $\partial B(x,r)$ and $L_i$ denotes the local time of $B(x,r)$ induced by each pioneer. By a slight variant of Theorem \ref{T:thick} (with spheres of radius $r$ instead of $1$),
\[
\P_{z_0,R/4}\Bigl(
\# \bigl\{ x \in 10 r \Z^4 : N_{x,r} \geq \frac{a+\eps}{2} r^2 (\log R/r)^2 \bigr\} \geq (R/r)^{4-a-2\eps} \,\Big\vert\, \zeta \geq R^2 \Bigr) \to 1,
\]
%\red{[TH: How annoying would it be to change the proof of Theorem 1.5 to give this more general statement?]}
as $R\to \infty$. Furthermore, and recalling the definition \eqref{E:m1} of $m_1$, for a fixed centre $x$, the local times $L_i$, $i=1, \dots, N_{x,r}$, are i.i.d. with average $(1+o(1))m_1 r^2$. Because the $L_i$'s are nonnegative, it is very costly for the sum to be much smaller than its average. Specifically, by standard large deviations for nonnegative random variables (sometimes called Benett's inequality, see, e.g., \cite[Theorem 2.9]{boucheron2013concentration}), for any fixed $x$, the probability
\begin{align*}
\P_{z_0,R/4} \Bigl( N_{x,r} \geq \frac{a+\eps}{2} r^2 (\log R/r)^2 \text{ and } L^1_{x,r} \leq \frac{am_1}{2} r^4 (\log R/r)^2  \,\Big\vert\, \zeta \geq R^2\Bigr)
\end{align*}
decays faster than any polynomial in $R$.
This proves \eqref{E:discrete1}.

We will now use Corollary \ref{C:coupling} to transfer this result to $S_\Tf^2$. Conditionally on surviving at least $R^2$ generations, $\Tf$ (as a discrete tree) has size $O(R^4)$, or equivalently contains $O(R^4)$ branch points including leaves. We therefore define $E_\Tf$ to be the event that $\Tf$ contains at most $R^5$ such branch points. By Corollary \ref{C:coupling}, on this event, we can couple $S_\Tf^1$ and $S_\Tf^2$ in such a way that $\max_{v \in \Tf} |S_\Tf^1(v) - S_\Tf^2(v)| \leq C (\log R)^2$ with large probability. Recall that particles of $S_\Tf^1$ are killed on $\partial B(R/4)$ and those of $S_\Tf^2$ are killed on $\partial B(R/2)$. By slightly enlarging the balls $B(x,r)$ into balls $B(x,r+C(\log R)^2)$, we deduce that
\[
\P_{z_0,R/2}^{\brw,\theta_2}\Bigl( \# \bigl\{ x \in 10 r \Z^4 : L^2_{x,r+C(\log R)^2} \geq \frac{am_1}{2} r^4 (\log R/r)^2 \bigr\} \geq (R/r)^{4-a-2\eps} \,\Big\vert\, \zeta \geq R^2 \Bigr) \to 1.
\]
The same procedure can now be done between $S_\Tf$ and $S_\Tf^2$. The only difference comes from the fact that the covariance matrix $\Gamma$ of $\theta$ is equal to $\frac14 I$ instead of $I$, so that the balls $B(x,r+C(\log R)^2)$ should be replaced by %the ellipsoids
\[
C_{x,r} := \{ z \in \Z^4 : |\Gamma^{-1/2} z - x| < r + C(\log R)^2 \}.
\]
Denoting $L(C_{x,R})$ the local time accumulated by $S_\Tf$ in $C_{x,r}$, we obtain that
\begin{equation}
\label{E:local_ellipse}
\P_{z_0,R}^{\brw,\theta}\Bigl( \# \bigr\{ x \in 10 r \Z^4 : L(C_{x,r}) \geq \frac{am_1}{2} r^4 (\log R/r)^2 \bigr\} \geq (R/r)^{4-a-\eps} \,\Big\vert\, \zeta \geq R^2 \Bigr) \to 1.
\end{equation}
Since the number of points of $C_{x,r}$ is at most $(1-\eps)^{-1} \frac{\pi^2}{2} \sqrt{\det \Gamma} r^4$ (where $\pi^2/2$ is the volume of the unit ball in $\R^4$),
if $L(C_{x,r}) \geq \frac{am_1}{2} r^4 (\log R/r)^2$ then there must be a point $z \in C_{x,r}$ such that $\ell_z \geq (1-\eps) \frac{am_1}{\pi^2 \sqrt{\det \Gamma}} (\log R/r)^2$.
In other words, we have shown that
\[
\P_{z_0,R}^{\brw,\theta}\Bigl( \# \bigl\{ z \in \Z^4 : \ell_z \geq (1-\eps) \frac{am_1}{\pi^2 \sqrt{\det \Gamma}} (\log R/r)^2 \bigr\} \geq (R/r)^{4-a-\eps} \,\Big\vert\, \zeta \geq R^2 \Bigr) \to 1.
\]
Recall that $\sqrt{\det(\Gamma)} = 1/16$.
Since $r = R^\eta$ and $\eta$ and $\eps$ can be arbitrary small, this concludes the proof.

\medskip

We now move to the proof of the upper bound. Let $a \in (0,a_0)$, where $a_0>0$ is from Theorem \ref{T:local_time_tail}. As before, let $\eps >0$, $\eta>0$ be much smaller than $\eps$ and $r = R^\eta$.
Let
\[
\widetilde{C}_{x,r} := \{ z \in \Z^4 : |\Gamma^{-1/2} z - x| < r \} = \{ z \in \Z^4 : |2 z - x| < r \}.
\]
Using the exact same strategy as above, we can show that
\begin{equation}
\label{E:local_ellipse2}
\P_{z_0,R}^{\brw,\theta} \Bigl( \# \bigl\{ x \in \Z^4 : L(\widetilde{C}_{x,r}) \geq \frac{am_1}{2} r^4 (\log R/r)^2 \bigr\} \leq (R/r)^{4-a+\eps} \,\Big\vert\, \zeta \geq R^2\Bigr) \to 1,
\end{equation}
where $L(\widetilde{C}_{x,R})$ stands for the local time accumulated by $S_\Tf$ in $\widetilde{C}_{x,r}$.
In the proof of the lower bound, we then transferred the analogous estimate \eqref{E:local_ellipse} to an estimate about local time of points simply by using the fact that for all $x$, there must exist a point $y \in C_{x,r}$ such that $\ell_y \geq L(C_{x,r}) / \# C_{x,r}$.
This argument does not work directly for the upper bound. We now explain how to circumvent this issue.
As in the branching Brownian motion case, let us denote by $N_{x,r}$ the number of pioneers of $\widetilde{C}_{x,r}$ of our branching random walk. Because of the lack of rotational invariance, we will need to consider the sigma algebra $\Fc_{x,r}$ generated by the number of pioneers together with their locations.
From \eqref{E:local_ellipse2}, one can deduce an upper bound on $N_{x,r}$ of the type
\begin{equation}
\label{E:local_ellipse3}
\P_{z_0,R}^{\brw,\theta} \Bigl( \# \bigl\{ x \in \Z^4 : N_{x,r} \geq \frac{a+\eps}{2} r^2 (\log R/r)^2 \bigr\} \leq (R/r)^{4-a+\eps} \,\Big\vert\, \zeta \geq R^2\Bigr) \to 1.
\end{equation}
Indeed, for each $x \in \Z^4$, we can decompose
\[
L(\widetilde{C}_{x,r}) = \sum_{i=1}^{N_{x,r}} L_i,
\]
where $L_i$ is the local time of $\widetilde{C}_{x,r}$ accumulated by the $i$-th pioneer on $\widetilde{C}_{x,r}$. Conditionally on $\Fc_{x,r}$, the $L_i$'s are independent, nonnegative with mean equal to $(1+o(1)) m_1 r^2$. By Benett's inequality, we deduce that for all $x \in \Z^4$,
\[
\P_{z_0,R}^{\brw,\theta} \Bigl(L(\widetilde{C}_{x,r}) \leq \frac{am_1}{2} r^4 (\log R/r)^2 \,\Big\vert\, N_{x,r} \geq \frac{a+\eps}{2} r^2 (\log R/r)^2 \text{ and } \zeta \geq R^2\Bigr)
\]
decays faster than any polynomial in $R$. Together with \eqref{E:local_ellipse2}, this leads to \eqref{E:local_ellipse3}. As already alluded to, transferring directly an upper bound on the number of pioneers from the continuous to the discrete without relying on local time of balls might not be an easy task. This is why we took this convoluted approach to get \eqref{E:local_ellipse3}.

We can now transfer this upper bound on the pioneers to an upper bound on local times at points. Indeed, for all $x \in \Z^4$, we can write
\[
\ell_x = \sum_{i=1}^{N_{x,r}} \ell_i
\]
where $\ell_i$ is the local time at $x$ generated by the $i$-th pioneer on $B(x,r)$. In the proof of Theorem \ref{T:local_time_tail}, we transferred the estimate we had on the number of pioneers to estimates about local time. We can proceed in the exact same way here. Note in particular that the discrete analogue of \eqref{E:local_time_laplace} is known \cite{angel2021tail, asselah2022time}.
\end{proof}

\section{Hitting probabilities in dimension \texorpdfstring{$d \neq 4$}{d neq 4} (proof of Theorem \ref{T:hitting_neq4})}\label{S:123}

In this section, we assume that $d \neq 4$ and we will prove the expansion of the hitting probability stated in Theorem \ref{T:hitting_neq4}. We recall the definitions  of $\beta$ and $(\alpha_\ell)_{\ell \geq 0}$ from \eqref{E:beta} and \eqref{E:alpha}:
\begin{equation}
\label{E:beta2}
    \beta = \beta(d)  =
    \begin{cases}
        \displaystyle \frac{d-6 + \sqrt{d^2 - 20 d + 68}}{2}  & \text{ if } d \le 3,\\
          d-4 & \text{ if } d \ge 5,
         \end{cases}
    % ~=~
    % \left\{ \begin{array}{ll}
    % 1     &  \text{if~} d=1,\\
    % 2(\sqrt{2} -1) \approx 0.83     & \text{if~} d=2,\\
    % (\sqrt{17} - 3)/2 \approx 0.56 & \text{if~} d=3,\\
    %  d-4 & \text{if~} d\ge 5.
    % \end{array} \right.
\end{equation}
and
\begin{equation}
\label{E:alpha2}
\alpha_0 = (8-2d)_+, \quad
\alpha_1 = 1 \quad \text{and} \quad \alpha_\ell = \frac{1}{(\beta^2\ell + |2d-8|)(\ell-1)} \sum_{k=1}^{\ell - 1} \alpha_k \alpha_{\ell-k}, \quad \ell \geq 2
\end{equation}
where for $x \in \R$, $x_+ = \max(x,0)$. We want to prove that (when $d\neq 4$) there exists a constant $\mu_1=\mu_1(d)\in \R$ such that
\[
\mathbb{P}_{r,\infty}(N_1>0) = \sum_{\ell=0}^\infty \alpha_\ell \mu_1^\ell r^{-2-\beta\ell}
\]
for every $r$ such that this sequence converges absolutely. We also want to prove that the constant $\mu_1$ is negative when $d\in \{1,2,3\}$, is strictly between $0$ and $1$ when $d\geq 5$, and is the unique solution to $\sum_{\ell=1}^\infty \alpha_\ell \mu^\ell = 1$ when $d\geq 5$.

% \paragraph{High dimensions}
\subsection{High dimensions}
 We start with the easier case in which $d\geq 5$.
% In this paragraph, we assume that $d \geq 5$. We will prove Theorem \ref{T:hitting_neq4} in this case.
For each $\ell \geq 1$, $\alpha_\ell \in (0,1]$, so the function $\mu \in [0,1] \mapsto \sum_{\ell \geq 1} \alpha_\ell \mu^\ell$ is continuous, increasing, vanishes at 0 and is larger than $\alpha_1 = 1$ at $\mu =1$. Therefore, for all $s \in [0,1]$, there exists a unique $\mu_s \in [0,1)$ such that
\begin{equation}
\label{E:mu5}
\sum_{\ell \geq 1} \alpha_\ell \mu_s^\ell = s.
\end{equation}
The high-dimensional case of Theorem \ref{T:hitting_neq4} follows by taking $s=1$ in the following lemma:
% will be a special case of the following lemma when $s=1$:

\begin{lemma}\label{L:dimension5}
Let $d \geq 5$. For each $s \in [0,1]$ and $r \geq 1$,
\begin{equation}
\label{eq:v_explicit}
1 - \EXPECT{r,\infty}{(1-s)^{N_1}} = \sum_{\ell \geq 1} \alpha_\ell \mu_s^\ell r^{-(d-4)\ell -2}.
\end{equation}
\end{lemma}

\begin{proof}
Let $w(r)$ be the right hand side of \eqref{eq:v_explicit}. Using the recurrence equation \eqref{E:alpha2}, one can easily verify that 
$
w'' + \frac{d-1}{r} w' = w^2
$ for every $r\geq 1$.
Moreover, we also have that  $w(r) \to 0$ as $r \to \infty$ and that $w(1) = s$ by definition of $\mu_s$. By \eqref{E:L_ODEv}, the left hand side of \eqref{eq:v_explicit} also satisfies this boundary value problem, and we conclude by uniqueness of nonnegative solutions to this problem as established in Lemma~\ref{L:appendix_uniqueness}.
\end{proof}

\begin{remark}
Since both sides of \eqref{eq:v_explicit} are analytic in $s$, the identity \eqref{eq:v_explicit} extends to at least some negative values of $s$. This could prove useful in the study of large deviations of ${N_1}$ when $d \geq 5$;  see the proof of Lemma \ref{L:invariance} for arguments in this direction.
\end{remark}

Lemma \ref{L:dimension5} has the following immediate corollary, which is related to the existence of an ``excursion from infinity'' (or ``branching'' interlacements \cite{sznitman}) measure for high-dimensional BBM. 

\begin{corollary}
The law of ${N_1}$ under $\P_r(\cdot \vert {N_1}>0)$ converges weakly as $r \to \infty$ to some law $\P_\infty$ satisfying 
\begin{equation}
\label{eq:cor2}
\EXPECT{\infty}{(1-s)^{N_1}} = 1 - \frac{\mu_s}{\mu_1}
\end{equation}
for all $s \in (0,1)$.
\end{corollary}

\begin{comment}
We emphasise that \cite{LeGall16} and \cite{Zhu17} study analogous hitting probabilities in the context of branching random walk and obtain estimates with a similar flavour as \eqref{eq:cor1}. In this branching random walk setting, \cite{Zhu17} expresses the constant $\mu_1$ in terms of the branching capacity of the relevant set. Using our strong coupling result (Corollary \ref{C:coupling}) between branching Brownian motion and branching random walk, it should be possible to obtain the discrete result from the continuous one and vice versa. In particular, it is likely that the constant $\mu_1$ appearing in \eqref{eq:cor1} can be expressed in terms of the branching capacity of \cite{Zhu17}.
See also \cite{dawson1989super} for related questions in the setting of super-Brownian motion.
\end{comment}

\subsection{Low dimensions}

Let us now consider the low-dimensional case $d \in \{1,2,3\}$. The main result of this section is the following proposition, the $s=1$ case of which implies
% By Lemma \ref{L:v_ODE} and by considering the special case $s=1$, 
Theorem \ref{T:hitting_neq4} by Lemma \ref{L:v_ODE}.
% will directly follow Proposition \ref{P:123}.

\begin{proposition}\label{P:123}
Let $d \in \{1,2,3\}$. There exists a continuous increasing function $s \in (0,8-2d] \mapsto \mu_s \in (-\infty,0]$ such that the following holds. For each $s \in (0,8-2d]$, there exists a unique nonnegative solution $v_s$ to the boundary value problem
\begin{equation}
\label{E:P_123}
\left\{ \begin{array}{l}
\Delta v_s = v_s^2 \quad \text{in} \quad \R^d \setminus \overline{B(1)}, \\
v_s = s \quad \text{on} \quad \partial B(1), \\
v_s(y) \to 0 \quad \text{as} \quad y \to \infty.
\end{array}\right.
\end{equation}
Moreover,  the function $v_s$ is equal to
\[
v_s(y) = \sum_{\ell=0}^\infty \alpha_\ell \mu_s^\ell \norme{y}^{-2-\beta \ell}
\]
for each $y \in \R^d$ with $\norme{y} > (|\mu_s|/R_\alpha^{1/\beta})$, where $R_\alpha>0$ is the radius of convergence of $x\mapsto\sum_{\ell=0}^\infty \alpha_\ell x^\ell$.
\end{proposition}

Note that this expansion holds for large $\norme{y}$, in contrast to the high-dimensional case. More importantly, the existence of the function $s \in (0,8-2d] \mapsto \mu_s \in (-\infty, 0]$ is nontrivial. Heuristically (letting aside the possible divergence issues), one needs to know the possible values that
\[
y \mapsto \sum_{\ell \geq 0} \alpha_\ell \mu^\ell \norme{y}^{-2-\beta\ell}
\]
can take when $y \in \partial B(1)$, by letting $\mu$ vary. If we are targeting smaller values than $\alpha_0 = 8-2d$, we need to take $\mu < 0$ (possibly very negative).
This has a different flavour than the high-dimensional case and requires extra work that we now explain.

\medskip

We are concerned with rotationally invariant solutions to $\Delta v = v^2$. As before, expressing the Laplacian in spherical coordinates leads to the ODE $v''(r) + (d-1) v'(r) / r = v(r)^2$. If one makes the change of variable $v(r) = r^{-2} f(r^{-\beta})$, one obtains the following equation for $f$:
\begin{equation}
\label{E:june4}
\beta^2 x^2 f''(x) + \beta (\beta+6-d) xf'(x) + (8-2d) f(x) = f(x)^2.
\end{equation}
The exponent
$\beta$ was chosen exactly so that the multiplicative constants in front of $xf'(x)$ and $f(x)$ agree. To see why this is natural, we can first neglect $x^2 f''(x)$ and look at the solutions to $c_1 x f'(x) + c_2 f(x) = f(x)^2$, which are of the form 
$x \mapsto C c_2 / (C + x^{c_2/c_1})$ for some constant $C$. These solutions are therefore analytic at $x=0$ only when $c_2/c_1$ is an integer. We chose to fix $c_2/c_1 = 1$.

\begin{remark}
    This qualitative behaviour is quite robust: if the branching mechanism were different but still critical, this would change $f(x)^2$ into a more general function of $f(x)$, but would not change the local behaviour of solutions to \eqref{E:june4} near $x=0$. This means that for other critical branching mechanism, we would still want $\beta$ to satisfy $\beta (\beta+6-d) = 8-2d$, indicating that $\beta$ is indeed a universal exponent.
\end{remark}

To proceed, we will fix two positive real numbers $p$ and $q$ and study the solutions to $p x^2 f''(x) + q x f'(x) + q f(x) = f(x)^2$ on $[0,\infty)$. We will later specify our result to $p = \beta^2$ and $q = 8-2d$. See Figure \ref{fig123} for numerical approximations of solutions to this equation.
Let $\alpha_0(p,q) = q$, $\alpha_1(p,q) =1$ and for all $\ell \geq 2$,
\[
\alpha_\ell(p,q) = \frac{1}{(p \ell + q) (\ell-1)} \sum_{k=1}^{\ell-1} \alpha_k(p,q) \alpha_{\ell-k}(p,q).
\]
Let $R(p,q)$ be the radius of convergence of the power series $x \mapsto \sum_{\ell=0}^\infty \alpha_\ell(p,q) x^\ell$. It can be shown by induction that for all $\ell \geq 1$, $\alpha_\ell(p,q) \leq (2p+q)^{1-\ell}$, so $R(p,q) \geq (2p+q)$.

\medskip

We now state and prove a lemma concerning solutions to $p x^2 f''(x) + q x f'(x) + q f(x) = f(x)^2$, which we will use to prove Proposition \ref{P:123}.

\begin{lemma}\label{L:d123}
    Let $p, q >0$ be such that $q/p \geq 3 + 2 \sqrt{2}$. For each $\mu \leq 0$, there exists a unique nonnegative function $f_\mu$ in $C^\infty([0,\infty))$ solving the initial value problem
    \[
    \left\{ \begin{array}{l}
    p x^2 f''(x) + q x f'(x) + q f(x) = f(x)^2, \quad x \geq 0,\\
    f(0) = q, \quad f'(0) = \mu.
    \end{array} \right.
    \]
    The function
    $f_0$ is the constant function $x \mapsto q$, while for each $\mu<0$, the function $f_\mu$ is positive, decreasing, convex, and equal to $x \mapsto \sum_{\ell=0}^\infty \alpha_\ell(p,q) \mu^\ell x^\ell$ on $[0,R_\alpha(p,q)/|\mu|)$. This family of functions also has the following properties:
    \begin{enumerate}
        \item $(x,\mu) \in [0,\infty) \times (-\infty,0] \mapsto f_\mu(x)$ is analytic;
        \item For all $\mu_1 < \mu_2 \leq 0$, $f_{\mu_1} < f_{\mu_2}$ on $(0,\infty)$;
        \item For all $x>0$, $f_\mu(x) \to 0$ as $\mu \to - \infty$.
    \end{enumerate}
\end{lemma}

\begin{figure}
   \centering
   \begin{subfigure}{.45\columnwidth}
    \def\svgwidth{\columnwidth}
   %% Creator: Inkscape 1.1 (c4e8f9e, 2021-05-24), www.inkscape.org
%% Accompanies image file 'plot123.pdf' (pdf, eps, ps)
\begingroup%
  \makeatletter%
  \providecommand\color[2][]{%
    \errmessage{(Inkscape) Color is used for the text in Inkscape, but the package 'color.sty' is not loaded}%
    \renewcommand\color[2][]{}%
  }%
  \providecommand\transparent[1]{%
    \errmessage{(Inkscape) Transparency is used (non-zero) for the text in Inkscape, but the package 'transparent.sty' is not loaded}%
    \renewcommand\transparent[1]{}%
  }%
  \providecommand\rotatebox[2]{#2}%
  \newcommand*\fsize{\dimexpr\f@size pt\relax}%
  \newcommand*\lineheight[1]{\fontsize{\fsize}{#1\fsize}\selectfont}%
  \ifx\svgwidth\undefined%
    \setlength{\unitlength}{653.29775075bp}%
    \ifx\svgscale\undefined%
      \relax%
    \else%
      \setlength{\unitlength}{\unitlength * \real{\svgscale}}%
    \fi%
  \else%
    \setlength{\unitlength}{\svgwidth}%
  \fi%
  \global\let\svgwidth\undefined%
  \global\let\svgscale\undefined%
  \makeatother%
  \begin{picture}(1,0.59988833)%
    \lineheight{1}%
    \setlength\tabcolsep{0pt}%
    \put(0,0){\includegraphics[width=\unitlength,page=1]{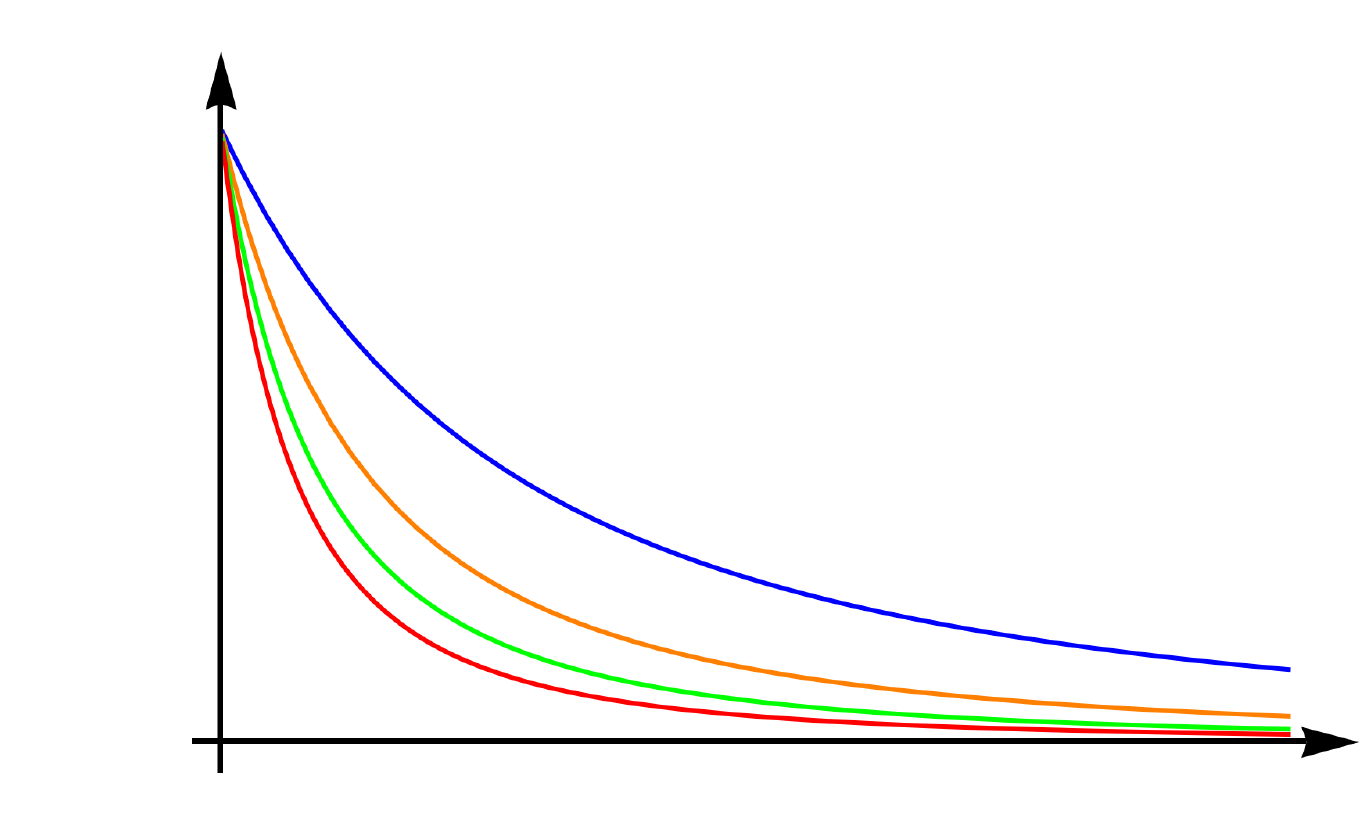}}%
    \put(0.11535853,0.58104167){\color[rgb]{0,0,0}\makebox(0,0)[lt]{\lineheight{1.25}\smash{\begin{tabular}[t]{l}$f_\mu(x)$\\\end{tabular}}}}%
    \put(0.92674569,0.00431703){\color[rgb]{0,0,0}\makebox(0,0)[lt]{\lineheight{1.25}\smash{\begin{tabular}[t]{l}$x$\end{tabular}}}}%
    \put(0,0){\includegraphics[width=\unitlength,page=2]{plot123.pdf}}%
    \put(-0.00185357,0.49168659){\color[rgb]{0,0,0}\makebox(0,0)[lt]{\lineheight{1.25}\smash{\begin{tabular}[t]{l}$8-2d$\\\end{tabular}}}}%
  \end{picture}%
\endgroup%

   \end{subfigure}
\caption{Plot of numerical approximations of $f_\mu(x), x \in [0,15]$, for different values of $\mu$, when $p = \beta^2$, $q = 8-2d$ and $d=2$. From top to bottom, $\mu = -1, -2, -3$ and $-4$.}\label{fig123}
\end{figure}

Before we start the proof of this lemma we make a few comments. As explained above, our analysis is predicated on the fact that the term containing $x^2 f''(x)$ is negligible compared to the other terms in \eqref{E:june4}. It is thus reasonable to make an assumption that $p/q$ is sufficiently small;
% (or $q /p$ sufficiently large).% in order to guarantee uniqueness.
the threshold $q/p = 3+ 2 \sqrt{2}$ is sufficient for the argument in the proof below.
Moreover, numerical simulations suggest that if the initial slope $\mu$ is a sufficiently large negative number and $q/p<3+2\sqrt{2}$, then the associated solution does not stay nonnegative. As such, the lemma is likely to be false when this condition is violated.

\medskip

In the concrete cases where we will apply this lemma, we will have $p = \beta^2$ and $q = 8-2d$, where $\beta$ is as in \eqref{E:beta2}. Interestingly, as $d$ varies in the interval $d \in [0,4)$, the minimum value of $q/p$ is attained at $d=2$ and is precisely equal to $3+2 \sqrt{2}$. As such, the condition $q/p \geq 3+2\sqrt{2}$ holds for all $d=1,2,3$, but $d=2$ seems to be a borderline case for uniqueness. We do not know the reason why $d=2$ seems to be critical from that point of view, or what interpretations this should have for the behaviour of $2d$ BBM.
%will be in the regime $q/p \geq 3 + 2 \sqrt{2}$. Interestingly, there is in fact equality when $d=2$.

\begin{proof}[Proof of Lemma \ref{L:d123}.]
Let $\mu < 0$ and let \[\Tilde{f} : x \in [0,R_\alpha(p,q)/|\mu|) \mapsto \sum_{\ell=0}^\infty \alpha_\ell(p,q) \mu^\ell x^\ell.\] Using the recursive definition of $(\alpha_n(p,q))_{n \geq 0}$, one can check that $\tilde{f}$ satisfies \[\tilde{f}(0) = q,\quad \tilde{f}'(0) = \mu \qquad \text{and}\qquad p x^2 \tilde{f}''(x) + q x \tilde{f}'(x) + q \tilde{f}(x) = \tilde{f}(x)^2 \quad \text{ for all $x \in [0,R_\alpha(p,q)/|\mu|)$}.\] 
Let $x_0$ be any positive real number in $(0,R_\alpha(p,q)/|\mu|)$. Since we are away from the singularity at $x=0$, we can safely consider the unique maximal solution $\hat{f}$ to $p x^2 \hat{f}''(x) + q x \hat{f}'(x) + q \hat{f}(x) = \hat{f}(x)^2$ in $[x_0,x_*)$ with $\hat{f}(x_0) = \tilde{f}(x_0)$ and $\hat{f}'(x_0) = \tilde{f}'(x_0)$. Finally, let $f_\mu: [0,x_*) \to \R$ be the function defined by: $f_\mu(x) = \tilde{f}(x)$ if $x \in [0,x_0]$ and $f_\mu(x) = \hat{f}(x)$ if $x \in [x_0,x_*)$.
Notice that this definition does not depend on the choice of $x_0$ by the uniqueness of forward solutions when $x_0>0$.

\medskip

We are going to show that $f_\mu$ is defined on $[0,\infty)$ (i.e. that $x_* = +\infty$) and that it satisfies all the desired properties. In the rest of the proof we will denote by $\gamma = q/p$ and we will simply write $f$ instead of $f_\mu$ when there is no ambiguity.

\medskip

% $\bullet~$ Proof that $f > 0$:
\noindent \textbf{$f$ is positive:}
Let $b_1$ and $b_2$ be the roots of the polynomial $\gamma + (1-\gamma)b + b^2$:
\begin{equation}
\label{E:roots}
b_1 = \frac{\gamma-1-\sqrt{1-6\gamma+\gamma^2}}{2} \quad \text{and} \quad b_2 = \frac{\gamma-1+\sqrt{1-6\gamma+\gamma^2}}{2}.
\end{equation}
The assumption that the ratio $\gamma = q/p$ is at least $3+2\sqrt{2}$ is used to ensure that $b_1$ and $b_2$ are real and positive. Since $b_1$ and $b_2$ satisfy $b_1 + b_2 + 1 = \gamma$ and $b_1 b_2 = \gamma$, we get that
\begin{multline*}
 \frac{\d (x^{b_1+1} f'(x))}{\d x} + b_2 \frac{\d (x^{b_1} f(x))}{\d x} = x^{b_1+1} f''(x) + (b_1+b_2+1) x^{b_1} f'(x) + b_1 b_2 x^{b_1 -1} f(x) \\
 = \frac{1}{p} x^{b_1-1} (p x^2 f''(x) + q xf'(x) + q f(x) ) = \frac{1}{p} x^{b_1-1} f(x)^2 \geq 0.
\end{multline*}
Integrating this relation between $0$ and $x$ leads to
$x^{b_1+1}f'(x) + b_2 x^{b_1} f(x) \geq 0.$
Since $x^{b_2}f(x)$ has zero derivative at $x=0$, this implies by Lemma~\ref{L:Gronwall} that the derivative of $x \mapsto x^{b_2} f(x)$ is nonnegative. Since $f >0$ in a neighbourhood of the origin, this shows that $f$ stays positive as desired.

\medskip

\noindent \textbf{$f$ is strictly decreasing:}
Since $f'(0)<0$ there exists $x_1 >0$ such that $f'(x) < 0$ for all $x \in [0,x_1]$. Let $x_2 := \inf \{ x \geq x_1 : f(x) \geq q \}$. Using the nonnegativity of $f$, we see that on $[x_1,x_2)$, $f^2 - q f \leq 0$. This implies that $p x^2 f''(x) + q x f'(x) \leq 0$ on $[x_1,x_2)$. Integrating this relation (Lemma \ref{L:Gronwall}) shows that for all $x \in [x_1, x_2)$,
$x^{1+\gamma} f'(x) \leq x_1^{1+\gamma} f'(x_1) < 0$. The derivative $f'$ is therefore negative on $[x_1,x_2)$ and $x_2$ must be equal to $x_*$ (where we recall that, by definition, $f$ is defined on $[0,x_*)$). This shows that $f$ is strictly decreasing on its entire domain and also shows that $f < q$ on its entire domain.

\medskip

\noindent \textbf{$f$ is convex:}
Using the ODE satisfied by $f$ and then the fact, established in the previous paragraph, that $f'(f-q) \geq 0$, we have
\[
p \frac{\d (x^2 f''(x))}{\d x} = 2f'(x)(f(x)-q) - q x f''(x) \geq -qx f''(x).
\]
This implies that the derivative of $x^{\gamma+2} f''(x)$ is nonnegative and hence that $f''(x) \geq 0$.

\medskip

\noindent \textbf{$f$ is defined and analytic on $[0,\infty)$:} $f$ is defined on $[0,\infty)$ since $0 \leq f \leq q$ and $f'(0) \leq f' \leq 0$ (which are consequences of the previous properties); here we are using that maximal solutions to locally Lipschitz second order ODEs must either have $(x,f,f')$ unbounded or accumulating to a point where the ODE is not defined as $x$ approaches an endpoint of the domain. The fact that $(x,\mu) \mapsto f_\mu(x)$ is analytic follows quickly from the definition of $f_\mu$ and from Cauchy--Kovalevskaya theorem.

\medskip

\noindent \textbf{$f_\mu$ is a strictly increasing function of $\mu\in (-\infty,0]$:}
 Let $g = f_{\mu_2} - f_{\mu_1}$. It follows immediately from the series expansion of $f_{\mu_1}$ and $f_{\mu_2}$ around $0$ that there exists $x_1 >0$ such that for all $x \in (0,x_1]$, $g(x) >0$. Let $x_2 := \inf \{ x \geq x_1: g(x) = 0 \}$. On $[x_1,x_2)$, we have that $f_{\mu_2}^2 - f_{\mu_1}^2 = g (f_{\mu_2} + f_{\mu_1}) \geq 0$ and, using the equation satisfied by $f_{\mu_1}$ and $f_{\mu_2}$, we deduce that  $p x^2 g''(x) + q x g'(x) + q g(x) \geq 0$ on $[x_1,x_2)$. With the same line of argument as in the proof of the fact that solutions stay positive, this implies that  $x^{b_2} g(x) \geq x_1^{b_2} g(x_1) > 0$ for all $x \in [x_1,x_2)$, where $b_2$ is defined in \eqref{E:roots}. This shows that $x_2 = +\infty$, establishing the desired monotonicity property.

\medskip

\noindent
\textbf{$f_\mu$ converges to zero pointwise as $\mu \to -\infty$}: We fix a small positive real number $\eps$ and let $\mu <0$ be very negative. Since $f_\mu$ is decreasing and has initial value $q$, the function $x \mapsto (q - f_\mu(x)) f_\mu(x)$ increases until $f_\mu$ reaches $q/2$ and then decreases. 
Let $x_1 = \inf \{ x \geq 0: (q - f_\mu(x)) f(x) \geq \eps \}$ and $x_2 = \inf \{ x > x_1: (q - f_\mu(x)) f(x) \leq \eps\}$. On $[x_2, \infty)$, $f_\mu$ is at most
\[
\frac{q}{2} \left( 1 - \sqrt{1 - 4\eps/q^2} \right) = \frac{\eps}{q} + O(\eps^2).
\]
It is therefore enough to show that for $\eps >0$ fixed, $x_2 \to 0$ as $\mu \to -\infty$.
Using that $f'' \geq 0$, we have on $[x_1,x_2]$,
\[
-f'(x) = \frac{p}{q} f''(x) + \frac{(q-f(x))f(x)}{x} \geq \frac{\eps}{x}.
\]
Integrating between $x_1$ and $x_2$ leads to $q \geq f(x_1) - f(x_2) \geq \eps \log(x_2/x_1)$, i.e. $x_2 \leq e^{q/\eps} x_1$.
Finally, using the explicit power series for $x < R_\alpha(p,q)/|\mu|$, one can check that $x_1$ is of order $1/|\mu|$. This concludes the proof that $x_2 \to 0$ as $\mu \to -\infty$.
\end{proof}

We are now ready to prove Proposition \ref{P:123}.

\begin{proof}[Proof of Proposition \ref{P:123}]
Let $d \in \{1,2,3\}$ and recall the definition \eqref{E:beta2} of $\beta$.
We apply Lemma \ref{L:d123} with $p = \beta^2$ and $q = 8-2d$. As already mentioned, the assumption that $q / p$ is at least $3 + 2 \sqrt{2}$ is satisfied for $d \in \{1,2,3\}$ with equality when $d=2$.
By Lemma \ref{L:d123}, the function $\mu \in (-\infty,0] \mapsto f_\mu(1) \in (0,8-2d]$ is a continuous decreasing bijection. Let us denote by $s \in (0,8-2d] \mapsto \mu_s \in (-\infty,0]$ its inverse. Let $s \in (0,8-2d]$ and define for all $r >0$, $w_s(r) = r^{-2} f_{\mu_s}(r^{-\beta})$. The fact that $f_{\mu_s}$ satisfies $\beta^2 x^2 f'' + (8-2d)(xf'+f) = f^2$ implies that
\[
w_s''(r) + \left( 5+\beta - (8-2d)/\beta \right) \frac{w_s'(r)}{r} + (4+2\beta +(8-2d)(1-2/\beta)) \frac{w_s(r)}{r^2} = w_s(r)^2, \quad \quad r>0.
\]
The quadratic equation satisfied by $\beta$ implies that the multiplicative constant in front of $w_s'$ equals $d-1$ and the one in front of $w_s$ vanishes, i.e. $w_s''(r) + (d-1) w_s'(r)/r = w_s(r)^2$ for all $r>0$. By definition of $\mu_s$, $w_s(1) = s$. Wrapping up, the function $y \in \R^d \setminus B(1) \mapsto \norme{y}^{-2} f_{\mu_s}(\norme{y}^{-\beta})$ is a nonnegative solution to the boundary value problem \eqref{E:P_123}. Uniqueness of nonnegative solutions (see Lemma \ref{L:appendix_uniqueness}) together with Lemma \ref{L:d123} concludes the proof.
\end{proof}

\addcontentsline{toc}{section}{References}

{\small
\bibliographystyle{alpha}
\bibliography{BBM.bib}

\begin{thebibliography}{DKRV16}

\bibitem[AB22]{abe2022exceptional}
Yoshihiro Abe and Marek Biskup.
\newblock Exceptional points of two-dimensional random walks at multiples of
  the cover time.
\newblock {\em Probab. Theory Relat. Fields}, 183(1-2):1--55, 2022.

\bibitem[ABJL23]{ABJL}
Élie Aïdékon, Nathanaël Berestycki, Antoine Jego, and Titus Lupu.
\newblock Multiplicative chaos of the {B}rownian loop soup.
\newblock {\em Proc. London Math. Soc.}, 126(4):1254--1393, 2023.

\bibitem[ABL19]{abe2019exceptional}
Yoshihiro Abe, Marek Biskup, and Sangchul Lee.
\newblock Exceptional points of discrete-time random walks in planar domains.
\newblock {\em arXiv preprint arXiv:1911.11810}, 2019.

\bibitem[ADC21]{ADC}
Michael Aizenman and Hugo Duminil-Copin.
\newblock Marginal triviality of the scaling limits of critical {4D} {I}sing
  and $\lambda \phi_4^4$ models.
\newblock {\em Ann. of Math.}, 194(1):163--235, 2021.

\bibitem[AHJ21]{angel2021tail}
Omer Angel, Tom Hutchcroft, and Antal J{\'a}rai.
\newblock On the tail of the branching random walk local time.
\newblock {\em Probab. Theory Relat. Fields}, 180(1-2):467--494, 2021.

\bibitem[AHS20]{AidekonHuShi2018}
Elie A\"{\i}d\'{e}kon, Yueyun Hu, and Zhan Shi.
\newblock Points of infinite multiplicity of planar {B}rownian motion: measures
  and local times.
\newblock {\em Ann. Probab.}, 48(4):1785--1825, 2020.

\bibitem[Ald93]{aldous1993continuum}
David Aldous.
\newblock The continuum random tree {III}.
\newblock {\em Ann. Probab.}, pages 248--289, 1993.

\bibitem[AS22]{asselah2022time}
Amine Asselah and Bruno Schapira.
\newblock Time spent in a ball by a critical branching random walk.
\newblock {\em arXiv preprint arXiv:2203.14737}, 2022.

\bibitem[BBK94]{BBK}
Richard~F. Bass, Krzysztof Burdzy, and Davar Khoshnevisan.
\newblock {Intersection Local Time for Points of Infinite Multiplicity}.
\newblock {\em Ann. Probab.}, 22(2):566 -- 625, 1994.

\bibitem[BDG01]{bolthausen2001}
Erwin Bolthausen, Jean-Dominique Deuschel, and Giambattista Giacomin.
\newblock Entropic repulsion and the maximum of the two-dimensional harmonic
  crystal.
\newblock {\em Ann. Probab.}, 29(4):1670--1692, 10 2001.

\bibitem[BDR21]{10.1214/21-EJP678}
Anna Brandenberger, Luc Devroye, and Tommy Reddad.
\newblock {The Horton–Strahler number of conditioned Galton–Watson trees}.
\newblock {\em Electron. J. Probab.}, 26:1 -- 29, 2021.

\bibitem[BDZ16]{BramsonDingZeitouni}
Maury Bramson, Jian Ding, and Ofer Zeitouni.
\newblock Convergence in law of the maximum of the two-dimensional discrete
  gaussian free field.
\newblock {\em Comm. Pure Appl. Math.}, 69(1):62--123, 2016.

\bibitem[Ber17]{BerestyckiGMC}
Nathana{\"e}l Berestycki.
\newblock {An elementary approach to Gaussian multiplicative chaos}.
\newblock {\em Electron. Comm. Probab.}, 22:1 -- 12, 2017.

\bibitem[BL16]{BiskupLouidor_Extreme}
Marek Biskup and Oren Louidor.
\newblock Extreme local extrema of two-dimensional discrete {G}aussian free
  field.
\newblock {\em Comm. Math. Phys.}, 345:271--304, 2016.

\bibitem[BL19]{BiskupLouidor_intermediate}
Marek Biskup and Oren Louidor.
\newblock {On intermediate level sets of two-dimensional discrete {G}aussian
  free field}.
\newblock {\em Ann. Inst. Henri Poincaré Probab. Stat.}, 55(4):1948 -- 1987,
  2019.

\bibitem[BLM13]{boucheron2013concentration}
St{\'e}phane Boucheron, G{\'a}bor Lugosi, and Pascal Massart.
\newblock {\em Concentration inequalities: {A} nonasymptotic theory of
  independence}.
\newblock Oxford university press, 2013.

\bibitem[BP]{BerestyckiPowell}
Nathana{\"e}l Berestycki and Ellen Powell.
\newblock {\em {G}aussian free field and {L}iouville quantum gravity}.
\newblock Cambridge University Press, to appear.

\bibitem[BR07]{BassRosen2007}
Richard Bass and Jay Rosen.
\newblock Frequent points for random walks in two dimensions.
\newblock {\em Electron. J. Probab.}, 12:1--46, 2007.

\bibitem[Dav06]{daviaud2006}
Olivier Daviaud.
\newblock Extremes of the discrete two-dimensional gaussian free field.
\newblock {\em Ann. Probab.}, 34(3):962--986, 05 2006.

\bibitem[DIP89]{dawson1989super}
Donald~Andrew Dawson, Ian Iscoe, and Edwin~A Perkins.
\newblock {Super-Brownian motion: path properties and hitting probabilities}.
\newblock {\em Probab. Theory Relat. Fields}, 83(1-2):135--205, 1989.

\bibitem[DKRV16]{DKRV}
Fran{\c{c}}ois David, Antti Kupiainen, R{\'e}mi Rhodes, and Vincent Vargas.
\newblock Liouville quantum gravity on the {R}iemann sphere.
\newblock {\em Comm. Math. Phys.}, 342:869--907, 2016.

\bibitem[DPRZ01]{DPRZ}
Amir Dembo, Yuval Peres, Jay Rosen, and Ofer Zeitouni.
\newblock {Thick points for planar Brownian motion and the Erdős-Taylor
  conjecture on random walk}.
\newblock {\em Acta Math.}, 186(2):239 -- 270, 2001.

\bibitem[DS11]{DuplantierSheffield}
Bertrand Duplantier and Scott Sheffield.
\newblock Liouville quantum gravity and {KPZ}.
\newblock {\em Invent. Math.}, 185(2):333--393, 2011.

\bibitem[Dur19]{Durrett}
Rick Durrett.
\newblock {\em Probability: theory and examples}, volume~49.
\newblock Cambridge university press, 2019.

\bibitem[Dyn91]{dynkin1991probabilistic}
EB~Dynkin.
\newblock A probabilistic approach to one class of nonlinear differential
  equations.
\newblock {\em Probab. Theory Relat. Fields}, 89(1):89--115, 1991.

\bibitem[Dyn04]{Dynkin}
Evgenii~Borisovich Dynkin.
\newblock Superdiffusions and positive solutions of non-linear partial
  differential equations.
\newblock {\em Russian Mathematical Surveys}, 59(1):147, 2004.

\bibitem[Ein89]{einmahl1989extensions}
Uwe Einmahl.
\newblock Extensions of results of {K}oml{\'o}s, {M}ajor, and {T}usn{\'a}dy to
  the multivariate case.
\newblock {\em J. Multivariate Anal.}, 28(1):20--68, 1989.

\bibitem[ET60]{erdos_taylor1960}
Paul Erd\H{o}s and Samuel~J. Taylor.
\newblock Some problems concerning the structure of random walk paths.
\newblock {\em Acta Math. Acad. Sci. Hungar.}, 11:137--162, 1960.

\bibitem[GJ12]{GlimmJaffe}
James Glimm and Arthur Jaffe.
\newblock {\em Quantum physics: a functional integral point of view}.
\newblock Springer Science \& Business Media, 2012.

\bibitem[GM06]{LeGall-Merle}
Jean-Francois~Le Gall and Mathieu Merle.
\newblock {On the occupation measure of super-Brownian motion}.
\newblock {\em Electron. Comm. Probab.}, 11:252 -- 265, 2006.

\bibitem[Har00]{hardy2000divergent}
Godfrey~Harold Hardy.
\newblock {\em Divergent series}, volume 334.
\newblock American Mathematical Soc., 2000.

\bibitem[HMP10]{HuMillerPeres2010}
Xiaoyu Hu, Jason Miller, and Yuval Peres.
\newblock Thick points of the {G}aussian free field.
\newblock {\em Ann. Probab.}, 38(2):896--926, 03 2010.

\bibitem[Hul15]{MR3418547}
Tim Hulshof.
\newblock The one-arm exponent for mean-field long-range percolation.
\newblock {\em Electron. J. Probab.}, 20:no. 115, 26, 2015.

\bibitem[Hut20]{MR4055195}
Tom Hutchcroft.
\newblock Universality of high-dimensional spanning forests and sandpiles.
\newblock {\em Probab. Theory Related Fields}, 176(1-2):533--597, 2020.

\bibitem[Jeg20a]{jegoBMC}
Antoine Jego.
\newblock Planar {B}rownian motion and {G}aussian multiplicative chaos.
\newblock {\em Ann. Probab.}, 48(4):1597--1643, 2020.

\bibitem[Jeg20b]{jego2020}
Antoine Jego.
\newblock Thick points of random walk and the {G}aussian free field.
\newblock {\em Electron. J. Probab.}, 25:39 pp., 2020.

\bibitem[Jeg21]{jegoBMCc}
Antoine Jego.
\newblock Critical {B}rownian multiplicative chaos.
\newblock {\em Probab. Theory Relat. Fields}, 180(1):495--552, 2021.

\bibitem[Jeg23]{jegoRW}
Antoine Jego.
\newblock Characterisation of planar {B}rownian multiplicative chaos.
\newblock {\em Comm. Math. Phys.}, 399(2):971--1019, 2023.

\bibitem[Kah85]{Kahane}
Jean-Pierre Kahane.
\newblock Sur le chaos multiplicatif.
\newblock {\em Ann. Sci. Math. Qu{\'e}bec}, 9(2):105--150, 1985.

\bibitem[KNS66]{kesten1966galton}
Harry Kesten, P~Ney, and F~Spitzer.
\newblock {The Galton-Watson process with mean one and finite variance}.
\newblock {\em Theory of Probability \& Its Applications}, 11(4):513--540,
  1966.

\bibitem[KRV20]{DOZZ}
Antti Kupiainen, R{\'e}mi Rhodes, and Vincent Vargas.
\newblock Integrability of {L}iouville theory: proof of the {DOZZ} formula.
\newblock {\em Ann. of Math.}, 191(1):81--166, 2020.

\bibitem[Law13]{Lawler_intersections}
Gregory~F Lawler.
\newblock {\em Intersections of random walks}.
\newblock Springer Science \& Business Media, 2013.

\bibitem[LG99]{le1999spatial}
Jean-Fran{\c{c}}ois Le~Gall.
\newblock {\em Spatial branching processes, random snakes and partial
  differential equations}.
\newblock Springer Science \& Business Media, 1999.

\bibitem[LG05]{LeGall_survey}
Jean-Fran{\c{c}}ois Le~Gall.
\newblock {Random trees and applications}.
\newblock {\em Probab. Surv.}, 2:245 -- 311, 2005.

\bibitem[LGL15]{10.1214/14-AOP947}
Jean-Fran{\c{c}}ois Le~Gall and Shen Lin.
\newblock {The range of tree-indexed random walk in low dimensions}.
\newblock {\em Ann. Probab.}, 43(5):2701 -- 2728, 2015.

\bibitem[LGL16]{LeGall16}
Jean-François Le~Gall and Shen Lin.
\newblock The range of tree-indexed random walk.
\newblock {\em J. Inst. Math. Jussieu}, 15(2):271–317, 2016.

\bibitem[LJ11]{LeJan}
Yves Le~Jan.
\newblock {\em Markov Paths, Loops and Fields: {\'E}cole D'{\'E}t{\'e} de
  Probabilit{\'e}s de Saint-Flour XXXVIII--2008}, volume 2026.
\newblock Springer Science \& Business Media, 2011.

\bibitem[LL10]{lawler2010random}
Gregory~F Lawler and Vlada Limic.
\newblock volume 123.
\newblock Cambridge University Press, 2010.

\bibitem[LW04]{LawlerWerner}
Gregory~F. Lawler and Wendelin Werner.
\newblock The {B}rownian loop soup.
\newblock {\em Probab. Theory Relat. Fields}, 128(4):565--588, 2004.

\bibitem[LZ15]{Lalley15}
Steven Lalley and Bowei Zheng.
\newblock {Critical branching Brownian motion with killing}.
\newblock {\em Electron. J. Probab.}, 20:1 -- 29, 2015.

\bibitem[MM03]{Marckert}
Jean-Fran{\c{c}}ois Marckert and Abdelkader Mokkadem.
\newblock {The depth first processes of Galton--Watson trees converge to the
  same Brownian excursion}.
\newblock {\em Ann. Probab.}, 31(3):1655 -- 1678, 2003.

\bibitem[MR06]{MarcusRosen}
Michael~B Marcus and Jay Rosen.
\newblock {\em {Markov processes, Gaussian processes, and local times}}.
\newblock Number 100. Cambridge University Press, 2006.

\bibitem[Mse04]{Mselati}
Beno\^{i}t Mselati.
\newblock Classification and probabilistic representation of the positive
  solutions of a semilinear elliptic equation.
\newblock {\em Mem. Amer. Math. Soc.}, 168(798):xvi+121, 2004.

\bibitem[Ros05]{rosen2006}
Jay Rosen.
\newblock {A random walk proof of the Erd\H{o}s-Taylor conjecture}.
\newblock {\em Perido. Math. Hungar.}, 50(1):223--245, Aug 2005.

\bibitem[RV14]{RhodesVargas_survey}
R{\'e}mi Rhodes and Vincent Vargas.
\newblock {Gaussian multiplicative chaos and applications: A review}.
\newblock {\em Probab. Surv.}, 11:315 -- 392, 2014.

\bibitem[Szn10]{sznitman}
Alain-Sol Sznitman.
\newblock Vacant set of random interlacements and percolation.
\newblock {\em Annals of mathematics}, pages 2039--2087, 2010.

\bibitem[Var17]{VargasNotes}
Vincent Vargas.
\newblock Lecture notes on {L}iouville theory and the {DOZZ} formula.
\newblock {\em arXiv preprint arXiv:1712.00829}, 2017.

\bibitem[Wat68]{watanabe1968limit}
Shinzo Watanabe.
\newblock A limit theorem of branching processes and continuous state branching
  processes.
\newblock {\em Journal of Mathematics of Kyoto University}, 8(1):141--167,
  1968.

\bibitem[WF72]{wilsonfisher}
Kenneth~G Wilson and Michael~E Fisher.
\newblock Critical exponents in 3.99 dimensions.
\newblock {\em Physical Review Letters}, 28(4):240, 1972.

\bibitem[Wil71]{Wilson}
Kenneth~G Wilson.
\newblock Renormalization group and critical phenomena. {I.} {R}enormalization
  group and the {K}adanoff scaling picture.
\newblock {\em Physical review B}, 4(9):3174, 1971.

\bibitem[Zhu17]{Zhu17}
Qingsan Zhu.
\newblock On the critical branching random walk {I}: Branching capacity and
  visiting probability, 2017.

\bibitem[Zhu21]{Zhu3}
Qingsan Zhu.
\newblock {On the critical branching random walk III: The critical dimension}.
\newblock {\em Ann. Inst. Henri Poincaré Probab. Stat.}, 57(1):73 -- 93, 2021.

\end{thebibliography}
}
\end{document}